\documentclass[10pt]{article}
\usepackage{fullpage}
\usepackage[utf8]{inputenc}

\RequirePackage{natbib}

\RequirePackage{amsthm,amsmath,amsfonts,amssymb}

\RequirePackage[colorlinks,citecolor=blue,urlcolor=blue,backref=page]{hyperref}

\RequirePackage{amsgen,amsmath,amstext,amsbsy,amsopn,amssymb,stmaryrd}
\RequirePackage{booktabs}
\RequirePackage{blkarray}
\RequirePackage{algorithm}
\RequirePackage{algpseudocode}
\RequirePackage{url}
\RequirePackage{longtable}
\RequirePackage{mathtools}
\RequirePackage{multirow}
\RequirePackage{comment}
\RequirePackage{enumitem}
\setlist{  
  listparindent=\parindent,
  parsep=0pt,
}
\RequirePackage[capitalize]{cleveref}
\RequirePackage[dvipsnames]{xcolor}
\RequirePackage{graphicx}
\RequirePackage{tikz}
\usetikzlibrary{fit,positioning,arrows,automata,calc}
\tikzset{
   main/.style={circle, minimum size = 10mm, thick, 
        draw =black!80, node distance = 10mm},
   box/.style={rectangle, draw=black!100}
}
\usetikzlibrary{positioning}
\RequirePackage[margin=1cm,font=footnotesize]{caption}
\RequirePackage{subcaption}

\newcommand{\I}{{\mathbf{I}}}

\newcommand{\E}{{\mathbf{E}}}

\newcommand{\U}{{\mathbf{U}}}

\newcommand{\Cov}{{\rm Cov}}

\newcommand{\var}{{\rm var}}

\newcommand{\polylog}{{\rm polylog}}

\newcommand{\SVD}{{\rm SVD}}

\newcommand{\argmin}{\mathop{\rm arg\min}}

\newcommand{\br}{\boldsymbol{r}}

\renewcommand{\citet}{\cite}
\renewcommand{\t}{^{\top}}
\newcommand{\utilde}{\mathbf{\tilde{U}}}
\newcommand{\uhat}{\mathbf{\hat{U}}}
\newtheorem{lemma}{Lemma}
\newtheorem{theorem}{Theorem}
\newtheorem{remark}{Remark}

\newtheorem{assumption}{Assumption}
\newcommand{\eps}{\varepsilon}
\newcommand{\p}{\mathbb{P}}
\newcommand{\inv}{^{-1}}
\newcommand{\ku}{^{(k)}}
\newcommand{\dl}{\delta_{\mathrm{L}}}

\renewcommand{\E}{\mathbb{E}}
\newcommand{\del}{\partial}
\renewcommand{\tilde}{\widetilde}
\renewcommand{\hat}{\widehat}

\newcommand{\pure}{^{(\mathrm{pure})}}
\newcommand\numberthis{\addtocounter{equation}{1}\tag{\theequation}}

\renewcommand{\hat}{\widehat}
 \makeatletter
 \newcommand*{\rom}[1]{\expandafter\@slowromancap\romannumeral #1@}
\makeatother 

\setcitestyle{round}
\hypersetup{%
  colorlinks=true,% hyperlinks will be black
  linkcolor=blue,
  citecolor = blue,
  urlcolor = blue%,% hyperlink borders will be red
}
\usepackage{natbib}
\bibpunct{(}{)}{;}{a}{,}{,}

\begin{document}

\title{%The $\ell_2\to\ell_\infty$ Tensor Perturbation Bound with Applications to Statistics and Machine Learning\\
Statistical Inference for Low-Rank Tensors: Heteroskedasticity, Subgaussianity, and Applications}
\author{Joshua Agterberg\thanks{Department of Statistics, University of Illinois Urbana-Champaign, Email: jagt@illinois.edu} ~ and ~ Anru R. Zhang\thanks{Departments of Biostatistics \& Bioinformatics and Computer Science, Duke University. Email: anru.zhang@duke.edu.}}

\date{(\today)}

%\maketitle

\maketitle
% \begin{frontmatter}

% % "Title of the paper"

% %\runtitle{Statistical Inference for Tensors}

% % indicate corresponding author with \corref{}
% % \author{\fnms{John} \snm{Smith}\corref{}\ead[label=e1]{smith@foo.com}\thanksref{t1}}
% % \thankstext{t1}{Thanks to somebody} 
% % \address{line 1\\ line 2\\ printead{e1}}
% % \affiliation{Some University}
% \begin{aug}
% %%%%%%%%%%%%%%%%%%%%%%%%%%%%%%%%%%%%%%%%%%%%%%%
% %% Only one address is permitted per author. %%
% %% Only division, organization and e-mail is %%
% %% included in the address.                  %%
% %% Additional information can be included in %%
% %% the Acknowledgments section if necessary. %%
% %% ORCID can be inserted by command:         %%
% %% \orcid{0000-0000-0000-0000}               %%
% %%%%%%%%%%%%%%%%%%%%%%%%%%%%%%%%%%%%%%%%%%%%%%%

% \end{aug}
%\runauthor{Agterberg and Zhang}

\begin{abstract}
In this paper, we consider inference and uncertainty quantification for low Tucker rank tensors with additive noise in the high-dimensional regime. Focusing on the output of the \emph{higher-order orthogonal iteration} (\texttt{HOOI}) algorithm,  a commonly used algorithm for tensor singular value decomposition, we establish non-asymptotic distributional theory and study how to construct confidence regions and intervals for both the estimated singular vectors and the tensor entries in the presence of heteroskedastic subgaussian noise, which are further shown to be optimal for homoskedastic  Gaussian noise. Furthermore, as a byproduct of our theoretical results, we establish the entrywise convergence of \texttt{HOOI} when initialized via diagonal deletion. To further illustrate the utility of our theoretical results, we then consider several concrete statistical inference tasks.  First, in the tensor mixed-membership blockmodel, we consider a two-sample test for equality of membership profiles, and we propose a test statistic with consistency under local alternatives that exhibits a power improvement relative to the corresponding matrix test considered in several previous works. Next, we consider simultaneous inference for small collections of entries of the tensor, and we obtain consistent confidence regions.   Finally, focusing on the particular case of testing whether entries of the tensor are equal, we propose a consistent test statistic that shows how index overlap results in different asymptotic standard deviations. All of our  proposed procedures are fully data-driven, adaptive to noise distribution and signal strength, and do not rely on sample-splitting, and our main results highlight the effect of higher-order structures on estimation relative to the matrix setting. Our theoretical results are demonstrated through numerical simulations.  
\end{abstract}

%\begin{keyword}[class=MSC]
%\kwd[Primary ]{}
%\kwd{}
%\kwd[; secondary ]{}
%\end{keyword}

%\begin{keyword}
%\kwd{}
%\kwd{}
%\end{keyword}
% \begin{keyword}[class=MSC]
% \kwd[Primary ]{62G15}
% \kwd[; secondary ]{62E7}
% \end{keyword}

% \begin{keyword}
% \kwd{Tensor Data Analysis}
% \kwd{Spectral Methods}
% \kwd{Singular Value Decomposition}
% \end{keyword}

% \end{frontmatter}

\tableofcontents
\section{Introduction}\label{sec:intro}
%%%%%%%%%%%%%%%%%%%%%%%%%%%%%%%%%
Higher-order data, or tensor data, appears frequently in statistics, machine learning, and data science, and has applications in medical imaging \citep{li_parsimonious_2017,zhang_tensor_2019}, network analysis \citep{jing_community_2021,lyu_latent_2023,lei_consistent_2020}, electron microscopy \citep{zhang_denoising_2020}, and microbiome studies \citep{martino_context-aware_2021,han2023guaranteed}, to name a few. With the rise of the ubiquity of high-dimensional tensor data statistics researchers have begun focusing on models exhibiting low-dimensional structures such as low-rankness, and theoretical results have been derived under various observation models and low-rank structures.  However, despite the wide array of estimation guarantees in the literature, there are relatively few procedures that can adequately quantify the uncertainty inherent in the resulting estimates in a principled manner.  In addition, existing works have primarily focused on settings with homoskedastic or Gaussian noise, and hence cannot handle the general setting. 

%or focus on estimation. Much less work on statistical inference.  More detailed discussion of related work in \cref{sec:relatedwork}.
These observations motivate the main question considered in this work:
\begin{center}
    \textit{Can we reliably perform principled statistical inference for low-rank tensors in the presence of heteroskedastic, subgaussian noise?}
\end{center}
This work answers this question in the affirmative. Unlike matrices, there is no canonical notion of tensor rank, so we deliberately focus our attention on tensors with low-rank \emph{Tucker decomposition}, and we study the higher-order orthogonal iteration (\texttt{HOOI}) algorithm, an algorithm that performs Tucker decomposition of a tensor. We provide a suite of inferential tools for the output of the \texttt{HOOI} algorithm in the high-dimensional regime where the dimensions of the tensor are large and comparable, and we use our theoretical results to obtain solutions to several motivating statistical problems of theoretical and practical interest. All of our results hold under reasonable signal strength conditions, and our proposed confidence intervals and regions are data-driven and adaptive to heteroskedastic noise.

%The rest of this section is organized as follows.  
%In the following subsections we describe several motivating inferential tasks, discuss our contributions, and provide an overview of notation and tensor algebra. In \cref{sec:assumptions} we provide our main estimation methodology and technical assumptions.  \cref{sec:loadings} and \cref{sec:entries} contain our main inferential tools and associated theoretical guarantees
 
%%%%%%%%%%%%%%%%%%%%%%%%%%%%%%%%%
\subsection{Motivating Inference Tasks}
\label{sec:introapplications}
%%%%%%%%%%%%%%%%%%%%%%%%%%%%%%%%%
To further motivate the primary problem considered in this work, we consider several concrete inferential tasks of interest. 
\begin{itemize}
    \item \textbf{Testing membership profiles}.  In the \emph{tensor blockmodel} \citep{wu_general_2016}, or the \emph{tensor mixed-membership blockmodel} \citep{agterberg_estimating_2022}, nodes along different modes have community memberships associated with them, where the communities may be discrete or continuous (corresponding to the blockmodel and mixed-membership blockmodel setting respectively).   The estimation of community memberships in the blockmodel setting has been considered in a number of different works \citep{han_exact_2020,chi_provable_2020,wang_multiway_2019}, and in the mixed-membership blockmodel setting in \citet{agterberg_estimating_2022}. However, previous works have not considered the problem of \emph{testing memberships}.  Explicitly, given two nodes of interest along a fixed mode, can one test whether their memberships are the same given only their higher-order interactions?  This problem has been considered in the matrix (network) setting in \citet{fan_simple_2022} and \citet{du_hypothesis_2022}. %, and in this work we propose a test statistic with principled statistical guarantees for the tensor data analysis setting.  
\item  \textbf{Simultaneous confidence intervals}. In many situations one is not merely interested in individual entries of tensors, but rather collections of entries.  Unfortunately, depending on the collection there may be correlation between entries, particularly if the entries are all localized to a particular region of the tensor.  For example, in MRI data, collections of entries can correspond to tumor growth, with larger values indicating the possibility of a tumor.  Therefore, an important problem is constructing principled confidence intervals that are  \emph{simultaneously valid} for all entries in a collection.
\item  \textbf{Testing  entry equality}.  Beyond obtaining confidence intervals, in some settings, one may be interested in testing whether two entries are equal. For example, in time series of networks, one may be interested in testing whether the particular probability of an edge is equal between two distinct times.  Therefore, a practical but interesting theoretical problem is to design and analyze test statistics for testing this  hypothesis.  
%In this work we leverage our statistical theory to design and theoretically justify a test statistic 
\end{itemize}

%%%%%%%%%%%%%%%%%%%%%%%%%%%%%%%%%
\subsection{Our Contributions}
%%%%%%%%%%%%%%%%%%%%%%%%%%%%%%%%%
In light of our main question and the three applications in \cref{sec:introapplications}, the contributions of this paper are as follows.
\begin{itemize}
    \item \textbf{Singular vector distributional theory and inference}. We establish nonasymptotic distributional theory (\cref{thm:eigenvectornormality_v1} and \cref{thm:eigenvectornormality2_v1}) and confidence regions (\cref{thm:civalidity2_v1}) for the estimated tensor singular vectors obtained from the \texttt{HOOI} algorithm in the presence of heteroskedastic  subgaussian noise.  Our proposed confidence regions are data-driven, adaptive to heteroskedasticity and signal strength, and optimal for homoskedastic Gaussian noise (\cref{thm:efficiencyloadings,thm:efficiency_order}).  
    \item \textbf{Entrywise distributional theory, inference, and consistency}.  We establish nonasymptotic distributional theory (\cref{thm:asymptoticnormalityentries_v1}) and confidence intervals (\cref{thm:civalidity_v1}) for individual entries of the underlying tensor.  Again our proposed confidence intervals are   data-driven, adaptive to heteroskedasticity, and optimal for homoskedastic Gaussian noise (\cref{thm:efficiency,thm:efficiency_order_ijk}).   As a byproduct of our main results, we also establish the entrywise convergence of \texttt{HOOI} (\cref{cor:maxnormbound_v1}).  
    \item \textbf{Membership profile testing in tensor mixed-membership blockmodels}. We apply our results to  testing membership profiles in the tensor mixed-membership blockmodel.  We leverage our theory for the tensor singular vectors to study a test statistic for this hypothesis, and we show that our test statistic is consistent under the null as well as local alternatives (\cref{cor:testing}).  Our results show that this test statistic exhibits a power gain relative to the corresponding matrix test.  
    \item \textbf{Simultaneous confidence intervals}.  We study the problem of obtaining simultaneous confidence intervals for small collections of entries of the underlying tensor. Under reasonable assumptions on signal strengths and the size of the collection, we establish the consistency of our proposed confidence regions (\cref{thm:simultaneousinference_v1}).
    \item \textbf{Hypothesis tests for entries}.  We consider testing whether two tensor entries are equal, and we establish consistency for our proposed test procedure (\cref{thm:entrytesting_v1}).  Our results demonstrate  how the test depends on the overlap of the indices of the tensor; in particular, demonstrating that tensor entries that are ``further away'' are easier to test than those that are ``close.''
    \end{itemize}
All of our results hold under nearly-optimal signal to noise ratio conditions such that a polynomial-time estimator exists.  Throughout we compare our results to the matrix setting, highlighting fundamental differences between Tensor SVD and Matrix SVD.  For ease of presentation, we deliberately restrict our attention to order three tensors, though the ideas carry through straightforwardly to the higher-order setting.  

%%%%%%%%%%%%%%%%%%%%%%%%%%
\subsection{Paper Organization}
The rest of this paper is organized as follows.  In the following subsection, we set notation and provide background on tensor algebra that we will be using throughout this work.  In \cref{sec:assumptions}
we describe the \texttt{HOOI} algorithm and our model in detail.  In \cref{sec:loadings} and \cref{sec:entries} we study distributional theory and statistical inference for the estimated tensor singular vectors and entries respectively.  In \cref{sec:consequences} we study how to apply our theory to the motivating problems discussed in \cref{sec:introapplications}, and in \cref{sec:relatedwork} we discuss related work.  In \cref{sec:simulations} we present numerical simulations, and in \cref{sec:discussion} we include discussion.  We include several more general results as well as proof details in the appendices.

%Our analysis and more general results are in \cref{sec:analysis}, and our main proofs as well as more general theorem statements are deferred to the appendices. 

%%%%%%%%%%%%%%%%%%%%%%%%%%%%%%%%%
\subsection{Notation and Background on Tensor Algebra}
\label{sec:notation}
%%%%%%%%%%%%%%%%%%%%%%%%%%%%%%%%%
First, for two sequences of numbers $a_n$ and $b_n$, we say $a_n \lesssim b_n$ if there is some universal constant $C> 0$ such that $a_n \leq C b_n$, and we write $a_n \asymp b_n$ if $a_n \lesssim b_n$ and $b_n \lesssim a_n$. We also write $a_n = O(b_n)$ if $a_n \lesssim b_n$. In addition, we write $a_n \ll b_n$ if $a_n/b_n \to 0$ as $n \to \infty$, and we write $a_n = o(b_n)$ to mean $a_n \ll b_n$.  We write $a_n = \tilde O(b_n)$ if there is some $c > 0$ such that $a_n = O(b_n \log^c(n))$, and we write $a_n = \tilde o(b_n)$ if there is some $c>0$ such that $a_n = o(b_n \log^c(n))$.  For a random variable $Z$ we denote $\|Z\|_{\psi_2}$ as its subgaussian Orlicz norm defined via $\|Z\|_{\psi_2} := \inf\{t >0: \mathbb{E}(X^2/t^2) \leq 2 \}$ (see Chapter 2 of \citet{vershynin_high-dimensional_2018} for details). 

Next, we use bold letters to denote matrices, and for a matrix $\mathbf{M}$ we let $\mathbf{M}_{i\cdot}$ and $\mathbf{M}_{\cdot j}$ denote its $i$'th row and $j$'th column respectively, where we view both as column vectors.  We let $\mathbf{M}\t$ denote the transpose of a matrix, and we set $\| \mathbf{M}\|_F$ as the Frobenius norm on a matrix.  We use $\|\cdot\|$ to denote the Euclidean norm or matrix spectral norm of vectors and matrices respectively, and we let $\|\mathbf{M}\|_{2,\infty}$ be the $\ell_{2,\infty}$ norm of a matrix, defined as $\|\mathbf{M}\|_{2,\infty} = \max_{i} \|\mathbf{M}_{i\cdot} \|$.  We let $e_k$ denote the standard basis vector in the appropriate dimension, and we denote the identity as $\mathbf{I}$ or $\mathbf{I}_k$, the latter where the dimension is specified for clarity.  For a matrix $\mathbf{U}$ with orthonormal columns  satisfying $\mathbf{U}\t \mathbf{U} = \mathbf{I}_k$ we let $\mathbf{U}_{\perp}$ denote its orthogonal complement; i.e., the (non-unique) matrix with orthonormal columns satisfying $\mathbf{U}_{\perp}\t \mathbf{U} = 0$.  For two matrices $\U_1$ and $\U_2$ of same dimensions with orthonormal columns, we let $\|\sin\Theta(\U_1,\U_2)\|$ denote their (spectral) $\sin\Theta$ distance defined as $\|\sin\Theta(\U_1,\U_2)\| = \| (\mathbf{\U}_1)_{\perp}\t \U_2 \|$. For a matrix $\U$ with orthonormal columns, we write $\mathcal{P}_{\U}$ as the projection onto the subspace spanned by  $\U$. For a matrix $\mathbf{M}$, we let $\mathrm{SVD}_r(\mathbf{M})$ denote the leading $r$ left singular vectors of $\mathbf{M}$.  For a square matrix $\mathbf{M}$ with singular value decomposition $\mathbf{M} = \mathbf{U \Sigma V}\t$, we write $\mathrm{sgn}(\mathbf{M})$ to denote the \emph{matrix sign function} of $\mathbf{M}$, defined via
\begin{align}
    \mathrm{sgn}(\mathbf{M}) \coloneqq \mathbf{UV}\t. \label{sgnfunction}
\end{align}

A tensor $\mathcal{T}$ is a multidimensional array, and we use the calligraphic letters for tensors, except for $\mathcal{M}$ (defined momentarily), $\mathcal{P}$ (for projections), and $\mathcal{E}$ (for probabilistic events).  We write $\mathcal{M}_k(\mathcal{T})$ as the \emph{matricization} of $\mathcal{T}$ along its $k$'th mode, so that $\mathcal{M}_k(\mathcal{T})$ satisfies
\begin{align*}
    \big(\mathcal{M}_k(\mathcal{T}) \big)_{i_kj} = \mathcal{T}_{i_1i_2i_3}; \qquad j = 1 + \sum_{l=1,l\neq k}^{3} \bigg\{ (i_l-1) \prod_{m=1,m\neq k}^3 p_m \bigg\},
\end{align*}
where $p_k$ are the dimensions of the tensor.  We write $\mathrm{Vec}(\mathcal{T})$ to denote the vectorization of the tensor, organized according to the lexicographic ordering.  We write $p_{-k} = \prod_{m\neq k} p_k$.   We let $\br = (r_1,r_2,r_3)$ denote the multilinear rank of a tensor, where $r_k$ denotes the rank of the $k$'th matricization of $\mathcal{T}$.  We let $r_{-k}$ denote $\prod_{m\neq k} r_m$, and when referring to tensor modes we use the convention that each mode is understood mod three (i.e., $\mathcal{M}_{k+3}(\cdot) = \mathcal{M}_k(\cdot)$).  The mode-one  product of a tensor $\mathcal{T}\in\mathbb{R}^{p_1 \times p_2 \times p_3}$ with a matrix $\U\in \mathbb{R}^{p_1 \times r_1}$ is denoted $\mathcal{T} \times_1 \U\t \in \mathbb{R}^{r_1 \times p_2 \times p_3}$ as is defined by
\begin{align*}
    \big( \mathcal{T} \times_1 \U\t \big)_{ji_2i_3} &= \sum_{i_1=1}^{p_k} \mathcal{T}_{i_1i_2i_3} \U_{i_1j},
\end{align*}
with other mode-wise products defined similarly. See \citet{kolda_tensor_2009} for more details on tensor matricizations and tensor ranks.

We say a tensor $\mathcal{T} \in \mathbb{R}^{p_1 \times p_2 \times p_3}$ has Tucker decomposition of rank  $\br$ if
\begin{align*}
    \mathcal{T} = \mathcal{C} \times_1 \U_1 \times_2 \U_2 \times_3 \U_3,
\end{align*}
where $\mathcal{C}\in \mathbb{R}^{r_1 \times r_2 \times r_3}$ is the core tensor and $\U_k \in \mathbb{R}^{p_k\times r_k}$ satisfy  $\U_k =\mathrm{SVD}_{r_k}\big(\mathcal{M}_k(\mathcal{T})\big)$.   We let $\lambda_{\min}(\mathcal{T})$ denote the smallest nonzero singular value along each  matricization of $\mathcal{T}$, and we let $\kappa$ denote the condition number defined via
\begin{align*}
    \kappa := \max_k \frac{\| \mathcal{M}_k(\mathcal{T})\|}{\lambda_{r_k}(\mathcal{M}_k(\mathcal{T}))}.
\end{align*}
The \emph{incoherence parameter} of $\mathcal{T}$ with Tucker decomposition $\mathcal{T} = \mathcal{C}\times_1 \U_1\times_2 \U_2 \times_3 \U_3$ is given by the smallest number $\mu_0$ such that
\begin{align*}
    \max_k \sqrt{\frac{p_k}{r_k}} \| \U_k \|_{2,\infty} \leq \mu_0.
\end{align*}

%%%%%%%%%%%%%%%%%%%%%%%%%%%%%%%%%
\section{Methodology and Model}
\label{sec:assumptions}
%%%%%%%%%%%%%%%%%%%%%%%%%%%%%%%%%%%
In this section we describe our main methodology, the higher-order orthogonal iteration (\texttt{HOOI})  algorithm, and perhaps the most ubiquitous tensor SVD algorithm for Tucker low-rank tensors.  Roughly speaking, \texttt{HOOI} is the analog of power iteration for tensors, and a number of previous works have studied this algorithm due to its practical implementation and historical significance \citep{de_lathauwer_best_2000,zhang_tensor_2018,luo_sharp_2021,agterberg_estimating_2022}.

Given  a tensor $\mathcal{\tilde{T}}$ and initializations $\{\uhat_k^{(0)}\}_{k=1}^{2}$, the \texttt{HOOI} algorithm iteratively updates each subsequent iteration by first projecting the tensor onto the subspaces corresponding to the first two modes and then extracting the singular vectors of the reduced tensor from this new tensor. A convenient intuitive representation of \texttt{HOOI} was described in \citet{xia_inference_2022}: given previous iterates $\uhat_k^{(t-1)}$, \texttt{HOOI} attempts to solve the problem 
\begin{align*}
    \max_{\U_1} \| \mathcal{\tilde T} \times_1 \U_1 \times_2 \uhat_2^{(t-1)} \times_3 \uhat_3^{(t-1)} \|_F^2,
\end{align*}
with similar updates for the other modes.  The maximum above is achievable by the Ekhart-Young Theorem via the SVD; intuitively, given the previous   iterates, \texttt{HOOI} updates by keeping the other modes fixed and finding the subspace such that the projection onto that subspace is  maximized.  The formal procedure is described in \cref{al:tensor-power-iteration}.

\begin{algorithm}[t]
	\caption{Higher-Order Orthogonal Iteration (\texttt{HOOI})}
	\begin{algorithmic}[1]
		\State Input: $\mathcal{\tilde{T}}\in \mathbb{R}^{p_1 \times p_2 \times p_3}$, Tucker rank $\br = (r_1, r_2,r_3)$, initialization $\uhat_2^{(0)}, \uhat_3^{(0)}$.  
		\Repeat
		\State Let $t = t+1$
		\For {$k=1,2,3$}
		\begin{equation*}
		\uhat_k^{(t)} = 
		\SVD_{r_k}\left(\mathcal{M}_k\left(\mathcal{\widehat{T}}\times_{k' < k} (\uhat_{k'}^{(t)})^\top \times_{k'>k} (\uhat_{k'}^{(t-1)})^\top\right)\right).
		\end{equation*}
		\EndFor
		\Until Convergence or the maximum number of iterations is reached.
		\State Set 
		\begin{align*}
		    \mathcal{\hat T} &\coloneqq \mathcal{\tilde T}  \times_1 \uhat_1^{(t_{\max})} \times_2 \uhat_2^{(t_{\max})} \times_3 \uhat_3^{(t_{\max})}
		\end{align*}
		\State Output: $\uhat_k^{(t_{\max})}$, estimated  tensor $\mathcal{\hat T}$
	\end{algorithmic}\label{al:tensor-power-iteration}
\end{algorithm}

\begin{algorithm}[t]
	\caption{Diagonal-Deletion Initialization}
	\begin{algorithmic}[1]
		\State Input: $\mathcal{\widehat{T}}\in \mathbb{R}^{p_1 \times p_2 \times p_3}$, Tucker rank $\br = (r_1, r_2,r_3)$.
		\For {$k=2,3$}
		\State Set $\uhat_k^{(0)}$ as the leading $r_k$ eigenvectors of the matrix $\mathbf{\widehat{G}}_k$, with
		\begin{align*}
		\mathbf{\widehat{G}}_k \coloneqq 
	 \Gamma\big(\mathcal{M}_k ( \mathcal{\widehat{T}}) \mathcal{M}_k(\mathcal{\widehat{T}})\t \big), \qquad &\text{where $\Gamma(\cdot)$ is the \emph{hollowing operator} that sets} \\
  \qquad &\text{the diagonal of ``$\cdot$" to zero;}
		\end{align*}
		\EndFor
		\State Output: $\uhat_k^{(0)}$.
	\end{algorithmic}\label{al:dd}
\end{algorithm}

\subsection{Initialization via Diagonal Deletion}
The \texttt{HOOI} algorithm requires a suitably warm initialization.  One common procedure for initialization is via the Higher-Order SVD (HOSVD) procedure, which uses the leading $r_k$ singular vectors from each matricization of $\mathcal{\tilde{T}}$.  This procedure was analyzed in \citet{zhang_tensor_2018} and shown to yield a strong initialization for homoskedastic Gaussian noise, though, for heteroskedastic noise, it has been demonstrated in \citet{zhang_heteroskedastic_2022} to be biased, which may  not result in a sufficiently close initialization.  

 To understand  this bias, consider a matrix $\mathbf{M}$ corrupted by a noise matrix $\mathbf{Z}$.  Then the singular vectors of $\mathbf{M} + \mathbf{Z}$ are equivalent to the eigenvectors of the matrix $\mathbf{MM}\t + \mathbf{MZ}\t + \mathbf{ZM}\t + \mathbf{ZZ}\t$.  When the noise is homoskedastic, the matrix $\mathbb{E}\mathbf{ZZ}\t$ is a scalar multiple of the identity, and hence the singular vectors of $\mathbf{M + Z}$ may well approximate those of $\mathbf{M}$ as eigenvectors are invariant to adding scalar multiples of the identity.  However, when $\mathbf{Z}$ consists of heteroskedastic noise, the matrix $\mathbb{E} \mathbf{ZZ}\t$ is a diagonal matrix with unequal entries, and hence the singular vectors of $\mathbf{M + Z}$ may not approximate those of $\mathbf{M}$ unless the heteroskedasticity is small (i.e., $\mathbb{E}\mathbf{ZZ}\t$ is ``close'' to a scalar multiple of the identity). % diagonal matrix

To combat this bias we consider initialization via the diagonal deletion algorithm, which provides initialization singular vectors via the eigenvectors of the  \emph{hollowed} Gram matrix  (i.e., setting the diagonal of the Gram matrix $(\mathbf{M + Z})(\mathbf{M+Z)\t}$ to zero).  This procedure has previously been considered in the literature as both an initialization \citep{agterberg_estimating_2022,wang_implicit_2021}, and as an algorithm in its own right \citep{cai_subspace_2021}.  The full initialization procedure is described in \cref{al:dd}.  

\textcolor{black}{\begin{remark}[Estimation of $\br$] Throughout this paper, we assume that $\br = (r_1,r_2,r_3)$ is known, though in practice it needs to be estimated.  Even in the matrix setting, rank estimation is known to be difficult, and there are many different procedures to estimate the rank for dimensionality reduction tailored to different statistical models (e.g., \citet{jin_optimal_2023,han_universal_2023}).  As our theory demonstrates, the matrix $\mathbf{\widehat{G}}_k$ is approximately rank $r_k$, and, under our assumptions, will have $r_k$ ``large'' eigenvalues and $p_k - r_k$ ``small'' eigenvalues.  Therefore, one principled approach to obtain a rank estimate is to look for an elbow in the scree plot of the hollowed Gram matrix $\mathbf{\widehat{G}}_k$ from \cref{al:dd}; see, for example, \citet{zhu_automatic_2006}.  However, in principle, any rank estimation procedure for matrices can be applied to $\mathbf{\widehat{G}}_k$.  As this problem is worthwhile in its own right, we leave a more detailed explanation to future work. 
\end{remark}}

\subsection{Model and Technical Assumptions}
We now elucidate our main model, the tensor signal-plus-noise model (also referred to as the ``tensor denoising'' or ``tensor PCA'' model).  We assume that we observe
\begin{align*}
    \widetilde{\mathcal{T}} = \mathcal{T} + \mathcal{Z},
\end{align*}
where $\mathcal{T}, \mathcal{Z} \in \mathbb{R}^{p_1\times p_2\times p_3}$ are order three tensors, $\mathcal{T}$ has Tucker rank $\br = (r_1,r_2,r_3)$ of the form
\begin{align*}
    \mathcal{T} = \mathcal{C} \times_1 \U_1 \times_2 \U_2 \times_3 \U_3;
\end{align*}
(see \cref{sec:notation}), and $\mathcal{Z}$ consists of independent, heteroskedastic noise. Explicitly, we make the following assumption on the noise tensor $\mathcal{Z}$.

\begin{assumption}[Noise] \label{ass:noise}
The noise $\mathcal{Z}$ consists of independent mean-zero entries $\mathcal{Z}_{ijk}$ with $\|\mathcal{Z}_{ijk} \|_{\psi_2} \leq \sigma$, and $\mathrm{var}(\mathcal{Z}_{ijk}) = \sigma^2_{ijk}$ with $\sigma_{\min} \leq \sigma_{ijk} \leq \sigma$.  Finally, it holds that $\sigma \leq C \sigma_{\min} $.  \end{assumption}

In addition, throughout this work, we operate in the following ``quasi-asymptotic'' regime.

\begin{assumption}[Regime]
It holds that $p_{k} \asymp p$ and $r_k \asymp r$ for all $k$.%  (maybe assumption about regime goes here, and then repeated in more generality in the appendix)
\end{assumption}

While this second assumption is not strictly necessary for some of our results to hold, it renders a number of calculations much more straightforward.  Furthermore, we choose to focus on this regime as it highlights a number of fundamental differences from the matrix setting that we describe after stating our main results.

All of our results (both our main results and more general theorem statements in the appendix) will hold under these two assumptions.  Some of our results also do not require the assumption $\sigma \lesssim \sigma_{\min}$, but we make clear when this is the case.  In addition, it may be possible to extend our results to a broader regime, where $\sigma_{\min} \ll \sigma$, but this regime is beyond the scope of this paper.

Our results will be stated under general \emph{signal-to-noise ratio} (SNR) assumptions.   Define the signal-strength parameter:
\begin{align*}
    \lambda &\coloneqq \lambda_{\min}(\mathcal{T});
\end{align*}
that is, $\lambda$ is the smallest nonzero singular value of each matricization of the signal tensor $\mathcal{T}$.  It has been demonstrated in \citet{zhang_tensor_2018} that Tensor SVD suffers from a so-called statistical and computational \emph{gap}: the statistical lower bound requires $\lambda/\sigma \gtrsim \sqrt{p}$ for minimax optimal subspace estimation, whereas the condition $\lambda/\sigma \gtrsim p^{3/4}$ is required for a polynomial-time estimator to exist (under a complexity conjecture from computer science), and, moreover, \texttt{HOOI} achieves the minimax rate in this regime.  Therefore, in this work, we will focus on this latter regime $(\lambda/\sigma \gtrsim p^{3/4})$, as we emphasize data-driven and practical uncertainty quantification, which is not achievable computationally if the signal-to-noise ratio is below this level.  \textcolor{black}{In addition, we will assume throughout this work that $\lambda/\sigma \leq \exp(cp)$ for some small constant $c>0$.  Such a condition is only for technical purposes, as it guarantees that we have moderate noise.  Indeed, the regime of interest is $p^{3/4} \lesssim \lambda/\sigma \lesssim p$, as once $\lambda/\sigma \gtrsim p$, then no additional tensor power iterations are required to achieve the minimax rate (see, for example, \cref{remark:snr}).}

\section{Singular Vector Distributional Theory and Inference}
\label{sec:loadings}
%%%%%%%%%%%%%%%%%%%%%%%%
 In this section, we focus on studying the estimated singular vectors $\uhat_k^{(t)}$ from \cref{al:tensor-power-iteration} after sufficiently many iterations.   For convenience throughout all of our main results, we will assume that the condition number $\kappa$ and incoherence parameter $\mu_0$ are bounded, but more general results are available in \cref{sec:generaltheorems}.  

%In the subsequent theorems we will assume that $\kappa, \mu_0 = O(1)$l though more general theorems are available in appendix.

%We now turn our attention to distributional guarantees for the loading matrices $\U_k$.  

The theoretical results for the estimated singular vectors will be stated up to a rotational ambiguity $\mathbf{W}_k^{(t)}$, which is in general necessary as we do not make any assumptions on the multiplicity of the tensor singular values. Explicitly, we define $\mathbf{W}_k^{(t)}$ as the orthogonal matrix satisfying
\begin{align*}
    \mathbf{W}_k^{(t)} &\coloneqq \argmin_{\mathbf{WW}\t = \mathbf{I}_{r_k}} \| \uhat_k^{(t)} - \U_k \mathbf{W} \|_F.
\end{align*}
The matrix $\mathbf{W}_k^{(t)}$ satisfies $\mathbf{W}_k^{(t)} = \mathrm{sgn}( \U_k\t \uhat_k^{(t)} )$, where $\mathrm{sgn}(\cdot)$ is the matrix sign function defined in \eqref{sgnfunction}.

The following result establishes a first-order approximation of $\uhat_k^{(t)}$ to $\U_k$ up to the  orthogonal transformation $\mathbf{W}_k^{(t)}$.  The more general result with $\kappa, \mu_0$ permitted to grow can be found in \cref{thm:eigenvectornormality} in \cref{sec:generaltheorems}. 

\begin{theorem}[First-Order Expansion for Tensor Singular Vectors] \label{thm:eigenvectornormality_v1}
Suppose that $r \lesssim p^{1/2} $and $\lambda/\sigma \gtrsim p^{3/4} \sqrt{\log(p)}$, \textcolor{black}{and that $\lambda/\sigma \leq \exp(c p)$ for some small constant $c$}. Suppose that $\kappa,\mu_0 = O(1)$, and let $\uhat_k^{(t)}$ denote the estimated singular vectors from the output of \texttt{HOOI} (\cref{al:tensor-power-iteration}) with $t \asymp \textcolor{black}{\log( \frac{\lambda/\sigma}{C \sqrt{p\log(p)}})} $%\log(p/(\lambda/\sigma))$ 
iterations, initialized via \cref{al:dd}.  Suppose $\mathbf{T}_k = \mathcal{M}_k(\mathcal{T})$ has rank $r_k$ singular value decomposition $\U_k \mathbf{\Lambda}_k \mathbf{V}_k\t$.  Denote $\mathbf{Z}_k = \mathcal{M}_k(\mathcal{Z})$.  Then there exists an event  $\mathcal{E}_{\mathrm{\cref{thm:eigenvectornormality_v1}}}$ with $\p(\mathcal{E}_{\mathrm{\cref{thm:eigenvectornormality_v1}}}) \geq 1 - O(p^{-9})$ such that on this event for each $k$ it holds that
\begin{align*}
    \uhat_k^{(t)} (\mathbf{W}_k^{(t)})\t - \U_k &=\mathbf{Z}_k \mathbf{V}_k \mathbf{\Lambda}_k\inv + \mathbf{\Psi}^{(k)},
\end{align*}
where
\begin{align*}
    \| \mathbf{\Psi}^{(k)} \|_{2,\infty} \lesssim \frac{\sigma^2 \log(p) r \sqrt{p}}{\lambda^2} + \frac{\sigma r}{\lambda \sqrt{p}}.
\end{align*}
\end{theorem}

  \begin{center}
 \begin{table*} 
\begin{tabular}{|c|c|c|c|}
\hline
Work & Setting & SNR & Leading-Order Term \\
\hline 
\citet{chen_spectral_2021} & \multirow{2}{6em}{$p \times p$ Symmetric Matrix} & $\lambda/\sigma \gtrsim \sqrt{p}$ & $\mathbf{Z} \mathbf{U} \mathbf{\Lambda}\inv$ \\ &&& \\ &&& \\
\hline
\citet{agterberg_entrywise_2022} & \multirow{2}{6em}{$p_1 \times p_2$ rectangular matrix} & \multirow{2}{6em}{$\lambda/\sigma \gtrsim \sqrt{p_{\max}}$} & $\mathbf{Z} \mathbf{V} \mathbf{\Lambda}\inv$ \\ &&& \\ &&& \\
\hline
\citet{yan_inference_2021} & \multirow{2}{6em}{$p_1 \times p_2$ rectangular matrix} & \multirow{2}{8em}{$\lambda/\sigma \gtrsim (p_1 p_2)^{1/4}$} & $\mathbf{Z} \mathbf{V} \mathbf{\Lambda}\inv + \mathcal{P}_{\mathrm{Off-Diag}}( \mathbf{Z Z}\t ) \U \mathbf{\Lambda}^{-2}$ \\ &&& \\ &&& \\
\hline
\emph{This work} & \multirow{2}{6em}{$p_1 \times p_2 \times p_3$ Tensor with $p_k \asymp p$} & \multirow{2}{9em}{$\lambda/\sigma \gtrsim (p_1 p_2 p_3)^{1/4}$ } & $\mathbf{Z} \mathbf{V} \mathbf{\Lambda}\inv$ \\ &&& \\ &&& \\\hline
\end{tabular}
\caption{Leading-order terms for estimated singular vectors and eigenvectors under various SNR regimes, ignoring logarithmic terms, factors of $\kappa,\mu_0$, and $r$.  \label{tbl:eigenvectorcomparison}}
\end{table*}
\end{center}
 
 \cref{thm:eigenvectornormality_v1} continues to hold without assuming $\sigma \lesssim \sigma_{\min}$. In essence, \cref{thm:eigenvectornormality_v1} showcases that one has a leading-order expansion for the rows of $\uhat_k$ (modulo an orthogonal transformation) under the nearly optimal SNR condition $\lambda/\sigma \gtrsim p^{3/4} \sqrt{\log(p)}$.  The key feature of \cref{thm:eigenvectornormality_v1} is that this leading-order expansion is \emph{linear} in the corresponding matricization of the noise tensor $\mathcal{Z}$. This result forms the foundation of our  analysis, including suggesting the form of the asymptotic distribution of the rows of $\uhat_k^{(t)}$ that we will see in subsequent results. %Note that \cref{thm:eigenvectornormality_v1} already allows the rank $r$ to grow with $p$.  
 
  \begin{remark}[Comparison to Previous Tensor Perturbation Bounds]
 It has been demonstrated in \citet{zhang_tensor_2018} that when $\lambda/\sigma \gtrsim p^{3/4}$, \texttt{HOOI} achieves the error rate
 \begin{align*}
     \| \uhat_k^{(t)}  - \U_k \mathbf{W}_k^{(t)} \|_F &\lesssim \frac{\sqrt{r_k p_k}}{\lambda/\sigma} =  O\bigg( \frac{\sqrt{p_k}}{\lambda/\sigma}\bigg)
 \end{align*}
 provided $r = O(1)$.   More recently, it was demonstrated in \citet{agterberg_estimating_2022} that when $\lambda/\sigma \gtrsim p^{3/4} \sqrt{\log(p)}$ that
 \begin{align*}
     \max_{1\leq m\leq p_k} \| \big(\uhat_k^{(t)}  - \U_k \mathbf{W}_k^{(t)} \big)_{m\cdot} \| &\lesssim \frac{\sqrt{r_k \log(p)}}{\lambda/\sigma} = \tilde O\bigg( \frac{1}{\lambda/\sigma} \bigg)
 \end{align*}
 assuming that $\kappa,\mu_0, r = O(1)$. In contrast, \cref{thm:eigenvectornormality_v1} demonstrates that
 \begin{align*}
\max_{1\leq m\leq p_k}\bigg\| \bigg(     \uhat_k^{(t)}( \mathbf{W}_k^{(t)})\t - \U_k - \mathbf{Z}_k \mathbf{V}_k \mathbf{\Lambda}_k\inv \bigg)_{m\cdot}\bigg\| &\lesssim \frac{\sigma^2 \log(p) r\sqrt{p}}{\lambda^2} + \frac{\sigma r }{\lambda \sqrt{p}}\\&= \tilde o \bigg( \frac{1}{\lambda/\sigma} \bigg).
 \end{align*}
 Consequently, \cref{thm:eigenvectornormality_v1} demonstrates precisely how well the additional first-order correction term $\mathbf{Z}_k \mathbf{V}_k \mathbf{\Lambda}_k\inv$ can be used to approximate the estimated singular vectors.  
 \end{remark}

 \begin{remark}[Comparison to Matrix Singular Vector Estimation] \label{remark:snr}
 Such a leading-order expansion has also been developed for symmetric matrix denoising in  \citet{chen_spectral_2021} when both the row and column dimensions are of comparable size.  One natural point of comparison is the corresponding error estimates for HOSVD (higher-order singular value decomposition), or similar procedures.  In the tensor setting with $p_k \asymp p$, each matricization is of order $p \times p^2$, and the column dimension can be much larger than the row dimension.  \textcolor{black}{Therefore, when comparing asymptotic results for a $p \times p \times p$ tensor to asymptotic results for a $p_1 \times p_2$ matrix, the most natural point of comparison is to take $p_1 \asymp p$ and $p_2 \asymp p^2$.  In the subsequent discussion we will focus on matrices of dimension $p_1 \times p_2$, but to translate the results to tensors, we will consider a generic $p \times p \times p$ tensor, with $p_1 \asymp p$ and $p_2 \asymp p^2$.  }
 
 In \citet{agterberg_entrywise_2022}, who study entrywise singular vector analyses of rectangular matrices \textcolor{black}{of dimension $p_1 \times p_2$, it was shown that when $\lambda/\sigma \gtrsim \sqrt{p_{\max} \log(p_{\max})}$ (which translates to the condition $\lambda/\sigma \gtrsim p \sqrt{\log(p)}$))}, that one has a leading-order expansion similar to the one presented in \cref{thm:eigenvectornormality_v1} (see their equation 3) for estimated singular vectors.  However, the SNR condition in \citet{agterberg_entrywise_2022} may be too stringent for tensors, particularly in the high-noise regime $\lambda/\sigma \asymp p^{3/4} \polylog(p)$. In \citet{yan_inference_2021}, it was shown that when $\lambda/\sigma \asymp  (p_1 p_2)^{3/4}  \polylog(p)$ \textcolor{black}{(which translates to the condition $\lambda/\sigma \gtrsim p^{3/4} \polylog(p)$)}, that one has the leading-order expansion
 \begin{align*}
     \uhat_k^S \mathbf{W}_k - \U_k &= \mathbf{Z}_k \mathbf{V}_k \mathbf{\Lambda}_k\inv + \Gamma( \mathbf{Z}_k \mathbf{Z}_k\t ) \U_k \mathbf{\Lambda}_k^{-2} + \mathbf{\tilde \Psi}^{(k)},
 \end{align*}
 where $\mathbf{\tilde \Psi}^{(k)}$ is a residual term, $\uhat_k^S$ are estimates of the left singular vectors of the underlying low-rank matrix obtained via the \texttt{HeteroPCA} algorithm, and $\mathbf{Z}_k$ and $\mathbf{\Lambda}_k$ are the noise matrix and singular value matrix respectively. Observe that the term containing  $\Gamma(\mathbf{Z}_k \mathbf{Z}_k\t)$ is quadratic in the noise $\mathbf{Z}_k$--this additional quadratic term is dominant in the high-noise regime.    
 
 Therefore, \cref{thm:eigenvectornormality_v1} demonstrates how \texttt{HOOI} uses the tensorial structure to effectively eliminate the additional (dominant) ``quadratic'' term in the regime $p \gtrsim \lambda/\sigma \gtrsim p^{3/4}$.
 These results \textcolor{black}{(modulo logarithmic terms)} are summarized in Table \ref{tbl:eigenvectorcomparison}.
 \end{remark}

 \begin{remark}[Adaptivity to Heteroskedasticity]
 Note that the primary condition in \cref{thm:eigenvectornormality_v1} is essentially a signal-strength condition;  in fact, this result continues to hold even if one does not have the condition $\sigma \lesssim \sigma_{\min}$ in Assumption \ref{ass:noise}.  It has been previously demonstrated in \citet{zhang_heteroskedastic_2022} that in the absence of additional structure, heteroskedasticity may require additional debiasing for matrix singular vector estimation.  In contrast, the leading-order expansion in \cref{thm:eigenvectornormality_v1} continues to hold \emph{even with heteroskedasticity} (provided one initializes with diagonal deletion), which further demonstrates that tensor SVD is \emph{adaptive} to unknown variance profiles.  This phenomenon has also been discussed in a perturbative sense in \citet{agterberg_estimating_2022}.  
 \end{remark}

 Next, while the expansion in \cref{thm:eigenvectornormality_v1} demonstrates the leading-order approximation of the estimated singular vectors, it falls just short of establishing the asymptotic normality of the rows.  The following result shows that the rows of $\uhat_k^{(t)}$ are Gaussian about $\U_k$ modulo an orthogonal transformation.  \cref{thm:eigenvectornormality2} gives a result where $\kappa,\mu_0$ are allowed to grow.

 \begin{theorem}[Distributional Theory for Tensor Singular Vectors] \label{thm:eigenvectornormality2_v1}
 Instate the conditions of \cref{thm:eigenvectornormality_v1}.  
 Let $\Sigma^{(m)}_k$ denote the diagonal matrix of dimension $p_{-k} \times p_{-k}$, where the diagonal entries consist of the variances of $\mathcal{Z}_{mbc}$ if $k = 1$, $\mathcal{Z}_{amc}$ if $k = 2$, and $\mathcal{Z}_{abm}$ if $k =3$.  Define
 \begin{align*}
     \mathbf{\Gamma}^{(m)}_k &\coloneqq \mathbf{\Lambda}_k\inv \mathbf{V}_k\t \Sigma^{(m)}_k \mathbf{V}_k \mathbf{\Lambda}_k\inv.
 \end{align*}
  Let $\mathcal{A}$ denote the collection of convex sets in $\mathbb{R}^{r_k}$, and let $Z$ be an $r_k$-dimensional Gaussian random variable with the identity covariance matrix.  
   Then it holds that
 \begin{align*}
     \sup_{A \in \mathcal{A}} \bigg| &\mathbb{P}\bigg\{ (\mathbf{\Gamma}^{(m)}_k)^{-1/2} \bigg( \uhat_k^{(t)}(\mathbf{W}_k^{(t)})\t - \U_k \bigg)_{m\cdot} \in A \bigg\} - \mathbb{P}\{ Z \in A \} \bigg| \\
     &\lesssim  \frac{\sigma \log(p) r^{3/2} \sqrt{p}}{\lambda} + \frac{r^{3/2}}{\sqrt{p}}.
 \end{align*}
 \end{theorem}

%  With a little more analysis, we can obtain the following result demonstrating the asymptotic normality of the rows of the estimated loadings.

 % Observe that if $\kappa$ and $\mu_0$ are bounded, a sufficient condition for normality is that $r = o(p^{1/6}/\log(p))$, which still allows $r$ to grow polynomially in $p$. 
  Observe that \cref{thm:eigenvectornormality2_v1} allows the rank $r$ to grow slowly.  It is sufficient to have
 \begin{align*}
     r^{3/2} \ll \min\bigg\{ \frac{\lambda/\sigma}{\sqrt{p} \log(p)}, \sqrt{p}\bigg\}
 \end{align*}
  for asymptotic normality. In particular,  $r = o(p^{1/6})$ suffices.  
  
  We show in the proof of \cref{thm:eigenvectornormality2_v1} that the covariance matrix $\mathbf{\Gamma}^{(m)}_k$ is invertible with minimum eigenvalue lower bounded by $\sigma_{\min}^2/\lambda^2$; moreover, in the particular case that $\sigma_{abc} \equiv \sigma$, we note that $\mathbf{\Gamma}^{(m)}_k$ simplifies to $\sigma^2 \mathbf{\Lambda}_k^{-2}$. In \cref{thm:efficiencyloadings,thm:efficiency_order} we show that this covariance matrix is both optimal over all unbiased estimators and yields the order-wise optimal expected length over all valid confidence intervals for Gaussian noise.
  
   \begin{remark}[Relationship to Matrix Singular Vector Estimation]
   The  matrix $\mathbf{\Gamma}^{(m)}_k$ is the same limiting covariance matrix as in Corollary 2 of \citet{agterberg_entrywise_2022} (in the particular case the noise matrix therein has independent entries).  However, a key feature is that \cref{thm:eigenvectornormality2_v1} holds when $\lambda/\sigma \asymp p^{3/4} \sqrt{\log(p)}$  (modulo factors of $\kappa, \mu_0$, and $r$), whereas the results of \citet{agterberg_entrywise_2022} (when translated to the tensor setting) only hold when $\lambda/\sigma \asymp p \sqrt{\log(p)}$, which shows how the additional tensor structure affects the limiting properties of the estimated singular vector components. 
  \end{remark}

 %%%%%%%%%%%%%%%%%%%%%%%%%%%%%%%%%%%%%%%%%%%
  \subsection{Confidence Regions and Statistical Inference}
  \label{sec:loadingsinference}
  %%%%%%%%%%%%%%%%%%%%%%%%%%%%%%%%%%%%%%%%%%%
  
  Next, we consider uncertainty quantification for the estimated singular vectors $\uhat_k = \uhat_k^{(t)}$. By \cref{thm:eigenvectornormality2_v1} we can identify the limiting covariance matrix $\mathbf{\Gamma}_k^{(m)}$, in \cref{al:ci_eigenvector} we describe a plug-in approach to estimating this matrix and producing confidence regions.  The following result demonstrates the theoretical  validity of this procedure.  A more general statement can be found in \cref{thm:civalidity2} in \cref{sec:generaltheorems}.  

  \begin{algorithm}[t]
	\caption{Confidence Regions for $(\U_k)_{m\cdot}$}
	\begin{algorithmic}[1]
		\State Input:  Singular vector estimate $\uhat_k$ and tensor estimate $\mathcal{\hat T}$ from \cref{al:tensor-power-iteration}, coverage level $1 - \alpha$  
		
			\State Let $\mathbf{\hat V}_k$ and $\mathbf{\hat \Lambda}_k$ denote the $r_k$ right singular vectors and singular values of the matrix
		\begin{align*}
		    \mathcal{M}_k\big(  \mathcal{\tilde T} \big) \bigg( \big(\uhat_{k+1} \uhat_{k+1}\t \big) \otimes \big(\uhat_{k+2} \uhat_{k+2}\t \big) \bigg).
		\end{align*}
	\State Define $\mathcal{\hat Z} = \mathcal{\tilde T} - \mathcal{\hat T}$, and set $\hat \Sigma^{(m)}_k$ as the diagonal matrix of entries defined via:
		\begin{align*}
		    \bigg(\hat \Sigma^{(m)}_1 \bigg)_{(a-1)p_3 + b} = 		    \mathcal{\hat Z}^2_{mab};\qquad
		     \bigg(\hat \Sigma^{(m)}_2 \bigg)_{(b-1)p_1 + a} &= 		    \mathcal{\hat Z}^2_{amb}; \qquad 
		        \bigg(\hat \Sigma^{(m)}_3 \bigg)_{(a-1)p_2 + b} = 		    \mathcal{\hat Z}^2_{amb}.
		\end{align*}
		\State Define
		\begin{align*}
		    \mathbf{\hat \Gamma}_k^{(m)} \coloneqq \mathbf{\hat \Lambda}_k\inv \mathbf{\hat V}_k\t \hat \Sigma^{(m)}_k \mathbf{\hat V}_k \mathbf{\hat \Lambda}_k\inv;
		\end{align*}
		\State Compute the $1 - \alpha$ quantile $\tau_{\alpha}$ of $\chi^2_{r_k}$ random variable, and construct the ball $\mathcal{B}_{1-\alpha} \coloneqq \{z: \|z\|^2 \leq \tau_{\alpha}\}$
		\State Output the confidence region
		\begin{align*}
		   \mathrm{C.R.}_{k,m}^{\alpha}(\uhat_k) &\coloneqq \uhat_k + \big(  \mathbf{\hat \Gamma}_k^{(m)}  \big)^{1/2} \mathcal{B}_{1-\alpha} = \{ \uhat_k + \big(  \mathbf{\hat \Gamma}_k^{(m)}  \big)^{1/2} z: z \in  \mathcal{B}_{1-\alpha} \}
		\end{align*}
		\end{algorithmic}\label{al:ci_eigenvector}
\end{algorithm}
  
 \begin{theorem}[Validity of Confidence Intervals for the Loadings]\label{thm:civalidity2_v1}  Instate the conditions of \cref{thm:eigenvectornormality_v1}.  Suppose also that 
 \begin{align}
     r^{3/2}\sqrt{\log(p)} \lesssim p^{1/4}, \label{technicalcondition}
 \end{align}
 In addition, assume that %$ \frac{r^2}{p} = o(1)$ and that
 \begin{align*}
     \lambda/\sigma \gg \log^2(p) r^2 \sqrt{p}.
 \end{align*}
Let $\mathrm{C.R.}_{k,m}^{\alpha}(\uhat_k)$ denote the output of \cref{al:ci_eigenvector}.  Then it holds that
 \begin{align*}
   \p\bigg\{ \bigg( \U_k \mathbf{W}_k^{(t)}\bigg)_{m\cdot}  \in  \mathrm{C.R.}_{k,m}^{\alpha}(\uhat_k) \bigg\} = 1- \alpha  - o(1).
 \end{align*}
   \end{theorem}
 
 We note that the SNR condition $\lambda/\sigma \gg  r^2 \log(p) \sqrt{p}$ is automatically satisfied when  $r = o(p^{1/8}/\log^{3/2}(p))$.  In particular, the condition holds whenever  $r$ is fixed.  
 
 To the best of our knowledge, \cref{thm:civalidity2_v1} is the first to establish entrywise confidence region validity for estimated tensor singular vectors in the presence of heteroskedastic subgaussian noise.  Moreover, these results are entirely data-driven and do not require any sample splitting.

  %%%%%%%%%%%%%%%%%%%%%%%%%%%%%%%%%%
\subsection{Lower Bounds}
\label{sec:loadingslowerbounds}
%%%%%%%%%%%%%%%%%%%%%%%%%%%%%%%%%
   \cref{thm:civalidity2_v1} shows that the confidence regions from \cref{al:ci_eigenvector}, are, up to nonidentifiable orthogonal transformation, asymptotically valid.    In order to investigate the optimality of this result, we establish the following  lower bound showing that these regions are essentially statistically efficient.
 
 \begin{theorem}[Efficiency of Loadings] \label{thm:efficiencyloadings}
 Suppose that $\mathcal{Z}$ consists of independent Gaussians with variance lower bounded by $\sigma_{\min}^2$, and suppose that $\kappa^2 \mu_0 \sqrt{\frac{r}{p}} \ll 1$.  Then if $( \mathbf{\tilde \U}_k)_{m\cdot}$ is any unbiased estimator for $(\U_k)_{m\cdot}$, it holds that
 \begin{align*}
     \mathrm{Var}\big( (\mathbf{\tilde \U}_k)_{m\cdot} \big) \succcurlyeq \sigma_{\min}^2 (1- o(1)) \mathbf{\Lambda}_k^{-2},
 \end{align*}
 where $A \succcurlyeq B$ refers to the positive semidefinite ordering.
 \end{theorem}
 
 Recall that when $\sigma_{ijk} \equiv \sigma$ it holds that $\mathbf{\Gamma}^{(m)}_k = \sigma^2 \mathbf{\Lambda}_k^{-2}$.   Consequently, \cref{thm:efficiencyloadings} shows that when $\sigma_{abc} \equiv \sigma$ (i.e., the noise is homoskedastic), the rows of the estimated loadings are asymptotically efficient, and when $\sigma/\sigma_{\min} = O(1)$, the confidence regions for the rows of $\uhat$ are essentially optimal.  

A subtle nuance in \cref{thm:efficiencyloadings} is that it only holds for unbiased estimators of $(\mathbf{U}_k)_{m\cdot}$, and while $(\uhat_k)_{m\cdot}$ is \emph{asymptotically} unbiased, it may not be unbiased for finite samples. Therefore, \textcolor{black}{following ideas from \citet{cai_confidence_2017}}, we will also consider a lower bound for the expected length of any $1-\alpha$ confidence interval for (any linear functional of) $\big(\mathbf{U}_k\big)_{m\cdot}$.  In order to do so, we must first define some notation.

Define the parameter space
 \begin{align*}
     \Theta( \lambda,&\kappa,\mu_0,\sigma,\sigma_{\min}) \\
     := \bigg\{ &\mathcal{T} \in \mathbb{R}^{p_1 \times p_2 \times p_3}: \mathcal{T} = \mathcal{S} \times_1 \mathbf{U}_1 \times_2 \U_2 \times_3 \U_3;\  \mathcal{S} \in \mathbb{R}^{r_1 \times r_2 \times r_3}; \\
  &\mathbf{U}_k \in \mathbb{R}^{p_k \times r_k}, \U_k\t \U_k = \mathbf{I}_{r_k}; \| \U_k \|_{2,\infty} \leq \mu_0 \sqrt{\frac{r_k}{p_k}}; \\
  &\lambda \leq \lambda_{\min}(\mathcal{S}) \leq \lambda_{\max}(\mathcal{S}) \leq \kappa \lambda; \ \lambda/\sigma \geq C_0 \kappa r_{\max} \sqrt{p_{\max}}; \\
  &\sigma_{\min} \leq  \sigma_{ijk} \leq \sigma; \ \kappa \leq p_{\min}^{1/4}; \ r_{\max} \leq p_{\min}^{1/2} \bigg\}. \numberthis \label{parameterspace}
 \end{align*}
 Under the assumptions $p_k \asymp p$ and $\kappa r_{\max} \lesssim p^{1/4}$, the assumption that $\lambda/\sigma \geq C_0 \kappa r_{\max} \sqrt{p_{\max}}$ is not particularly stringent, and only precludes settings where the rank is prohibitively large or the tensor is extremely ill-conditioned.

Next, for a given deterministic vector $\xi \in \mathbb{R}^{r_k}$, define the set $\mathcal{I}_{\alpha}(\Theta,\xi)$ of all $1-\alpha$ confidence intervals for $\xi\t (\mathbf{U}_k)_{m\cdot}$ over the parameter space $\Theta$; that is,
 \begin{align*}
     \mathcal{I}_{\alpha}(\Theta,\xi) := \bigg\{ \mathrm{C.I.}_{k,m}^{\alpha}(\xi,\mathcal{Z},\mathcal{T}) = [l, u]: \inf_{\mathcal{T} \in \Theta} \p_{\mathcal{T}}\big( l \leq \pm \xi\t\big( \mathbf{U}_k\big)_{m\cdot} \leq u \big)\geq 1 - \alpha \bigg\},
 \end{align*}
 where the fact that we consider $\pm \xi\t (\U_k)_{m\cdot}$ is due to rotational ambiguity.  For a given confidence interval $\mathrm{C.I.}_{k,m}^{\alpha}(\xi,\mathcal{Z},\mathcal{T}) = [l,u]$, define its length $L\big( \mathrm{C.I.}_{k,m}^{\alpha}(\xi,\mathcal{Z},\mathcal{T}) \big) =|u - l|$.  \textcolor{black}{In words, the set $\mathcal{I}_{\alpha}(\Theta,\xi)$ is the set of all confidence intervals for $\xi\t\big( \mathbf{U}_k\big)_{m\cdot}$ that have coverage at least $1-\alpha$ uniformly over the parameter space $\Theta$.}
 
The following result yields a lower bound for the expected length of any linear functional over the class $\Theta(\lambda,\kappa,\mu_0,\sigma,\sigma_{\min})$.  
% shows that the smallest expected length of any linear functional of the $m$'th row of $(\mathbf{U}_k)_{m\cdot}$ is of order $\frac{\sigma^2_{\min}}{\lambda}$, thereby demonstrating the \emph{order-wise} optimality of $\uhat_k$. 

\begin{theorem} \label{thm:efficiency_order}
     Let $\xi \in \mathbb{R}^{r_k}$ be any deterministic vector satisfying $\max_{j\neq k} \frac{|\xi_j|}{|\xi_k|} \leq c_0$ for some fixed constant $c_0$, and suppose that  $\alpha$ satisfies $ 0 < \alpha < 1/2$. Then there is some constant $c > 0$ such that
     \begin{align*}
         \inf_{\mathrm{C.I.}_{k,m}^{\alpha}(\xi,\mathcal{Z},\mathcal{T}) \in \mathcal{I}_{\alpha}(\Theta,\xi)} \sup_{\mathcal{T} \in \Theta} \mathbb{E}_{\mathcal{T}} L\big( \mathrm{C.I.}_{k,m}^{\alpha}(\xi,\mathcal{Z},\mathcal{T}) \big) &\geq c \|\xi\|_{\infty} \frac{\sigma_{\min}}{\lambda_{r_k}}.
     \end{align*}
 \end{theorem}

%\cref{thm:efficiency_order} shows that the smallest expected length of any linear functional of the $m$'th row of $(\mathbf{U}_k)_{m\cdot}$ \textcolor{black}{(with coverage at least $1 - \alpha$)} must be of order $\frac{\sigma_{\min}}{\lambda}$. 
\textcolor{black}{Observe that the left-hand side above represents the minimax expected length of all confidence intervals for $\xi^\top (\mathbf{U}_k)_{m\cdot}$ with coverage probability of at least $1 - \alpha$, uniformly over all $\mathcal{T} \in \Theta$. In simple terms, \cref{thm:efficiency_order} shows that the minimax expected length of such intervals is on the order of $\|\xi\|_{\infty} \frac{\sigma_{\min}}{\lambda_{r_k}}$. Moreover, the confidence interval for $\xi^\top (\mathbf{U}_k)_{m\cdot}$ constructed via \cref{al:ci_eigenvector} has the same order of length, as demonstrated in the proof of \cref{thm:civalidity2_v1}, where the estimator $(\mathbf{\hat \Gamma}_k^{(m)})^{1/2}$ from \cref{al:ci_eigenvector} is shown to have a smallest eigenvalue of at least $\sigma_{\min}/\lambda$ (see \eqref{lowerbdongamma}). Since the length of this confidence interval aligns with the lower bound in \cref{thm:efficiency_order}, these results together demonstrate that \cref{al:ci_eigenvector} produces asymptotically valid and, in some sense, optimal confidence intervals.
%Observe that  the left hand side above can be viewed as the minimax expected length of all confidence intervals for $\xi\t (\mathbf{U}_k)_{m\cdot})$ with coverage at least $1 - \alpha$ \emph{uniformly for all $\mathcal{T}\in\Theta$}. Put simply, \cref{thm:efficiency_order} shows that the minimax expected length of any such confidence interval must have length of order $\|\xi\|_{\infty} \frac{\sigma_{\min}}{\lambda_{r_k}}$, up to some universal constant.  Furthermore, the confidence interval for $\xi\t (\mathbf{U}_k)_{m\cdot}$ constructed using \cref{al:ci_eigenvector} will have length of the same order, since it is shown in the proof of \cref{thm:civalidity2_v1} that the estimator $(\mathbf{\hat \Gamma}_k^{(m)})^{1/2}$ from \cref{al:ci_eigenvector} has the smallest eigenvalue of order at least $\sigma_{\min}/\lambda$ (see \eqref{lowerbdongamma}). Since the length of the resulting confidence interval matches the lower bound in \cref{thm:efficiency_order}, these results in tandem show that \cref{al:ci_eigenvector} yields asymptotically valid confidence intervals, that are, in some sense, optimal.
} 
 
 %We return to the problem of obtaining confidence intervals for the true loadings in \cref{sec:entrywiseinference}.
 
 %%%%%%%%%%%%%%%%%%%%%%%%%%%%%%
\section{Entrywise Distributional Theory, Inference, and Consistency}
\label{sec:entries}
%%%%%%%%%%%%%%%%%%%%%%%%%%%%%%%

We now turn our attention to estimating the entries of the underlying low-rank tensor, obtained via the estimate
\begin{align*}
    \mathcal{\hat T} \coloneqq \mathcal{\tilde T} \times_1 \uhat_1 \times_2 \uhat_2 \times_3 \uhat_3
\end{align*}
as described in \cref{al:tensor-power-iteration}, where $\uhat_k \coloneqq \uhat_k^{(t)}$ for $t$ iterations.  The following result characterizes the distribution of the entries of this estimator, with the more general statement permitting $\mu_0, \kappa$ to grow available in \cref{thm:asymptoticnormalityentries} in \cref{sec:generaltheorems}.

\begin{theorem}[Asymptotic Normality of the Estimated Entries] \label{thm:asymptoticnormalityentries_v1}  Instate the conditions of \cref{thm:eigenvectornormality_v1}, and suppose further that
\begin{align*}
  r^{3/2} \sqrt{\log(p)} \lesssim p^{1/4}.
\end{align*}
 Let $\Sigma^{(m)}_1 \in \mathbb{R}^{p_2 p_3 \times p_2 p_3}$ be the diagonal matrix whose $(a-1)p_3 + b$'th entry is the variance of the random variable $\mathcal{Z}_{mab}$, and let $\Sigma^{(m)}_2$ and $\Sigma^{(m)}_3$ be defined similarly.  Assume that
\begin{align*}
   \| e_{(j-1)p_3 + k}\t \mathbf{V}_1 \|^2& + \|e_{(k-1)p_1 + i} \mathbf{V}_2  \|^2 + \| e_{(i-1)p_2 + j} \mathbf{V}_3\|^2 \\
   &\gg \max\bigg\{\frac{ r^4 \log(p)}{p^{3}} ,\frac{\sigma^2 r^{4} \log^2(p)}{\lambda^2 p}\bigg\}.
  \end{align*}
Define
\begin{align*}
    s^2_{ijk} &\coloneqq  \| e_{(j-1)p_3 + k}\t \mathbf{V}_1 \mathbf{V}_1\t \big(\Sigma^{(i)}_1\big)^{1/2} \|^2 + \|e_{(k-1)p_1 + i}\t \mathbf{V}_2 \mathbf{V}_2\t\big(\Sigma^{(j)}_2\big)^{1/2} \|^2\\
    &\quad + \| e_{(i-1)p_2 + j}\t \mathbf{V}_3\mathbf{V}_3\t \big( \Sigma^{(k)}_3 \big)^{1/2}\|^2.
\end{align*}
Let $Z$ denote a standard Gaussian random variable and let $\Phi$ denote its cumulative distribution function. Then it holds that
\begin{align*}
    \sup_{t\in\mathbb{R}} \bigg| \p\bigg\{ \frac{ \mathcal{\hat T}_{ijk} - \mathcal{T}_{ijk}}{s_{ijk}} \leq t \bigg\} - \Phi(t) \bigg| &= o(1).  
\end{align*}
\end{theorem}

In \citet{xia_inference_2022} the authors obtain entrywise distributional theory under the assumption that $\mathcal{T}$ is rank one and that the noise is homoskedastic Gaussian.  In contrast, we allow subgaussian noise, arbitrary (possibly growing) rank, and unknown variances. In fact, for rank-one tensors, our results generalize those of \cite{xia_inference_2022}: in the rank-one setting, it holds that $\| e_{(j-1)p_3 + k}\t \mathbf{V}_1 \|^2 = (\U_2)_{j}^2 (\U_3)_{k}^2$ (and similarly for the other terms), so that the limiting variance in \cref{thm:asymptoticnormalityentries_v1} reduces to that of \cite{xia_inference_2022}.  On the other hand, our incoherence requirement is much stronger than that in \cite{xia_inference_2022}, but this is likely due to the fact that we allow arbitrary subgaussian noise.  In addition, our analysis is significantly different from \citet{xia_inference_2022}, who rely heavily on the rotational invariance of the Gaussian distribution, whereas our analysis is based on a leave-one-out argument via the constructions from \citet{agterberg_estimating_2022}. 
%The proof in \cite{xia_inference_2022} relies heavily on the rotational invariance of the Gaussian distribution, something that does not hold for more general subgaussian ensembles. In addition, our proof technique is significantly different from \citet{xia_inference_2022} to account for the possibility of non-Gaussian, heteroskedastic noise. 

Note that $s^2_{ijk}$ in \cref{thm:civalidity_v1} satisfies the lower bound
\begin{align*}
    s^2_{ijk} &\geq \sigma^2_{\min}  \left(\left\|e_{(j-1) p_{3}+k}^{\top} \mathbf{V}_{1}\right\|^{2}+\left\|e_{(k-1) p_{1}+i}\t \mathbf{V}_{2}\right\|^{2}+\left\|e_{(i-1) p_{2}+j}\t \mathbf{V}_{3}\right\|^{2}\right),
\end{align*}
with equality when $\sigma \equiv \sigma_{\min}$ (i.e., the noise is homoskedastic).  In \cref{thm:efficiency} we demonstrate that this lower bound is optimal when $\mathcal{Z}$ consists of homoskedastic Gaussian noise.  

\begin{remark}[Comparison to Entrywise Distributional  Theory for Matrices]
\cref{thm:asymptoticnormalityentries_v1} can be compared to several results on entrywise distributional guarantees for low-rank matrices, such as Theorem 4.10 of \citet{chen_spectral_2021} or Theorem 6 of \citet{yan_inference_2021}. %A key difference from our work and these works is that here we consider an order-three tensor, which means that the leading-order term stems from three separate quantities, each of which is linear in the noise tensor.  
Informally, the proof of \cref{thm:asymptoticnormalityentries_v1} shows that we have the approximate asymptotic expansion
\begin{align*}
    \mathcal{T}_{ijk} - \mathcal{\hat T}_{ijk} &\approx e_{i}^{\top} \mathbf{Z}_{1} \mathbf{V}_{1} \mathbf{V}_{1}^{\top} e_{(j-1) p_{3}+k}+e_{j}^{\top} \mathbf{Z}_{2} \mathbf{V}_{2} \mathbf{V}_{2}^{\top} e_{(k-1) p_{1}+i}+e_{k}^{\top} \mathbf{Z}_{3} \mathbf{V}_{3} \mathbf{V}_{3}^{\top} e_{(i-1) p_{2}+j}.
\end{align*}
Moreover, the asymptotic variance $s^2_{ijk}$ is simply the variance of each of the three leading-order terms, ignoring cross-terms.  While each of the three terms is not entirely uncorrelated (since they contain repetitions of elements of $\mathcal{Z}$), we show that this correlation is negligible due to the incoherence of singular vectors.  

On the other hand, by slightly modifying the results in, for example, \citet{chen_spectral_2021} or \citet{yan_inference_2021}, one can show that for a generic low-rank matrix $\mathbf{M = U \Lambda V\t}$ of dimension $p_1 \times p_2$, with $p_1 \asymp p_2$,  one has the approximate first-order decomposition 
\begin{align*}
    \mathbf{M}_{ij} - \mathbf{\hat M}_{ij} &\approx e_i\t \mathbf{Z} \mathbf{V} \mathbf{V}\t e_j + e_j\t \mathbf{Z}\t \U \U\t e_i,
\end{align*}
where $\mathbf{\hat M}$ is the truncated rank $r$ SVD of the observed matrix $\mathbf{M + Z}$, for $\mathbf{Z}$ consisting of independent noise.    Hence, in contrast to the matrix case, the tensor case results in \emph{three} (as opposed to two) leading-order perturbations, each of which is linear in the noise tensor $\mathcal{Z}$, which shows how the tensorial structure manifests in the asymptotics. 
\end{remark}

\begin{remark}[Signal-Strength Requirement] In addition to the signal-strength conditions from \cref{thm:eigenvectornormality_v1}, \cref{thm:asymptoticnormalityentries_v1} also requires a lower bound on the magnitude of the appropriate rows of the matrices $\mathbf{V}_k$. This condition implies a lower bound condition on the variance $s^2_{ijk}$ of the form
\begin{align*}
    s_{ijk} \gg \max\bigg\{ \frac{\sigma r^{3/2} \sqrt{\log(p)}}{p^{3/2}}, \frac{\sigma^2 r^2 \log(p)}{\lambda \sqrt{p}} \bigg\}.
\end{align*}
Ignoring factors of $r$ amounts to requiring that 
\begin{align*}
     s_{ijk} \gg \max\bigg\{ \frac{\sigma\sqrt{\log(p)}}{p^{3/2}}, \frac{\sigma^2 \log(p)}{\lambda \sqrt{p}} \bigg\}.
\end{align*}
A similar condition is required in the matrix setting in Theorem 4.10 of \cite{chen_spectral_2021} and Theorem 7 of \cite{yan_inference_2021}. Informally, this condition is required so that there is enough signal within that corresponding entry of the underlying tensor. In the more challenging regime $\lambda/\sigma \asymp p^{3/4}\polylog(p)$, this yields the condition
\begin{align*}
    s_{ijk} \gg \frac{\sigma}{p^{5/4} \polylog(p)}.
\end{align*}
Note that $s_{ijk}$ is of order $\sigma$ times the size of the maximum of the corresponding rows of $\mathbf{V}_k$.  Since $\| \mathbf{V}_k \|_{2,\infty} \lesssim  \frac{\sqrt{r}}{p}$ when $\mu_0 = O(1)$, we see that the corresponding rows are allowed to be as much as a factor of $p^{1/4} \polylog(p)$ smaller than the maximum row norm, which covers a wide range of possible values.  This regime is much broader than what is permitted in the matrix setting (e.g., Theorem 4.10 of \citet{chen_spectral_2021}). %), but this effect is due largely to the fact that the permitted lower bound $s^2_{ijk}$ depends on the value of $\lambda/\sigma$, and additional signal is required in order for the spectral initialization to succeed. 
\end{remark}

\subsection{Confidence Intervals} \label{sec:entrywiseinference}
%%%%%%%%%%%%%%%%%%%%%%%%%%%%%

We now turn our attention to uncertainty quantification for the  entries of the underlying low-rank tensor. 
In  \cref{al:ci_entries} we provide a data-driven plug-in estimator of the variance $s^2_{ijk}$, and the following result demonstrates the theoretical validity of this procedure. The general statement is available in \cref{thm:civalidity}.  

  \begin{algorithm}[t]
	\caption{Confidence Intervals for $\mathcal{T}_{ijk}$}
	\begin{algorithmic}[1]
		\State Input:  Singular vector estimate $\uhat_k$ and tensor estimate $\mathcal{\hat T}$ from \cref{al:tensor-power-iteration}, coverage level $1 - \alpha$.
		
			\State Let $\mathbf{\hat V}_k$ denote the $r_k$ right singular vectors of the matrix
		\begin{align*}
		    \mathcal{M}_k\big(  \mathcal{\tilde T} \big) \bigg( \big(\uhat_{k+1} \uhat_{k+1}\t \big) \otimes \big(\uhat_{k+2} \uhat_{k+2}\t \big) \bigg).
		\end{align*}
	\State Define $\mathcal{\hat Z} = \mathcal{\tilde T} - \mathcal{\hat T}$. Set
		\begin{align*}
		   \mathbf{\hat Z}_k &\coloneqq \mathcal{M}_k( \mathcal{\hat Z}).
		\end{align*}
		\State Set %{\red (Do you think it would be better to specify the upper and lower range of $a,b,c$ in your summation?)}
\begin{align*}
    \hat s^2_{ijk} &\coloneqq \sum_{a=1}^{p_{-1}} \big( \mathbf{\hat Z}_1 \big)_{ia}^2 \big( \mathbf{\hat V}_1 \mathbf{\hat V}_1\t \big)_{a,(j-1)p_3 + k}^2 + \sum_{b=1}^{p_{-2}} \big( \mathbf{\hat Z}_2 \big)_{jb}^2 \big( \mathbf{\hat V}_2 \mathbf{\hat V}_2\t \big)_{b,(k-1)p_1 + i}^2 \\
    &\quad + \sum_{c=1}^{p_{-3}} \big( \mathbf{\hat Z}_3 \big)_{kc}^2 \big( \mathbf{\hat V}_3 \mathbf{\hat V}_3\t \big)_{c,(i-1)p_2 + j}^2
\end{align*}
\State Let $z_{\alpha/2}$ denote the $1 - \alpha/2$ quantile of a standard Gaussian random variable.
\State Output confidence interval
\begin{align*}
    \mathrm{C.I.}^{\alpha}_{ijk}(\mathcal{\hat T}_{ijk}) &\coloneqq \big( \mathcal{\hat T}_{ijk} - z_{\alpha/2} \hat s_{ijk}, \mathcal{\hat T}_{ijk} + z_{\alpha/2} \hat s_{ijk} \big).
\end{align*}
		\end{algorithmic}\label{al:ci_entries}
\end{algorithm}

\begin{theorem}[Validity of Confidence Intervals of the Entries] \label{thm:civalidity_v1}
Instate the conditions of \cref{thm:asymptoticnormalityentries_v1}.  
Suppose further that
\begin{align*}
\left\|e_{(j-1) p_{3}+k}^{\top} \mathbf{V}_{1}\right\|^{2}+\left\|e_{(k-1) p_{1}+i}\t \mathbf{V}_{2}\right\|^{2}+\left\|e_{(i-1) p_{2}+j}\t \mathbf{V}_{3}\right\|^{2} &\gg  \frac{\sigma r^3  \log^{3/2}(p)}{\lambda p^{3/2}}.
\end{align*}
%Suppose that
%\begin{align*}
%    \kappa^2 \mu_0^2 r^{3/2} \sqrt{\log(p)} \lesssim p^{1/4},
%\end{align*}
%in addition to the conditions of \cref{thm:twoinfty}.  
Let $\mathrm{C.I.}^{\alpha}_{ijk}(\mathcal{\hat T}_{ijk})$ denote the output of \cref{al:ci_entries}.  Then it holds that
\begin{align*}
    \p\bigg( \mathcal{T}_{ijk} \in  \mathrm{C.I.}^{\alpha}_{ijk}(\mathcal{\hat T}_{ijk}) \bigg) = 1- \alpha  - o(1).
\end{align*}
\end{theorem}

Consequently, \cref{thm:civalidity_v1} shows that the confidence interval obtained by a plug-in estimate of the variance is asymptotically valid.   In addition, this result allows $r$ to grow --  a sufficient condition for  \cref{thm:civalidity_v1} to hold is that $r = o(p^{1/6}/\log(p))$.   To the best of our knowledge, this is the first result demonstrating the theoretical validity of a plug-in approach in the presence of heteroskedastic noise.

\begin{remark}[Signal-Strength Requirement] %\textcolor{green}{I think it is better to leave it stated in its current form, since it resembles existing works for the matrix case (e.g., the assumptions in Theorem 3 of https://arxiv.org/abs/2006.08580, Assumption 4 in https://arxiv.org/pdf/1909.00116.pdf, Equation 3.5 in https://arxiv.org/abs/2107.12365}
We note that \cref{thm:civalidity_v1} requires an additional signal-strength condition to \cref{thm:asymptoticnormalityentries_v1}. This additional requirement ensures that
\begin{align*}
    s^2_{ijk} \gg \frac{\sigma^3  \log^{3/2}(p)}{\lambda p^{3/2}},
\end{align*}
which implies that the variance dominates the bias in order to yield asymptotically valid confidence intervals. This is slightly more stringent than the condition in \cref{thm:asymptoticnormalityentries_v1}.  When $\lambda/\sigma \asymp p^{3/4} \polylog(p)$ and $r = O(1)$, then this requirement essentially states that
\begin{align*}
    s^2_{ijk} \gg \frac{\sigma^2}{p^{9/4}\polylog(p)}.
\end{align*}
On the other hand, under the same conditions, \cref{thm:asymptoticnormalityentries_v1} requires $s^2_{ijk}$ to satisfy
\begin{align*}
    s^2_{ijk} \gg \frac{\sigma^2}{p^{5/2} \polylog(p)},
\end{align*}
which is a factor of $p^{1/4}$ smaller than the condition in \cref{thm:civalidity_v1}.  However, this requirement still allows the appropriate rows of $\mathbf{V}_k$ to be a factor of $p^{1/8}\polylog(p)$ smaller than the maximum row norm.  %We believe that this requirement is due to our proof technique, where we bound relevant quantities with their maximum entries, which may be somewhat coarse. 
It may be possible to improve this result slightly by using a more refined analysis, %bounding scheme, 
but this is beyond the scope of this paper.
\end{remark}

    %%%%%%%%%%%%%%%%%%%%%%%%%%%%
    \subsection{Lower Bounds}
    \label{sec:entrylowerbounds}
    %%%%%%%%%%%%%%%%%%%%%%%%%%
    
We now turn our attention to the optimality of the plug-in estimate $\hat s^2_{ijk}$.  The following result shows that the variance $s^2_{ijk}$ is optimal over all unbiased estimators for $\mathcal{T}_{ijk}$ when $\mathcal{Z}$ consists of homoskedastic Gaussian noise.

%{\red (Your lower bound statement (Cramer-Rao type argument) states that: ``Then for any unbiased estimate"... Actually in the high-dimensional regimes, almost all estimators are biased. In particular, HOOI estimator is definitely biased (although it is hopeful to be nearly unbiased). So rigorously speaking, your lower bound statement does not match your upper bound statement and does not really apply here. Any thoughts on how to revise it?)}
\begin{theorem}[Efficiency Of Entrywise Confidence Intervals] \label{thm:efficiency}
Suppose that $\mathcal{Z}$ consists of independent Gaussian entries of variance lower bounded by $\sigma_{\min}^2$, and suppose that $\kappa^2 \mu_0 \sqrt{\frac{r}{p}} \ll 1$. Then for any unbiased estimate $\mathcal{\tilde T}_{ijk}$ of $\mathcal{T}_{ijk}$ it holds that
\begin{align*}
    \mathrm{Var}&(\mathcal{\tilde T}_{ijk}) \\
    &\geq (1 - o(1)) 
\sigma^2_{\min}  \left(\left\|e_{(j-1) p_{3}+k}^{\top} \mathbf{V}_{1}\right\|^{2}+\left\|e_{(k-1) p_{1}+i}\t \mathbf{V}_{2}\right\|^{2}+\left\|e_{(i-1) p_{2}+j}\t \mathbf{V}_{3}\right\|^{2}\right).
\end{align*}
\end{theorem}

Recall that $s^2_{ijk}$ in \cref{thm:civalidity_v1} satisfies the lower bound
\begin{align*}
    s^2_{ijk} &\geq \sigma^2_{\min}  \left(\left\|e_{(j-1) p_{3}+k}^{\top} \mathbf{V}_{1}\right\|^{2}+\left\|e_{(k-1) p_{1}+i}\t \mathbf{V}_{2}\right\|^{2}+\left\|e_{(i-1) p_{2}+j}\t \mathbf{V}_{3}\right\|^{2}\right),
\end{align*}
with equality when $\sigma \equiv \sigma_{\min}$ (i.e., the noise is homoskedastic). Consequently, taken together, Theorems \ref{thm:civalidity_v1} and \ref{thm:efficiency} show that Tensor SVD together plug-in variance estimation yields asymptotically efficient uncertainty quantification for the entries of the underlying low-rank tensor $\mathcal{T}$, with strict optimality for homoskedastic noise. However, as in the case for the estimated tensor singular vectors, \cref{thm:efficiency} only holds for unbiased estimators of $\mathcal{T}_{ijk}$, and while our results demonstrate that $\mathcal{\hat T}_{ijk}$ is \emph{asymptotically} unbiased, it may not be unbiased for finite samples.  Therefore, similar to the previous analysis, we provide a lower bound for the expected length for any confidence interval $\mathrm{C.I.}^{\alpha}_{ijk}(\mathcal{T},\mathcal{Z})$.

In what follows, recall we let $\Theta$ denote the parameter space given in \eqref{parameterspace}.  Define the set $\mathcal{I}_{\alpha}(\Theta,\{i,j,k\})$ via
\begin{align*}
     \mathcal{I}_{\alpha}(\Theta,\{i,j,k\}) := \bigg\{ \mathrm{C.I.}_{ijk}^{\alpha}(\mathcal{Z},\mathcal{T}) = [l, u]: \inf_{\mathcal{T} \in \Theta} \p_{\mathcal{T}}\big( l \leq  \mathcal{T}_{ijk} \leq u \big)\geq 1 - \alpha \bigg\};
 \end{align*}
 i.e., the set of valid confidence intervals such that $\mathcal{T}_{ijk}$ is contained in $\mathrm{C.I.}_{ijk}^{\alpha}(\mathcal{Z},\mathcal{T})$.  The following result quantifies the minimax length of any such confidence interval.

\begin{theorem}\label{thm:efficiency_order_ijk}
 Suppose that $\alpha$ satisfies $0 < \alpha < 1/2$, suppose $\mu_0 > 2$, and suppose that $r_{\max} \leq C r_{\min}$ and $p_k \leq C p_{\min}$. Then there is some constant $c > 0$ such that
 \begin{align*}
     &\inf_{\mathrm{C.I.}^{\alpha}_{ijk}(\mathcal{Z},\mathcal{T}) \in \mathcal{I}_{\alpha}(\Theta,\{i,j,k\})} \sup_{\mathcal{T} \in \Theta} \mathbb{E}_{\mathcal{T}} L\big(\mathrm{C.I.}^{\alpha}_{ijk}(\mathcal{Z},\mathcal{T}) \big) \\
     &\geq c \sigma_{\min} \sqrt{ \big\| e_{(j-1)p_3 + k}\t \mathbf{V}_1 \big\|^2 + \big\| e_{(k-1)p_1 + i}\t \mathbf{V}_2 \big\|^2 + \big\| e_{(i-1)p_2 + j}\t \mathbf{V}_3 \big\|^2 }.
 \end{align*}
 \end{theorem}

Analogous to the case of the singular vectors, we see that \cref{thm:efficiency,thm:efficiency_order_ijk} in tandem demonstrate that the confidence intervals obtained by \cref{al:ci_entries} is essentially optimal.  

\subsection{Entrywise Convergence of \texttt{HOOI}}
In the previous results, we require that there is sufficient signal strength in the entry of the underlying tensor in order to obtain valid confidence intervals. The following result shows that can still attain a strong rate of convergence in entrywise max-norm even when there is not sufficient signal strength.  As throughout the rest of this main paper, we focus on the regime $\kappa,\mu_0 = O(1)$, but a more general result is available in \cref{cor:maxnormbound}. 

\begin{theorem}\label{cor:maxnormbound_v1}
Instate the conditions of \cref{thm:eigenvectornormality_v1}, and suppose that
\begin{align*}
     r^{3/2} \sqrt{\log(p)} \lesssim p^{1/4}.
\end{align*}
Then the following bound holds with probability at least $1 - O(p^{-6})$:
\begin{align*}
    \| \mathcal{\hat T} - \mathcal{T} \|_{\max} &\lesssim \frac{\sigma \sqrt{r\log(p)}  }{p}  %\frac{\sigma \kappa \mu_0^3 r^{2} \sqrt{\log(p)}}{p^{3/2}} 
    + \frac{\sigma^2  r^3 \log(p)}{\lambda \sqrt{p}}. %\frac{\mu_0\sigma \sqrt{r\log(p)}  }{p} + 
    %\frac{\sigma \kappa \mu_0 \sqrt{r\log(p)}}{p} + \frac{\sigma^2 \mu_0^4 \kappa^3 r^3 \log(p)}{\lambda \sqrt{p}}
\end{align*}
Consequently, when following condition holds:
\begin{align*}
   % \kappa \mu_0^2 r \lesssim \sqrt{p}; \qquad 
   \lambda/\sigma \gtrsim   r^{5/2} \sqrt{p\log(p)},
\end{align*}
the bound above reduces to
\begin{align*}
     \| \mathcal{\hat T} - \mathcal{T} \|_{\max} &\lesssim \frac{ \sigma \sqrt{r\log(p)}}{p}.
\end{align*}
%In particular, this bound holds if $\mu_0 = O(1)$ and $\kappa^3 r^{5/2} = o(p^{1/4})$.
\end{theorem}

\cref{cor:maxnormbound_v1} continues to hold without the assumption $\sigma \lesssim \sigma_{\min}$.  To the best of our knowledge, this is the first entrywise $\|\cdot \|_{\max}$ convergence guarantee for tensor denoising in the low Tucker rank setting.  In \citet{wang_implicit_2021} the authors study the entrywise convergence of a Riemannian algorithm for tensor completion, but they focus on the noiseless setting, and, as such, they do not need to consider additional complications arising from additive noise, so our results are not directly comparable.  Similarly, \citet{cai_nonconvex_2022} consider the entrywise convergence of their gradient descent algorithm. %but they assume a low CP-rank, which is more stringent than the low Tucker rank setting.  Moreover, 
Their results only hold with probability $1 -\delta$ for an arbitrary (but fixed) small constant $\delta$, which is weaker than our results.  

%.  Their argument hinges on a leave-one-out analysis idea as well.  However, an important note is that their leave-one-out construction considers only missingness (i.e. Bernoulli noise), and, as such, . Our analysis considers additive, heteroskedastic, subgaussian noise, and therefore requires a different type of analysis that takes into account the noise.  

%%%%%%%%%%%%%%%%%%%%%%%%%%%%%%%%%%%%

%%%%%%%%%%%%%%%%%%%%%%%%%%%%%%%%%

%{\red (Change the following title to ``Applications to Inference Tasks"?)}
%%%%%%%%%%%%%%%%%%%%%%%%%%%%%%%%%
\section{Applications to Inference Tasks}
%%%%%%%%%%%%%%%%%%%%%%%%%%%%%%%%%
\label{sec:consequences}
In this section, we discuss the consequences of our main theoretical results in the context of the inference problems discussed in \cref{sec:introapplications}.   In \cref{sec:testing}, we apply our results to testing in the tensor mixed-membership blockmodel. In \cref{sec:simultaneous}, we consider simultaneous confidence intervals. In \cref{sec:testing2}, we consider the problem of testing the equality of tensor entries.

\subsection{Testing  Membership Profiles in the Tensor Mixed-Membership Blockmodel} \label{sec:testing}
%%%%%%%%%%%%%%%%%%%%%%%%%%%%%%%%%
We now consider an application of our theoretical results to the tensor mixed-membership blockmodel. %whose analysis was initiated in \citet{agterberg_estimating_2022}.  
We say the signal tensor $\mathcal{T}$ is a \emph{tensor mixed-membership blockmodel} \citep{agterberg_estimating_2022} if
\begin{align*}
    \mathcal{T} = \mathcal{S} \times_1 \mathbf{\Pi}_1 \times_2 \mathbf{\Pi}_2 \times_3 \mathbf{\Pi}_3,
\end{align*}
where $\mathcal{S} \in \mathbb{R}^{r_1 \times r_2 \times r_3}$ is a \emph{mean tensor} and $\mathbf{\Pi}_k$ are \emph{mixed-membership matrices} satisfying
\begin{align*}
    \mathbf{\Pi}_k \in [0,1]^{p_k \times r_k}; \qquad \sum_{l=1}^{r_k} \big( \mathbf{\Pi}_k \big)_{i_k l} = 1  \text{ for all } 1 \leq i_k \leq p_k.
\end{align*}
Informally, along each mode there are $r_k$ communities, and the rows of $\mathbf{\Pi}_k$ describe the memberships of each node in each community, where the total membership for each node in each community sums to one.  This model generalizes the tensor blockmodel studied in a number of previous works, as if every row of $\mathbf{\Pi}_k$ is $\{0,1\}$ valued, one recovers the tensor blockmodel.  The identifiability of this model was studied in \citet{agterberg_estimating_2022}, who demonstrated that the existence of \emph{pure nodes} is necessary and sufficient when $\mathcal{S}$ is assumed full-rank; here a \emph{pure node} is a node $i_k$ such that $\big(\mathbf{\Pi}_k\big)_{i_k} \in \{0,1\}^{r_k}$.  

Now consider the setting that one has two particular nodes of interest $i_k$ and $i_k'$, and consider the null hypothesis
\begin{align*}
    H_0&: (\mathbf{\Pi}_k)_{i_k\cdot} = (\mathbf{\Pi}_k)_{i_k'\cdot}.
\end{align*}
In essence, $H_0$ determines whether two nodes have the same community memberships.  To test this hypothesis, we consider a test statistic partially motivated by \citet{fan_simple_2022} for the matrix setting.  Define
\begin{align*}
    \hat T_{i_k i_k'} \coloneqq \big((\uhat_k)_{i_k\cdot} - (\uhat_k)_{i_k'\cdot} \big)\t \bigg( \mathbf{\hat \Gamma}_k^{(i_k)} + \mathbf{\hat \Gamma}_k^{(i_k')} \bigg)\inv \big((\uhat_k)_{i_k\cdot} - (\uhat_k)_{i_k' \cdot} \big),
\end{align*}
where $\uhat_k$ is as in the previous section and $\mathbf{\hat \Gamma}_k^{(i_k)}$ and $\mathbf{\hat \Gamma}_k^{(i_k')}$ are the plug-in estimators from \cref{al:ci_eigenvector} (for the $i_k$'th and $i_k'$'th row of the singular vector estimates for mode $k$ respectively).

The following result establishes the asymptotic distribution of our test statistic  $\hat T_{i_k i_k'}$ when $r_k$ is fixed in $p$. 
\begin{theorem}\label{cor:testing}
Consider the tensor mixed-membership model, where $r_k$ is fixed along each mode.  Suppose that
\begin{itemize}
\item (Regularity) Each matricization of $\mathcal{S}$ has a bounded condition number.
\item (Identifiability) There is at least one pure node for each community along each mode.  
\item (Signal strength) The smallest singular value of $\mathcal{S}$ satisfies $\lambda_{\min}(\mathcal{S})^2/\sigma^2 \gg \frac{\log(p)}{p^{3/2}}.$
\item (Approximately equal community sizes) The community membership matrices $\mathbf{\Pi}_k$ satisfy $\lambda_{\min} \bigg( \mathbf{\Pi}_k\t \mathbf{\Pi}_k \bigg) \gtrsim p.$
\end{itemize}
Then:
\begin{enumerate}
    \item (Consistency under the null) Under the null hypothesis $(\mathbf{\Pi}_k)_{i_k\cdot} = (\mathbf{\Pi}_k)_{i_k'\cdot}$, it holds that
    %\begin{align*}
    $\hat T_{i_k i_k'} \to \chi^2_{r_k}$
%\end{align*}
in distribution as $p \to \infty$.
\item (Consistency against local alternatives) If it holds that $%\begin{align*}
    \frac{\lambda_{\min}(\mathcal{S}) p }{\sigma} \| \mathbf{\Pi}_{i_k\cdot} - \mathbf{\Pi}_{i_k'\cdot} \| \to \infty,$
%\end{align*}
then for any constant $C > 0$, 
%\begin{align*}
$    \p( \hat T_{i_k i_k'} > C ) \to 1.$
%\end{align*}
If instead, it holds that
\begin{align*}
    \big( (\U_k)_{i_k\cdot} - (\U_k)_{i_k'\cdot} \big)\t \bigg( \mathbf{\Gamma}^{(i_k)}_k + \mathbf{\Gamma}^{(i_k')}_{k} \bigg)^{-1}\big((\U_k)_{i_k\cdot} - (\U_k)_{i_k'\cdot} \big) \to \gamma < \infty
\end{align*}
then %
%\begin{align*}
$    \hat T_{i_k i_k'} \to \chi^2(\gamma),$
%\end{align*}
as $p \to \infty$, where $\chi^2(\gamma)$ denotes a noncentral $\chi^2$ distribution with noncentrality parameter $\gamma$. 
\end{enumerate}
\end{theorem}

Suppose that both $i_k$ and $i_k$' are pure nodes in the sense that $(\mathbf{\Pi}_k)_{i_k\cdot} = e_l$ for some basis vector $e_l \in \{0,1\}^{r_k}$.  It is straightforward to check that
\begin{align*}
    \| (\mathbf{\Pi}_k)_{i_k\cdot} - (\mathbf{\Pi}_k)_{i_k'\cdot} \| = \sqrt{2},
\end{align*}
and, hence it holds that
\begin{align*}
    \frac{\lambda_{\min}(\mathcal{S}) p}{\sigma}  \| (\mathbf{\Pi}_k)_{i_k\cdot} - (\mathbf{\Pi}_k)_{i_k'\cdot} \| &\geq  p^{1/4}\sqrt{\log(p)},
\end{align*}
which diverges.  Consequently, the test statistic consistently rejects whenever each node belongs solely to separate communities.

\begin{remark}[Signal Strength Condition]
In \citet{agterberg_estimating_2022}, the condition $\lambda_{\min}(\mathcal{S})^2/\sigma^2 \gtrsim \frac{\log(p)}{p^{3/2}}$ was shown to be sufficient for $\ell_{2,\infty}$ membership recovery when $r$ is fixed; our signal-strength condition is only slightly stronger than theirs.  %In addition, \citet{han_exact_2020} consider community detection in the tensor blockmodel, and they define the signal-strength parameter
%\begin{align*}
%    \tilde \lambda_{\min}(\mathcal{S})^2&\coloneqq \min_k\min_{l\neq l'} \| \mathcal{M}_k\big( \mathcal{S} \big)_{l\cdot} - \mathcal{M}_k\big( \mathcal{S} \big)_{l'\cdot} \|^2.
%\end{align*}
%It was discussed in \citet{agterberg_estimating_2022} that the signal-strength parameter $\tilde\lambda_{\min}(\mathcal{S})$ coincides with $\lambda_{\min}(\mathcal{S})$ up to a factor of the condition number.  In \citet{han_exact_2020} iw was demonstrated that the condition $\frac{\lambda_{\min}(\mathcal{S})^2}{\sigma^2} \gtrsim  \frac{1}{p^{3/2}}$ is required to achieve exact community detection in polynomial time.  Therefore, our assumption that $\frac{\lambda_{\min}(\mathcal{S})^2}{\sigma^2} \gg \frac{\log(p)}{p^{3/2}}$ nearly optimal up to the logarithmic term and implicit condition number dependence for .
However, in contrast to \citet{agterberg_estimating_2022}, % we allow mixed membership, meaning that each vertex can belong to a convex combination of communities, and 
our results apply to testing whether two vertices have the same community memberships,  whereas \citet{agterberg_estimating_2022} focus only on estimation.

In addition, we assume that the membership matrices satisfy $\lambda_{\min}\big( \mathbf{\Pi}_k\t \mathbf{\Pi}_k \big) \gtrsim p$. A similar condition was imposed in \citet{agterberg_estimating_2022}; implicitly this condition requires that the communities are approximately balanced. 
\end{remark}

\begin{remark}[Comparison to  Matrix Two-Sample Testing]
Our result can also be compared to Theorem 1 of \citet{fan_simple_2022}, where they prove (under a Bernoulli noise model), that their test statistic exhibits similar convergence in distribution (a similar result was obtained in \citet{du_hypothesis_2022}). Informally, their results demonstrate that 
\begin{align*}
    \sqrt{p}  \times  \mathrm{SNR} \times \| \mathbf{\Pi}_{i\cdot} - \mathbf{\Pi}_{j\cdot} \| \to \infty
\end{align*}
in order to achieve power converging to one, where their result holds for testing the $i$'th and $j$'th rows for a $p\times p$ symmetric matrix of the form $\mathbf{\Pi} \mathbf{B} \mathbf{\Pi}\t$, and $\mathrm{SNR}$ denotes a measurement of the signal-to-noise ratio.  In contrast, our results demonstrate that
\begin{align*}
    p \times \mathrm{SNR} \times \| \big(\mathbf{\Pi}_k\big)_{i_k} - \big(\mathbf{\Pi}_k\big)_{i_k'} \| \to \infty
\end{align*}
is required 
in order to achieve power converging to one.  Therefore our result yields an improvement of order $\sqrt{p}$ in the local power of our test statistic.  
%Therefore, since $\lambda_{\min}(\mathcal{S})^2/\sigma^2$ can be viewed as the analogue of $\theta$ for the subgaussian noise setting, our result yields an improvement of $\sqrt{p}$ in the power of our test statistic, which showcases the effect of the third-order structure on this problem.
\end{remark}

%As a special case of our result, we note that if there are exac
%
%We note that our result continues to hold if we instead assume that each $\mathbf{B}^{(i)}$ is a convex combination of $r$ distinct $\mathbf{B}^{(k)}$'s. Define the $L \times r$ matrix $\mathbf{\tilde \Pi} \in [0,1]^{L \times r}$ whose $l,k$ entry corresponds to the coefficient for the $l$'th network for the $k$'th distinct $\mathbf{B}$ matrix.  If we also assume that $\lambda_{\min}\big(\mathbf{\tilde \Pi}\t \mathbf{\tilde \Pi}\big) \gtrsim L$, then \cref{cor:testing} continues to hold without any additional modification.  Furthermore, our results also apply to the setting of testing if two distinct networks have the same $\mathbf{B}$ matrices, where now we need only use the third mode instead of the first two modes.  The result continues to hold in this setting as well.  

%%%%%%%%%%
%\subsubsection{Relationship to Prior Work}

%%%%%%%%%%%%%%%%%%%%%%%%%%%%%%%%%
\subsection{Simultaneous Confidence Intervals} \label{sec:simultaneous}
%%%%%%%%%%%%%%%%%%%%%%%%%%%%%%%%%

In many applications, one is often interested in more than just one single entry of $\mathcal{T}$, and instead one may wish to design simultaneous confidence intervals for multiple entries in a small localized region of the tensor; for example, in image denoising, one may be interested in particular collections of entries that may correspond to regions of interest in the underlying image.  Na\"ively applying \cref{thm:civalidity_v1} to all these entries will not result in confidence intervals that are simultaneously valid, since there may be a correlation between entries. In fact, the proof of \cref{thm:asymptoticnormalityentries_v1} shows that if two entries contain the same indices (e.g., the index $\{i,j,k\}$ and the index $\{i',j,k\}$), then they will be highly correlated as their leading terms will depend on the same rows of the matricizations of $\mathcal{Z}$.  Therefore, in order to obtain simultaneous confidence intervals, one may need to correct for the covariance arising due to the close proximity of the entries of interest.

To formalize this problem, let $J \subset [p_1] \times [p_2] \times [p_3]$ denote an index set.  When $|J|$ is sufficiently small relative to $p$, we can still obtain valid simultaneous confidence intervals for the vector $\mathrm{Vec}(\mathcal{\hat T}_J)$, where $\mathcal{\hat T}_J$ denotes the entries of $\mathcal{\hat T}$ corresponding to indices in $J$.   \cref{al:simultaneousci} describes an approach to obtain simultaneous confidence intervals for $\mathrm{Vec}(\mathcal{\hat T}_J)$. 
The following result establishes the validity of this procedure.  The more general result with $\mu_0$ and $\kappa$ permitted to grow can be found in \cref{sec:applicationproofs}.  % For convenience we will assume that $\mu_0, \kappa = O(1)$.  

  \begin{algorithm}[t]
	\caption{Simultaneous Confidence Intervals for $\mathcal{T}_{J}$}
	\begin{algorithmic}[1]
		\State Input:  Singular vector estimate $\uhat_k$ and tensor estimate $\mathcal{\hat T}$ from \cref{al:tensor-power-iteration}, coverage level $1 - \alpha$, and index set $J$.
		
			\State Let $\mathbf{\hat V}_k$ denote the $r_k$ right singular vectors of the matrix
		\begin{align*}
		    \mathcal{M}_k\big(  \mathcal{\tilde T} \big) \bigg( \big(\uhat_{k+1} \uhat_{k+1}\t \big) \otimes \big(\uhat_{k+2} \uhat_{k+2}\t \big) \bigg).
		\end{align*}
	\State Define $\mathcal{\hat Z} = \mathcal{\tilde T} - \mathcal{\hat T}$. Set
		\begin{align*}
		   \mathbf{\hat Z}_k &\coloneqq \mathcal{M}_k( \mathcal{\hat Z}).
		\end{align*}
		\State Let $\hat \Sigma^{(i)}_1$ denote the $p_{-1} \times p_{-1}$ diagonal matrix with entries consisting of the squared values of $e_i\t \mathbf{\hat Z}_1$, and define $\hat \Sigma_2^{(j)}$ and $\hat \Sigma_3^{(k)}$ similarly.  
		\State 
		Define the matrix $\hat S_J$ via
\begin{align*}
    \big(\hat S_{J} \big)_{\{i,j,k\},\{i',j',k'\}} &\coloneqq \mathbb{I}_{\{i= i'\}} e_{(j-1)p_3 + k}\t  \mathbf{\hat V}_1 \mathbf{\hat V}_1\t \hat \Sigma_1^{(i)} \mathbf{\hat V}_1 \mathbf{\hat V}_1\t e_{(j'-1)p_3 + k'} \\
    &\quad + \mathbb{I}_{\{j= j'\}} e_{(k-1)p_1 + i}\t  \mathbf{\hat V}_2 \mathbf{\hat V}_2\t\hat \Sigma_2^{(j)} \mathbf{\hat V}_2 \mathbf{\hat V}_2\t e_{(k'-1)p_3 + i'}  \\
    &\quad + \mathbb{I}_{\{k= k'\}} e_{(i-1)p_2 + j}\t  \mathbf{\hat V}_3 \mathbf{\hat V}_3\t \hat \Sigma_3^{(k)} \mathbf{\hat V}_3 \mathbf{\hat V}_3\t e_{(i'-1)p_2 +j'}.
\end{align*}
\State Compute the $1 - \alpha$ quantile $\tau_{\alpha}$ of a $\chi^2_{|J|}$ random variable, and construct the ball $\mathcal{B}_{1-\alpha} \coloneqq \{ z: \|z\|^2 \leq \tau_{\alpha} \}$.
\State Output confidence interval
\begin{align*}
    \mathrm{C.I.}^{\alpha}_{J}(\mathcal{\hat T}) &\coloneqq \mathrm{Vec}\big( \mathcal{\hat T}_J\big) + \big( \hat S_J\big)^{1/2} \mathcal{B}_{1-\alpha} = \{\mathrm{Vec}\big( \mathcal{\hat T}_J\big) + \big( \hat S_J\big)^{1/2} z: z \in \mathcal{B}_{1-\alpha}\}.
\end{align*}
		\end{algorithmic}\label{al:simultaneousci}
\end{algorithm}

\begin{theorem}[Simultaneous Inference for Sparse Collections of Entries] \label{thm:simultaneousinference_v1}
Instate the conditions in \cref{thm:eigenvectornormality_v1}, and suppose that 
\begin{align*}
    r^{3/2} \sqrt{\log(p)} \lesssim p^{1/4}.
\end{align*}
Let $J$ be a given index set with $|J| = o(p^{1/6})$.  Define the $|J|\times |J|$ matrix $S_J$ via
\begin{align*}
    (S_J)_{\{i,j,k\},\{i',j',k'\}} &\coloneqq  \mathbb{I}_{\{i= i'\}} e_{(j-1)p_3 + k}\t  \mathbf{V}_1 \mathbf{V}_1\t \Sigma_1^{(i)} \mathbf{ V}_1 \mathbf{ V}_1\t e_{(j'-1)p_3 + k'} \\
    &\quad + \mathbb{I}_{\{j= j'\}} e_{(k-1)p_1 + i}\t  \mathbf{ V}_2 \mathbf{ V}_2\t \Sigma_2^{(j)} \mathbf{ V}_2 \mathbf{ V}_2\t e_{(k'-1)p_3 + i'}  \\
    &\quad + \mathbb{I}_{\{k= k'\}} e_{(i-1)p_2 + j}\t  \mathbf{ V}_3 \mathbf{ V}_3\t \Sigma_3^{(k)} \mathbf{ V}_3 \mathbf{ V}_3\t e_{(i'-1)p_2 + j'},
\end{align*}
where $\Sigma_k^{(m)}$ is as in \cref{thm:asymptoticnormalityentries_v1}.  Suppose $S_J$ is invertible, and let $s^2_{\min}$ denote its smallest eigenvalue.  Suppose that 
\begin{align*}
    s_{\min}/\sigma &\gg \max\bigg\{ |J|^{3/2}  \frac{r^{3/2} \sqrt{\log(p)}}{p^{3/2}}, |J|^{3/2} \frac{ r^2 \log(p)}{(\lambda/\sigma) \sqrt{p}}, |J| \frac{  r^{3/2}  \log^{3/4}(p)}{(\lambda/\sigma)^{1/2} p^{3/4}}, \\
      &\qquad \qquad \qquad  |J|^{1/6} \frac{ r^{3/2} \sqrt{\log(p)}}{p^{4/3}}, |J|^{1/6} \frac{r^{7/6} \log^{5/6}(p)}{(\lambda/\sigma)^{1/3} p^{5/6}}\bigg\}.
\end{align*}
Let $\mathrm{C.I.}_J^{\alpha}(\mathcal{\hat T})$ denote the output of \cref{al:simultaneousci}.  Then it holds that
\begin{align*}
 \p\bigg\{ \mathrm{Vec}(\mathcal{ T}_J) \in  \mathrm{C.I.}^{\alpha}_{J}(\mathcal{\hat T}) \bigg\} = 1- \alpha - o(1).
\end{align*}
\end{theorem}

Suppose $\mathcal{T}$ is the constant tensor, $\mathcal{Z}$ consists of homoskedastic noise, and $J$ consists of two index sets of the form $\{i,j,k\}$ and $\{i',j,k\}$ so that $j$ and $k$ are shared. Then $S_J$ is simply the matrix
\begin{align*}
    \sigma^2 \begin{pmatrix} \frac{3}{p^2} & \frac{2}{p^2} \\ \frac{2}{p^2} & \frac{3}{p^2} \end{pmatrix},
\end{align*}
which has smallest eigenvalue $\sigma^2\frac{1}{p^2}$.  Consequently,  $s_{\min}/\sigma = p\inv$, and the condition on $s_{\min}/\sigma$ holds very straightforwardly.  More general bounds on $s_{\min}$ may not be possible unless $\mathcal{T}$ has additional structure.

\begin{remark}[Signal Strength Condition]
The condition on $s_{\min}$ can be viewed as both a signal-strength requirement on the magnitudes of the corresponding entries of the rows of the $\mathbf{V}_k$'s as well as a condition governing how much the indices overlap, with more overlap requiring more minimum signal.  Note that if $J$ consists of $|J|$ disjoint index sets, then the condition on $s_{\min}$ is comparable to the condition in \cref{thm:civalidity_v1} accounting for the size of the set $J$.   If some terms share an index, then the condition on $s_{\min}$ governs how many terms can be shared.  
\end{remark}

 \begin{remark}[Allowable Size of $|J|$]
 While the condition on $s_{\min}$ is hard to parse, consider the setting that $r = O(1)$, and that $\lambda/\sigma \asymp p^{3/4} \polylog(p)$.  Then our condition translates to the requirement
\begin{align*}
   s_{\min}/\sigma \gg \max\bigg\{  |J|^{3/2} \frac{1}{ p^{5/4} \polylog(p) }, |J| \frac{ 1}{p^{9/8} \polylog(p)},   |J|^{1/6} \frac{1}{p^{13/12}\polylog(p)}\bigg\},
\end{align*}
which demonstrates a tradeoff in the size of the index set $|J|$.  Consider the setting that $|J| = O(p^{1/6 - \eps})$.  Then, ignoring logarithmic terms, this condition translates to the requirement %{\red (Change to 
$$s_{\min}/\sigma
     \gg \max\bigg\{ p^{-\eps-23/24} , p^{-\eps/6-19/18}\bigg\}$$%}
%\begin{align*}
%    s_{\min}/\sigma
%     &\gg \max\bigg\{  \frac{p^{-\eps}}{p^{23/24}} , \frac{p^{-\eps/6}}{p^{19/18}}\bigg\},
%end{align*}
From this, we see that we need at least that $\eps \geq \frac{1}{24}$ since the largest eigenvalue of $S_J/\sigma^2$ is of order at most $\frac{1}{p^2}$ by incoherence.  
However, as $\eps$ increases (i.e., $|J|$ gets smaller), we see that we require less signal strength in each row of $\mathbf{V}_1$, $\mathbf{V}_2$, and $\mathbf{V}_3$.  Therefore, smaller index sets require less signal strength.
 \end{remark}

%In the proof of \cref{thm:simultaneousinference}, we show that the corresponding population counterpart of $\hat S_J$ is invertible with smallest eigenvalue bounded below by the smallest magnitude of the entries, which is what yields our condition of the signal strength.  

%%%%%%%%%%%%%%%%%%%%%%%%

%%%%%%%%%%%%%%%%%%%%%%%%%%%%%%%%%
\subsection{Testing Equality of Entries} 
\label{sec:testing2}
%%%%%%%%%%%%%%%%%%%%%%%%%%%%%%%%%
In \cref{thm:simultaneousinference_v1} the results depend on the minimum eigenvalue $s_{\min}$ of $S_J$, which may be hard to interpret in general.  However, in many settings, one may only be interested in two entries of the underlying tensor.  Therefore, in this section, we consider the null hypothesis
\begin{align*}
    H_0: \mathcal{T}_{ijk} = \mathcal{T}_{i'j'k'}
\end{align*}
for some prespecified indices $\{ijk\}$ and $\{i'j'k'\}$.  By modifying the proof of \cref{thm:simultaneousinference_v1}, we can establish the  consistency of a procedure using a plug-in estimate for the variance, with the general result available in \cref{sec:applicationproofs}.   The procedure is summarized in \cref{al:entrytesting}.

\begin{algorithm}[t]
	\caption{Confidence Intervals for $\mathcal{T}_{ijk} - \mathcal{T}_{i'j'k'}$}
	\begin{algorithmic}[1]
		\State Input:  Singular vector estimate $\uhat_k$ and tensor estimate $\mathcal{\hat T}$ from \cref{al:tensor-power-iteration}, coverage level $1 - \alpha$.
		
			\State Let $\mathbf{\hat V}_k$ denote the $r_k$ right singular vectors of the matrix
		\begin{align*}
		    \mathcal{M}_k\big(  \mathcal{\tilde T} \big) \bigg( \big(\uhat_{k+1} \uhat_{k+1}\t \big) \otimes \big(\uhat_{k+2} \uhat_{k+2}\t \big) \bigg).
		\end{align*}
	\State Define $\mathcal{\hat Z} = \mathcal{\tilde T} - \mathcal{\hat T}$. Set
		\begin{align*}
		   \mathbf{\hat Z}_k &\coloneqq \mathcal{M}_k( \mathcal{\hat Z}).
		\end{align*}
	\State Let $\hat \Sigma^{(i)}$ denote the $p_2 \times p_3$ diagonal matrix with entries consisting of the squared values of $e_i\t \mathbf{\hat Z}_1$, and define $\hat \Sigma^{(j)}$ and $\hat \Sigma^{(k)}$ similarly.  \State Set
\begin{align*}
    \hat s^2_{ijk} &\coloneqq \sum_{a} \big( \mathbf{\hat Z}_1 \big)_{ia}^2 \big( \mathbf{\hat V}_1 \mathbf{\hat V}_1\t \big)_{a,(j-1)p_3 + k}^2 + \sum_{b} \big( \mathbf{\hat Z}_2 \big)_{jb}^2 \big( \mathbf{\hat V}_2 \mathbf{\hat V}_2\t \big)_{b,(k-1)p_1 + i}^2 \\
    &\quad + \sum_{c} \big( \mathbf{\hat Z}_3 \big)_{kc}^2 \big( \mathbf{\hat V}_3 \mathbf{\hat V}_3\t \big)_{c,(i-1)p_2 + j}^2.
\end{align*}
\State Set
\begin{align*}
    \hat s_{\{ijk\}\{i'j'k'\}}^2 &= \hat s^2_{ijk} + \hat s^2_{i'j'k'}  \\
    &\quad - \mathbb{I}_{\{i= i'\}} e_{(j-1)p_3 + k}\t  \mathbf{\hat V}_1 \mathbf{\hat V}_1\t \hat \Sigma^{(i)} \mathbf{\hat V}_1 \mathbf{\hat V}_1\t e_{(j'-1)p_3 + k'} \\
    &\quad - \mathbb{I}_{\{j= j'\}} e_{(k-1)p_1 + i}\t  \mathbf{\hat V}_2 \mathbf{\hat V}_2\t\hat \Sigma^{(j)} \mathbf{\hat V}_2 \mathbf{\hat V}_2\t e_{(k'-1)p_3 + i'}  \\
    &\quad - \mathbb{I}_{\{k= k'\}} e_{(i-1)p_2 + j}\t  \mathbf{\hat V}_3 \mathbf{\hat V}_3\t \hat \Sigma^{(k)} \mathbf{\hat V}_3 \mathbf{\hat V}_3\t e_{(i'-1)p_2 +j'}.
\end{align*}
\State Let $z_{\alpha/2}$ denote the $1 - \alpha/2$ quantile of a standard Gaussian random variable.
\State Output confidence interval
\begin{align*}
    \mathrm{C.I.}^{\alpha}_{\{ijk\},\{i'j'k'\}}( \mathcal{\hat T}) &\coloneqq \big( \mathcal{\hat T}_{ijk} - \mathcal{\hat T}_{i'j'k'} - z_{\alpha/2} \hat s_{\{ijk\}\{i'j'k'\}}, \mathcal{\hat T}_{ijk} - \mathcal{\hat T}_{i'j'k'} + z_{\alpha/2} \hat s_{\{ijk\}\{i'j'k'\}}\big).
\end{align*}
		\end{algorithmic}
		\label{al:entrytesting}
\end{algorithm}

\begin{theorem}\label{thm:entrytesting_v1}
Instate the conditions of \cref{thm:eigenvectornormality_v1}, and suppose that $r^{3/2} \sqrt{\log(p)} \lesssim p^{1/4}$.  
Suppose that  \begin{align*}
 \min\bigg\{ \| e_{(j-1)p_3 + k }\t \mathbf{V}_1\|^2, \|& e_{(j'-1)p_3 + k'}\t \mathbf{V}_1\|^2, \| e_{(k-1)p_1 + i}\t \mathbf{V}_2 \|^2\\
 &, \| e_{(k'-1)p_1 + i}\t \mathbf{V}_2 \|^2, \| e_{(i-1)p_2 + j}\t \mathbf{V}_3 \|^2, \| e_{(i'-1)p_2 + j'}\t \mathbf{V}_3 \|^2 \bigg\} \\
    &\gg \max\bigg\{ \frac{r^{3} \log(p)}{p^3} , \frac{r^{3} \log^{3/2}(p)}{(\lambda/\sigma)p^{3/2}}\bigg\} .
    \end{align*} 
    Let $\mathrm{C.I.}^{\alpha}_{\{ijk\},\{i'j'k'\}}( \mathcal{\hat T}) $ denote the output of \cref{al:entrytesting}.  
Then it holds that
\begin{align*}
   \p\bigg\{ \mathcal{T}_{ijk} - \mathcal{T}_{i'j'k'} \in   \mathrm{C.I.}^{\alpha}_{\{ijk\},\{i'j'k'\}}( \mathcal{\hat T}) \bigg\} =  1- \alpha - o(1).
\end{align*}
%\begin{align*}
%    \frac{\mathcal{\hat T}_{ijk} - \mathcal{\hat T}_{i'j'k'} - \big( \mathcal{T}_{ijk} - \mathcal{T}_{i'j'k'}\big)}{\hat s_{\{ijk\}\{i'j'k'\}}} \to N(0,1).
%\end{align*}
\end{theorem}

\begin{remark}[Index Overlap and Correlation]
\cref{thm:entrytesting_v1}  demonstrates how the closeness of indices induces correlation via the additional correction terms required in $\hat s^2_{\{ijk\}\{i'j'k'\}}$ in \cref{al:entrytesting}.  For example, if one wishes to consider uncertainty quantification for the entries $\{i,j,k\}$ and $\{i',j,k\}$ simultaneously, the asymptotic variance in \cref{thm:entrytesting_v1} will have additional correlation since these two entries both share the indices $j$ and $k$.  Consequently, the closeness of indices corresponds to higher correlation, with the strength of correlation depending on how corresponding entries of right singular vectors interact.  
\end{remark}

%%%%%%%%%%%%%%%%%%%%%%%%%%%%%%%%%%%%%%%%%%%%%%%%%%%%
\section{Related Work}
\label{sec:relatedwork}
%%%%%%%%%%%%%%%%%%%%%%%%%%%%%%%%%
A number of authors have obtained theoretical results for tensor data under various structured models. Under  the Tucker low-rank model, \citet{zhang_tensor_2018} study the statistical and computational limits of estimation,  \citet{luo_sharp_2021} provide sharp  perturbation bounds for the \texttt{HOOI} algorithm,  \citet{zhang_optimal_2019} consider a version of Tucker decomposition where some modes have additional sparsity structure, and \citet{han_optimal_2022} consider a general framework for estimating Tucker low-rank tensors.  The works \citet{richard2014statistical,auddy_estimating_2022}, and \citet{huang_power_2022} consider the special case where the underlying tensor is rank one, with the latter focusing on the convergence of the power iteration algorithm, and \citet{auddy_estimating_2022} considering heavy-tailed errors.  Under the CP low-rank model, general perturbation bounds have been developed in \citet{auddy_perturbation_2023}, and \citet{han_tensor_2023} consider probabilistic bounds for their proposed algorithm.  Both \citet{zhou_optimal_2022} and \citet{cai_provable_2022} consider the low-rank tensor train model, and \citet{hao_sparse_2020} considers a setting where there are sparse corruptions.  Finally, a series of works  have considered clustering in the tensor blockmodel \citep{han_exact_2020,luo_tensor_2022,wu_general_2016,chi_provable_2020,wang_multiway_2019} or generalizations thereof \citep{agterberg_estimating_2022,hu_multiway_2022,lyu_optimal_2022,lyu_optimal_2022-1,jing_community_2021,hu_multiway_2022}.

These previous works have primarily focused on estimation guarantees in, for example, the $\sin\Theta$ distance and theoretical results on uncertainty quantification or distributional theory are comparatively lacking.
Perhaps the most related work is in \citet{cai_uncertainty_2022}, in which the authors consider uncertainty quantification for noisy tensor completion for tensors with low CP-rank.  However, in \citet{cai_uncertainty_2022}, they assume that the underlying tensor and noise are supersymmetric.  On the other hand, our results require independent noise (i.e., absence of symmetries), but our results allow for a general Tucker low-rank structure.   Therefore,  \cref{thm:asymptoticnormalityentries_v1} is not directly comparable to the results of \citet{cai_uncertainty_2022}, but our results complement theirs by generalizing to a broader model class and filling out the picture to the asymmetric setting.  

In addition, as discussed in \cref{sec:entrywiseinference}, the related work \citet{xia_inference_2022} considers statistical inference for tensors under homoskedastic Gaussian noise.  Besides the entrywise distribution of rank-one tensors, they also establish confidence regions for the error metric $\|\sin\Theta(\uhat_k,\U_k)\|^2_F$, which corresponds to a ``coarse'' confidence region for $\U_k$.  In contrast to this work, we establish fine-grained confidence regions for $(\uhat_k)_{m\cdot}$, and our results hold under heteroskedastic subgaussian noise.  Moreover, our proof techniques are significantly different from \citet{xia_inference_2022}, which rely heavily on the rotational invariance of the Gaussian distribution.  Finally, \citet{huang_power_2022} establish an asymptotic theory for the low CP-rank tensor signal-plus-noise model, and they use these results to obtain confidence intervals for linear functionals of the signals.  Similar to \citet{xia_inference_2022}, their analysis relies on the assumption of homoskedastic Gaussian noise.

In this work, we also study the entrywise convergence of the \texttt{HOOI} algorithm.  The previous work \citet{wang_implicit_2021} considers the $\|\cdot\|_{\max}$ convergence of \cref{al:tensor-power-iteration} in the noiseless tensor completion setting.  Our work is not directly comparable as we focus on the fully observed noisy setting.  Similarly,  \citet{cai_nonconvex_2022} provide bounds for their procedure to estimate general CP-rank tensors and our results are not directly comparable as they assume symmetry and allow for missingness.  In addition, we study the \texttt{HOOI} procedure, a ubiquitous algorithm for computing the tensor SVD, whereas  \citet{cai_nonconvex_2022} study a more specific gradient descent procedure for their problem.

%Since any low CP-rank tensor is also a low Tucker rank tensor, our results also hold for low CP-rank tensors, albeit without being able to estimate the individual components directly, but rather their underlying subspaces.  

Our work is closely related to that of \citet{agterberg_estimating_2022}, who study estimation in the tensor mixed-membership model as well as provide general $\ell_{2,\infty}$ perturbation bounds for tensor denoising.  While our proofs are closely related to their proofs insofar as we use their leave-one-out constructions,  \citet{agterberg_estimating_2022} focus on providing perturbation bounds, whereas we focus on distributional theory and uncertainty quantification.  In addition, we use several of their intermediate results to establish the validity of our test procedure in \cref{sec:testing}.  Beyond \citet{agterberg_estimating_2022}, our work is also related to leave-one-out analyses for matrix and tensor data, such as \citet{abbe_entrywise_2020,cai_nonconvex_2022,cai_subspace_2021,cai_uncertainty_2022,yan_inference_2021}.

\section{Numerical Simulations}

\label{sec:simulations}

In this section, we conduct numerical simulations for our proposed procedures.  In every simulation we run $2000$ independent Monte Carlo iterations.  \textcolor{black}{For all simulations, we fix the significance level (Type I error rate) at $\alpha = 0.05$.}

%\subsection{Distributional Theory}

\ \\ 
\noindent    
\textbf{Setup:} 
We design our simulation as follows.  First, we generate our tensor by drawing a mean tensor $\mathcal{S} \in \mathbb{R}^{r \times r \times r}$ with independent Gaussian entries, and then drawing $\mathbf{\Pi}_1,\mathbf{\Pi}_2$ and $\mathbf{\Pi}_3 \in [0,1]^{p \times r}$ independently from a Dirichlet distribution with all parameters set to one. We then form the signal tensor $\mathcal{T} = \mathcal{S} \times_1 \mathbf{\Pi}_1 \times_2 \mathbf{\Pi}_2 \times_3 \mathbf{\Pi}_3$.  Note that this procedure guarantees that $\mu_0 = O(1)$ with high probability.  Finally, we manually set the smallest singular value of $\mathcal{T}$ to be $\lambda = 1$.  This procedure is done once for each $p$.

To generate the noise, for a given value $\lambda/\sigma$ (where due to our parameterization, $\lambda/\sigma = \sigma\inv$), we first draw the standard deviations according to $U(0,\sigma)$, and then we generate the noise tensor $\mathcal{Z}_{ijk} \sim  N(0,\sigma^2_{ijk})$.  The standard deviations are drawn once for each value of $\lambda/\sigma$, but the noise tensor is redrawn at each Monte Carlo iteration. \textcolor{black}{In this section we only present coverage rates, but more general simulation results can be found in the appendix.}   %{\red ($N(0,\sigma_{ijk})$ seems a bit strange. Can you double-check?)} \textcolor{ForestGreen}{Agreed -- it is $\sigma^2_{ijk}$.}
\\ \ \\ \noindent
\textbf{Empirical Coverage Rates}: We now consider the approximate validity of Algorithms \ref{al:ci_eigenvector} and \ref{al:ci_entries} as demonstrated by \cref{thm:civalidity_v1} and \cref{thm:civalidity2_v1}, respectively.  We use the same setup as the previous setting, only in both cases do we take $r = 4$.  In \cref{fig:coveragerates} we display the empirical coverage rates and standard deviations for both  $(\U_{1})_{1\cdot}$ and $\mathcal{T}_{111}$, where we use the plug-in estimate $\mathbf{\hat \Gamma}_{1}^{(1)}$ and $\hat s_{111}$.

 \begin{table}[htb]
    \centering
    \begin{tabular}{|c|c | c | c |}
    \hline 
\multicolumn{4}{|c|}{Coverage Rates for $(\U_1)_{1\cdot}$}   \\ [1ex]
\hline 
  $p$ & $ \lambda/\sigma = p^{\gamma}$ & Mean  & Std  \\
  \hline 
  100 & $\gamma = 3/4$ & 0.984 &   0.0028 
\\ 
\hline 
  
 150 & $\gamma =  3/4$ &  1.000  & 0.0000 
 \\
\hline 
  
 100 & $\gamma = 7/8$ & 0.939   & 0.0054 
\\
\hline 
  
 150 & $\gamma = 7/8$ & 0.991
  &   0.0022
 \\
 \hline 
  
 100  & $\gamma = 1$ &  0.883 
&  0.0072 
\\\hline 
  
 150  & $\gamma =  1$ & 0.946  & 0.0051
\\ \hline 
  
\end{tabular} 
\hspace{20pt}
\begin{tabular}{|c|c | c | c |}
      \hline 
\multicolumn{4}{|c|}{Coverage Rates for $\mathcal{T}_{111}$}   \\ [1ex]
\hline 
  $p$ & $ \lambda/\sigma = p^{\gamma}$ & Mean  & Std  \\
  \hline 
  100 & $\gamma = 3/4$ & 0.949 &  0.0049 
\\ 
\hline 
  
 150 & $\gamma =  3/4$ &  0.943 &   0.0052 
\\
\hline 
  
 100 & $\gamma = 7/8$ & 0.949   & 0.0049
\\
\hline 
  
 150 & $\gamma = 7/8$ &0.938 &  0.0054 
 \\
 \hline 
  
 100  & $\gamma = 1$ & 0.938 & 0.0054
\\\hline 
  
 150  & $\gamma =  1$ &0.955 &   0.0047
\\ \hline 
  
\end{tabular}
    \caption{Empirical coverage rates for confidence intervals for both  $(\U_1)_{1\cdot}$ (left) and $\mathcal{T}_{111}$ (right) using \cref{al:ci_eigenvector} and \cref{al:ci_entries} respectively for varying $p$ and $\lambda/\sigma$.   The column ``Mean'' represents the empirical probability of coverage averaged over $200$ Monte Carlo iterations, and the column ``Std'' denotes the standard deviation of this coverage rate.}
    \label{fig:coveragerates}
\end{table}
% \begin{table}[]
%     \centering
 
%     \caption{Empirical coverage rates (Mean) and standard deviation (Std) for confidence intervals $\mathcal{T}_{111}$ using \cref{al:ci_entries} with $200$ Monte Carlo iterations.}
%     \label{tab:my_label}
% \end{table}

\ \\
\textbf{Tensor Mixed-Membership  Blockmodel}: We now consider applying \cref{cor:testing} to testing if the first two rows of $\mathbf{\Pi}_1$ are equal.  To generate our tensor mixed-membership blockmodel, we use the same procedure as in the previous simulations with  $p = 150$ and $r = 3$, only we also manually guarantee that there are pure nodes for each community (as required in \cref{cor:testing}), and we manually set the first two nodes via $\big(\mathbf{\Pi}_1\big)_{1\cdot} = \{.2,.6,.2\}$, and $\big(\mathbf{\Pi}_1\big)_{2\cdot} = \{.2,.6-\eps/2,.2+\eps/2\}$, where $\eps = \|(\mathbf{\Pi}_1)_{1\cdot} - (\mathbf{\Pi}_1)_{2\cdot} \|_1$ represents a local departure from the null hypothesis.  In \cref{tab:power} we display the empirical size and power of our test at $\alpha = .05$, with each column representing the empirical power (size for the first column) for varying values of $\eps$ \textcolor{black}{under both Gaussian (left) and Bernoulli (right) noise. The Bernoulli noise is generated by drawing the $\mathbf{\Pi}_k$ matrices the same as in the Gaussian case, but by setting the underlying mean tensor $\mathcal{S}$ to have entries within $\{.2,.3,.4,.5,.6,.8,.9\}$ (recycled), and then having entries rescaled by $\rho$, with smaller $\rho$ corresponding to sparser tensors (and hence weaker signal strength). Observe that for the Gaussian setting the power increases to one as $\eps$ increases, and it increases at a slower rate for smaller values of $\lambda/\sigma$.  Similarly, while the Bernoulli model exhibits weaker power, it still improves as the tensor becomes denser and $\eps$ increases.}
\begin{table}[htb]
    \centering
    \begin{tabular}{|c||c | c | c |c|c|} 
    \hline 
    \multicolumn{6}{|c|}{Size and Power of Test Statistic $\hat T_{i_k i_k'}$ (Gaussian)} \\ [1ex] \hline
 &   \multicolumn{5}{|c|}{$\|(\mathbf{\Pi}_1)_{1\cdot} - (\mathbf{\Pi}_1)_{2\cdot} \|_1 = \eps$}\\[0.5ex] \hline 
$ \gamma $  &$\eps = $ 0 &0.05  &  0.1 & 0.15 &  0.2  \\ \hline \hline
     $ 3/4$ & 0.052 & 0.120& 0.725 &0.977 & 0.811\\ \hline 
$7/8$ & 0.056& 0.304& 1.000& 1.000 &0.999 \\ \hline 
 $1$ &   0.046& 0.814& 1.000& 1.000 &1.000 \\ \hline 
  \end{tabular} \hspace{10pt}
   \begin{tabular}{|c||c | c | c |c|c|} 
    \hline 
    \multicolumn{6}{|c|}{Size and Power of Test Statistic $\hat T_{i_k i_k'}$ (Bernoulli)} \\ [1ex] \hline
 &   \multicolumn{5}{|c|}{$\|(\mathbf{\Pi}_1)_{1\cdot} - (\mathbf{\Pi}_1)_{2\cdot} \|_1 = \eps$}\\[0.5ex] \hline 
$\rho $  &$\eps = $ 0 &0.05  &  0.1 & 0.15 &  0.2  \\ \hline \hline
     $.8 $ & 0.056 &0.072& 0.123 &0.216& 0.352\\ \hline 
$ .9$ &0.056 &0.076& 0.146 &0.248& 0.442 \\ \hline 
 $ 1$ &  0.059 &0.086 &0.167& 0.308 &0.500 \\ \hline 
  \end{tabular}
      \caption{Empirical power (first column $=$ size) of testing the null hypothesis $H_0: (\mathbf{\Pi}_1)_{1\cdot} = (\mathbf{\Pi}_2)_{2\cdot}$ under varying signal to noise ratios and local alternatives  (as quantified via $\|(\mathbf{\Pi}_1)_{1\cdot} - (\mathbf{\Pi}_1)_{2\cdot} \|_1 = \eps$).  The left hand table denotes Gaussian noise with  $\lambda/\sigma = p^{\gamma}$ with the leftmost column denoting different values of $\gamma$.  The right hand table denotes Bernoulli noise with the leftmost column conidering varying levels of sparsity $\rho$.       
      }
    \label{tab:power}
\end{table}
\\ \ \\ \noindent
\textbf{Entrywise Testing}:
In \cref{fig:entrywiseci} we examine empirical coverage rates for $\mathcal{T}_{111} - \mathcal{T}_{112}$ and $\mathcal{T}_{111} - \mathcal{T}_{122}$ using \cref{al:entrytesting} with varying $p$ and noise levels.  As before we focus on the setting of $r = 4$.  By \cref{thm:entrytesting_v1}, the confidence intervals are asymptotically valid, so we display both the empirical coverage rate (``Mean'') and empirical standard deviation (``Std'').  \textcolor{black}{In the appendix we also plot the associated joint distribution of $\hat S_J^{-1/2} \big ( \mathcal{T}_{J} - \mathcal{T}_J \big)$ with $J = \{111,112\}$ and $J = \{111,122\}$.}
\begin{table}[htb]
    \centering
    \begin{tabular}{|c|c | c | c |}
    \hline 
\multicolumn{4}{|c|}{Coverage Rates for $\mathcal{T}_{111}- \mathcal{T}_{112}$}   \\ [1ex]
\hline 
  $p$ & $ \lambda/\sigma = p^{\gamma}$ & Mean  & Std  \\
  \hline

  100 & $\gamma = 3/4$ & 0.984&  0.0028 
\\ 
\hline 
  
 150 & $\gamma =  3/4$ &   0.991  
   &  0.0021
\\
\hline 
  
 100 & $\gamma = 7/8$ &   0.990 
 & 0.0022
\\
\hline 
  
 150 & $\gamma = 7/8$ & 0.993 
  &   0.0019
 \\
 \hline 
  
 100  & $\gamma = 1$ &  0.986   
&  0.0027
\\\hline 
  
 150  & $\gamma =  1$ &0.995  & 0.0017 
\\ \hline 
  
\end{tabular} 
\hspace{20pt}
\begin{tabular}{|c|c | c | c |}
      \hline 
\multicolumn{4}{|c|}{Coverage Rates for $\mathcal{T}_{111} - \mathcal{T}_{122}$}   \\ [1ex]
\hline 
  $p$ & $ \lambda/\sigma = p^{\gamma}$ & Mean  & Std  \\
  \hline 
  100 & $\gamma = 3/4$ & 0.973   &  0.0036 
\\ 
\hline 
  
 150 & $\gamma =  3/4$ &0.946    &   0.0051 
\\
\hline 
  
 100 & $\gamma = 7/8$ &  0.981   &  0.0031
\\
\hline 
  
 150 & $\gamma = 7/8$ & 0.943   &   0.0052
 \\
 \hline 
  
 100  & $\gamma = 1$ & 0.973    &   0.0037 
\\\hline 
  
 150  & $\gamma =  1$ & 0.937  &  0.0054
\\ \hline 
  
\end{tabular}
    \caption{Empirical coverage rates for confidence intervals for both  $\mathcal{T}_{111} - \mathcal{T}_{112}$ (left) and $\mathcal{T}_{111} - \mathcal{T}_{122}$ (right) using \cref{al:entrytesting} for varying $p$ and $\lambda/\sigma$.   The column ``Mean'' represents the empirical probability of coverage averaged over $200$ Monte Carlo iterations, and the column ``Std'' denotes the standard deviation of this coverage rate.}
    \label{fig:entrywiseci}
\end{table}

\section{Discussion}
\label{sec:discussion}
In this paper, we have studied a suite of inferential procedures for tensor data in the presence of heteroskedastic, subgaussian noise.  Our main results depend only on the structural properties of the underlying tensor, and our confidence intervals and regions are shown to be optimal for independent homoskedastic Gaussian noise.  We have also seen how our results can be used in three different concrete applications, resulting in several interesting insights for these problems.

In future work, it would be interesting to study other types of structures beyond the Tucker decomposition.  For example, can similar distributional theory and inference be obtained for tensors with low tensor train rank \citep{zhou_optimal_2022,cai_provable_2022} or low with additional sparsity structure \citep{zhang_optimal_2019}?  In our results a leading-order term of the form $\mathbf{Z}_k \mathbf{V}_k \mathbf{\Lambda}_k\inv$ manifested; it would be of interest to see if similar leading-order  terms arise in these other settings. In addition, throughout all of this paper, we have assumed knowledge of the underlying ranks  $r_k$; however, in practice, this is typically not known \emph{a priori}.  Therefore, a practical and interesting theoretical problem is to develop inferential tools when the rank is either over or under-specified.   

Throughout this work we have assumed that the noise is subgaussian, meaning that it exhibits certain tail behavior.  In many settings, such as network data, the noise satisfies other distributional assumptions (e.g., Bernoulli noise), so it would be useful to establish the statistical theory for other noise settings.  Moreover, the tools in this paper require that there are no outliers; it would be of interest to study statistical inference for models permitting outlier (e.g., heavy-tailed) noise. \textcolor{black}{In \citet{auddy_estimating_2022}, the authors showed that Tensor SVD is suboptimal when the noise has a finite $\alpha$-th moment for some $2 < \alpha < 4$. They propose an alternative procedure based on sample splitting to address this issue. Extending their analysis to provide valid inferential guarantees in this regime would be of interest.}
%In \citet{auddy_estimating_2022} the authors show that Tensor SVD is suboptimal when the noise has finite $\alpha$'th moment for some $2 < \alpha < 4$.  They propose an alternative procedure to correct for this discrepancy based on sample splitting.  It would be of interest to extend their analysis to obtain valid inferential guarantees in such a regime.}

In addition, our results in the main paper require the condition number $\kappa$ to be bounded, though we permit $\kappa$ to grow slowly in our general results stated in \cref{sec:generaltheorems}. Nonetheless, the recent work \citet{zhou_deflated_2023} proposes an intriguing subspace estimation procedure \texttt{Deflated-HeteroPCA} that is shown to be optimal in both $\ell_2$ and $\ell_{2,\infty}$ norm for unbalanced matrices, and they propose applying their algorithm as an initialization procedure for \texttt{HOOI}, showing optimal $\ell_2$ error rates that are condition number free.  It would be interesting to combine their procedure with our statistical theory to obtain theoretical guarantees that are independent of the condition number.  

Finally, our theory requires that $p_k \asymp p$ for all $k$.  The work \citet{luo_sharp_2021} establishes sharp perturbation bounds for tensors of varying order $p_k$'s, resulting in different phenomena for different regimes depending on the order of $p_k$ and the signal strength.  In many settings, one does not have $p_k \asymp p$, so it would be of theoretical and practical interest to develop statistical theory under varying  signal strengths and orders of $p_k$.

\section*{Acknowledgements}
JA’s research was partially supported by a fellowship from the Johns Hopkins Mathematical Institute of Data Science (MINDS) via its NSF TRIPODS award CCF-1934979, through the Charles and Catherine Counselman Fellowship, and through support from the Acheson J. Duncan Fund for the Advancement of Research in Statistics. ARZ's research was partially supported by the NSF Grant CAREER-2203741.

\appendix
%\addtocontents{toc}{\protect\setcounter{tocdepth}{0}}

\section{Appendix Structure and More General Theorems}
\label{sec:generaltheorems}

In this section, we  describe the structure of the rest of the appendix and the analysis.   We then state the more general results where we allow $\kappa, \mu_0$ to grow with $p$.  Our results in the main paper are readily seen to be implied by these more general results.  

The rest of the appendix is structured as follows.  First, in the subsequent subsections, we present generalizations of our distributional theory results, namely, Theorems \ref{thm:eigenvectornormality}, \ref{thm:eigenvectornormality2}, \ref{thm:asymptoticnormalityentries}, and \cref{cor:maxnormbound}, which are generalizations of Theorems  \ref{thm:eigenvectornormality_v1}, \ref{thm:eigenvectornormality2_v1}, \ref{thm:asymptoticnormalityentries_v1}, and \ref{cor:maxnormbound_v1} respectively.  Next, we present generalizations of our confidence interval and region validity, namely,  \cref{thm:civalidity2} and \cref{thm:civalidity}, generalizations of \cref{thm:civalidity2_v1} and \cref{thm:civalidity_v1} respectively. \cref{sec:analysispreliminaries} sets the stage for our analysis, including stating results and introducing notation from  \citet{agterberg_estimating_2022}.  In \cref{sec:distributionalproofs} we prove our main distributional theory results for the estimated singular vectors, and \cref{sec:entrywisedistributionaltheoryproofs} is concerned with proving the distributional theory for the entries.  In \cref{sec:ciproofs} we prove the validity of our confidence intervals and regions.  \cref{sec:applicationproofs} contains the proofs from \cref{sec:consequences}, as well as more general statements of Theorems \ref{thm:simultaneousinference_v1} and \cref{thm:entrytesting_v1}.  Finally, \cref{sec:efficiencyproof} contains a self-contained proof of both \cref{thm:efficiencyloadings} and \cref{thm:efficiency}, as well as proofs of \cref{thm:efficiency_order,thm:efficiency_order_ijk}.  Additional simulation results are presented in \cref{sec:additionalsimulations}.

\subsection{Distributional Theory and Entrywise Consistency Generalizations} \label{sec:generaldistributionaltheory}

The following result generalizes \cref{thm:eigenvectornormality_v1} to the setting where $\kappa$ and $\mu_0$ are permitted to grow. 

%^Suppose that $r \lesssim p^{1/2} $and $\lambda/\sigma \gtrsim p^{3/4} \sqrt{\log(p)}$. Suppose that $\kappa,\mu_0 = O(1)$, and let $\uhat_k^{(t)}$ denote the estimated singular vectors from the output of \texttt{HOOI} (\cref{al:tensor-power-iteration}) with $t \asymp \log(p/(\lambda/\sigma))$ iterations, initialized via \cref{al:dd}.  Suppose $\mathbf{T}_k = \mathcal{M}_k(\mathcal{T})$ has rank $r_k$ singular value decomposition $\U_k \mathbf{\Lambda}_k \mathbf{V}_k\t$.  Denote $\mathbf{Z}_k = \mathcal{M}_k(\mathcal{Z})$.  Then there exists an event  $\mathcal{E}_{\mathrm{\cref{thm:eigenvectornormality_v1}}}$ with $\p(\mathcal{E}_{\mathrm{\cref{thm:eigenvectornormality_v1}}}) \geq 1 - O(p^{-9})$ such that on this event for each $k$ it holds that
%\begin{align*}
%    \uhat_k^{(t)} (\mathbf{W}_k^{(t)})\t - \U_k &=\mathbf{Z}_k \mathbf{V}_k \mathbf{\Lambda}_k\inv + \mathbf{\Psi}^{(k)},
%\end{align*}
%where
%\begin{align*}
%    \| \mathbf{\Psi}^{(k)} \|_{2,\infty} \lesssim \frac{\sigma^2 \log(p) r \sqrt{p}}{\lambda^2} + \frac{\sigma r}{\lambda \sqrt{p}}.
%\end{align*}

\begin{theorem}[Generalization of \cref{thm:eigenvectornormality_v1}] \label{thm:eigenvectornormality}
Suppose that $\mu_0^2 r \lesssim p^{1/2}$, that $\kappa^2 \lesssim p^{1/4}$,  that $\lambda/\sigma \gtrsim \kappa p^{3/4} \sqrt{\log(p)}$, \textcolor{black}{and that $\lambda/\sigma \leq \exp(c p)$ for some small constant $c$}. Let $\uhat_k^{(t)}$ denote the estimated singular vectors from the output of \texttt{HOOI} (\cref{al:tensor-power-iteration}) with $t \asymp \textcolor{black}{\log( \frac{\lambda/\sigma}{C \kappa \sqrt{p\log(p)}}})$%\log(\kappa p/(\lambda/\sigma))$
iterations, initialized via \cref{al:dd}.  Suppose $\mathbf{T}_k = \mathcal{M}_k(\mathcal{T})$ has rank $r_k$ singular value decomposition $\U_k \mathbf{\Lambda}_k \mathbf{V}_k\t$. %Suppose further that $\lambda/\sigma \lesssim \mu_0 \kappa p \log(p)$.  
Denote $\mathbf{Z}_k = \mathcal{M}_k(\mathcal{Z})$.
Then there exists an event $\mathcal{E}_{\mathrm{Theorem} \ \ref{thm:eigenvectornormality}}$ with $\p(\mathcal{E}_{\mathrm{Theorem} \ \ref{thm:eigenvectornormality}}) \geq 1 - O(p^{-9})$ such that on this event it holds that
 \begin{align*}
     \uhat_k^{(t)} ( \mathbf{W}_k^{(t)})\t - \U_k &= \mathbf{Z}_k \mathbf{V}_k \mathbf{\Lambda}_k\inv + \mathbf{\Psi}^{(k)},
 \end{align*}
 where 
 \begin{align*}
     \big\| \mathbf{\Psi}^{(k)} \big\|_{2,\infty} &\lesssim \frac{\sigma^2 \kappa^2 \mu_0^2 \log(p) r \sqrt{p}}{\lambda^2} + \frac{\mu_0 r \kappa}{\lambda \sqrt{p}}.
 \end{align*}
 \end{theorem}

 \noindent
 The following result is needed for our entrywise distributional theory results.
  \begin{theorem} \label{cor:asymptoticnormality_projection}
  Instate the conditions of \cref{thm:eigenvectornormality}. Under the event $\mathcal{E}_{\mathrm{\cref{thm:eigenvectornormality}}}$ it holds that
\begin{align*}
    \uhat_k^{(t)} (\uhat_k^{(t)})\t - \U_k \U_k\t &= \U_k \mathbf{\Lambda}_k\inv \mathbf{V}_k \mathbf{Z}_k\t + \mathbf{Z}_k \mathbf{V}_k \mathbf{\Lambda}_k\inv \U_k\t + \mathbf{\Phi}^{(k)},
\end{align*}
where
\begin{align*}
    \big\| \mathbf{\Phi}^{(k)} \big\|_{2,\infty} &\lesssim %\frac{\sigma^2\kappa^2 \mu_0^3 \log(p) r^{3/2} \sqrt{p}}{\lambda^2}.  
    \frac{\sigma^2\kappa^2 \mu_0^3 \log(p) r^{3/2} \sqrt{p}}{\lambda^2}   + \frac{\sigma \mu_0^2 r^{3/2} \kappa}{\lambda \sqrt{p}}
\end{align*}
\end{theorem}
 
\noindent 
The following result generalizes \cref{thm:eigenvectornormality2_v1}.  

 \begin{theorem}[Generalization of \cref{thm:eigenvectornormality2_v1}]\label{thm:eigenvectornormality2}
 Instate the conditions of \cref{thm:eigenvectornormality}.    Let $\Sigma_k^{(m)}$ denote the diagonal matrix of dimension  $p_{-k} \times p_{-k}$, where the diagonal entries consist of the variances of $\mathcal{Z}_{mbc}$ if $k = 1$, $\mathcal{Z}_{amc}$ if $k = 2$, and $\mathcal{Z}_{abm}$ if $k = 3$.  Define
  \begin{align*}
     \mathbf{\Gamma}^{(m)}_k &\coloneqq \mathbf{\Lambda}_k\inv \mathbf{V}_k\t \Sigma^{(m)} \mathbf{V}_k \mathbf{\Lambda}_k\inv.
 \end{align*}
 Let $\mathcal{A}$ denote the collection of all convex sets in  $\mathbb{R}^{r_k}$, and let $Z$  be an $r_k$-dimensional Gaussian random variable with the identity covariance matrix. Then it holds that
 \begin{align*}
      \sup_{A \in \mathcal{A}} | &\p\bigg\{   \big( \mathbf{\Gamma}^{(m)}_k \big)^{-1/2}  \bigg( \uhat_k^{(t)} (\mathbf{W}_k^{(t)})\t - \U_k \bigg)_{m\cdot} \in A \bigg\} - \p \{ Z \in A \} | \\
      &\lesssim \mu_0 \frac{r^2}{p} + \frac{\sigma\kappa^3 \mu_0^2 \log(p) r^{3/2} \sqrt{p}}{\lambda}  + \frac{\mu_0 r^{3/2} \kappa }{ \sqrt{p}}.
 \end{align*}
 Therefore, asymptotic normality holds as long as $\kappa^2 \mu_0^2 r^{3/2} = o\big(p^{1/4}/\sqrt{\log(p)}\big)$. Furthermore, when $\kappa$ and $\mu_0$ are bounded, a sufficient condition for asymptotic normality is that $r = o(p^{1/6}/\log(p))$. 
 \end{theorem}
 \noindent
 Next, the following result generalizes \cref{thm:asymptoticnormalityentries_v1}.  
 
 \begin{theorem}[Generalization of \cref{thm:asymptoticnormalityentries_v1}] \label{thm:asymptoticnormalityentries}
 Instate the conditions of \cref{thm:eigenvectornormality}, and suppose further that \begin{align*}
 \kappa^2 \mu_0^2 r^{3/2} \sqrt{\log(p)} \lesssim p^{1/4}.
\end{align*}
Let $\Sigma_1^{(m)} \in \mathbb{R}^{p_2p_3 \times p_2p_3}$ be the diagonal matrix whose $(a-1)p_3 + b$'th entry is the variance of the random variable $\mathcal{Z}_{mab}$, and let $\Sigma_2^{(m)}$ and $\Sigma_{3}^{(m)}$ be defined similarly.  Assume that
\begin{align*}
   \| e_{(j-1)p_3 + k}\t &\mathbf{V}_1 \|^2 + \|e_{(k-1)p_1 + i} \mathbf{V}_2  \|^2 + \| e_{(i-1)p_2 + j} \mathbf{V}_3\|^2 \\
   &\gg \max\bigg\{\frac{ \kappa^4 \mu_0^6 r^4 \log(p)}{p^{3}} ,\frac{\sigma^2\mu_0^8 \kappa^6 r^{4} \log^2(p)}{\lambda^2 p}\bigg\}.
  \end{align*}
Define 
 \begin{align*}
    s^2_{ijk} &\coloneqq  \| e_{(j-1)p_3 + k}\t \mathbf{V}_1 \mathbf{V}_1\t \big(\Sigma^{(i)}\big)^{1/2} \|^2 + \|e_{(k-1)p_1 + i} \mathbf{V}_2 \mathbf{V}_2\t\big(\Sigma^{(j)}\big)^{1/2} \|^2 \\
    &\quad + \| e_{(i-1)p_2 + j} \mathbf{V}_3\mathbf{V}_3\t \big( \Sigma^{(k)} \big)^{1/2}\|^2.
\end{align*}
Let $Z$ denote a standard Gaussian random variable and let $\Phi$ denote its cumulative distribution function. Then it holds that
\begin{align*}
    \sup_{t\in\mathbb{R}} \bigg| \p\bigg\{ \frac{ \mathcal{\hat T}_{ijk} - \mathcal{T}_{ijk}}{s_{ijk}} \leq t \bigg\} - \Phi(t) \bigg| &= o(1).  
\end{align*}
\end{theorem}
\noindent 
Finally, the following result generalizes  \cref{cor:maxnormbound_v1}.  
\begin{theorem}[Generalization of \cref{cor:maxnormbound_v1}]\label{cor:maxnormbound}
Instate the conditions of \cref{thm:eigenvectornormality}, and suppose that
\begin{align*}
    \kappa^2 \mu_0^2 r^{3/2} \sqrt{\log(p)} \lesssim p^{1/4}.
\end{align*}
Then the following bound holds with probability at least $1 - O(p^{-6})$:
\begin{align*}
    \| \mathcal{\hat T} - \mathcal{T} \|_{\max} &\lesssim \frac{\mu_0\sigma \sqrt{r\log(p)}  }{p}  %\frac{\sigma \kappa \mu_0^3 r^{2} \sqrt{\log(p)}}{p^{3/2}} 
    + \frac{\sigma^2 \mu_0^4 \kappa^3 r^3 \log(p)}{\lambda \sqrt{p}} %\frac{\mu_0\sigma \sqrt{r\log(p)}  }{p} + 
    %\frac{\sigma \kappa \mu_0 \sqrt{r\log(p)}}{p} + \frac{\sigma^2 \mu_0^4 \kappa^3 r^3 \log(p)}{\lambda \sqrt{p}}
\end{align*}
Consequently, when the following condition holds:
\begin{align*}
   % \kappa \mu_0^2 r \lesssim \sqrt{p}; \qquad 
   \lambda/\sigma \gtrsim  \mu_0^3 \kappa^3 r^{5/2} \sqrt{p\log(p)},
\end{align*}
the bound above reduces to
\begin{align*}
     \| \mathcal{\hat T} - \mathcal{T} \|_{\max} &\lesssim \frac{ \kappa \sigma \mu_0 \sqrt{r\log(p)}}{p}.
\end{align*}
In particular, this bound holds if $\mu_0 = O(1)$ and $\kappa^3 r^{5/2} = o(p^{1/4})$.
\end{theorem}

 \subsection{Confidence Interval Validity}
The following results generalize \cref{thm:civalidity2_v1} and \cref{thm:civalidity_v1} respectively.  
 \begin{theorem}[Generalization of \cref{thm:civalidity2_v1}]\label{thm:civalidity2}
 Instate the conditions of \cref{thm:eigenvectornormality}.   Suppose also that 
 \begin{align}
     \kappa^2 \mu_0^2 r^{3/2}\sqrt{\log(p)} \lesssim p^{1/4}. \label{technicalcondition}
 \end{align}
 In addition, assume that $\mu_0 \frac{r^2}{p} = o(1)$ and that
 \begin{align*}
     \lambda/\sigma \gg  \kappa^3 \mu_0^2 \log^2(p) r^2 \sqrt{p}.
 \end{align*}
 Let $\mathrm{C.R.}^{\alpha}_{k,m}(\uhat_k)$ denote  the output of \cref{al:ci_eigenvector}.  Then it holds that
 \begin{align*}
     \p\bigg\{ \bigg( \U_k \mathbf{W}_k^{(t)} \bigg)_{m\cdot} \in \mathrm{C.R.}^{\alpha}_{k,m}(\uhat_k) \bigg\} = 1- \alpha  - o(1).
 \end{align*}
   \end{theorem}

\begin{theorem}[Generalization of \cref{thm:civalidity_v1}] \label{thm:civalidity}
Instate the conditions of \cref{thm:asymptoticnormalityentries}.  Suppose further that 
\begin{align*}
\left\|e_{(j-1) p_{3}+k}^{\top} \mathbf{V}_{1}\right\|^{2}+\left\|e_{(k-1) p_{1}+i}\t \mathbf{V}_{2}\right\|^{2}+\left\|e_{(i-1) p_{2}+j}\t \mathbf{V}_{3}\right\|^{2} &\gg  \frac{\sigma \mu_0^5 r^3 \kappa^2 \log^{3/2}(p)}{\lambda p^{3/2}}.
\end{align*}
 Let $\mathrm{C.I.}^{\alpha}_{ijk}(\mathcal{\hat T}_{ijk})$ denote the output of \cref{al:ci_entries}.  Then it holds that
\begin{align*}
    \p\bigg( \mathcal{T}_{ijk} \in  \mathrm{C.I.}^{\alpha}_{ijk}(\mathcal{\hat T}_{ijk}) \bigg) = 1- \alpha  - o(1).
\end{align*}
\end{theorem}

\section{Analysis Preliminaries} 
\label{sec:analysispreliminaries}
In this section we introduce notation and present several previous results concerning the output of \texttt{HOOI} from \citet{agterberg_estimating_2022}.    We also describe the dependencies of all of our main results.

\subsection{Initial Bounds and the Leave-One-Out Sequence} \label{sec:leaveoneout}

Our analysis is based on the theory developed in \citet{agterberg_entrywise_2022}.  First we state several results concerning the output of Tensor SVD.   Throughout our proofs we assume that $t$ is taken to be $t_0 + 1$, with $t_0$ as in \cref{thm:twoinfty} below.    
\begin{theorem}[Restatement of Theorem 2 \citet{agterberg_entrywise_2022}] \label{thm:twoinfty}
Suppose $\mathcal{T}$ is a Tucker low-rank tensor with incoherence parameter $\mu_0$ and condition number $\kappa$.  Suppose that $\lambda/\sigma \gtrsim \kappa p^{3/4} \sqrt{\log(p)}$ and that $r_k \asymp r$.  Suppose further that $\kappa^2 \lesssim p^{1/4}$ and that $\mu_0^2 r \lesssim p^{1/2}$.  Then  for $t\asymp \textcolor{black}{\log( \frac{\lambda/\sigma}{C \kappa \sqrt{p\log(p)}}})$ %\log(\kappa p/(\lambda/\sigma)$
it holds \textcolor{black}{with probability at least $1 - p^{-10}$} that
\begin{align*}
    \| \uhat_k^{(t)} \mathbf{W}_k^{(t)} - \U_k \|_{2,\infty} &\lesssim \frac{ \kappa\mu_0 \sqrt{r_k \log(p)}}{\lambda/\sigma}.
\end{align*}
\end{theorem}

%\subsection{Leave-One-Out Sequences and Related Results}
We now recall the definition of the leave-one-out sequences defined in \citet{agterberg_estimating_2022}.  We define $\utilde_k^{S,j-m}$ as follows.  First, let $\mathbf{Z}_k^{k-m}$ be the $k$'th matricization of $\mathcal{Z}$ with its $m$'th row set to zero, and let $\mathcal{Z}^{k-m}$ be the corresponding tensor.  We then define $\mathcal{Z}_k^{j-m}$ as the matrix $\mathcal{M}_k(\mathcal{Z}^{j-m})$, which is the $k$'th matricization of the tensor $\mathcal{Z}$ with entries assocciated to the $m$'th row of $\mathbf{Z}_j$ set to zero.  We then define $\utilde_k^{(S,j-m)}$ as the leading $r_k$ eigenvectors of the matrix
\begin{align*}
    \Gamma\big( \mathbf{T}_k \mathbf{T}_k\t + \mathbf{Z}_k^{j-m} \mathbf{T}_k\t + \mathbf{T}_k (\mathbf{Z}_k^{j-m})\t + \mathbf{Z}_k^{j-m} (\mathbf{Z}_k^{j-m})\t \big).
\end{align*}
We then define $\utilde_k^{(t,j-m)}$ inductively as follows. For a given iteration $t$, we set
\begin{align*}
    \utilde_k^{(t,j-m)} = \begin{cases}
    \mathrm{SVD}_{r_1}\big( \mathbf{T}_1 + \mathbf{Z}_1^{j-m} \mathcal{P}_{\utilde_{2}^{(t-1,j-m)} \otimes \utilde_{3}^{(t-1,j-m)}} \big)& k = 1; \\
    \mathrm{SVD}_{r_2}\big( \mathbf{T}_2 + \mathbf{Z}_2^{j-m} \mathcal{P}_{\utilde_{1}^{(t,j-m)} \otimes \utilde_{3}^{(t-1,j-m)}} \big)& k =2; \\\mathrm{SVD}_{r_3}\big( \mathbf{T}_3 + \mathbf{Z}_3^{j-m} \mathcal{P}_{\utilde_{1}^{(t,j-m)} \otimes \utilde_{2}^{(t,j-m)}} \big)& k = 3. \end{cases}
\end{align*}

 \subsection{Additional Notation} \label{sec:notation2}
 Finally, we define the following additional notation defined in \citet{agterberg_estimating_2022}.  We set $\mathcal{\hat P}_{k}^{t}$ via
 \begin{align*}
     \mathcal{\hat P}_k^{t} &\coloneqq \begin{cases}
     \mathcal{P}_{\uhat_{2}^{(t-1)} \otimes \uhat_3^{(t-1)}} &k = 1; \\
        \mathcal{P}_{\uhat_{1}^{(t)} \otimes \uhat_3^{(t-1)}} &k =2;\\
           \mathcal{P}_{\uhat_{1}^{(t)} \otimes \uhat_2^{(t)}} &k =3. \end{cases}
 \end{align*}
 We define $\mathcal{\tilde P}_k^{t,j-m}$ similarly.  We also define the terms
 \begin{align*}
     \mathbf{L}_k^{(t)} &\coloneqq \U_{k\perp} \U_{k\perp}\t \mathbf{Z}_k \mathcal{\hat P}_k^{t} \mathbf{T}_k\t \uhat_k^{(t)} \big( \mathbf{\hat \Lambda}_k^{(t)}\big)^{-2}; \\
     \mathbf{Q}_k^{(t)} &\coloneqq \U_{k\perp} \U_{k\perp}\t \mathbf{Z}_k \mathcal{\hat P}_k^{t} \mathbf{Z}_k\t \uhat_k^{(t)} \big( \mathbf{\hat \Lambda}_k^{(t)}\big)^{-2},
 \end{align*}
 representing the \emph{linear error} and \emph{quadratic error} respectively.  We also define
 \begin{align*}
    \tau_k &\coloneqq\sup_{\substack{ \| \mathbf{U}_1 \| = 1, \mathrm{rank}(\U_1) \leq 2 r_{k+1}\\ \|\mathbf{U}_2\| =1, \mathrm{rank}(\U_2) \leq 2 r_{k+2}} } \|  \mathbf{Z}_k \bigg( \mathcal{P}_{\mathbf{U}_1} \otimes \mathcal{P}_{\mathbf{U}_2} \bigg)\|; \\  
  %  \tilde \tau_k &\coloneqq \sup_{\substack{ \| \mathbf{U}_1 \| = 1, \mathrm{rank}(\U_1) \leq 2 r_{k+1}\\ \|\mathbf{U}_2\| =1, \mathrm{rank}(\U_2) \leq 2 r_{k+2}} } \|  \mathbf{Z}_k \bigg( \mathcal{P}_{\mathbf{U}_1} \otimes \mathcal{P}_{\mathbf{U}_2} \bigg) \mathbf{V}_k\|; \\  
    \xi_{k}^{(t,j-m)} &\coloneqq 
    \bigg\| \bigg(\mathbf{Z}_k^{j-m} - \mathbf{Z}_k \bigg) \mathcal{\tilde P}_{k}^{t,j-m}  \bigg\| \\
    \tilde \xi_{k}^{(t,j-m)} &\coloneqq   \bigg\| \bigg(\mathbf{Z}_k^{j-m} - \mathbf{Z}_k \bigg) \mathcal{\tilde P}_{k}^{t,j-m}  \mathbf{V}_k  \bigg\| \\
     \eta_{k}^{(t,j-m)} &\coloneqq 
     \begin{cases}
         \| \sin\Theta( \utilde_{k+1}^{(t-1,j-m)}, \uhat_{k+1}^{(t-1)}) \| + \| \sin\Theta( \utilde_{k+2}^{(t-1,j-m)}, \uhat_{k+2}^{(t-1)}) \| & k = 1  \\
          \| \sin\Theta( \utilde_{k+1}^{(t-1,j-m)}, \uhat_{k+1}^{(t-1)}) \| + \| \sin\Theta( \utilde_{k+2}^{(t,j-m)}, \uhat_{k+2}^{(t)}) \| & k = 2  \\
           \| \sin\Theta( \utilde_{k+1}^{(t,j-m)}, \uhat_{k+1}^{(t)}) \| + \| \sin\Theta( \utilde_{k+2}^{(t,j-m)}, \uhat_{k+2}^{(t)}) \| & k = 3  \end{cases} \\
    \eta_k^{(t)} &\coloneqq \begin{cases}\| \sin\Theta( \U_{k+1}, \uhat_{k+1}^{(t-1)}) \| + \| \sin\Theta( \U_{k+2}, \uhat_{k+2}^{(t-1)}) \| & k = 1 \\
        \| \sin\Theta( \U_{k+1}, \uhat_{k+1}^{(t-1)}) \| + \| \sin\Theta( \U_{k+2}, \uhat_{k+2}^{(t)}) \| & k = 2 \\
        \| \sin\Theta( \U_{k+1}, \uhat_{k+1}^{(t)}) \| + \| \sin\Theta( \U_{k+2}, \uhat_{k+2}^{(t)}) \| & k = 3.
    \end{cases}
\end{align*}
Denote $\delta_L^{(k)} \coloneqq C_0\kappa \sqrt{p_k \log(p)}$, where $C_0$ is some appropriately large constant, and let $\delta_L = C_0 \kappa \sqrt{p_{\max} \log(p)}$.  We will use the following events from \citet{agterberg_estimating_2022} (where $\sigma = 1$ without loss of generality):
\begin{align*}
    \mathcal{E}_{\mathrm{Good}} &\coloneqq \bigg\{ \max_k \tau_k \leq C \sqrt{pr} \bigg\} \\
    &\qquad \bigcap \bigg\{ \| \sin\Theta(\uhat_k^{(t)}, \U_k ) \| \leq \frac{\dl\ku}{\lambda} + \frac{1}{2^{t}} \text{ for all $t \leq t_{\max}$ and $1\leq k \leq 3$ } \bigg\} \\
    &\qquad \bigcap \bigg\{ \max_k \bigg\| \U_k\t \mathbf{Z}_k \mathbf{V}_k \bigg\| \leq C \left( \sqrt{r} + \sqrt{\log(p)} \right) \bigg\}; \\
    &\qquad \bigcap \bigg\{ \max_k  \bigg\| \U_k\t \mathbf{Z}_k \mathcal{P}_{\U_{k+1}}  \otimes   \mathcal{P}_{\U_{k+2}}  \bigg\| \leq C \left( r + \sqrt{\log(p)} \right) \bigg\}; \\
    &\qquad \bigcap \bigg\{ \max_k \bigg\| \mathbf{Z}_k \mathbf{V}_k \bigg\| \leq C \sqrt{p_k} \bigg\}. \\
    \mathcal{E}_{2,\infty}^{t,k} &\coloneqq \bigg\{  \| \uhat_k^{(t)} - \U_k \mathbf{W}_k^{(t)} \|_{2,\infty} \leq \bigg( \frac{\dl\ku}{\lambda} + \frac{1}{2^{t}} \bigg) \mu_0 \sqrt{\frac{r_k}{p_k}} \bigg\}; \\
    %\bigg\{ \max_k \| \uhat_k^{(t)} - \U_k \mathbf{W}_k^{(t)} \|_{2,\infty} \leq \bigg( \frac{\dl}{\lambda} + \frac{1}{2^{t}} \bigg) \mu_0 \sqrt{\frac{r}{p}} \text{ for all $t \leq t_0 - 1$}\bigg\}; \\
    \mathcal{E}_{j-m}^{t,k} &\coloneqq \bigg\{ \| \sin\Theta (\utilde_k^{t,j-m}, \uhat_k^{(t)}) \| \leq  \bigg( \frac{\dl\ku}{\lambda} + \frac{1}{2^{t}} \bigg) \mu_0 \sqrt{\frac{r_k}{p_j}} \bigg\}; \\
    %\bigg\{ \max_j \| \sin\Theta(\utilde_j^{t,k-m}, \uhat_j^{(t)}) \| \leq \bigg( \frac{\dl}{\lambda} + \frac{1}{2^{t}} \bigg) \mu_0 \sqrt{\frac{r}{p}} \text{ for all $t \leq t_0 - 1$} \bigg\}; \\
  \mathcal{E}_{\mathrm{main}}^{t_0-1,1} &\coloneqq \bigcap_{t=1}^{t_0-1} \Bigg\{ \bigcap_{k=1}^{3}  \mathcal{E}^{t,k}_{2,\infty} \cap  \bigcap_{j=1}^{3} \bigcap_{m=1}^{p_j} \mathcal{E}_{k-m}^{t,j} \Bigg\}; \\
       \mathcal{E}_{\mathrm{main}}^{t_0-1,2} &\coloneqq\mathcal{E}_{\mathrm{main}}^{t_0-1,1} \cap \bigg\{ \bigcap_{k=1}^{3} \bigcap_{m = 1}^{p_k} \mathcal{E}_{k-m}^{t_0,1} \bigg\} \cap \mathcal{E}_{2,\infty}^{t_0,1} \\
        \mathcal{E}_{\mathrm{main}}^{t_0-1,3} &\coloneqq\mathcal{E}_{\mathrm{main}}^{t_0-1,2} \cap \bigg\{ \bigcap_{k=1}^{3} \bigcap_{m = 1}^{p_k} \mathcal{E}_{k-m}^{t_0,2} \bigg\} \cap \mathcal{E}_{2,\infty}^{t_0,2}.\\
        \mathcal{\tilde E}_{j-m}^{t,k} &\coloneqq \Bigg\{ \|\mathcal{\tilde P}_k^{t_0,j-m}  \mathbf{V}_k \|_{2,\infty} \leq c \mu_0^2 \frac{\sqrt{r_{-k}}}{p_j} \bigg( \frac{\dl^{(k+1)}}{\lambda} + \frac{1}{2^{t_0-1}}\bigg)\bigg( \frac{\dl^{(k+2)}}{\lambda} + \frac{1}{2^{t_0-1}}\bigg) \\
     &\quad + c \mu_0^2 \frac{\sqrt{r_{-k}}}{\sqrt{p_jp_{k+2}}} \bigg( \frac{\dl^{(k+1)}}{\lambda} + \frac{1}{2^{t_0-1}} \bigg) + c \mu_0^2 \frac{\sqrt{r_{-k}}}{\sqrt{p_j p_{k+1}}} \bigg( \frac{\dl^{(k+2)}}{\lambda} + \frac{1}{2^{t_0-1}} \bigg)\\
     &\quad + c \mu_0^2 \frac{\sqrt{r_{-k}}}{\sqrt{p_{-k}}} \bigg( \frac{\dl^{(k+1)}}{\lambda} + \frac{1}{2^{t_0-1}} \bigg) + c\mu_0^2 \frac{\sqrt{r_{-k}}}{\sqrt{p_{-k}}} \bigg( \frac{\dl^{(k+2)}}{\lambda} + \frac{1}{2^{t_0-1}} \bigg) + c \mu_0 \sqrt{\frac{r_k}{p_{-k}}}.\Bigg\} \\
     &\bigcap \Bigg\{ \|\mathcal{\tilde P}_k^{t_0,j-m}   \|_{2,\infty} \leq c \mu_0^2 \frac{\sqrt{r_{-k}}}{p_j}  \bigg( \frac{\dl^{(k+1)}}{\lambda} + \frac{1}{2^{t_0-1}} \bigg)  \bigg( \frac{\dl^{(k+2)}}{\lambda} + \frac{1}{2^{t_0-1}} \bigg) \\
    &\quad +  c\mu_0^2 \frac{\sqrt{r_{-k}}}{\sqrt{p_jp_{k+2}}} \bigg( \frac{\dl^{(k+1)}}{\lambda} + \frac{1}{2^{t_0-1}} \bigg) + c\mu_0^2 \frac{\sqrt{r_{-k}}}{\sqrt{p_jp_{k+1}}} \bigg( \frac{\dl^{(k+2)}}{\lambda} + \frac{1}{2^{t_0-1}} \bigg) + c\mu_0^2 \frac{\sqrt{r_{-k}}}{\sqrt{p_{-k}}}.\Bigg\},
  \end{align*}
  where $c$ is some deterministic constant.  These events are analyzed explicitly in \citet{agterberg_estimating_2022}. 
  \subsection{Initial Lemmas}
  \textcolor{black}{Without loss of generality, throughout this section we assume that $\sigma = 1$.  }  \textcolor{black}{First, we record the following lemma concerning the event $\mathcal{E}_{\mathrm{Good}}$ from \citet{agterberg_estimating_2022}.
  \begin{lemma}
      Let $\mathcal{E}_{\mathrm{Good}}$ be defined as above.  Under the conditions of \cref{thm:twoinfty}, it holds that $\p\{ \mathcal{E}_{\mathrm{Good}}\} \geq 1- O(p^{-30})$.  
  \end{lemma}
  \begin{proof}
  See the proof of Lemma 19 of \citet{agterberg_estimating_2022}.
  \end{proof}}
  \textcolor{black}{
 We note with the choice of $t_0 = C \log( \frac{\lambda}{C_0 \kappa \sqrt{p\log(p)}})$, we have that
  \begin{align*}
 \log(   2^{t_0} ) &=\log\bigg( 2^{C \log( \frac{\lambda}{C_0 \kappa \sqrt{p_{\min}\log(p)}})} \bigg) = C \log(2) \log( \frac{\lambda}{C_0 \kappa \sqrt{p_{\min}\log(p)}}) \geq \log\bigg( \frac{\lambda}{C_0 \kappa \sqrt{p_{\min}\log(p)}}\bigg),
  \end{align*}
  provided the constant $C$ is sufficiently large.  Hence it holds that $\frac{1}{2^{t_0}} \leq \frac{\dl}{\lambda}$.  Therefore, on the event $\mathcal{E}_{\mathrm{Good}}$ for this choice of $t_0$, it holds that
  \begin{align*}
      \| \sin\Theta(\uhat_k^{(t)}, \U_k ) \| \lesssim \frac{\dl^{(k)}}{\lambda}. \numberthis \label{sinthetaegood}
  \end{align*}
  In addition, for this choice of $t_0$, on the event $\mathcal{E}_{{\mathrm{main}}}^{t_0-1,k}$, it holds that
  \begin{align*}
      \| \uhat_k^{(t)} - \U_k \mathbf{W}_k^{(t)} \|_{2,\infty} &\lesssim \frac{\dl^{(k)}}{\lambda} \mu_0 \sqrt{\frac{r_k}{p_k}}.\numberthis \label{twoinftygood}
  \end{align*}
}
\textcolor{black}{
 In addition, the following result characterizes the properties of the leave-one-out sequences.
\begin{lemma} \label{sinthetaloo14}
In the setting of \cref{thm:twoinfty}, on the event $\mathcal{E}_{\mathrm{main}}^{t_0-1,k}$ it holds that for each $1 \leq j \leq 3$ and $1 \leq k \leq 3$ that
\begin{align*}
    \| \sin\Theta(\uhat_j^{(t)},\utilde_j^{(t,k-m)}) \| \lesssim \frac{\kappa \sqrt{p_k \log(p)}}{\lambda} \mu_0 \sqrt{\frac{r_k}{p_j}}.
\end{align*}
\end{lemma}
\begin{proof}
The proof of Theorem 2 of \citet{agterberg_estimating_2022} shows that on the event $\mathcal{E}_{\mathrm{main}}^{t_0-1,k}$, for all $t \leq t_{\max} \asymp  C \log \bigg( \frac{\lambda}{C \kappa \sqrt{p\log(p)}} \bigg)$, one has the bound
\begin{align*}
    \| \sin\Theta(\uhat_j^{(t)}, \utilde_j^{(t,k-m)}) \| \leq \frac{C_0 \kappa \sqrt{p_k\log(p)}}{\lambda} \mu_0 \sqrt{\frac{r_k}{p_j}} + \frac{1}{2^t} \mu_0 \sqrt{\frac{r_k}{p_j}}.
\end{align*}
For the choice of $t = C \log \bigg( \frac{\lambda}{C \kappa \sqrt{p\log(p)}} \bigg)$, it holds that
\begin{align*}
    \frac{1}{2^t} \leq \frac{C_0 \kappa \sqrt{p_k\log(p)}}{\lambda},
\end{align*}
which completes the proof, with the implicit constant $2 C_0$.  
\end{proof}  
Finally, we record the following result concerning the empirical singular values $\mathbf{\hat \Lambda}_k^{(t)}$.
\begin{lemma}\label{lem:eigengaps}
Let $\mathbf{\Lambda}_k$ denote the diagonal matrix of leading $r_k$ nonzero singular values of $\mathbf{T}_k$, and let $\mathbf{\hat \Lambda}_k^{(t)}$ denote the leading $r$ singular values of $(\mathbf{T}_k + \mathbf{Z}_k\big) \mathcal{\hat P}_k^{(t)}.$ Under the conditions of \cref{thm:twoinfty}, tor all $t \geq 2$, on the event $\mathcal{E}_{{\mathrm{Good}}}$ it holds that
\begin{align*}
    \| \big(\mathbf{\hat \Lambda}_k^{(t)} \big)\inv \| &\leq  \frac{2}{\lambda}.
\end{align*}
\end{lemma}
\begin{proof}
Without loss of generality, we prove the result for $k = 1$.  First, observe that on the event $\mathcal{E}_{\mathrm{Good}}$, it holds that under the assumptions $\lambda \geq C_0 \kappa p^{3/4} \sqrt{\log(p)}$ and $r \lesssim p^{1/2}$ that
    \begin{align*}
        \| \mathbf{Z}_1 \mathcal{\hat P}_k^{(t)} \| \lesssim \sqrt{pr} \leq \lambda/8.
    \end{align*}
    As a result, letting $\mathbf{\tilde \Lambda}_1^{(t)}$ denote the singular values of the matrix $\mathbf{T}_1 \mathcal{\hat P}_1^{(t)}$, Weyl's inequality implies
    \begin{align*}
        \| \mathbf{\tilde \Lambda}_1^{(t)} - \mathbf{\hat \Lambda}_1^{(t)} \| \leq \frac{\lambda}{8}.
    \end{align*}
Furthermore, since by definition $\mathbf{T}_1 = \mathbf{T}_1 \mathcal{P}_{\mathbf{U}_{2} \otimes \U_{3}}$, we have that
\begin{align*}
  \lambda_{r_1} \bigg( \mathbf{T}_1 \mathcal{\hat {P}}_1^{(t)} \bigg)  &=  \lambda_{r_1} \bigg(  \mathbf{T}_1\mathcal{P}_{\mathbf{U}_{2} \otimes \U_{3}} \mathcal{\hat P}_1^{(t)} \bigg) \\
  &\geq \lambda_{r_1}\bigg(  \mathbf{T}_1 \U_{2} \otimes \U_{2} \bigg) \lambda_{\min} \bigg( \big( \U_{2} \otimes \U_{3} \big)\t  \big( \uhat_{3}^{(t-1)} \otimes \uhat_{3}^{(t-1)} \big)  \bigg) \\
  &= \lambda_{r_1} \bigg( \mathbf{T}_1 \bigg) \lambda_{\min} \big( \U_2\t \uhat_2^{(t)} \big) \lambda_{\min} \big( \U_3\t \uhat_3^{(t)} \big).
\end{align*}
Next, by Lemma 1 of \citet{cai_rate-optimal_2018} it holds that
\begin{align*}
    \| \sin\Theta(\U_2, \uhat_2^{(t)})\|^2 &= 1 - \lambda_{\min}\big( \U_2\t \uhat_2^{(t)} \big)^2
\end{align*}
which implies that
\begin{align*}
    \lambda_{\min}\big( \U_2\t \uhat_2^{(t)} \big) &= \sqrt{1 - \|\sin\Theta(\U_2,\uhat_2^{(t)})\|^2} \geq \sqrt{1 - \frac{15}{64}} \geq \frac{7}{8},
\end{align*}
where we have used the fact that by \eqref{sinthetaegood}, on the event $\mathcal{E}_{\mathrm{Good}}$ one has
\begin{align*}
    \|\sin\Theta(\U_2,\uhat_2^{(t)})\| \lesssim \frac{\dl}{\lambda} \leq \frac{3}{8} \leq \sqrt{\frac{15}{64}}
\end{align*}
since $\lambda \gg \dl = C_0 \kappa \sqrt{p\log(p)}$ by assumption.  By a similar argument, it holds that $\lambda_{\min}\big(\U_3\t \uhat_3^{(t)} \big) \geq \frac{7}{8}$.  
Therefore, this demonstrates that
\begin{align*}
    \lambda_{r_1} \bigg( \mathbf{T}_1 \mathcal{\hat P}_1^{(t)} \bigg) \geq \lambda_{r_1} \bigg( \mathbf{T}_1 \bigg) \frac{49}{64} \geq \lambda_{r_1} \bigg( \mathbf{T}_1 \bigg) \frac{3}{4}.
\end{align*}
Consequently, combining these bounds, we see that
\begin{align*}
 \lambda_{r_1}\big( \mathbf{\hat \Lambda}_1^{(t)} \big) &\geq \lambda_{r_1}\big(\mathbf{T}_1 \mathcal{\hat P}_1^{(t)} \big) - \|  \mathbf{\tilde \Lambda}_1^{(t)} - \mathbf{\hat \Lambda}_1^{(t)} \| \geq \frac{5}{8} \lambda.
\end{align*}
As a result, one has that $\|(\mathbf{\hat \Lambda}_1^{(t)})\inv\| \leq \frac{2}{\lambda}$ as required.  
\end{proof}
 }

 \subsection{Proof Dependencies}
As our main technical results have a rather complicated dependency structure, for convenience we have included the following diagram describing the dependencies of the results.  We note that \cref{thm:efficiency} and \cref{thm:efficiencyloadings} are self-contained and do not rely on any previous results.  

\begin{center}
{\footnotesize 
\begin{tikzpicture}
   % put nodes
   \node[main] (t1)   {\shortstack{First-order \\
   expansion\\ (\cref{thm:eigenvectornormality})}};
   \node[main] (t2) [below left= 2.5cm of t1] {\shortstack{Singular vector \\ distributional \\ theory \\ (\cref{thm:eigenvectornormality2})}};
   \node[main] (t3) [below right= 2.5cm of t1] {\shortstack{Entrywise \\ Distributional \\ Theory \\ (\cref{thm:asymptoticnormalityentries})}};
   \node[main] (t6) [below = 1.3cm of t1] {\shortstack{Projection \\ expansion \\ (\cref{cor:asymptoticnormality_projection})}};
        \node[main] (t7) [below = 1cm of t6]{\shortstack{Entrywise \\ convergence \\ (\cref{cor:maxnormbound})}};
        \node[main] (t4) [left = 1.5cm of t7]{\shortstack{Confidence \\region \\ validity\\ (\cref{thm:civalidity2})}};
     \node[main] (t5) [right = 1.5cm of t7]{\shortstack{Confidence \\interval \\ validity \\ (\cref{thm:civalidity})}};
  \node[main](t8) [below = .6cm of t4]{\shortstack{Testing \\ memberships \\ (\cref{cor:testing})}};
  \node[main](t9) [right = 1.2cm of t8]{\shortstack{Simultaneous \\ Inference \\ (\cref{thm:simultaneousinference})}}; 
  \node[main](t10) [right = 1.2cm of t9]{\shortstack{Testing entries\\ \cref{thm:entrytesting}}};

   % make path ...
   %\path (t1) -- node[auto=false]{\ldots} (t2);
   %\path (tt) -- node[auto=false]{\ldots} (t0);

\draw [->] (t1) to (t2);
\draw [->] (t1) to (t6);
\draw [->] (t6) to (t3);
\draw[->](t7) to (t5);
\draw[->] (t7) to (t4);
\draw[->](t4) to node[midway,left]{{\tiny  $r,\kappa,\mu_0$ bounded }}(t8);
\draw[->](t7) to node[midway,left]{{\tiny \shortstack{Condition \\ on $s_{\min}$}}} (t9);
\draw [->] (t7) to node[midway,right]{{\tiny \shortstack{Additional \\condition \\ on $s_{ijk}$}}} (t10);
\draw [->] (t1) to node[midway,right]{{\tiny \shortstack{$\kappa^2 \mu_0^2 r^{3/2}\sqrt{\log(p)}\lesssim p^{1/4}$\\Condition on $s_{ijk}$}}} (t3);
\draw [->] (t2) to node[midway,left]{{\tiny \shortstack{$\kappa^2 \mu_0^2 r^{3/2}\sqrt{\log(p)}\lesssim p^{1/4}$\\$\lambda/\sigma \gg \kappa^3 \mu_0^2 \log^2(p) r^2 \sqrt{p}$}}} (t4);
\draw[->] (t3) to node[midway,right]{{\tiny \shortstack{Additional condition\\ on $s_{ijk}$}}} (t5);
\draw[->] (t6) to [out = 210,in=135] node[midway,right]{{\tiny $\kappa^2 \mu_0^2 r^{3/2} \sqrt{\log(p)} \lesssim p^{1/4}$}}(t7);

%\draw[->](t2) to node[midway,left]{{\tiny \substack{$\kappa^2 \mu_0^2 r^{3/2} \sqrt{\log(p)}\lesssim p^{1/4}$\\ }}} (t4);
   % draw arrows
  % \draw [->] (t1) to [out=45,in=135] node [midway, above]{$P(k=n|\lambda=i)$} (tn); 
   %\draw [->] (t1) to [out=-45,in=-135] node [midway, below]{$P(k=i+2|\lambda=i)$}(t3);
   %\draw [->] (t1) to  node [midway, above] {$P(k=i+1|\lambda=i)$}(t2);
   %\draw [->] (t1) to  node [midway, above] {$P(k=i-1|\lambda=i)$} (t);
   %\draw [->] (t1) to [out=-135,in=-45] node [midway, below]{$P(k=i-2|\lambda=i)$}(tt);
   %\draw [->] (t1) to [out=135,in=35] node [midway, above]{$P(k=0|\lambda=i)$}(t0);
\end{tikzpicture} 
}
\end{center}

 \section{Proof of Distributional Guarantees for the Loadings  (Theorem \ref{thm:eigenvectornormality}, Theorem \ref{thm:eigenvectornormality2}, and Theorem  \ref{cor:asymptoticnormality_projection})}
 \label{sec:distributionalproofs}

 This section contains the proof of \cref{thm:eigenvectornormality}, \cref{thm:eigenvectornormality2}, and \cref{cor:asymptoticnormality_projection}.  The following subsection introduces the auxiliary lemmas needed for the proofs,  \cref{sec:eigenvectornormalityproof} contains the proof of \cref{thm:eigenvectornormality}, \cref{sec:eigenvectornormality2proof} contains the proof of \cref{thm:eigenvectornormality2}, and \cref{sec:corollaryprojectionproof} contains the proof of \cref{cor:asymptoticnormality_projection}. 
Throughout we assume that $t = t_0 + 1$, where $t_0$ is such that \cref{thm:twoinfty} holds.   Throughout this section we assume without loss of generality that $\sigma = 1$.  
 %In addition, we note that the proof of \cref{thm:twoinfty} shows that if $t = t_0 + 1$ where $t_0$ is the first iteration such that \cref{thm:twoinfty} holds, then \cref{thm:twoinfty} still holds. 
 
 \subsection{Preliminary Lemmas: First Order Approximations}
 \label{sec:prelimlemmas}
 In this section we present several lemmas that are useful for the proofs of the main results in this section.  The proofs are deferred to \cref{sec:prelimproofs}.  We assume throughout this section without loss of generality that $\sigma = 1$.
 
 The following result shows that the linear-term approximation is sufficiently strong.

\begin{lemma}[Linear term approximation] \label{lem:linearapprox}
Under the conditions of \cref{thm:eigenvectornormality}, with probability at least $1 - O(p^{-9})$ it holds that
\begin{align*}
    \bigg\|&\bigg(\uhat_k^{(t)} - \U_k \mathbf{W}_k^{(t)}  - (\mathbf{I} - \U_k \U_k\t) \mathbf{Z}_k \mathcal{\hat P}^{(t)}_k \mathbf{V}_k \mathbf{\Lambda}_k \U_k\t \uhat_k^{(t)}(\mathbf{\hat \Lambda}_k^{(t)})^{-2} \bigg)\bigg\|_{2,\infty} \\
    &\lesssim  \frac{\kappa \mu_0^2 r^{3/2} p \log(p)}{\lambda^3} + \frac{\mu_0^2 \big(r^2 \sqrt{\log(p)} + r \log(p) \big)}{\lambda^2}\\
    &\qquad +  \frac{\mu_0 \kappa^2 \sqrt{pr} \log(p)}{\lambda^2}.
\end{align*}
\end{lemma}

\noindent The next lemma shows that the contribution of the projection onto $\U_k$ is sufficiently small.
\begin{lemma}[Small Projection] \label{uperplemma}
Under the conditions of \cref{thm:eigenvectornormality} it holds that
\begin{align*}
    \| \U_k \U_k\t \mathbf{Z}_k \mathcal{\hat P}_k^{(t)} \mathbf{V}_k \mathbf{\Lambda}_k \U_k\t \uhat_k^{(t)} ( \mathbf{\hat \Lambda}_k^{(t)})^{-2} \|_{2,\infty} &\lesssim \frac{\mu_0 r \kappa^2 \sqrt{p\log(p)}}{\lambda^2} + \frac{\mu_0 r \kappa}{\lambda \sqrt{p}}.
\end{align*}

\end{lemma}

\noindent 
The next lemma replaces the empirical linear term with the population linear term.

\begin{lemma}[Replacing the empirical linear term with the population linear term] \label{lem:empiricallinearreplacement}
Under the conditions of \cref{thm:eigenvectornormality}, with probability at least $1 - O(p^{-9})$ it holds that
\begin{align*}
    \bigg\| \mathbf{Z}_k \bigg( \mathcal{\hat P}_{k}^{(t)} - \mathcal{P}_k  \bigg) \mathbf{V}_k \mathbf{\Lambda}_k \U_k\t \uhat_k^{(t)}  (\mathbf{\hat \Lambda}_k^{(t)})^{-2} \bigg\|_{2,\infty} &\lesssim \frac{\kappa^2 \mu_0^2 \log(p) r \sqrt{p}}{\lambda^2}.
\end{align*}
\end{lemma}

\noindent 
Finally, the following result shows that $\mathbf{\Lambda}_k^2$ and $(\mathbf{\hat \Lambda}_k^{(t)})^2$ approximately commute.  
\begin{lemma}[Approximate commutation of $\mathbf{\Lambda}_k^2$ and $(\mathbf{\hat{\Lambda}}_k^{(t)})^2$] \label{lem:approximatecommute}
Under the conditions of \cref{thm:eigenvectornormality}, with probability at least $1 - O(p^{-9})$ it holds that
\begin{align*}
    \bigg\| \mathbf{\Lambda}_k^2 \U_k\t \uhat_k - \U_k\t \uhat_k (\mathbf{\hat{\Lambda}}_k^{(t)})^{2} \bigg\| &\lesssim \lambda_1 \sqrt{pr}.
\end{align*}
\end{lemma}

\subsection{Proof of \cref{thm:eigenvectornormality}}
\label{sec:eigenvectornormalityproof}

With the lemmas from the previous section in place, we are now prepared to prove \cref{thm:eigenvectornormality}.  Again without loss of generality we assume $\sigma = 1$.  
\begin{proof}[Proof of \cref{thm:eigenvectornormality}]
% First we show that with probability at least $1 - O(p^{-9})$ that
% \begin{align*}
%      e_m\t \bigg( \uhat_k^{(t)} (\mathbf{W}_k^{(t)})\t - \U_k \bigg) &= e_m\t \mathbf{Z}_k \mathbf{V}_k \mathbf{\Lambda}_k\inv + O\bigg( \frac{\kappa^2 \mu_0^2 \log(p) r \sqrt{p}}{\lambda^2} \bigg),
% \end{align*}
% which justifies the first part of \cref{thm:eigenvectornormality}.
First, we note that with probability at least $1 - O(p^{-9})$ it holds that
\begin{align*}
    e_m\t &\bigg( \uhat_k^{(t)} - \U_k \mathbf{W}_k^{(t)} \bigg) \\
    &\overset{\scalebox{.6}{\text{\cref{lem:linearapprox}}}}{=} e_m\t (\mathbf{I} - \U_k \U_k\t ) \mathbf{Z}_k \mathcal{\hat P}_k^{(t)} \mathbf{V}_k \mathbf{\Lambda}_k \U_k\t \uhat_k^{(t)} (\mathbf{\hat{\Lambda}}_k^{(t)})^{-2}  \\
    &\qquad + O\bigg( \frac{\kappa \mu_0^2 r^{3/2}p \log(p)}{\lambda^3} + \frac{\mu_0^2 \big( r^2 \sqrt{\log(p)} + r \log(p)\big) + \mu_0 \kappa^2 \sqrt{pr} \log(p)}{\lambda^2} \bigg) \\
    &\overset{\scalebox{.6}{\text{\cref{uperplemma}}}}{=} e_m\t  \mathbf{Z}_k \mathcal{\hat P}_k^{(t)} \mathbf{V}_k \mathbf{\Lambda}_k \U_k\t \uhat_k^{(t)} (\mathbf{\hat{\Lambda}}_k^{(t)})^{-2}  \\
    &\qquad + O\bigg( \frac{\mu_0 r \kappa^2 \sqrt{p\log(p)}}{\lambda^2} + \frac{\mu_0 r \kappa}{\lambda \sqrt{p}} \bigg) \\
    &\quad + O\bigg( \frac{\kappa \mu_0^2 r^{3/2}p \log(p)}{\lambda^3} + \frac{\mu_0^2 \big( r^2 \sqrt{\log(p)} + r \log(p)\big) + \mu_0 \kappa^2 \sqrt{pr} \log(p)}{\lambda^2} \bigg) \\
    &\overset{\scalebox{.6}{\text{\cref{lem:empiricallinearreplacement}}}}{=} e_m\t \mathbf{Z}_k \mathcal{ P}_k \mathbf{V}_k \mathbf{\Lambda}_k \U_k\t \uhat_k^{(t)}  (\mathbf{\hat{\Lambda}}_k^{(t)})^{-2} \\
    &\qquad + O\bigg( \frac{\kappa^2 \mu_0^2 \log(p) r \sqrt{p}}{\lambda^2} \bigg) \\
    &\qquad + O\bigg( \frac{\kappa \mu_0^2 r^{3/2}p \log(p)}{\lambda^3} + \frac{\mu_0^2 \big( r^2 \sqrt{\log(p)} + r \log(p)\big) + \mu_0 \kappa^2 \sqrt{pr} \log(p)}{\lambda^2} \bigg) \\
    &\qquad + O\bigg( \frac{\mu_0 r \kappa^2 \sqrt{p\log(p)}}{\lambda^2} + \frac{\mu_0 r \kappa}{\lambda \sqrt{p}} \bigg) \\
    &= e_m\t \mathbf{Z}_k \mathbf{V}_k \mathbf{\Lambda}_k\inv \U_k\t \uhat_k^{(t)} \\
    &\qquad + e_m\t \mathbf{Z}_k \mathbf{V}_k \mathbf{\Lambda}_k\inv \bigg( \mathbf{\Lambda}_k^2 \U_k\t \uhat_k^{(t)} - \U_k \uhat_k^{(t)} (\mathbf{\hat{\Lambda}}_k^{(t)})^2 \bigg) (\mathbf{\hat{\Lambda}}_k^{(t)})^{-2} \\
    &\qquad + O\bigg( \frac{\kappa^2 \mu_0^2 \log(p) r \sqrt{p}}{\lambda^2} \bigg) \\
    &\qquad + O\bigg( \frac{\kappa \mu_0^2 r^{3/2}p \log(p)}{\lambda^3} + \frac{\mu_0^2 \big( r^2 \sqrt{\log(p)} + r \log(p)\big) + \mu_0 \kappa^2 \sqrt{pr} \log(p)}{\lambda^2} \bigg) \\
    &\qquad + O\bigg( \frac{\mu_0 r \kappa^2 \sqrt{p\log(p)}}{\lambda^2} + \frac{\mu_0 r \kappa}{\lambda \sqrt{p}} \bigg) \\
    &= e_m\t \mathbf{Z}_k \mathbf{V}_k \mathbf{\Lambda}_k\inv \mathbf{W}_k^{(t)} + e_m\t \mathbf{Z}_k \mathbf{V}_k \mathbf{\Lambda}_k\inv \bigg( \U_k\t \uhat_k^{(t)} - \mathbf{W}_k^{(t)} \bigg) \\
    &\qquad + e_m\t \mathbf{Z}_k \mathbf{V}_k \mathbf{\Lambda}_k\inv \bigg( \mathbf{\Lambda}_k^2 \U_k\t \uhat_k^{(t)} - \U_k \uhat_k^{(t)} (\mathbf{\hat{\Lambda}}_k^{(t)})^2 \bigg) (\mathbf{\hat{\Lambda}}_k^{(t)})^{-2} \\
    &\qquad + O\bigg( \frac{\kappa^2 \mu_0^2 \log(p) r \sqrt{p}}{\lambda^2} \bigg) \\
    &\qquad + O\bigg( \frac{\kappa \mu_0^2 r^{3/2}p \log(p)}{\lambda^3} + \frac{\mu_0^2 \big( r^2 \sqrt{\log(p)} + r \log(p)\big) + \mu_0 \kappa^2 \sqrt{pr} \log(p)}{\lambda^2} \bigg) \\
     &\qquad + O\bigg( \frac{\mu_0 r \kappa^2 \sqrt{p\log(p)}}{\lambda^2} + \frac{\mu_0 r \kappa}{\lambda \sqrt{p}} \bigg).
\end{align*}
Note that by Lemma 16 of \citet{agterberg_estimating_2022} it holds that with probability at least $1 - O(p^{-30})$ that
\begin{align*}
    \| e_m\t \mathbf{Z}_k \mathbf{V}_k \mathbf{\Lambda}_k\inv \| &\lesssim p \sqrt{\log(p)} \| \mathbf{V}_k \mathbf{\Lambda}_k\inv \|_{2,\infty} \\
    &\lesssim \frac{\mu_0 \sqrt{r\log(p)} }{\lambda}.
\end{align*}
\textcolor{black}{Furthermore, suppose $\U_k\t \uhat_k^{(t)}$ has singular value decomposition  $\mathbf{W}_1 \mathbf{\Sigma} \mathbf{W}_2\t$. 
 Then since $\mathbf{W}_k^{(t)} = \mathrm{sgn}(\U_k, \uhat_k^{(t)})$, we have $\mathbf{W}_k^{(t)} = \mathbf{W}_1 \mathbf{W}_2\t$, and $\mathbf{\Sigma}$ contains the diagonal entries equal to $\cos\theta_i$, with $\theta_i$ the canonical angles between the subspace spanned by $\U_k$ and $\uhat_k^{(t)}$ (see, e.g, \citet{kato2013perturbation}) and hence it holds that
\begin{align*}
    \| \U_k\t \uhat_k^{(t)} - \mathbf{W}_k^{(t)} \| &= \| \mathbf{W}_1 \mathbf{\Sigma} \mathbf{W}_2\t - \mathbf{W}_1 \mathbf{W}_2\t \| = \| \mathbf{\Sigma} - \mathbf{I} \| \\
    &= \max_{1\leq i \leq r} 1 - \cos\theta_i \leq \max_{1 \leq i\leq r} 1 - \cos^2 \theta_i = \max_{1\leq i \leq r} \sin^2\theta_i = \| \sin\Theta(\uhat_k^{(t)}, \U_k) \|^2.\numberthis \label{sintheta14}
\end{align*}
}
Therefore, it holds on the event $\mathcal{E}_{\mathrm{Good}}$ that
\begin{align*}
    \| e_m\t \mathbf{Z}_k \mathbf{V}_k \mathbf{\Lambda}_k\inv \bigg( \U_k\t \uhat_k^{(t)} - \mathbf{W}_k^{(t)} \bigg) \| &\lesssim \frac{\mu_0 \sqrt{r\log(p)} }{\lambda} \| \U_k\t \uhat_k^{(t)} - \mathbf{W}_k^{(t)} \| \\
    &\lesssim\frac{\mu_0 \sqrt{r\log(p)} }{\lambda} \| \sin\Theta(\uhat_k^{(t)}, \U_k ) \|^2 \\
    &\overset{\scalebox{.6}{\text{\cref{sinthetaegood}}}}\lesssim \frac{\mu_0 \sqrt{r\log(p)} }{\lambda} \frac{\dl^2}{\lambda^2} \\
    &\lesssim \frac{\mu_0 \kappa^2 p \sqrt{r} \log^{3/2}(p)  }{\lambda^3},
\end{align*}
and that
\begin{align*}
   \bigg\| e_m\t &\mathbf{Z}_k \mathbf{V}_k \mathbf{\Lambda}_k\inv \bigg( \mathbf{\Lambda}_k^2 \U_k\t \uhat_k^{(t)} - \U_k \uhat_k^{(t)} (\mathbf{\hat{\Lambda}}_k^{(t)})^2 \bigg) (\mathbf{\hat{\Lambda}}_k^{(t)})^{-2} \bigg\| \\&\lesssim \frac{\mu_0 \sqrt{r\log(p)} }{\lambda} \| \mathbf{\Lambda}_k^2 \U_k\t \uhat_k^{(t)} - \U_k \uhat_k^{(t)} (\mathbf{\hat{\Lambda}}_k^{(t)})^2 \| \| (\mathbf{\hat{\Lambda}}_k^{(t)})^{-2} \| \\
   &\overset{\scalebox{.6}{\text{\cref{lem:eigengaps}}}}{\lesssim} \frac{\mu_0 \sqrt{r\log(p)} }{\lambda^3} \| \mathbf{\Lambda}_k^2 \U_k\t \uhat_k^{(t)} - \U_k \uhat_k^{(t)} (\mathbf{\hat{\Lambda}}_k^{(t)})^2 \| \\
   &\overset{\scalebox{.6}{\text{\cref{lem:approximatecommute}}}}{\lesssim} \frac{\mu_0 \sqrt{r\log(p)} }{\lambda^3} \lambda_1 \sqrt{pr} \\
   &\lesssim \frac{\mu_0 \kappa r \sqrt{p\log(p)} }{\lambda^2}.
\end{align*}
Therefore, we have shown that there is an event $\mathcal{E}_{\mathrm{Very Good}}$  with  $\p\big( \mathcal{E}_{\mathrm{Very Good}}\big) \geq 1 - O(p^{-9})$ such that on this event
\begin{align*}
    e_m\t &\bigg( \uhat_k^{(t)} - \U_k \mathbf{W}_k^{(t)} \bigg) \\
    &= e_m\t \mathbf{Z}_k \mathbf{V}_k \mathbf{\Lambda}_k\inv \mathbf{W}_k^{(t)} \\
    &\qquad + O\bigg( \frac{\mu_0 \kappa r \sqrt{p\log(p)} }{\lambda^2} + \frac{\mu_0 \kappa^2 p \sqrt{r} \log^{3/2}(p)  }{\lambda^3} \bigg) \\
    &\qquad + O\bigg( \frac{\kappa^2 \mu_0^2 \log(p) r \sqrt{p}}{\lambda^2} \bigg) \\
    &\qquad + O\bigg( \frac{\kappa \mu_0^2 r^{3/2}p \log(p)}{\lambda^3} + \frac{\mu_0^2 \big( r^2 \sqrt{\log(p)} + r \log(p)\big) + \mu_0 \kappa^2 \sqrt{pr} \log(p)}{\lambda^2} \bigg) \\ 
    &\qquad + O\bigg( \frac{\mu_0 r \kappa^2 \sqrt{p\log(p)}}{\lambda^2} + \frac{\mu_0 r \kappa}{\lambda \sqrt{p}} \bigg) \\
    &= e_m\t \mathbf{Z}_k \mathbf{V}_k \mathbf{\Lambda}_k\inv \mathbf{W}_k^{(t)} + O\bigg( \frac{\kappa^2 \mu_0^2 \log(p) r \sqrt{p}}{\lambda^2} + \frac{\mu_0 r \kappa}{\lambda \sqrt{p}} \bigg) 
\end{align*}
where we have used the fact that $\lambda \gtrsim \kappa \mu_0 \sqrt{pr \log(p)}$, which holds under the conditions of \cref{thm:eigenvectornormality} since $\mu_0^2 r \lesssim \sqrt{p}$. 
Therefore, by rotational invariance of vector norms, it holds that on this event
\begin{align*}
    e_m\t \bigg( \uhat_k^{(t)} (\mathbf{W}_k^{(t)})\t - \U_k \bigg) &= e_m\t \mathbf{Z}_k \mathbf{V}_k \mathbf{\Lambda}_k\inv + O\bigg( \frac{\kappa^2 \mu_0^2 \log(p) r \sqrt{p}}{\lambda^2} + \frac{\mu_0 r \kappa}{\lambda \sqrt{p}}  \bigg).
\end{align*}
This completes the proof.
\end{proof}
\subsection{Proof of \cref{cor:asymptoticnormality_projection}} \label{sec:corollaryprojectionproof}

\begin{proof}[Proof of \cref{cor:asymptoticnormality_projection}]
Again we assume that $\sigma = 1$ without loss of generality. We suppress the dependence of $\uhat_k^{(t)}$ and $\mathbf{\hat{\Lambda}}_k^{(t)}$ on $t$ for convenience.  Denoting $\mathbf{W}_k$ as the orthogonal matrix in \cref{thm:eigenvectornormality}, by \cref{thm:eigenvectornormality} on $\mathcal{E}_{\mathrm{Theorem} \ \ref{thm:eigenvectornormality}}$ it holds that
\begin{align*}
    \uhat_k &\uhat_k\t - \U_k \U_k\t \\
    &= \uhat_k (\uhat_k - \U_k \mathbf{W}_k )\t + (\uhat_k \mathbf{W}_k\t - \U_1)\U_1\t \\
    &= \uhat_k \mathbf{W}_k\t (\uhat_k \mathbf{W}_k\t - \U_k  )\t + (\uhat_k \mathbf{W}_k\t - \U_1)\U_1\t \\
    &= (\uhat_k \mathbf{W}_k\t -\U_k) (\uhat_k \mathbf{W}_k\t - \U_k  )\t +\U_k (\uhat_k \mathbf{W}_k\t - \U_k  )\t+ (\uhat_k \mathbf{W}_k\t - \U_1)\U_1\t \\
    &= (\mathbf{Z}_k \mathbf{V}_k \mathbf{\Lambda}_k\inv + \mathbf{\Psi}^{(k)}) (\mathbf{Z}_k \mathbf{V}_k \mathbf{\Lambda}_k\inv + \mathbf{\Psi}^{(k)})\t \\
    &\quad +\U_k (\mathbf{Z}_k \mathbf{V}_k \mathbf{\Lambda}_k\inv + \mathbf{\Psi}^{(k)}  )\t+ (\mathbf{Z}_k \mathbf{V}_k \mathbf{\Lambda}_k\inv + \mathbf{\Psi}^{(k)})\U_1\t \\
    &= \mathbf{Z}_k \mathbf{V}_k \mathbf{\Lambda}_k^{-2} \mathbf{V}_k\t \mathbf{Z}_k\t + \mathbf{\Psi}^{(k)} \mathbf{\Lambda}_k^{-1} \mathbf{V}_k\t \mathbf{Z}_k\t + \mathbf{Z}_k \mathbf{V}_k \mathbf{\Lambda}_k^{-1} (\mathbf{\Psi}^{(k)})\t + \mathbf{\Psi}^{(k)}(\mathbf{\Psi}^{(k)})\t\\
    &\quad + \U_k \mathbf{\Lambda}_k\inv \mathbf{V}_k\t \mathbf{Z}_k\t + \U_k (\mathbf{\Psi}^{(k)})\t  + \mathbf{Z}_k \mathbf{V}_k \mathbf{\Lambda}_k\inv \U_k\t + \mathbf{\Psi}^{(k)} \U_k\t \\
    &\coloneqq \U_k \mathbf{\Lambda}_k\inv \mathbf{V}_k\t \mathbf{Z}_k\t + \mathbf{Z}_k \mathbf{V}_k \mathbf{\Lambda}_k\inv \U_k\t + \mathbf{\Phi}^{(k)}.
\end{align*}
We now show the bound on $\mathbf{\Phi}^{(k)}$ holds.  We need to bound the terms
\begin{align*}
    \mathbf{\Phi}^{(k)}_1 &= \mathbf{Z}_k \mathbf{V}_k \mathbf{\Lambda}_k^{-2} \mathbf{V}_k\t \mathbf{Z}_k\t; \\
    \mathbf{\Phi}^{(k)}_2 &= \mathbf{\Psi}^{(k)}\mathbf{\Lambda}_k^{-1} \mathbf{V}_k\t \mathbf{Z}_k\t; \\
     \mathbf{\Phi}^{(k)}_3 &= \mathbf{Z}_k \mathbf{V}_k \mathbf{\Lambda}_k^{-1}(\mathbf{\Psi}^{(k)})\t; \\
      \mathbf{\Phi}^{(k)}_4 &= \mathbf{\Psi}^{(k)} (\mathbf{\Psi}^{(k)})\t; \\
       \mathbf{\Phi}^{(k)}_5 &= \U_k (\mathbf{\Psi}^{(k)})\t; \\
       \mathbf{\Phi}^{(k)}_6 &= \mathbf{\Psi}^{(k)} \U_k\t.
\end{align*}
Note that by the proof of \cref{thm:eigenvectornormality} it holds on $\mathcal{E}_{\mathrm{Theorem} \ \ref{thm:eigenvectornormality}}$ that
\begin{align*}
    \| \mathbf{Z}_k \mathbf{V}_k \mathbf{\Lambda}_k^{\inv} \|_{2,\infty} &\lesssim \frac{\mu_0 \sqrt{r\log(p)}}{\lambda}; \\
    \| \mathbf{Z}_k \mathbf{V}_k \| &\lesssim \sqrt{p},
\end{align*}
where the second inequality holds on the event $\mathcal{E}_{\mathrm{Good}}$  (which is included implicitly in the event $\mathcal{E}_{\mathrm{Theorem} \ \ref{thm:eigenvectornormality}}$).  Therefore,
\begin{align*}
    \| \mathbf{\Phi}_1^{(k)} \|_{2,\infty} &\lesssim \| \mathbf{Z}_k \mathbf{V}_k \mathbf{\Lambda}_k\inv \|_{2,\infty} \frac{ \| \mathbf{Z}_k \mathbf{V}_k \|}{\lambda} \\
    &\lesssim \frac{\mu_0 \sqrt{r\log(p)}}{\lambda} \frac{\sqrt{p}}{\lambda} \\
    &\lesssim \frac{\kappa^2 \mu_0^2 \log(p) r \sqrt{p}}{\lambda^2}.
\end{align*}
The second term can be bounded directly by noting that
\begin{align*}
    \| \mathbf{\Phi}_2^{(k)} \|_{2,\infty} &\leq \| \mathbf{\Psi}^{(k)} \|_{2,\infty} \| \mathbf{Z}_k \mathbf{V}_k \mathbf{\Lambda}_k\inv \| \\
    &\lesssim \| \mathbf{\Psi}^{(k)} \|_{2,\infty} \frac{\sqrt{p}}{\lambda} \\
    &\ll \| \mathbf{\Psi}^{(k)} \|_{2,\infty} \\
    &\lesssim \frac{\kappa^2 \mu_0^2 \log(p) r \sqrt{p}}{\lambda^2}.
\end{align*}
Next,
\begin{align*}
    \| \mathbf{Z}_k \mathbf{V}_k \mathbf{\Lambda}_k\inv  (\mathbf{\Psi}^{(k)})\t \|_{2,\infty} &\leq   \| \mathbf{Z}_k \mathbf{V}_k \mathbf{\Lambda}_k\inv \|_{2,\infty} \| \mathbf{\Psi}^{(k)} \| \\
    &\lesssim \frac{\mu_0 \sqrt{r\log(p)}}{\lambda} \sqrt{p} \| \mathbf{\Psi}^{(k)} \|_{2,\infty} \\
    &\lesssim \frac{\mu_0 \sqrt{rp\log(p)}}{\lambda} \bigg( \frac{ \kappa^2 \mu_0^2 \log(p) r \sqrt{p}}{\lambda^2} + \frac{\mu_0 r \kappa}{\lambda \sqrt{p}} \bigg) \\%\frac{\kappa^2 \mu_0^2 \log(p) r p}{\lambda^2} \\
   % &= \bigg( \frac{\mu_0 \sqrt{rp \log(p)}}{\lambda} \bigg) \frac{ \kappa^2 \mu_0^2 \log(p) r\sqrt{p}}{\lambda^2} \\
    &\lesssim \frac{ \kappa^2 \mu_0^2 \log(p) r\sqrt{p}}{\lambda^2},
\end{align*}
since $\mu_0^2 r \leq \sqrt{p}$ and $\lambda \gtrsim \kappa p^{3/4} \sqrt{\log(p)}$.  Next,
\begin{align*}
    \| \mathbf{\Phi}^{(k)}_4 \|_{2,\infty} &= \| \mathbf{\Psi}^{(k)} (\mathbf{\Psi}^{(k)})\t \|_{2,\infty} \\
    &\leq \| \mathbf{\Psi}^{(k)} \|_{2,\infty} \sqrt{p} \| \mathbf{\Psi}^{(k)}\|_{2,\infty} \\
    &\lesssim  \bigg( \frac{ \kappa^2 \mu_0^2 \log(p) r \sqrt{p}}{\lambda^2} + \frac{\mu_0 r \kappa}{\lambda \sqrt{p}} \bigg)  \bigg( \frac{ \kappa^2 \mu_0^2 \log(p) r \sqrt{p}}{\lambda^2} + \frac{\mu_0 r \kappa}{\lambda \sqrt{p}} \bigg) \sqrt{p} \\%\frac{ \kappa^2 \mu_0^2 \log(p) r\sqrt{p}}{\lambda^2} \bigg( \frac{ \kappa^2 \mu_0^2 \log(p) rp}{\lambda^2} \bigg) \\
    &\lesssim \frac{ \kappa^2 \mu_0^2 \log(p) r\sqrt{p}}{\lambda^2}.
\end{align*}
Next,
\begin{align*}
    \| \mathbf{\Phi}_5^{(k)} \|_{2,\infty} &= \| \U_k (\mathbf{\Psi}^{(k)})\t \|_{2,\infty} \\
    &\leq \| \U_k \|_{2,\infty} \| \mathbf{\Psi}^{(k)} \| \\
    &\lesssim \sqrt{p} \|\U_k\|_{2,\infty} \| \mathbf{\Psi}^{(k)} \|_{2,\infty} \\
     &\lesssim \mu_0 \sqrt{r} \| \mathbf{\Psi}^{(k)} \|_{2,\infty} \\
     &\lesssim  \mu_0 \sqrt{r}  \bigg( \frac{ \kappa^2 \mu_0^2 \log(p) r \sqrt{p}}{\lambda^2} + \frac{\mu_0 r \kappa}{\lambda \sqrt{p}} \bigg) \\
     &\lesssim \frac{\kappa^2 \mu_0^3 \log(p) r^{3/2} \sqrt{p}}{\lambda^2}   + \frac{\mu_0^2 r^{3/2} \kappa}{\lambda \sqrt{p}}.
     %\frac{ \kappa^2 \mu_0^3 \log(p) r^{3/2}\sqrt{p}}{\lambda^2}.
\end{align*}
The bound on $\|\mathbf{\Phi}_6^{(k)}\|_{2,\infty}$ holds from the same bound as $\|\mathbf{\Psi}^{(k)}\|_{2,\infty}$, which completes the proof.
\end{proof}

\subsection{Proof of \cref{thm:eigenvectornormality2}} \label{sec:eigenvectornormality2proof}

\begin{proof}
 By \cref{thm:eigenvectornormality}, it holds that
\begin{align*}
    e_m\t \bigg(\uhat_k^{(t)} (\mathbf{W}_k^{(t)} )\t - \U_k\bigg) &= e_m\t \mathbf{Z}_k \mathbf{V}_k \mathbf{\Lambda}_k\inv + O\bigg( \frac{\sigma^2 \kappa^2 \mu_0^2 \log(p) r \sqrt{p}}{\lambda^2} + \frac{\mu_0 r \kappa \sigma}{\lambda \sqrt{p}}\bigg).
\end{align*}
It is straightforward to verify that the covariance of the vector $e_m\t \mathbf{Z}_k \mathbf{V}_k \mathbf{\Lambda}_k\inv$ is given by 
\begin{align*}
    \mathbf{\Lambda}_k\inv \mathbf{V}_k\t \Sigma^{(m)}_k \mathbf{V}_k \mathbf{\Lambda}_k\inv = \mathbf{\Gamma}^{(m)}_k.
\end{align*}
It is also straightforward to note that 
\begin{align*}
    \lambda_{\min} \bigg( \mathbf{\Gamma}^{(m)}_k \bigg) &= \min_{x \in \mathbb{R}^{r_k}: \|x \| = 1} \| (\Sigma^{(m)}_k)^{1/2} \mathbf{V}_k \mathbf{\Lambda}_k\inv x \|^2 \\
    &\geq \sigma_{\min}^2 \| \mathbf{V}_k \mathbf{\Lambda}_k\inv x \|^2 \\
    &\geq \frac{\sigma_{\min}^2 }{\lambda^2}.
\end{align*}
Hence $\mathbf{\Gamma}^{(m)}_k$ is invertible, and its inverse has maximum eigenvalue at most $\lambda^2/\sigma_{\min}^2 \lesssim \lambda^2 \sigma^2$ by Assumption \ref{ass:noise}.  Therefore, we have that
\begin{align*}
    e_m\t \bigg(& \uhat_k^{(t)} (\mathbf{W}_k^{(t)}) - \U_k \bigg) \big( \mathbf{\Gamma}^{(m)}_k \big)^{-1/2} \\
    &= e_m\t \mathbf{Z}_k \mathbf{V}_k \mathbf{\Lambda}_k\inv \big( \mathbf{\Gamma}^{(m)}_k \big)^{-1/2} + O\bigg( \frac{\sigma^2 \kappa^2 \mu_0^2 \log(p) r\sqrt{p}}{\lambda^2} \| (\mathbf{\Gamma}^{(m)}_k)^{-1/2} \| \bigg) \\ 
    &\quad + O\bigg(  \frac{\sigma\mu_0 r \kappa}{\lambda \sqrt{p}}  \| (\mathbf{\Gamma}^{(m)}_k)^{-1/2} \| \bigg) \\
    &= e_m\t \mathbf{Z}_k \mathbf{V}_k \mathbf{\Lambda}_k\inv \big( \mathbf{\Gamma}^{(m)}_k \big)^{-1/2} + O\bigg( \frac{\sigma \kappa^3 \mu_0^2 \log(p) r\sqrt{p}}{\lambda} + \frac{\mu_0 r \kappa}{\sqrt{p}} \bigg).
\end{align*}
We will apply Corollary 2.2 of \citet{shao_berryesseen_2022}, with (in their notation)
\begin{align*}
    \xi_l &= (\mathbf{Z}_k)_{ml} \bigg( \mathbf{V}_k \mathbf{\Lambda}_k\inv \big( \mathbf{\Gamma}^{(m)}_k \big)^{-1/2} \bigg)_{l\cdot}; \\
    \Delta^{(l)} = \Delta &= C \bigg(\frac{\sigma \kappa^3 \mu_0^2 \log(p) r\sqrt{p}}{\lambda} + \frac{\mu_0 r \kappa}{\sqrt{p}}\bigg).
\end{align*}
By Corollary 2.2 of \citet{shao_berryesseen_2022}, it holds that
\begin{align*}
    \sup_{A \in \mathcal{A}} | &\p\bigg\{ e_m\t \bigg( \uhat_k^{(t)} (\mathbf{W}_k^{(t)})\t - \U_k \bigg) (\mathbf{\Gamma}^{(m)}_k)^{-1/2} \in A \bigg\} - \p \{ Z \in A \} | \\
    &\lesssim r^{1/2} \sum_{l=1}^{p_{-k}} \E \| (\mathbf{Z}_k)_{ml} (\mathbf{V}_k \mathbf{\Lambda}_k\inv \mathbf{\Gamma}^{(m)}_k)^{-1/2}  )_{l\cdot}\|^3 \\
    &\quad + \bigg( \frac{\sigma\kappa^3 \mu_0^2 \log(p) r \sqrt{p}}{\lambda} + \frac{\mu_0 r \kappa}{\sqrt{p}} \bigg) \E \| e_m\t \mathbf{Z}_k \mathbf{V}_k \mathbf{\Lambda}_k\inv(\mathbf{\Gamma}^{(m)}_k)^{-1/2} \| + O(p^{-9}),
\end{align*}
where $\mathcal{A}$ is the collection of convex sets in $\mathbb{R}^r$ and $Z$ is an $r$-dimensional Gaussian random vector with identity covariance. To bound the first term, observe that by subgaussianity (Proposition 2.5.2 of  \citet{vershynin_high-dimensional_2018}),
\begin{align*}
    \sum_{l=1}^{p_{-k}} &\E \| (\mathbf{Z}_k)_{ml} (\mathbf{V}_k \mathbf{\Lambda}_k\inv (\mathbf{\Gamma}^{(m)}_k)^{-1/2} )_{l\cdot}\|^3 \\
    &\lesssim \sigma^3 \sum_{l=1}^{p_{-k}} \| (\mathbf{V}_k \mathbf{\Lambda}_k\inv (\mathbf{\Gamma}^{(m)}_k)^{-1/2} )_{l\cdot}\|^3  \\
    &\lesssim \sigma^3 \max_{l} \| (\mathbf{V}_k \mathbf{\Lambda}_k\inv (\mathbf{\Gamma}^{(m)}_k)^{-1/2} )_{l\cdot}) \| \sum_{l=1}^{p_{-k}}  \| (\mathbf{V}_k \mathbf{\Lambda}_k\inv (\mathbf{\Gamma}^{(m)}_k)^{-1/2} )_{l\cdot}\|^2 \\
    &\lesssim \sigma^3 \| \mathbf{V}_k \|_{2,\infty} \frac{\| (\mathbf{\Gamma}^{(m)}_k)^{-1/2} \|}{\lambda} \| \mathbf{V}_k \mathbf{\Lambda}_k\inv (\mathbf{\Gamma}^{(m)}_k)^{-1/2} \|^2_F \\
    &\lesssim \sigma^3 \mu_0 \frac{\sqrt{r}}{p} \frac{\lambda/\sigma_{\min}}{\lambda} \| \mathbf{V}_k \|_F^2 \| \mathbf{\Lambda}_k\inv \|^2 \| \mathbf{\Gamma}^{(m)}_k)^{-1/2} \|^2 \\
    &\lesssim \sigma^3 \mu_0 \frac{\sqrt{r}}{p} \frac{1}{\sigma_{\min}} r \frac{\lambda^2/\sigma_{\min}^2}{\lambda^2} \\
    &\lesssim \frac{\sigma^3}{\sigma_{\min}^3} \frac{\mu_0 r^{3/2}}{p} \\
    &\lesssim \mu_0 \frac{r^{3/2}}{p},
\end{align*}
where we have used the fact that $\|\mathbf{V}_k\|^2_F = r$ and that $\sigma/\sigma_{\min} = O(1)$.  In addition, we note that
\begin{align*}
    \E \| e_m\t \mathbf{Z}_k \mathbf{V}_k \mathbf{\Lambda}_k\inv (\mathbf{\Gamma}^{(m)}_k)^{-1/2} \| &\lesssim \sqrt{r},
\end{align*}
since the vector $e_m\t\mathbf{Z}_k\mathbf{V}_k \mathbf{\Lambda}_k\inv (\mathbf{\Gamma}^{(m)}_k)^{-1/2} $ is isotropic and subgaussian.  Therefore,
\begin{align*}
     \sup_{A \in \mathcal{A}} | &\p\bigg\{ e_m\t \bigg( \uhat_k^{(t)} (\mathbf{W}_k^{(t)})\t - \U_k \bigg) (\mathbf{\Gamma}^{(m)}_k)^{-1/2} \in A \bigg\} - \p \{ Z \in A \} | \\
    &\lesssim r^{1/2} \sum_{l=1}^{p_{-k}} \E \| (\mathbf{Z}_k)_{ml} (\mathbf{V}_k \mathbf{\Lambda}_k\inv \mathbf{\Gamma}^{(m)}_k)^{-1/2}  )_{l\cdot}\|^3 \\
    &\quad + \bigg( \frac{\sigma\kappa^3 \mu_0^2 \log(p) r \sqrt{p}}{\lambda} + \frac{\mu_0 r \kappa }{ \sqrt{p}}\bigg)\E \| e_m\t \mathbf{Z}_k \mathbf{V}_k \mathbf{\Lambda}_k\inv(\mathbf{\Gamma}^{(m)}_k)^{-1/2} \| + O(p^{-9}) \\
    &\lesssim \mu_0 \frac{r^2}{p} + \frac{\sigma\kappa^3 \mu_0^2 \log(p) r^{3/2} \sqrt{p}}{\lambda}  + \frac{\mu_0 r^{3/2} \kappa }{ \sqrt{p}} + O(p^{-9}) \\
    &\lesssim \mu_0 \frac{r^2}{p} + \frac{\sigma\kappa^3 \mu_0^2 \log(p) r^{3/2} \sqrt{p}}{\lambda}  + \frac{\mu_0 r^{3/2} \kappa }{ \sqrt{p}}.
\end{align*}
This completes the proof.
%
%
%
%Note that by subgaussianity,
%\begin{align*}
%    \sum_{l=1}^{p_{-k}} \E \| (\mathbf{Z}_k)_{ml} (\mathbf{V}_k)_{l\cdot} \|^3 %&\lesssim \sum_{l=1}^{p_{-k}} (\mathbf{V}_k)_{l\cdot} \|^3 \E | %(\mathbf{Z}_k)_{ml} |^3 \\
%    &\lesssim \mu_0 \frac{\sqrt{r}}{p} \sum_{l=1}^{p_{-k}} %(\mathbf{V}_k)_{l\cdot} \|^2  \\
%    &\leq  \mu_0 \frac{r^{3/2}}{p}.
%\end{align*}
%In addition, also by subgaussianity, it holds that $\E \| e_m\t \mathbf{Z}_k %\mathbf{V}_k \| \lesssim \sqrt{r}$. Combinining these together yields
%\begin{align*}
%     \sup_{A \in \mathcal{A}} | &\p\bigg\{ e_m\t \bigg( \uhat_k^{(t)} %(\mathbf{W}_k^{(t)})\t - \U_k \bigg) \mathbf{\Lambda}_k \in A \bigg\} - \p \{ Z \in A \} | %&\lesssim \mu_0 \frac{r^2}{p} + \frac{\kappa^3 \mu_0^2 \log(p) r^{3/2} %\sqrt{p}}{\lambda},
%\end{align*}
%which completes the proof.
\end{proof}

% We rely heavily on the following corollary to characterize the limiting distribution of the entries of the tensor.
% \begin{corollary} \label{cor:asymptoticnormality_projection}
% Under the conditions of \cref{thm:eigenvectornormality} it holds that under the event $\mathcal{E}_{\mathrm{VeryGood}}$
% \begin{align*}
%     \uhat_k^{(t)} (\uhat_k^{(t)})\t - \U_k \U_k\t &= \U_k \mathbf{\Lambda}_k\inv \mathbf{V}_k \mathbf{Z}_k\t + \mathbf{Z}_k \mathbf{V}_k \mathbf{\Lambda}_k\inv \U_k\t + \mathbf{\Phi}^{(k)},
% \end{align*}
% where
% \begin{align*}
%     \big\| \mathbf{\Phi}^{(k)} \big\|_{2,\infty} &\lesssim \frac{\sigma^2\kappa^2 \mu_0^3 \log(p) r^{3/2} \sqrt{p}}{\lambda^2}.  
% \end{align*}
% \end{corollary}

\subsection{Proofs of Preliminary Lemmas from \cref{sec:prelimlemmas}} \label{sec:prelimproofs}
This section contains all of the proofs from \cref{sec:prelimlemmas}.  

\subsubsection{Proof of \cref{lem:linearapprox}}
\begin{proof}[Proof of \cref{lem:linearapprox}]
First, the following expansion holds:
\begin{align*}
    e_m\t \bigg(\uhat_k^{(t)} - &\U_k \mathbf{W}_k^{(t)}  - (\mathbf{I} - \U_k \U_k\t) \mathbf{Z}_k \mathcal{\hat P}^{(t)}_k \mathbf{V}_k \mathbf{\Lambda}_k \U_k\t \uhat_k^{(t)} \mathbf{\hat{\Lambda}}_k^{-2} \bigg) \\
    &= e_m\t \mathbf{Q}_k^{(t)} + e_m\t \U_k (\U_k\t \uhat_k^{(t)} - \mathbf{W}_k^{(t)}),
\end{align*}
where $\mathbf{Q}_k^{(t)}$ is the quadratic term defined in \cref{sec:notation2} \textcolor{black}{(see the proof of Theorem 2 of \citet{agterberg_estimating_2022} for details on this expansion)}.  
Without loss of generality consider the case $k = 1$.  
Recall we assume $t_0$ is such that \cref{thm:twoinfty} holds.  For $t = t_0 + 1$, where $t_0$ is such that \cref{thm:twoinfty} holds, on the event $\mathcal{E}_{\mathrm{main}}^{t,1} \cap \mathcal{E}_{\mathrm{Good}}$, %\textcolor{black}{by \eqref{sinthetaegood}}, 
the following bounds hold:
\begin{align*}
    \tau_1 &\lesssim \sqrt{pr}; \\
    \eta_1^{(t_0,1-m)}  &\overset{\scalebox{.6}{\text{\cref{sinthetaloo14}}}}{\lesssim} \mu_0 \sqrt{\frac{r}{p}} \frac{\dl}{\lambda} ;\\
    \eta_1^{(t_0)} &\overset{\scalebox{.6}{\text{\cref{sinthetaegood}}}}{\lesssim} \frac{\dl}{\lambda} ;\\
    \| \sin\Theta(\uhat_1^{(t_0-1)}, \U_1^{(t_0-1)}) \| &\overset{\scalebox{.6}{\text{\cref{sinthetaegood}}}}{\lesssim} \frac{\dl}{\lambda}; \\
    \| \U_1\t \mathbf{Z}_1 \mathcal{P}_{\U_2}\otimes\mathcal{P}_{\U_3} \| &\lesssim r + \sqrt{\log(p)}.
\end{align*}
Next by Lemma 5 of \citet{agterberg_estimating_2022}, it holds that
\begin{align*}
    \| e_m\t \mathbf{Q}_1^{(t)} \| &\leq \frac{4}{\lambda^2} \| \U_1 \|_{2,\infty} \bigg( \tau_1 \eta_1^{(t)} + \| \U_k\t \mathbf{Z}_k\big[ \mathcal{P}_{\U_{k+1} \otimes \U_{k+2}} \big] \| \bigg) + \frac{16}{\lambda^2} \tau_1^2 \eta_k^{(t,k-m)} \\
    &\quad + \frac{4}{\lambda^2} \xi_1^{t,1-m} \bigg( \tau_k \|\sin\Theta(\uhat_k^{(t)}, \U_k) \| + \tau_k \eta_k^{(t-1)} + \| \U_k\t \mathbf{Z}_k \big[ \mathcal{P}_{\U_{k+1} \otimes \U_{k+2}} \big] \| \bigg),
\end{align*}
where the notation is defined in \cref{sec:notation2}, and the bound above holds whenever $\lambda/2 \leq \lambda_{r_1}\big( \mathbf{\hat \Lambda}_k^{(t)}\big)$, which holds by \cref{lem:eigengaps} on the event $\mathcal{E}_{\mathrm{Good}}$.  Combining this bound with the bounds on the event $\mathcal{E}_{\mathrm{main}}^{t,1} \cap \mathcal{E}_{\mathrm{Good}}$ we obtain
\begin{align*}
    \| e_m\t \mathbf{Q}_1^{t_0} \| &\lesssim  \frac{\|\U_1\|_{2,\infty}}{\lambda^2} \bigg( \sqrt{pr} \frac{\dl}{\lambda} + r + \sqrt{\log(p)} \bigg) + \frac{pr}{\lambda^2} \frac{\dl}{\lambda} \mu_0 \sqrt{\frac{r}{p}};\\
    &\qquad + \frac{\xi_1^{t,1-m}}{\lambda^2} \bigg( \sqrt{pr} \frac{\dl}{\lambda} + r + \sqrt{\log(p)} \bigg),
\end{align*}
where we recall
\begin{align*}
    \xi_1^{t,1-m} &= \bigg\| \bigg( \mathbf{Z}_1^{1-m} - \mathbf{Z}_1 \bigg) \mathcal{\tilde P}_1^{t_0,1-m} \bigg\|.
\end{align*}
By Lemma 11 of \citet{agterberg_estimating_2022} \textcolor{black}{and the fact that $\frac{1}{2^{t_0}} \leq \frac{\dl}{\lambda}$}, on the event $\mathcal{E}_{\mathrm{Good}} \cap \mathcal{E}_{\mathrm{main}}^{t_0,1}$  it holds  with probability at least $1 - O(p^{-30})$ that
\begin{align*}
    \xi_1^{t,1-m} &\lesssim r \mu_0^2 \sqrt{\log(p)},
\end{align*}
where their proof of that lemma uses the fact that the events $\mathcal{\tilde E}$ defined in \cref{sec:notation2} have empty intersection with the event $\mathcal{E}_{\mathrm{Good}} \cap \mathcal{E}_{\mathrm{main}}^{t_0,1}$.  
Recalling $\dl = C_0 \kappa \sqrt{p\log(p)}$ and simplifying these bounds yields
\begin{align*}
    \| e_m\t \mathbf{Q}_1^{t_0} \| &\lesssim  \frac{\mu_0}{\lambda^2} \sqrt{\frac{r}{p}} \bigg( \sqrt{pr} \frac{C_0 \kappa\sqrt{p\log(p)}}{\lambda} + r + \sqrt{\log(p)} \bigg) + \frac{pr}{\lambda^2} \frac{C_0 \kappa\sqrt{p\log(p)}}{\lambda} \mu_0 \sqrt{\frac{r}{p}};\\
    &\qquad + \frac{r \mu_0^2 \sqrt{\log(p)}}{\lambda^2} \bigg( \sqrt{pr} \frac{C_0 \kappa\sqrt{p\log(p)}}{\lambda} + r + \sqrt{\log(p)} \bigg) \\
    &\lesssim \frac{\mu_0 r \kappa \sqrt{p\log(p)}}{\lambda^3} + \mu_0\frac{r^{3/2} + \sqrt{r\log(p)}}{\sqrt{p}\lambda^2} + \frac{pr^{3/2} \kappa \mu_0 \sqrt{\log(p)}}{\lambda^3}; \\
    &\qquad +  \frac{\kappa \mu_0^2 r^{3/2} p \log(p)}{\lambda^3} + r \mu_0^2 \sqrt{\log(p)} \frac{r + \sqrt{\log(p)}}{\lambda^2}  \\
    &\lesssim \frac{\kappa \mu_0^2 r^{3/2} p \log(p)}{\lambda^3} + \frac{\mu_0^2 \big(r^2 \sqrt{\log(p)} + r \log(p) \big)}{\lambda^2}.
\end{align*} 
Similarly, it holds on the event $\mathcal{E}_{\mathrm{Good}}$ that
\begin{align*}
    \|e_m\t \U_1 \big( \U_1\t \uhat_1^{(t)} - \mathbf{W}_1^{(t)} \big) \| &\leq \| \U_1\|_{2,\infty} \|\sin\Theta(\uhat_1^{(t)},\U_1) \|^2 \\
    &\overset{\text{\cref{sintheta14} and \eqref{sinthetaegood}}}{\lesssim} \mu_0 \sqrt{\frac{r}{p}} \frac{\dl^2}{\lambda^2} \\
    &\lesssim \mu_0 \frac{ \kappa^2 \sqrt{pr} \log(p)}{\lambda^2}.
\end{align*}
Combining these bounds and taking a union bound over all $p_1$ rows yields that with probability at least $1 - O(p^{-9})$ that
\begin{align*}
    \|   e_m\t &\bigg(\uhat_k^{(t)} - \U_k \mathbf{W}_k^{(t)}  - (\mathbf{I} - \U_k \U_k\t) \mathbf{Z}_k \mathcal{\hat P}^{(t)}_k \mathbf{V}_k \mathbf{\Lambda}_k \U_k\t \uhat_k^{(t)} \mathbf{\hat{\Lambda}}_k^{-2} \bigg)\|  \\
    &\leq  \| e_m\t \mathbf{Q}_k^{(t)} \| + \|e_m\t \U_k (\U_k\t \uhat_k^{(t)} - \mathbf{W}_k^{(t)}) \| \\
    &\lesssim \frac{\kappa \mu_0^2 r^{3/2} p \log(p)}{\lambda^3} + \frac{\mu_0^2 \big(r^2 \sqrt{\log(p)} + r \log(p) \big)}{\lambda^2} + \frac{\mu_0 \kappa^2 \sqrt{pr}\log(p)}{\lambda^2}
\end{align*}
which completes the proof.
\end{proof}

\subsubsection{Proof of \cref{uperplemma}}
\begin{proof}[Proof of \cref{uperplemma}]
Observe that
\begin{align*}
    \| \U_k \U_k\t &\mathbf{Z}_k\mathcal{\hat P}_k^{(t)} \mathbf{V}_k \mathbf{\Lambda}_k \U_k\t \uhat_k^{(t)} ( \mathbf{\hat \Lambda}_k^{(t)})^{-2} \|_{2,\infty} \\
    &\leq \| \U_k \|_{2\infty} \| \mathbf{\Lambda}_k \| \|(\mathbf{\hat \Lambda}_k^{(t)})^{-2} \| \bigg( \| \U_k\t \mathbf{Z}_k \big( \mathcal{P}_k -\mathcal{\hat P}_k^{(t)} \big) \mathbf{V}_k \| + \| \U_k\t \mathbf{Z}_k \mathbf{V}_k \|  \bigg)\\
    &\lesssim \mu_0 \sqrt{\frac{r}{p}} \frac{\kappa}{\lambda} \bigg( \| \U_k\t \mathbf{Z}_k \big( \mathcal{P}_k -\mathcal{\hat P}_k^{(t)} \big) \mathbf{V}_k \| + \| \U_k\t \mathbf{Z}_k \mathbf{V}_k \|  \bigg) \\
    &\lesssim \mu_0 \sqrt{\frac{r}{p}} \frac{\kappa}{\lambda} \bigg( \| \mathbf{Z}_k \big( \mathcal{P}_{\U_{k+1}} \otimes \mathcal{P}_{\U_{k+2}}  -\mathcal{P}_{\uhat_{k+1}^{(t)}} \otimes \mathcal{P}_{\uhat_{k+2}^{(t)}}  \big)  \| + \sqrt{r} \bigg) \\
    &\lesssim \mu_0 \sqrt{\frac{r}{p}} \frac{\kappa}{\lambda} \bigg( \sqrt{rp} \frac{\kappa \sqrt{p\log(p)}}{\lambda} + \sqrt{r} \bigg)  \\
    &\asymp \frac{\mu_0 r \kappa^2 \sqrt{p\log(p)}}{\lambda^2} + \frac{\mu_0 r \kappa}{\lambda \sqrt{p}},
\end{align*}
where we have implicitly used the fact that
\begin{align*}
    \| \mathbf{Z}_k \big(& \mathcal{P}_{\U_{k+1}} \otimes \mathcal{P}_{\U_{k+2}}  -\mathcal{P}_{\uhat_{k+1}^{(t)}} \otimes \mathcal{P}_{\uhat_{k+2}^{(t)}}  \big)  \|\\
    &\leq  \| \mathbf{Z}_k \big( \mathcal{P}_{\U_{k+1}} \otimes \bigg( \mathcal{P}_{\uhat^{(t)}_{k+2}}  - \mathcal{P}_{\uhat_{k+2}^{(t)}}  \bigg)   \| + \| \mathbf{Z}_k \bigg( \mathcal{P}_{\U_{k+1}} - \mathcal{P}_{\uhat^{(t)}_{k+1}} \bigg) \otimes  \mathcal{P}_{\uhat^{(t)}_{k+2}} \| \\
    &\leq \tau_k \bigg( \|\sin\Theta(\uhat_{k+1}, \U_{k+1}) \| + \| \sin\Theta(\uhat_{k+2}, \U_{k+2}) \| \bigg)\\
    &\lesssim \sqrt{pr} \frac{\kappa \sqrt{p\log(p)}}{\lambda},
\end{align*}
where the final bound holds on the event $\mathcal{E}_{\mathrm{Good}}$ \textcolor{black}{by \eqref{sinthetaegood}}.  
\end{proof}

\subsubsection{Proof of \cref{lem:empiricallinearreplacement}}
\begin{proof}[Proof of \cref{lem:empiricallinearreplacement}]
Without loss of generality, we consider $k = 1$.  On the event $\mathcal{E}_{\mathrm{Good}} \cap \mathcal{E}_{\mathrm{main}}^{t,1}$, it holds that
\begin{align*}
    \bigg\| e_m\t \mathbf{Z}_1 \bigg( \mathcal{\hat P}_1^{(t)} -  \mathcal{P}_1 \bigg) \mathbf{V}_1 \mathbf{\Lambda}_1 \U_1\t \uhat_1^{(t)}  (\mathbf{\hat\Lambda}_1^{(t)})^{-2} \bigg\| &\lesssim \frac{\kappa}{\lambda} \bigg\| e_m\t \mathbf{Z}_1 \bigg( \mathcal{\tilde P}_{1}^{t,1-m} - \mathcal{\hat P}_1^{(t)} \bigg) \mathbf{V}_k \bigg\| \\
    &\qquad + \frac{\kappa}{\lambda} \bigg\| e_m\t \mathbf{Z}_1 \bigg( \mathcal{\tilde P}_{1}^{t,1-m} - \mathcal{ P}_1 \bigg) \mathbf{V}_k \bigg\| \\
    &\coloneqq (I) + (II).
\end{align*}
For term $(I)$, we note that
\begin{align*}
    (I) &= \frac{\kappa}{\lambda} \bigg\| e_m\t \mathbf{Z}_1 \bigg( \mathcal{\tilde P}_{1}^{t,1-m} - \mathcal{\hat P}_1^{(t)} \bigg) \mathbf{V}_k \bigg\| \\
    &= \frac{\kappa}{\lambda} \bigg\| e_m\t \mathbf{Z}_1 \bigg( \mathcal{P}_{\utilde_2^{(t-1,1-m)}} \otimes \mathcal{P}_{\utilde_3^{(t-1,1-m)}} - \mathcal{P}_{\uhat_2^{(t-1)}} \otimes \mathcal{P}_{\uhat_3^{(t-1)}} \bigg) \mathbf{V}_k \bigg\| \\
    &\leq \frac{\kappa}{\lambda} \bigg\| e_m\t \mathbf{Z}_1 \bigg( \bigg[\mathcal{P}_{\utilde_2}^{(t-1,1-m)} - \mathcal{P}_{\uhat_2^{(t-1)}} \bigg] \otimes \mathcal{P}_{\uhat_3^{(t-1)}} \bigg) \mathbf{V}_k \bigg\| \\
    &\qquad + \frac{\kappa}{\lambda} \bigg\| e_m\t \mathbf{Z}_1 \bigg( \mathcal{P}_{\utilde_2}^{(t-1,1-m)} \otimes \bigg[  \mathcal{P}_{\utilde_3^{(t-1,1-m)}} -  \mathcal{P}_{\uhat_3^{(t-1)}} \bigg]\bigg) \mathbf{V}_k \bigg\| \\
    &\lesssim \frac{\kappa \tau_k}{\lambda}  \bigg(\| \sin\Theta(\utilde_2^{(t-1,1-m)},\uhat_2^{(t-1)}) \|  + \| \sin\Theta(\utilde_3^{(t-1,1-m)},\uhat_3^{(t-1)}) \| \bigg) \\
    &\overset{\scalebox{.6}{\text{\cref{sinthetaloo14}}}}{\lesssim} \frac{\kappa \sqrt{pr}}{\lambda} \frac{\dl}{\lambda} \mu_0 \sqrt{\frac{r}{p}}\\
    &\lesssim \frac{ \mu_0 \kappa^2 r \sqrt{p\log(p)}}{\lambda^2},
\end{align*}
where the penultimate bound holds on the event $\mathcal{E}_{\mathrm{main}}^{t_0,1}$.  
Similarly, for term $(II)$, by Lemma 16 of \citet{agterberg_estimating_2022}, it holds that with probability at least $1 - O(p^{-30})$ that
\begin{align*}
    (II) &=  \frac{\kappa}{\lambda} \bigg\| e_m\t \mathbf{Z}_1 \bigg( \mathcal{\tilde P}_{1}^{t,1-m} - \mathcal{ P}_1 \bigg) \mathbf{V}_k \bigg\| \\
    &\lesssim \frac{\kappa}{\lambda} p \sqrt{\log(p)} \bigg\| \bigg( \mathcal{\tilde P}_{1}^{t,1-m} - \mathcal{ P}_1 \bigg) \mathbf{V}_k \bigg\|_{2,\infty}.
\end{align*}
We now observe that
\begin{align*}
    \| \bigg( \mathcal{\tilde P}_1^{t,1-m} - \mathcal{P}_1 \bigg) \mathbf{V}_k \bigg\|_{2,\infty} &=  \bigg\| \bigg( \mathcal{P}_{\utilde_2^{(t-1,1-m)}} \otimes \mathcal{P}_{\utilde_3^{(t-1,1-m)}} - \mathcal{P}_{\U_2} \otimes \mathcal{P}_{\U_3} \bigg) \mathbf{V}_k \bigg\|_{2,\infty}  \\
    &\leq \bigg\| \bigg( \mathcal{P}_{\utilde_2^{(t-1,1-m)}} \otimes \bigg[ \mathcal{P}_{\utilde_3^{(t-1,1-m)}} - \mathcal{P}_{\U_3} \bigg]\bigg) \mathbf{V}_k \bigg\|_{2,\infty}  \\
    &\qquad + \bigg\| \bigg(\bigg[ \mathcal{P}_{\utilde_2^{(t-1,1-m)}} - \mathcal{P}_{\U_2} \bigg] \otimes  \mathcal{P}_{\U_3} \bigg] \bigg) \mathbf{V}_k \bigg\|_{2,\infty} \\
    &\leq \| \mathcal{P}_{\utilde_2^{(t-1,1-m)}} \|_{2,\infty} \| \mathcal{P}_{\utilde_3^{(t-1,1-m)}} - \mathcal{P}_{\U_3} \|_{2,\infty} \\
    &\qquad + \| \mathcal{P}_{\utilde_2^{(t-1,1-m)}} - \mathcal{P}_{\U_2} \|_{2,\infty}  \mu_0 \sqrt{\frac{r}{p}} \\
    &\leq \| \mathcal{P}_{\utilde_2^{(t-1,1-m)}} - \mathcal{P}_{\U_2} \|_{2,\infty} \| \mathcal{P}_{\utilde_3^{(t-1,1-m)}} - \mathcal{P}_{\U_3} \|_{2,\infty} \\
    &\qquad + \| \mathcal{P}_{\utilde_3^{(t-1,1-m)}} - \mathcal{P}_{\U_3} \|_{2,\infty} \mu_0 \sqrt{\frac{r}{p}} \\
    &\qquad + \| \mathcal{P}_{\utilde_2^{(t-1,1-m)}} - \mathcal{P}_{\U_2} \|_{2,\infty}  \mu_0 \sqrt{\frac{r}{p}}.
\end{align*}
It is straightforward to check that
\begin{align*}
    \| \mathcal{P}_{\utilde_2^{(t-1,1-m)}} - \mathcal{P}_{\U_2} \|_{2,\infty} &\leq \|  \mathcal{P}_{\utilde_2^{(t-1,1-m)}} - \mathcal{P}_{\uhat_2^{(t-1)}} \|_{2,\infty} + \| \mathcal{P}_{\uhat_2^{(t-1)}}  - \mathcal{P}_{\U_2}\|_{2,\infty} \\
    &\lesssim \|\sin\Theta(\utilde_2^{(t-1,1-m)},\uhat_2^{(t-1)} ) \| \\
    &\quad + \| \U_2 \|_{2,\infty} \|\sin\Theta(\uhat_2^{(t-1)},\U_2) \| + \| \uhat_2^{(t-1)} - \U_2 \mathbf{W}_2^{(t-1)} \|_{2,\infty} \\
    &\lesssim \mu_0 \sqrt{\frac{r}{p}} \frac{\dl}{\lambda},
\end{align*}
where the final bound holds on the event \textcolor{black}{$\mathcal{E}_{\mathrm{main}}^{t_0,1} \cap \mathcal{E}_{\mathrm{Good}}$ by \cref{sinthetaloo14,sinthetaegood}}.  
The same bounds holds with $\mathcal{P}_{\utilde_3^{(t-1,1-m)}}$ in place of $\mathcal{P}_{\utilde_2^{(t-1,1-m)}}$.  Plugging this bound in yields that with probability at least $1 - O(p^{-11})$,
\begin{align*}
    (II)  &\lesssim \frac{\kappa}{\lambda} p \sqrt{\log(p)} \bigg\| \bigg( \mathcal{\tilde P}_{1}^{t,1-m} - \mathcal{ P}_1 \bigg) \mathbf{V}_k \bigg\|_{2,\infty} \\
    &\lesssim \frac{\kappa}{\lambda} p \sqrt{\log(p)} \bigg( \| \mathcal{P}_{\utilde_2^{(t-1,1-m)}} - \mathcal{P}_{\U_2} \|_{2,\infty} \| \mathcal{P}_{\utilde_3^{(t-1,1-m)}} - \mathcal{P}_{\U_3} \|_{2,\infty} \\
    &\qquad + \| \mathcal{P}_{\utilde_3^{(t-1,1-m)}} - \mathcal{P}_{\U_3} \|_{2,\infty} \mu_0 \sqrt{\frac{r}{p}} \\
    &\qquad + \| \mathcal{P}_{\utilde_2^{(t-1,1-m)}} - \mathcal{P}_{\U_2} \|_{2,\infty}  \mu_0 \sqrt{\frac{r}{p}} \bigg) \\
    &\lesssim \frac{\kappa}{\lambda} p \sqrt{\log(p)} \bigg( \mu_0^2 \frac{r}{p} \frac{\dl}{\lambda} \bigg) \\
    &\lesssim \frac{\kappa^2 \mu_0^2 \log(p) r \sqrt{p}}{\lambda^2}.
\end{align*}
Combining our bounds for both $(I)$ and $(II)$, we obtain that
\begin{align*}
    (I) + (II) &\lesssim \frac{\mu_0 \kappa^2 r \sqrt{p\log(p)}}{\lambda^2} + \frac{\kappa^2 \mu_0^2 \log(p) r \sqrt{p}}{\lambda^2} \\
    &\asymp \frac{\kappa^2 \mu_0^2 \log(p) r \sqrt{p}}{\lambda^2}.
\end{align*}
These bounds hold cumulatively with probability at most $1 - O(p^{-10})$, and the proof is completed by taking a union bound over all $p_1$ rows.
\end{proof}

\subsubsection{Proof of \cref{lem:approximatecommute}}

\begin{proof}[Proof of \cref{lem:approximatecommute}]
First, by the eigenvector-eigenvalue equation, it holds that
\begin{align*}
    \U_k \mathbf{\Lambda}_k^2 &= \mathbf{T}_k \mathbf{T}_k\t \U_k; \\
    \uhat_k (\mathbf{\hat \Lambda}_k^{(t)} )^2 &= \bigg( \mathbf{T}_k \mathcal{\hat P}_k^{(t)} \mathbf{T}_k\t + \mathbf{Z}_k \mathcal{\hat P}_k^{(t)} \mathbf{T}_k\t +\mathbf{T}_k\mathcal{\hat P}_k^{(t)} \mathbf{Z}_k\t +\mathbf{Z}_k\mathcal{\hat P}_k^{(t)} \mathbf{Z}_k\t \bigg) \uhat_k.
\end{align*}
Therefore,
\begin{align*}
    \bigg\| \mathbf{\Lambda}_k^2& \U_k\t \uhat_k - \U_k\t \uhat_k (\mathbf{\hat\Lambda}_k^{(t)})^2 \bigg\|\\
    &=   \bigg\| \U_k\t\bigg( \mathbf{T}_k \mathbf{T}_k\t - \mathbf{T}_k \mathcal{\hat P}_k^{(t)} \mathbf{T}_k\t - \mathbf{Z}_k \mathcal{\hat P}_k^{(t)} \mathbf{T}_k\t -\mathbf{T}_k\mathcal{\hat P}_k^{(t)} \mathbf{Z}_k\t -\mathbf{Z}_k\mathcal{\hat P}_k^{(t)} \mathbf{Z}_k\t \bigg) \uhat_k \bigg\|
    \\
    &\leq +     \bigg\| \U_k\t \bigg( \mathbf{T}_k \mathbf{T}_k\t - \mathbf{T}_k \mathcal{\hat P}_k^{(t)} \mathbf{T}_k\t \bigg) \uhat_k \bigg\|
    \\
    &\quad +  \bigg\| \U_k\t \bigg( \mathbf{Z}_k \mathcal{\hat P}_k^{(t)} \mathbf{T}_k\t + \mathbf{T}_k \mathcal{\hat P}_k^{(t)} \mathbf{Z}_k\t + \mathbf{Z}_k \mathcal{\hat P}_k^{(t)} \mathbf{Z}_k\t \bigg) \uhat_k^{(t)} \bigg\| \\
    &\leq \bigg\| \U_k\t \mathbf{T}_k \bigg( \mathcal{P}_{\U_{k+1} \otimes \U_{k+2}} - \mathcal{\hat P}_k^{(t)} \bigg) \mathbf{T}_k\t \uhat_k \bigg\| \\
    &\quad + \bigg\| \U_k\t \mathbf{Z}_k \mathcal{\hat P}_k^{(t)} \mathbf{T}_k\t \uhat_k^{(t)} \bigg\| + \bigg\| \U_k\t \mathbf{T}_k \mathcal{\hat P}_k^{(t)} \mathbf{Z}_k\t \uhat_k^{(t)} \bigg\| + \bigg\| \U_k\t \mathbf{Z}_k \mathcal{\hat P}_k^{(t)} \mathbf{Z}_k\t \uhat_k^{(t)} \bigg\| \\
    &\lesssim \lambda_1^2 \frac{\kappa \sqrt{p\log(p)}}{\lambda} + \lambda_1 \tau_k + \tau_k^2 \\
    &\lesssim \kappa^2 \sqrt{p\log(p)} + \lambda_1 \sqrt{pr} \\
    &\lesssim \lambda_1 \sqrt{pr},
\end{align*}
where the final inequality holds when $\lambda_1 \gtrsim \sqrt{p\log(p)}$ together with the assumption that $\kappa \leq p^{1/4}$.  \textcolor{black}{In addition, we have implicitly used the bound
\begin{align*}
    \| \mathcal{P}_{{\mathbf{U}_{k+1}\otimes \U_{k+2}}} - \mathcal{\hat P}_k^{(t)} \| \leq \| \sin\Theta(\uhat_{k+1}^{(t)},\U_{k+1}) \| + \| \sin\Theta(\uhat_{k+2}^{(t)},\U_{k+2})\| 
    &\lesssim \frac{\dl}{\lambda},
\end{align*}
which holds on the event $\mathcal{E}_{\mathrm{Good}}$ by \eqref{sinthetaegood}.
}
\end{proof}

 \section{Proof of Entrywise Distributional Theory (Theorem \ref{thm:asymptoticnormalityentries}) and $\|\cdot\|_{\max}$ Convergence (Corollary \ref{cor:maxnormbound})}
% Entrywise Guarantees (\cref{thm:asymptoticnormalityentries} and \cref{cor:maxnormbound})}
\label{sec:entrywisedistributionaltheoryproofs}
This section contains the proofs of the entrywise distributional theory and entrywise $\|\cdot\|_{\max}$ convergence.  Throughout this section we suppress the dependence of $\uhat_k^{(t)}$,  and $\mathbf{\hat{\Lambda}}_k^{(t)}$ on $t$.  \cref{sec:entrywiseaux} gives preliminary lemmas, \cref{sec:entrywiseprooc} gives the proof of \cref{thm:asymptoticnormalityentries},
 and \cref{sec:maxnormproof} gives the proof of \cref{cor:maxnormbound}. % and \cref{sec:civalidityproofs} gives the proof of \cref{thm:civalidity}.  

 \subsection{Preliminary Lemmas: Entrywise Residual Bounds and Leading-Order Approximations}
 \label{sec:entrywiseaux}
 This section presents several lemmas needed en route to the proof of \cref{thm:asymptoticnormalityentries}, whose proofs are deferred to \cref{sec:entrywiseauxproofs}.  
% The following two results (\cref{lem:zentrywiselemma} and \cref{lem:tentrywiselemma}) obtain strong concentration bounds for residual terms.  
The following lemma shows that the effect of the projection matrices on the random noise tensor $\mathcal{Z}$ is sufficiently small.

 \begin{lemma} \label{lem:zentrywiselemma}
 Under the conditions of \cref{thm:asymptoticnormalityentries},
 the following bounds hold with probability at least $1 - O(p^{-9})$:
 \begin{align*}
    \bigg| \bigg( &\mathcal{Z} \times_1 \U_1 \U_1\t \times_2 \U_2 \U_2\t \times_3 \U_3 \U_3\t \bigg)_{ijk} \bigg| \\
    &\lesssim \sigma\sqrt{\log(p)} \mu_0^3 \frac{r^{3/2}}{p^{3/2}}; \\
    \bigg| \bigg( &\mathcal{Z} \times_1 (\uhat_1 \uhat_1\t - \U_1 \U_1\t) \times_2 \U_2 \U_2\t \times_3 \U_3 \U_3\t \bigg)_{ijk} \bigg|\\ &\lesssim \frac{\sigma^2 \mu_0^3 r^{3/2} \log(p)}{\lambda \sqrt{p}}; \\
     \bigg| \bigg(& \mathcal{Z} \times_1 (\uhat_1 \uhat_1\t - \U_1 \U_1\t) \times_2 (\uhat_2 \uhat_2\t - \U_1 \U_1\t) \times_3 \U_3 \U_3\t \bigg)_{ijk} \bigg|\\
     &\lesssim  \frac{\sigma^2 \mu_0^3 r^{3/2} \log(p)}{\lambda \sqrt{p}}; \\
     \bigg| \bigg(& \mathcal{Z} \times_1 (\uhat_1 \uhat_1\t - \U_1 \U_1\t) \times_2 (\uhat_2 \uhat_2\t - \U_1 \U_1\t) \times_3 (\uhat_3 \uhat_3\t - \U_3 \U_3\t) \bigg)_{ijk} \bigg| \\&\lesssim  \frac{\sigma^2 \mu_0^3 r^{3/2} \log(p)}{\lambda \sqrt{p}}.
\end{align*}
\end{lemma}
\noindent 
The next lemma shows that the terms involving $\mathcal{T}$ and at least two differences of projection matrices is sufficiently small.

\begin{lemma}\label{lem:tentrywiselemma}
Under the conditions of \cref{thm:asymptoticnormalityentries}, the following bounds hold with probability at least $1 - O(p^{-9})$:
\begin{align*}
  \bigg| \bigg( \mathcal{T}& \times_1 (\uhat_1 \uhat_1\t - \U_1\U_1\t) \times_2 ( \uhat_2 \uhat_2\t - \U_2 \U_2\t) \times_3 (\U_3 \U_3\t) \bigg)_{ijk}  \bigg|\\&\lesssim \frac{\sigma^2 \mu_0^3 r^{3/2} \kappa \log(p)}{\lambda\sqrt{p}}  \\
  \bigg| \bigg( \mathcal{T}& \times_1 (\uhat_1 \uhat_1\t - \U_1\U_1\t) \times_2 ( \uhat_2 \uhat_2\t - \U_2 \U_2\t) \times_3 (\uhat_3 \uhat_3\t - \U_3 \U_3\t) \bigg)_{ijk} \bigg| \\&\lesssim \frac{\sigma^2 \mu_0^3 r^{3/2} \kappa \log(p)}{\lambda\sqrt{p}}.
\end{align*}
\end{lemma}

\noindent 
Finally, the following lemma shows that the leading-order term $\xi_{ijk}/s_{ijk}$ (defined below) is approximately Gaussian.
\begin{lemma}\label{lem:xiijkgaussian}
Assume the conditions of \cref{thm:asymptoticnormalityentries} hold. Define 
\begin{align*}
    \xi_{ijk} &=  e_{i}^{\top} \mathbf{Z}_{1} \mathbf{V}_{1} \mathbf{V}_{1}^{\top} e_{(j-1) p_{3}+k}+e_{j}^{\top} \mathbf{Z}_{2} \mathbf{V}_{2} \mathbf{V}_{2}^{\top} e_{(k-1) p_{1}+i}+e_{k}^{\top} \mathbf{Z}_{3} \mathbf{V}_{3} \mathbf{V}_{3}^{\top} e_{(i-1) p_{2}+j}.
\end{align*}
Then it holds that
\begin{align*}
      \sup_{t\in \mathbb{R}}\bigg| \p\bigg\{ \frac{\xi_{ijk}}{s_{ijk}} \leq t \bigg\} - \Phi(t) \bigg| &\leq \frac{C_1}{\sqrt{p\log(p)}} + \frac{C_2 \mu_0^2 r}{p}.
\end{align*}
\end{lemma}

\subsection{Proof of \cref{thm:asymptoticnormalityentries}} \label{sec:entrywiseprooc}

With these lemmas in hand, we are now prepared to prove \cref{thm:asymptoticnormalityentries}.
\begin{proof}[Proof of \cref{thm:asymptoticnormalityentries}]
First, note that
\begin{align*}
    \mathcal{\hat T}_{ijk} &= \bigg( \big( \mathcal{T + Z} \big) \times_1 (\uhat_1 \uhat_1\t) \times_2 (\uhat_2 \uhat_2\t) \times_3 (\uhat_3 \uhat_3)\t \bigg)_{ijk} \\
    &= \bigg( \mathcal{T} \times_1 (\uhat_1 \uhat_1\t) \times_2 (\uhat_2 \uhat_2\t) \times_3 (\uhat_3 \uhat_3)\t \bigg)_{ijk} \\
    &\quad + \bigg(\mathcal{Z} \times_1 (\uhat_1 \uhat_1\t) \times_2 (\uhat_2 \uhat_2\t) \times_3 (\uhat_3 \uhat_3)\t \bigg)_{ijk}.
\end{align*}
We consider each term separately. First, we will show that the second term is a residual term.  Observe that % For the second term, we note that
\begin{align*}
    \bigg(\mathcal{Z} &\times_1 (\uhat_1 \uhat_1\t) \times_2 (\uhat_2 \uhat_2\t) \times_3 (\uhat_3 \uhat_3)\t \bigg)_{ijk} \\
    &=  \bigg(\mathcal{Z} \times_1 (\uhat_1 \uhat_1\t - \U_1 \U_1\t ) \times_2 (\uhat_2 \uhat_2\t) \times_3 (\uhat_3 \uhat_3)\t \bigg)_{ijk} \\
    &\quad + \bigg(\mathcal{Z} \times_1 ( \U_1 \U_1\t ) \times_2 (\uhat_2 \uhat_2\t) \times_3 (\uhat_3 \uhat_3)\t \bigg)_{ijk} \\
    &= \bigg(\mathcal{Z} \times_1 (\uhat_1 \uhat_1\t - \U_1 \U_1\t ) \times_2 (\uhat_2 \uhat_2\t - \U_2 \U_2\t) \times_3 (\uhat_3 \uhat_3 - \U_3 \U_3)\t \bigg)_{ijk} \\
    &\quad + \bigg(\mathcal{Z} \times_1 (\uhat_1 \uhat_1\t - \U_1 \U_1\t ) \times_2 (\uhat_2 \uhat_2\t - \U_2 \U_2\t) \times_3 ( \U_3 \U_3)\t \bigg)_{ijk} \\
    &\quad + \bigg(\mathcal{Z} \times_1 (\uhat_1 \uhat_1\t - \U_1 \U_1\t ) \times_2 ( \U_2 \U_2\t) \times_3 (\uhat_3 \uhat_3 - \U_3 \U_3)\t \bigg)_{ijk} \\
        &\quad + \bigg(\mathcal{Z} \times_1 (\uhat_1 \uhat_1\t - \U_1 \U_1\t ) \times_2 ( \U_2 \U_2\t) \times_3 (\U_3 \U_3)\t \bigg)_{ijk} \\
    &\quad + \bigg(\mathcal{Z} \times_1 ( \U_1 \U_1\t ) \times_2 (\uhat_2 \uhat_2\t - \U_2 \U_2\t ) \times_3 (\uhat_3 \uhat_3 - \U_3 \U_3)\t \bigg)_{ijk} \\
     &\quad + \bigg(\mathcal{Z} \times_1 ( \U_1 \U_1\t ) \times_2 (\uhat_2 \uhat_2\t - \U_2 \U_2\t ) \times_3 (\U_3 \U_3)\t \bigg)_{ijk} \\
    &\quad + \bigg(\mathcal{Z} \times_1 ( \U_1 \U_1\t ) \times_2 ( \U_2 \U_2\t ) \times_3 (\uhat_3 \uhat_3-\U_3 \U_3)\t \bigg)_{ijk} \\
     &\quad + \bigg(\mathcal{Z} \times_1 ( \U_1 \U_1\t ) \times_2 ( \U_2 \U_2\t ) \times_3 ( \U_3 \U_3)\t \bigg)_{ijk}.
\end{align*}
Each term consists of terms containing either $\U_k \U_k\t$ or the difference $\uhat_k \uhat_k\t - \U_k \U_k\t$.  Therefore, without loss of generality, since $r_k \asymp r$ and $p_k \asymp p$, it suffices to analyze the following four terms:
\begin{align*}
    (I) &\coloneqq \bigg(\mathcal{Z} \times_1 ( \U_1 \U_1\t ) \times_2 ( \U_2 \U_2\t ) \times_3 ( \U_3 \U_3)\t \bigg)_{ijk}; \\
    (II) &\coloneqq \bigg(\mathcal{Z} \times_1 (\uhat_1 \uhat_1\t - \U_1 \U_1\t ) \times_2 ( \U_2 \U_2\t) \times_3 (\U_3 \U_3)\t \bigg)_{ijk} \\
    (III) &\coloneqq \bigg(\mathcal{Z} \times_1 (\uhat_1 \uhat_1\t - \U_1 \U_1\t ) \times_2 (\uhat_2 \uhat_2\t - \U_2 \U_2\t) \times_3 ( \U_3 \U_3)\t \bigg)_{ijk}; \\
    (IV) &\coloneqq \bigg(\mathcal{Z} \times_1 (\uhat_1 \uhat_1\t - \U_1 \U_1\t ) \times_2 (\uhat_2 \uhat_2\t - \U_2 \U_2\t) \times_3 (\uhat_3 \uhat_3 - \U_3 \U_3)\t \bigg)_{ijk}.
\end{align*}
Each of these terms are analyzed in  \cref{lem:zentrywiselemma}.  Therefore, with probability at least $1 - O(p^{-9})$, it holds that
\begin{align*}
   \bigg|  \bigg(\mathcal{Z} \times_1 (\uhat_1 \uhat_1\t) \times_2 (\uhat_2 \uhat_2\t) \times_3 (\uhat_3 \uhat_3)\t \bigg)_{ijk} \bigg| &\lesssim \sigma \mu_0^3 \sqrt{\log(p)} \frac{r^{3/2}}{p^{3/2}} + \frac{\sigma^2 \mu_0^3 r^{3/2} \log(p)}{\lambda \sqrt{p}}.
\end{align*}
We now focus on the term containing $\mathcal{T}$.  The strategy will be similar, only now appealing to the distributional characterization for the projections in  \cref{cor:asymptoticnormality_projection}.  We note that
\begin{align*}
   \bigg(\mathcal{T} \times_1 &\uhat_1 \uhat_1\t \times_2 \uhat_2 \uhat_2\t \times_3 \uhat_3 \uhat_3\t \bigg)_{ijk} \\
   &= \bigg( \mathcal{T} \times_1 (\uhat_1 \uhat_1\t - \U_1\U_1\t) \times_2 \uhat_2 \uhat_2\t \times_3 \uhat_3\uhat_3\t \bigg)_{ijk} \\
   &\quad + \bigg( \mathcal{T} \times_1  \U_1\U_1\t \times_2 \uhat_2 \uhat_2\t \times_3 \uhat_3\uhat_3\t \bigg)_{ijk} \\
   &= \bigg( \mathcal{T} \times_1 (\uhat_1 \uhat_1\t - \U_1\U_1\t) \times_2 ( \uhat_2 \uhat_2\t - \U_2 \U_2\t) \times_3 \uhat_3\uhat_3\t \bigg)_{ijk} \\
   &\quad + \bigg( \mathcal{T} \times_1 (\uhat_1 \uhat_1\t - \U_1\U_1\t) \times_2 \U_2 \U_2\t \times_3 \uhat_3\uhat_3\t \bigg)_{ijk} \\
   &\quad + \bigg( \mathcal{T} \times_1  \U_1\U_1\t \times_2 (\uhat_2 \uhat_2\t - \U_2 \U_2\t) \times_3 \uhat_3\uhat_3\t \bigg)_{ijk} \\
    &\quad + \bigg( \mathcal{T} \times_1  \U_1\U_1\t \times_2 \U_2 \U_2\t \times_3 \uhat_3\uhat_3\t \bigg)_{ijk} \\
    %here
    &= \bigg( \mathcal{T} \times_1 (\uhat_1 \uhat_1\t - \U_1\U_1\t) \times_2 ( \uhat_2 \uhat_2\t - \U_2 \U_2\t) \times_3 (\uhat_3\uhat_3\t - \U_3 \U_3\t) \bigg)_{ijk} \\
    &\quad + \bigg( \mathcal{T} \times_1 (\uhat_1 \uhat_1\t - \U_1\U_1\t) \times_2 ( \uhat_2 \uhat_2\t - \U_2 \U_2\t) \times_3 (\U_3 \U_3\t) \bigg)_{ijk} \\
   &\quad + \bigg( \mathcal{T} \times_1 (\uhat_1 \uhat_1\t - \U_1\U_1\t) \times_2 \U_2 \U_2\t \times_3 (\uhat_3\uhat_3\t - \U_3 \U_3\t)\bigg)_{ijk} \\
   &\quad + \bigg( \mathcal{T} \times_1 (\uhat_1 \uhat_1\t - \U_1\U_1\t) \times_2 \U_2 \U_2\t \times_3 ( \U_3 \U_3\t)\bigg)_{ijk} \\
   &\quad + \bigg( \mathcal{T} \times_1  \U_1\U_1\t \times_2 (\uhat_2 \uhat_2\t - \U_2 \U_2\t) \times_3 (\uhat_3\uhat_3\t - \U_3 \U_3\t) \bigg)_{ijk} \\
   &\quad + \bigg( \mathcal{T} \times_1  \U_1\U_1\t \times_2 (\uhat_2 \uhat_2\t - \U_2 \U_2\t) \times_3 ( \U_3 \U_3\t) \bigg)_{ijk} \\
    &\quad + \bigg( \mathcal{T} \times_1  \U_1\U_1\t \times_2 \U_2 \U_2\t \times_3 (\uhat_3\uhat_3\t  - \U_3 \U_3\t)\bigg)_{ijk} \\
    &\quad + \bigg( \mathcal{T} \times_1  \U_1\U_1\t \times_2 \U_2 \U_2\t \times_3  \U_3 \U_3\t \bigg)_{ijk}.
\end{align*}
The final term is simply $\mathcal{T}_{ijk}$. Similar to the previous case, the terms appearing all appear with either $\U_k \U_k\t$ or the difference $\uhat_k \uhat_k\t - \U_k \U_k\t$.  We will show that terms with at least two projection-norm differences are small-order terms. Again, since $r_k \asymp r$ and $p_k \asymp p$, it suffices to analyze the two terms
\begin{align*}
    (I) &\coloneqq  \bigg( \mathcal{T} \times_1 (\uhat_1 \uhat_1\t - \U_1\U_1\t) \times_2 ( \uhat_2 \uhat_2\t - \U_2 \U_2\t) \times_3 (\U_3 \U_3\t) \bigg)_{ijk}; \\
    (II) &\coloneqq \bigg( \mathcal{T} \times_1 (\uhat_1 \uhat_1\t - \U_1\U_1\t) \times_2 ( \uhat_2 \uhat_2\t - \U_2 \U_2\t) \times_3 (\uhat_3 \uhat_3\t - \U_3 \U_3\t) \bigg)_{ijk}.
\end{align*}
By \cref{lem:tentrywiselemma}, it holds with probability at least $1 - O(p^{-9})$ that
\begin{align*}
    (I) + (II) &\lesssim \frac{\sigma^2 \mu_0^3 r^{3/2} \kappa \log(p)}{\lambda \sqrt{p}}.
\end{align*}
By symmetry, we have shown so far that with probability at least $1 - O(p^{-9})$,
\begin{align*}
    \mathcal{\hat T}_{ijk} - \mathcal{T}_{ijk} &=  \bigg( \mathcal{T} \times_1 (\uhat_1 \uhat_1\t - \U_1\U_1\t) \times_2 \U_2 \U_2\t \times_3 ( \U_3 \U_3\t)\bigg)_{ijk} \\
    &\quad +  \bigg( \mathcal{T} \times_1  \U_1\U_1\t \times_2 (\uhat_2 \uhat_2\t - \U_2 \U_2\t) \times_3 ( \U_3 \U_3\t) \bigg)_{ijk} \\
    &\quad + \bigg( \mathcal{T} \times_1  \U_1\U_1\t \times_2 \U_2 \U_2\t \times_3 (\uhat_3\uhat_3\t  - \U_3 \U_3\t)\bigg)_{ijk} \\
    &\quad + O\bigg( \sigma \mu_0^3 \sqrt{\log(p)} \frac{r^{3/2}}{p^{3/2}} \bigg) \\
    &\quad + O\bigg( \frac{\sigma^2 \mu_0^3 r^{3/2} \kappa \log(p)}{\lambda \sqrt{p}} \bigg).
\end{align*}
We will now argue that the difference terms consist of another leading-order term.  More specifically, considering $k=1$, we will show that
\begin{align*}
    \bigg( \mathcal{T} \times_1 (\uhat_1 \uhat_1\t - \U_1\U_1\t) \times_2 \U_2 \U_2\t \times_3 ( \U_3 \U_3\t)\bigg)_{ijk} &= e_i\t \mathbf{Z}_1 \mathbf{V}_1 \mathbf{V}_1\t e_{(j-1)p_3 + k} + o( s_{ijk}).
\end{align*}
The other indices will follow by symmetry.  

By \cref{cor:asymptoticnormality_projection} on the event $\mathcal{E}_{\mathrm{\cref{thm:eigenvectornormality}}}$ it holds that
\begin{align*}
     \bigg( \mathcal{T} \times_1 &(\uhat_1 \uhat_1\t - \U_1\U_1\t) \times_2 \U_2 \U_2\t \times_3 ( \U_3 \U_3\t)\bigg)_{ijk} \\
     &=  \bigg( \mathcal{T} \times_1 (\mathbf{U}_1 \mathbf{\Lambda}_1\inv \mathbf{V}_1 \mathbf{Z}_1\t + \mathbf{Z}_1 \mathbf{V}_1 \mathbf{\Lambda}_1\inv \U_1\t + \mathbf{\Phi}^{(1)}) \times_2 \U_2 \U_2\t \times_3 ( \U_3 \U_3\t)\bigg)_{ijk} \\
     &= \bigg( \mathcal{T} \times_1 (\mathbf{U}_1 \mathbf{\Lambda}_1\inv \mathbf{V}_1 \mathbf{Z}_1\t ) \times_2 \U_2 \U_2\t \times_3 ( \U_3 \U_3\t)\bigg)_{ijk}  \\
     &\quad +\bigg( \mathcal{T} \times_1 ( \mathbf{Z}_1 \mathbf{V}_1 \mathbf{\Lambda}_1\inv \U_1\t ) \times_2 \U_2 \U_2\t \times_3 ( \U_3 \U_3\t)\bigg)_{ijk} \\
     &\quad +\bigg( \mathcal{T} \times_1 ( \mathbf{\Phi}^{(1)}) \times_2 \U_2 \U_2\t \times_3 ( \U_3 \U_3\t)\bigg)_{ijk}.
\end{align*}
The first term satisfies
\begin{align*}
    \bigg( \mathcal{T} &\times_1 ( \U_1 \mathbf{\Lambda}_1\inv \mathbf{V}_1 \mathbf{Z}_1\t) \times_2 \U_2 \U_2\t \times_3 ( \U_3 \U_3\t)\bigg)_{ijk} \\
    &= \mathcal{M}_1 \bigg( \mathcal{T} \times_1 ( \U_1 \mathbf{\Lambda}_1\inv \mathbf{V}_1 \mathbf{Z}_1\t) \times_2 \U_2 \U_2\t \times_3 ( \U_3 \U_3\t ) \bigg)_{i, (j-1)p_3 + k} \\
    &= e_i\t \mathcal{M}_1 \bigg( \mathcal{T} \times_1 ( \U_1 \mathbf{\Lambda}_1\inv \mathbf{V}_1 \mathbf{Z}_1\t) \bigg) ( \U_2 \U_2\t \otimes \U_3 \U_3\t ) e_{(j-1)p_3 + k} \\
    &= e_i\t \U_1 \mathbf{\Lambda}_1\inv \mathbf{V}_1 \mathbf{Z}_1\t \mathbf{T}_1 ( \U_2 \U_2\t \otimes \U_3 \U_3\t ) e_{(j-1)p_3 + k} \\
    &= e_i\t \U_1 \mathbf{\Lambda}_1\inv \mathbf{V}_1 \mathbf{Z}_1\t \U_1 \mathbf{\Lambda}_1 \mathbf{V}_1\t e_{(j-1)p_3 + k}.
\end{align*}
On the event $\mathcal{E}_{\mathrm{VeryGood}}$ it holds that
\begin{align*}
    \|\U_1 \t \mathbf{Z}_1 \mathbf{V}_1 \| &\lesssim \sigma \sqrt{r}.
\end{align*}
Therefore, on this event,
\begin{align*}
    | e_i\t \U_1 \mathbf{\Lambda}_1\inv \mathbf{V}_1 \mathbf{Z}_1\t \U_1 \mathbf{\Lambda}_1 \mathbf{V}_1\t e_{(j-1)p_3 + k} | &\leq \kappa \| e_i\t \U_1 \| \| e_{(j-1)p_3 + k}\t \mathbf{V}_1 \| \|\U_1 \t \mathbf{Z}_1 \mathbf{V}_1 \| \\
    &\lesssim \kappa \sigma \mu_0^2 \frac{r^{3/2}}{p^{3/2}}.
\end{align*}
In addition,
\begin{align*}
   | \bigg( \mathcal{T} &\times_1 ( \mathbf{\Phi}^{(1)}) \times_2 \U_2 \U_2\t \times_3 ( \U_3 \U_3\t)\bigg)_{ijk} |\\
   &= | e_i\t \mathbf{\Phi}^{(1)} \mathbf{T}_1 (\U_2 \U_2\t \otimes \U_3 \U_3\t) e_{(j-1)p_3 + k} | \\
   &\leq \| e_i\t \mathbf{\Phi}^{(1)} \| \| \mathbf{T}_1 e_{(j-1)p_3 + k} \| \\
   &\leq \| \mathbf{\Phi}^{(1)} \|_{2,\infty} \| \mathbf{T}_1\t \|_{2,\infty}  \\
   &\lesssim \bigg(  \frac{\sigma^2 \kappa^2 \mu_0^3 \log(p) r^{3/2} \sqrt{p}}{\lambda^2} + \frac{\sigma \mu_0^2 r^{3/2} \kappa}{\lambda \sqrt{p}}  \bigg)\lambda_1 \mu_0 \frac{\sqrt{r}}{p} \\
   &\lesssim 
   \frac{\kappa^3 \sigma^2 \mu_0^4 \log(p) r^2}{\lambda \sqrt{p}} + \frac{\kappa^2 \mu_0^3 r^2 \sigma}{ p^{3/2}}
\end{align*}
The remaining term satisfies
\begin{align*}
    \bigg( \mathcal{T} &\times_1 ( \mathbf{Z}_1 \mathbf{V}_1 \mathbf{\Lambda}_1\inv \U_1\t ) \times_2 \U_2 \U_2\t \times_3 ( \U_3 \U_3\t)\bigg)_{ijk} \\
    &= e_i\t \mathbf{Z}_1 \mathbf{V}_1 \mathbf{\Lambda}_1\inv \U_1\t \mathbf{T}_1 ( \U_2 \U_2\t \otimes \U_3 \U_3\t ) e_{(j-1)p_3 + k} \\
    &= e_i\t \mathbf{Z}_1 \mathbf{V}_1 \mathbf{\Lambda}_1\inv \U_1 \t \U_1 \mathbf{\Lambda}_1 \mathbf{V}_1\t  e_{(j-1)p_3 + k} \\
    &= e_i\t \mathbf{Z}_1 \mathbf{V}_1 \mathbf{V}_1\t  e_{(j-1)p_3 + k}.
\end{align*}
Therefore, we have shown that with probability at least $1 - O(p^{-9})$,
\begin{align*}
    \bigg( \mathcal{T}& \times_1 (\uhat_1 \uhat_1\t - \U_1\U_1\t) \times_2 \U_2 \U_2\t \times_3 ( \U_3 \U_3\t)\bigg)_{ijk} \\
    &= e_i\t \mathbf{Z}_1 \mathbf{V}_1 \mathbf{V}_1\t  e_{(j-1)p_3 + k} \\
    &\quad + O\bigg( \frac{\kappa^2 \sigma \mu_0^3 r^{2}}{p^{3/2}} + \frac{\kappa^3 \sigma^2 \mu_0^4 \log(p) r^2}{\lambda\sqrt{p}}\bigg).
\end{align*}
By symmetry among indices, it holds that with probability at least $1 - O(p^{-9})$ that
\begin{align}
    \mathcal{\hat T}_{ijk} - \mathcal{T}_{ijk} &= e_i\t \mathbf{Z}_1 \mathbf{V}_1 \mathbf{V}_1\t e_{(j-1)p_3 + k} + e_j\t \mathbf{Z}_2 \mathbf{V}_2 \mathbf{V}_2\t e_{(k-1)p_1 + i} + e_k\t \mathbf{Z}_3 \mathbf{V}_3 \mathbf{V}_3\t e_{(i-1)p_2 + j} \nonumber\\
    &\quad + O\bigg( \frac{\kappa^2 \sigma^2 \mu_0^4 \log(p) r^2}{\lambda \sqrt{p}} + \frac{\kappa^2 \mu_0^3 r^2 \sigma}{p^{3/2}} \bigg) \nonumber \\
    &\quad + O\bigg( \frac{\sigma \mu_0^3 \sqrt{\log(p)} r^{3/2}}{p^{3/2}} + \frac{\sigma^2 \mu_0^3 r^{3/2} \kappa \log(p)}{\lambda \sqrt{p}} \bigg) \nonumber \\
   % &\quad + O \bigg( \frac{\kappa \sigma^2 \mu_0^2 r^{3/2}}{p^{3/2}} + \frac{\kappa^3 \mu_0^4 \log(p) r^2}{\lambda\sqrt{p}}\bigg) \nonumber\\
   % &\quad +O( \sigma \mu_0^3 \sqrt{\log(p)} \frac{r^{3/2}}{p^{3/2}} ) \nonumber\\
 %   &\quad + O( \frac{\sigma^2 \mu_0^3 r^{3/2} \kappa \log(p)}{\lambda \sqrt{p}}) \nonumber\\
    &= e_i\t \mathbf{Z}_1 \mathbf{V}_1 \mathbf{V}_1\t e_{(j-1)p_3 + k} + e_j\t \mathbf{Z}_2 \mathbf{V}_2 \mathbf{V}_2\t e_{(k-1)p_1 + i} + e_k\t \mathbf{Z}_3 \mathbf{V}_3 \mathbf{V}_3\t e_{(i-1)p_2 + j} \nonumber\\
    &\quad + O\bigg(\frac{\sigma \kappa^2 \mu_0^3 r^2 \sqrt{ \log(p)}}{p^{3/2}} \bigg)+  O\bigg( \frac{\sigma^2 \mu_0^4 \kappa^3 r^{2} \log(p)}{\lambda \sqrt{p}}\bigg). \label{mainentrywiseexpansion}
\end{align}
This establishes the leading-order expansion.  

Therefore, defining $\xi_{ijk}$ as in \cref{lem:xiijkgaussian}, with probability at least $1 - O(p^{-9})$ it holds that
\begin{align*}
\frac{\mathcal{T}_{ijk} - \mathcal{\hat T}_{ijk}}{s_{ijk}} &= \frac{\xi_{ijk}}{s_{ijk}} + \frac{C_1}{s_{ijk}} \frac{\sigma \kappa^2 \mu_0^3 r^2 \sqrt{ \log(p)}}{p^{3/2}} + \frac{C_2}{s_{ijk}}   \frac{\sigma^2 \mu_0^4 \kappa^3 r^{2} \log(p)}{\lambda \sqrt{p}} 
\\
&\coloneqq \frac{\xi_{ijk}}{s_{ijk}} + \frac{\mathrm{err}}{s_{ijk}}
,\end{align*}
where $C_1$ and $C_2$ are some universal constants.  Therefore,  for any $t \in \mathbb{R}$, by \cref{lem:xiijkgaussian}, it holds that
\begin{align*}
    \p\Bigg\{ \frac{\mathcal{T}_{ijk} - \mathcal{\hat T}_{ijk}}{s_{ijk}} \leq t \Bigg\} &\leq \p\Bigg\{\frac{\xi_{ijk}}{s_{ijk}} \leq t + \frac{\mathrm{err}}{s_{ijk}} \Bigg\}  + C_3 p^{-9} \\
    &\leq \Phi\Bigg\{ t + \frac{\mathrm{err}}{s_{ijk}} \Bigg\}+ C_3 p^{-9} + \frac{C_4}{\sqrt{p\log(p)}} + \frac{C_5 \mu_0^2 r}{p} \\
    &\leq \Phi(t) + \frac{\mathrm{err}}{s_{ijk}} + C_3 p^{-9} + \frac{C_4}{\sqrt{p\log(p)}} + \frac{C_5 \mu_0^2 r}{p} \\
    &= \Phi(t) + o(1),
\end{align*}
where the final result holds since
\begin{align*}
    s_{ijk} &\geq \sigma_{\min} \left(\left\|e_{(j-1) p_{3}+k}^{\top} \mathbf{V}_{1}\right\|^{2}+\left\|e_{(k-1) p_{1}+i} \mathbf{V}_{2}\right\|^{2}+\left\|e_{(i-1) p_{2}+j} \mathbf{V}_{3}\right\|^{2}\right)^{1 / 2} \\
    &\gg \sigma_{\min} \max\bigg\{ \frac{\kappa^2 \mu_0^3 r^2 \sqrt{\log(p)}}{p^{3/2}}, \frac{\sigma \mu_0^4 \kappa^3 r^2 \log(p)}{\lambda \sqrt{p}} \bigg\}, %\\
   % \max\bigg\{ \frac{\kappa \mu_0^3 r^{3/2} \sqrt{\log(p)}}{p^{3/2}}, \frac{\sigma \mu_0^4 \kappa^3 r^{2} \log(p)}{\lambda \sqrt{p}} \bigg\},
\end{align*}
and the fact that $\sigma/\sigma_{\min} = O(1)$.  By applying the same argument to the other direction, the proof is complete.
\end{proof}

% With these lemmas in hand, we are now prepared to prove \cref{thm:asymptoticnormalityentries}.
% \begin{proof}[Proof of \cref{thm:asymptoticnormalityentries}]
\subsection{Proof of \cref{cor:maxnormbound}} \label{sec:maxnormproof}
%proof of max norm perturbation bound password
\begin{proof}[Proof of \cref{cor:maxnormbound}]
We start with the leading-order expansion in \eqref{mainentrywiseexpansion} partway through the proof of \cref{thm:asymptoticnormalityentries}, which demonstrates that with probability at least $1 - O(p^{-9})$
\begin{align}
     \mathcal{\hat T}_{ijk} - \mathcal{T}_{ijk} %&= e_i\t \mathbf{Z}_1 \mathbf{V}_1 \mathbf{V}_1\t e_{(j-1)p_3 + k} + e_j\t \mathbf{Z}_2 \mathbf{V}_2 \mathbf{V}_2\t e_{(k-1)p_1 + i} + e_k\t \mathbf{Z}_3 \mathbf{V}_3 \mathbf{V}_3\t e_{(i-1)p_2 + j} \nonumber\\
    % &\quad + O \bigg( \frac{\kappa \sigma^2 \mu_0^2 r^{3/2}}{p^{3/2}} + \frac{\kappa^3 \mu_0^4 \log(p) r^2}{\lambda\sqrt{p}}\bigg) \nonumber\\
    % &\quad +O( \sigma \mu_0^3 \sqrt{\log(p)} \frac{r^{3/2}}{p^{3/2}} ) \nonumber\\
    % &\quad + O( \frac{\sigma^2 \mu_0^3 r^{3/2} \kappa \log(p)}{\lambda \sqrt{p}}) \nonumber\\
    &= e_i\t \mathbf{Z}_1 \mathbf{V}_1 \mathbf{V}_1\t e_{(j-1)p_3 + k} + e_j\t \mathbf{Z}_2 \mathbf{V}_2 \mathbf{V}_2\t e_{(k-1)p_1 + i} + e_k\t \mathbf{Z}_3 \mathbf{V}_3 \mathbf{V}_3\t e_{(i-1)p_2 + j} \nonumber\\
     &\quad + O\bigg(\frac{\sigma \kappa^2 \mu_0^3 r^2 \sqrt{ \log(p)}}{p^{3/2}} \bigg)+  O\bigg( \frac{\sigma^2 \mu_0^4 \kappa^3 r^{2} \log(p)}{\lambda \sqrt{p}}\bigg). \nonumber  %O\bigg(\frac{\sigma \kappa \mu_0^3 \sqrt{r^3 \log(p)}}{p^{3/2}} \bigg)+  O\bigg( \frac{\sigma^2 \mu_0^4 \kappa^3 r^{2} \log(p)}{\lambda \sqrt{p}}\bigg). \nonumber %\label{mainentrywiseexpansion}
\end{align}
A straightforward Hoeffding inequality argument shows that with high probability,
\begin{align*}
    |e_i\t \mathbf{Z}_1 \mathbf{V}_1 \mathbf{V}_1\t e_{(j-1)p_3 + k} | &\lesssim \sigma \sqrt{\log(p)} \| e_{(j-1)p_3 + k}\t \mathbf{V}_1 \| \\
    &\lesssim \frac{\mu_0\sigma \sqrt{r\log(p)}  }{p}.
\end{align*}
The same bound holds for the other two terms; moreover, this bound is uniform in $i,j$, and $k$.  Consequently, by taking a union bound over all $O(p^{3})$ entries, we obtain
\begin{align*}
    \| \mathcal{\hat T} - \mathcal{T} \|_{\max} &\lesssim \frac{\mu_0\sigma \sqrt{r\log(p)}  }{p} + \frac{\sigma \kappa^2 \mu_0^3 r^{2} \sqrt{\log(p)}}{p^{3/2}} + \frac{\sigma^2 \mu_0^4 \kappa^3 r^3 \log(p)}{\lambda \sqrt{p}} \\
    &\lesssim \frac{\mu_0\sigma \kappa \sqrt{r\log(p)}  }{p} +  \frac{\sigma^2 \mu_0^4 \kappa^3 r^3 \log(p)}{\lambda \sqrt{p}},
\end{align*}
where the final bound is due to the condition  $\kappa^2 \mu_0^2 r^{3/2} \sqrt{\log(p)} \lesssim p^{1/4}$.  This holds with probability at least $1 - O(p^{-6})$.  The ``consequently'' part is immediate.
\end{proof}

\subsection{Proofs of Preliminary Lemmas from \cref{sec:entrywiseaux}} \label{sec:entrywiseauxproofs}
In this section we prove the preliminary lemmas from \cref{sec:entrywiseaux}. 
\subsubsection{Proof of \cref{lem:zentrywiselemma}}
\begin{proof}[Proof of \cref{lem:zentrywiselemma}]
The first bound follows by noting that
\begin{align*}
        \bigg( \mathcal{Z} \times_1 \U_1 \U_1\t \times_2 \U_2 \U_2\t \times_3 \U_3 \U_3\t \bigg)_{ijk}  &= \sum_{abc} \mathcal{Z}_{abc} (\U_1 \U_1\t)_{ia}(\U_2 \U_2\t)_{jb}(\U_3 \U_3\t)_{kc},
\end{align*}
which is a linear combination of Subgaussian random variables with Orlicz norm of coefficients bounded by
\begin{align*}
    \sigma^2 \sum_{a,b,c}  (\U_1 \U_1\t)_{ia}^2(\U_2 \U_2\t)_{jb}^2(\U_3 \U_3\t)_{kc}^2 &\leq \sigma^2 \| e_i\t \U_1 \|^2 \| e_j\t \U_2 \|^2 \| e_k\t \U_3 \|^2 \\
    &\leq \sigma^2 \mu_0^6 \frac{r^3}{p^3}.
\end{align*}
Consequently, Hoeffding's inequality shows that this term is bounded by $C \sigma \sqrt{\log(p)} \mu_0^3 \frac{r^{3/2}}{p^{3/2}}$ with probability at least $1 - O(p^{-10})$.

Next, observe that \cref{cor:asymptoticnormality_projection} implies that with probability at least $1- O(p^{-9})$ that
\begin{align*}
    \| \uhat_1 \uhat_1\t - \U_1 \U_1\t \|_{2,\infty} &\lesssim \frac{\sigma\mu_0 \sqrt{r \log(p)}}{\lambda} + \frac{\sigma^2 \kappa^2 \mu_0^3 \log(p) r^{3/2} \sqrt{p}}{\lambda^2}   + \frac{\sigma \mu_0^2 r^{3/2} \kappa}{\lambda \sqrt{p}}.  
\end{align*}
In addition, under the conditions of \cref{thm:asymptoticnormalityentries}, it holds that
\begin{align*}
  \kappa^2 \mu_0^2 r^{3/2} \sqrt{\log(p)} \lesssim p^{1/4}
\end{align*}
which implies that
\begin{align*}
\frac{\sigma^2 \kappa^2 \mu_0^3 \log(p) r^{3/2} \sqrt{p}}{\lambda^2} &=  \frac{\sigma \mu_0 \sqrt{r\log(p)}}{\lambda} \bigg( \frac{\sigma \kappa^2 \mu_0^2 r^{3/2} \sqrt{p\log(p)} }{\lambda} \bigg) \\
&\lesssim \frac{\sigma \mu_0 \sqrt{r\log(p)}}{\lambda} \bigg( \frac{p^{3/4} \sqrt{\log(p)}}{\lambda/\sigma} \bigg)\\
&\lesssim \frac{\sigma \mu_0 \sqrt{r\log(p)}}{\lambda}.
\end{align*}
Similarly,
\begin{align*}
    \frac{\sigma \mu_0^2 r^{3/2} \kappa}{\lambda \sqrt{p}} &\lesssim \frac{\sigma \mu_0 \sqrt{r\log(p)}}{\lambda}.
\end{align*}
Hence, with probability $1 - O(p^{-9})$,
\begin{align}
    \| \uhat_1 \uhat_1\t - \U_1 \U_1\t \|_{2,\infty} &\lesssim \frac{\sigma \mu_0 \sqrt{r\log(p)}}{\lambda}. \label{eq:projbound}
\end{align}
Similar bounds hold for the other modes as well. Therefore,
\begin{align*}
     \bigg| &\bigg( \mathcal{Z} \times_1 (\uhat_1 \uhat_1\t - \U_1 \U_1\t) \times_2  \U_2 \U_2\t\times_3 \U_3 \U_3\t \bigg)_{ijk} \bigg| \\
     &= \bigg| \sum_{abc}\mathcal{Z}_{abc}(\uhat_1 \uhat_1\t - \U_1 \U_1\t)_{ia} (\U_2 \U_2\t)_{jb} (\U_3 \U_3\t)_{kc} \bigg| \\
     &\leq \sqrt{p_1} \| \uhat_1 \uhat_1- \U_1\U_1\t \|_{2,\infty} \max_{a} \bigg| \sum_{bc} \mathcal{Z}_{abc}(\U_2 \U_2\t)_{jb} (\U_3 \U_3\t)_{kc} \bigg| \\
     &\lesssim \bigg(\frac{\sigma \mu_0 \sqrt{pr\log(p)}}{\lambda} \bigg) \max_{a} \bigg| \sum_{bc} \mathcal{Z}_{abc}(\U_2 \U_2\t)_{jb} (\U_3 \U_3\t)_{kc} \bigg|
\end{align*}
Hoeffding's inequality and a union bound reveals that
\begin{align*}
    \max_a \bigg| \sum_{bc} \mathcal{Z}_{abc}(\U_2 \U_2\t)_{jb} (\U_3 \U_3\t)_{kc} \bigg| &\lesssim \sigma \mu_0^2 \frac{r}{p} \sqrt{\log(p)}
\end{align*}
with probability at least $1 - O(p^{-9})$.  Therefore,
\begin{align*}
      \bigg| &\bigg( \mathcal{Z} \times_1 (\uhat_1 \uhat_1\t - \U_1 \U_1\t) \times_2  \U_2 \U_2\t\times_3 \U_3 \U_3\t \bigg)_{ijk} \bigg| \\
      &\lesssim \sigma \mu_0^2 \frac{r}{p} \sqrt{\log(p)} \bigg(\frac{\sigma \mu_0 \sqrt{pr\log(p)}}{\lambda} \bigg) \\
      &\lesssim \frac{\sigma^2 \mu_0^3 r^{3/2} \log(p)}{\lambda \sqrt{p}}.
\end{align*}
For the next term, we note that
\begin{align*}
    \bigg| &\bigg( \mathcal{Z} \times_1 (\uhat_1 \uhat_1\t - \U_1 \U_1\t) \times_2 (\uhat_2 \uhat_2\t - \U_2 \U_2\t) \times_3 \U_3 \U_3\t \bigg)_{ijk} \bigg| \\
    &=  \bigg| (\uhat_1 \uhat_1\t - \U_1\U_1\t) \bigg(\mathbf{Z}_1 ( \uhat_2 \uhat_2\t - \U_2\U_2\t) \otimes \U_3 \U_3\t \bigg) _{i, (j-1)p_3 + k}\bigg| \\
    &\leq \| \big(\uhat_1 \uhat_1\t - \U_1\U_1\t \big)_{i\cdot}\| \big\| \bigg(\mathbf{Z}_1 ( \uhat_2 \uhat_2\t - \U_2\U_2\t) \otimes \U_3 \U_3\t\bigg)_{\cdot, (j-1)p_3 + k} \big\| \\
    &\leq \| \uhat_1 \uhat_1\t - \U_1\U_1\t \|_{2,\infty} \big\| \bigg(\mathbf{Z}_1 ( \uhat_2 \uhat_2\t - \U_2\U_2\t) \otimes \U_3 \U_3\t\bigg)_{\cdot, (j-1)p_3 + k} \big\|.
\end{align*}
Define the matrix $\mathbf{A}^{(2)}(j)$ as the $p_2 \times p_2$ matrix whose rows are all zero except for the $j$'th row, which is equal to the $j$'th row of $\uhat_2 \uhat_2\t - \U_2 \U_2\t$, and define $\mathbf{A}^{(3)}(k)$ as the $p_3 \times p_3$ matrix whose rows are all zero except for the $k$'th row, which is equal to the $k$'th row of $\U_3 \U_3\t$.  Observe that since both $\mathbf{A}^{(2)}(j)$ and $\mathbf{A}^{(3)}(k)$ are rank at most $2r$, it holds on the event $\mathcal{E}_{\mathrm{Good}}$ that
\begin{align*}
    \| \mathbf{Z}_1 (\mathbf{A}^{(2)}(j) \otimes \mathbf{A}^{(3)}(k)) \| &\leq \tau_k \|\mathbf{A}^{(2)}(j) \| \|\mathbf{A}^{(3)}(k)\| \\
    &\lesssim \sigma \sqrt{pr}\| \| \big(\uhat_2 \uhat_2\t - \U_2 \U_2\t \big)_{j\cdot}\| (\U_3 \U_3\t )_{k\cdot}\| \\
    &\lesssim \sigma \sqrt{pr} \| \uhat_2 \uhat_2\t - \U_2 \U_2\t \|_{2,\infty} \|\U_3 \U_3\t\|_{2,\infty} \\
    &\lesssim \sigma \sqrt{pr} \frac{\sigma \mu_0 \sqrt{r\log(p)}}{\lambda} \mu_0 \sqrt{\frac{r}{p}} \\
    &\lesssim \frac{\sigma^2 \mu_0^2 r^{3/2} \sqrt{\log(p)}}{\lambda},
\end{align*}
where we have implicitly used the bound in Equation \eqref{eq:projbound}.  Putting it together, with probability at least $1-O(p^{-9})$ it holds that
\begin{align*}
     \bigg| &\bigg( \mathcal{Z} \times_1 (\uhat_1 \uhat_1\t - \U_1 \U_1\t) \times_2 (\uhat_2 \uhat_2\t - \U_2 \U_2\t) \times_3 \U_3 \U_3\t \bigg)_{ijk} \bigg|  \\
     &\lesssim \| \U_1 \U_1\t - \U_1 \U_1 \|_{2,\infty}\frac{\sigma^2 \mu_0^2 r^{3/2} \sqrt{\log(p)}}{\lambda} \\
     &\lesssim \frac{\sigma^3 \mu_0^3 r^2 \log(p)}{\lambda^2} \\
     &\lesssim \frac{\sigma^2 \mu_0^2 r^{3/2} \log(p)}{\lambda \sqrt{p}},
\end{align*}
since $\lambda/\sigma \gtrsim \sqrt{pr}$.  

By a similar argument, it is straightforward to show that with this same probability,
\begin{align*}
    \bigg| &\bigg( \mathcal{Z} \times_1 (\uhat_1 \uhat_1\t - \U_1 \U_1\t) \times_2 (\uhat_2 \uhat_2\t - \U_2 \U_2\t) \times_3 (\uhat_3 \uhat_3\t - \U_3 \U_3\t) \bigg)_{ijk} \bigg| \\
    &\lesssim \sigma \sqrt{pr} \|   \uhat_1 \uhat_1\t - \U_1 \U_1\t \|_{2,\infty} \| \uhat_2 \uhat_2\t - \U_2 \U_2\t \|_{2,\infty} \| \uhat_3 \uhat_3\t - \U_3 \U_3\t \|_{2,\infty} \\
    &\lesssim \sigma \sqrt{pr} \frac{ \sigma^3 \mu_0^3 r^{3/2}\log^{3/2}(p) }{\lambda^3} \\
    &\lesssim  \frac{ \sigma^3 \mu_0^3 r^{3/2}\log^{3/2}(p) }{\lambda^2} \\
    &\lesssim \frac{\sigma^2 \mu_0^3 r^{3/2} \log(p)}{\lambda \sqrt{p}}.
\end{align*}
Aggregating these bounds completes the proof.
\end{proof}

\subsubsection{Proof of \cref{lem:tentrywiselemma}}

\begin{proof}[Proof of \cref{lem:tentrywiselemma}]
The proof is similar to \cref{lem:zentrywiselemma}.  Again, observe that \cref{cor:asymptoticnormality_projection} and the conditions of \cref{thm:asymptoticnormalityentries} implies that with probability at least $1 - O(p^{-9})$,
\begin{align*}
    \max_k \| \uhat_k \uhat_k\t - \U_k \U_k\t \|_{2,\infty} &\lesssim \frac{\sigma \mu_0 \sqrt{r\log(p)}}{\lambda}.
\end{align*}
Therefore, we note that on this event,
\begin{align*}
    \bigg| \bigg( \mathcal{T}& \times_1 (\uhat_1 \uhat_1\t - \U_1\U_1\t) \times_2 ( \uhat_2 \uhat_2\t - \U_2 \U_2\t) \times_3 (\U_3 \U_3\t) \bigg)_{ijk} \bigg| \\
    &= \bigg| e_k\t \U_3 \U_3\t \mathbf{T}_3 \bigg( (\uhat_1 \uhat_1\t - \U_1\U_1\t) \otimes ( \uhat_2 \uhat_2\t - \U_2 \U_2\t) \bigg) e_{(i-1)p_2 + j} \bigg| \\
    &\leq \| \U_3 \|_{2,\infty} \| \U_3\t \mathbf{T}_3 \| \bigg\| \bigg( (\uhat_1 \uhat_1\t - \U_1\U_1\t) \otimes ( \uhat_2 \uhat_2\t - \U_2 \U_2\t) \bigg) \bigg\|_{2,\infty} \\
    &\leq \lambda_1 \| \U_3 \|_{2,\infty} \| \uhat_1 \uhat_1\t - \U_1\U_1\t \|_{2,\infty} \| \uhat_2 \uhat_2\t - \U_2 \U_2\t \|_{2,\infty} \\
    &\lesssim \lambda_1 \mu_0 \sqrt{\frac{r}{p}} \frac{ \sigma^2 \mu_0^2 r\log(p)}{\lambda^2} \\
    &\lesssim \frac{\sigma^2 \mu_0^3 r^{3/2} \kappa \log(p)}{\lambda\sqrt{p}}.
\end{align*}
Similarly, also on this event,
\begin{align*}
    \bigg|\bigg( &\mathcal{T} \times_1 (\uhat_1 \uhat_1\t - \U_1\U_1\t) \times_2 ( \uhat_2 \uhat_2\t - \U_2 \U_2\t) \times_3 (\uhat_3 \uhat_3\t - \U_3 \U_3\t) \bigg)_{ijk} \bigg| \\
    &= \bigg| e_i\t (\uhat_1 \uhat_1\t - \U_1\U_1\t) \mathbf{T}_1\bigg[ \bigg( \uhat_2 \uhat_2\t - \U_2 \U_2\t \bigg) \otimes \bigg( \uhat_3 \uhat_3\t - \U_3 \U_3\t \bigg) \bigg] e_{(j-1)p_3 + k} \bigg| \\
    &\leq \| \uhat_1 \uhat_1\t - \U_1\U_1\t \|_{2,\infty} \| \mathbf{T}_1 \| \| \uhat_2 \uhat_2\t - \U_2 \U_2\t \|_{2,\infty} \| \uhat_3 \uhat_3\t - \U_3 \U_3\t \|_{2,\infty} \\
    &\leq \frac{\sigma^3 \mu_0^3 r^{3/2} \kappa\log^{3/2}(p)}{\lambda^2} \\
    &\lesssim \frac{\sigma^2 \mu_0^3 r^{3/2} \kappa \log(p)}{\lambda\sqrt{p}},
\end{align*} 
since $\lambda/\sigma \gtrsim \sqrt{p\log(p)}$.  This completes the proof.
\end{proof}

\subsubsection{Proof of \cref{lem:xiijkgaussian}}
\begin{proof}[Proof of \cref{lem:xiijkgaussian}]
 We first observe that the random variable
\begin{align*}
    \xi_{ijk} &\coloneqq  e_{i}^{\top} \mathbf{Z}_{1} \mathbf{V}_{1} \mathbf{V}_{1}^{\top} e_{(j-1) p_{3}+k}+e_{j}^{\top} \mathbf{Z}_{2} \mathbf{V}_{2} \mathbf{V}_{2}^{\top} e_{(k-1) p_{1}+i}+e_{k}^{\top} \mathbf{Z}_{3} \mathbf{V}_{3} \mathbf{V}_{3}^{\top} e_{(i-1) p_{2}+j}
\end{align*}
is a linear combination of random variables belonging to $\mathcal{Z}$. Note that $\mathrm{Var}(\xi_{ijk})$ may not equal $s^2_{ijk}$.  First, we will show that  
\begin{align*}
    \mathrm{Var}(\xi_{ijk}) &= s^2_{ijk} + o( s^2_{ijk}).
\end{align*}
Next, we will calculate the moment bounds needed to apply the Berry-Esseen Theorem, and finally we will put it all together.
\begin{itemize}
    \item 
\textbf{Step 1: Variance Calculation:} Since $\xi_{ijk}$ is a sum of three separate terms, we will first calculate the contribution of the cross terms to the variance of $\xi_{ijk}$.  Observe that 
\begin{align*}
\bigg| \E \bigg[ e_i\t& \mathbf{Z}_1 \mathbf{V}_1 \mathbf{V}_1\t e_{(j-1)p_3 + k}  \bigg] \bigg[ e_j\t \mathbf{Z}_2 \mathbf{V}_2 \mathbf{V}_2\t e_{(k-1)p_1 + i} \bigg] \bigg| \\
&\leq \sigma^2  \sum_{(l_1,l_2) \in \Omega}  | e_{l_1}\t \mathbf{V}_1 \mathbf{V}_1\t e_{(j-1)p_3 + k} | \big| e_{l_2}\t \mathbf{V}_2 \mathbf{V}_2\t e_{(k-1)p_1 + i} \big|,
\end{align*}
where the sum is over the set $\Omega$ containing indices $(l_1,l_2)$ such that $(\mathbf{Z}_1)_{il_1} = (\mathbf{Z}_2)_{kl_2}$ (i.e., the indices corresponding to the same elements of the underlying tensor $\mathcal{Z}$).  We note that the general formula is given by
\begin{align*}
    (\mathbf{Z}_1)_{i,(j-1)p_3 + b} = (\mathbf{Z}_2)_{j,(b-1)p_1 + i},
\end{align*}
which shows that the two terms have $p_3 \lesssim p$ terms in common (since $1\leq b \leq p_3$).  Therefore,
\begin{align*}
     \sum_{(l_1,l_2) \in \Omega}  | e_{l_1}\t \mathbf{V}_1 &\mathbf{V}_1\t e_{(j-1)p_3 + k} | \big| e_{l_2}\t \mathbf{V}_2 \mathbf{V}_2\t e_{(k-1)p_1 + i} \big| \\
     &\lesssim p \max_{l_1,l_2}| e_{l_1}\t \mathbf{V}_1 \mathbf{V}_1\t e_{(j-1)p_3 + k} | \big| e_{l_2}\t \mathbf{V}_2 \mathbf{V}_2\t e_{(k-1)p_1 + i} \big| \\
     &\leq p \| \mathbf{V}_1 \|_{2,\infty} \|e_{(j-1)p_3 + k} \t \mathbf{V}_1 \| \| \mathbf{V}_2 \|_{2,\infty} \|e_{(k-1)p_1 + i}\t \mathbf{V}_2 \| \\
     &\leq p \mu_0^2 \frac{r}{p^2} \|e_{(j-1)p_3 + k} \t \mathbf{V}_1 \|\|e_{(k-1)p_1 + i}\t \mathbf{V}_2 \| \\
     &\leq \mu_0^2 \frac{r}{p}  \|e_{(j-1)p_3 + k} \t \mathbf{V}_1 \|\|e_{(k-1)p_1 + i}\t \mathbf{V}_2 \| \\
     &\leq \frac{\mu_0^2 r}{2p} \bigg(  \|e_{(j-1)p_3 + k} \t \mathbf{V}_1 \|^2 + \|e_{(k-1)p_1 + i}\t \mathbf{V}_2 \|^2\bigg),
\end{align*}
where we have used the inequality $2ab \leq a^2 + b^2$.  Therefore, by symmetry,
\begin{align*}
    \mathrm{Var}&\bigg( e_i\t \mathbf{Z}_1 \mathbf{V}_1 \mathbf{V}_1\t e_{(j-1)p_3 + k} + e_j\t \mathbf{Z}_2 \mathbf{V}_2 \mathbf{V}_2\t e_{(k-1)p_1 + i} + e_k\t \mathbf{Z}_3 \mathbf{V}_3 \mathbf{V}_3\t e_{(i-1)p_2 + j} \bigg) \\
    &= \mathrm{Var}\bigg( e_i\t \mathbf{Z}_1 \mathbf{V}_1 \mathbf{V}_1\t e_{(j-1)p_3 + k}\bigg) + \mathrm{Var}\bigg( e_j\t \mathbf{Z}_2 \mathbf{V}_2 \mathbf{V}_2\t e_{(k-1)p_1 + i} \bigg) \\
    &\quad + \mathrm{Var}\bigg( e_k\t \mathbf{Z}_3 \mathbf{V}_3 \mathbf{V}_3\t e_{(i-1)p_2 + j} \bigg) \\
    &\quad + O\bigg( \frac{\sigma^2 \mu_0^2 r}{p} \bigg(  \|e_{(j-1)p_3 + k} \t \mathbf{V}_1 \|^2 + \|e_{(k-1)p_1 + i}\t \mathbf{V}_2 \|^2 + \| e_{(i-1)p_2 + j} \mathbf{V}_3 \|^2 \bigg) \bigg).
\end{align*}
To calculate the remaining terms, we simply note that
\begin{align*}
    \E \bigg(e_i\t \mathbf{Z}_1 \mathbf{V}_1 \mathbf{V}_1\t e_{(j-1)p_3 + k}\bigg)^2 &= \sum_{l=1}^{p_2p_3} \E (\mathbf{Z}_1)_{il}^2 ( e_l\t \mathbf{V}_1 \mathbf{V}_1\t e_{(j-1)p_3 + k} )^2 \\
    &= \sum_{l=1}^{p_2p_3} \sigma^2_{il} ( e_l\t \mathbf{V}_1 \mathbf{V}_1\t e_{(j-1)p_3 + k} )^2 \\
    &= \| e_{(j-1)p_3 + k}\t \mathbf{V}_1 \mathbf{V}_1\t\big( \Sigma^{(i)}_1\big)^{1/2} \|^2,% \\
%    &= \| e_{(j-1)p_3 + k}\t \mathbf{V}_1\|^2
\end{align*}
where we recall that $\Sigma^{(i)}_1$ is the diagonal matrix whose entries are the variances $\sigma^2_{il}$.   Consequently,
\begin{align*}
    \mathrm{Var}\bigg( &e_i\t \mathbf{Z}_1 \mathbf{V}_1 \mathbf{V}_1\t e_{(j-1)p_3 + k}\bigg) + \mathrm{Var}\bigg( e_j\t \mathbf{Z}_2 \mathbf{V}_2 \mathbf{V}_2\t e_{(k-1)p_1 + i} \bigg)  \\
    &\quad + \mathrm{Var}\bigg( e_k\t \mathbf{Z}_3 \mathbf{V}_3 \mathbf{V}_3\t e_{(i-1)p_2 + j} \bigg) \\
    &= \| e_{(j-1)p_3 + k}\t \mathbf{V}_1 \mathbf{V}_1\t \big(\Sigma^{(i)}_1\big)^{1/2} \|^2 + \|e_{(k-1)p_1 + i} \mathbf{V}_2 \mathbf{V}_2\t \big(\Sigma^{(j)}_2\big)^{1/2} \|^2 \\
    &\quad + \| e_{(i-1)p_2 + j} \mathbf{V}_3 \mathbf{V}_3\t \big( \Sigma^{(k)}_3\big)^{1/2} \|^2 \\
    &= s_{ijk}^2.
\end{align*}
Therefore,
\begin{align*}
    \mathrm{Var}(\xi_{ijk}) &= s_{ijk}^2 + O\bigg( \frac{\sigma^2\mu_0^2 r}{p} \bigg(  \|e_{(j-1)p_3 + k} \t \mathbf{V}_1 \|^2 + \|e_{(k-1)p_1 + i}\t \mathbf{V}_2 \|^2 + \| e_{(i-1)p_2 + j} \mathbf{V}_3 \|^2 \bigg) \bigg), \\
    &= s^2_{ijk} + o( s^2_{ijk}),
\end{align*}
where the final inequality holds since 
\begin{align*}
    s^2_{ijk} &= \| e_{(j-1)p_3 + k}\t \mathbf{V}_1 \mathbf{V}_1\t \big(\Sigma^{(i)}_1\big)^{1/2} \|^2 + \|e_{(k-1)p_1 + i} \mathbf{V}_2 \mathbf{V}_2\t \big(\Sigma^{(j)}_2\big)^{1/2} \|^2 \\
    &\quad + \| e_{(i-1)p_2 + j} \mathbf{V}_3 \mathbf{V}_3\t \big(\Sigma^{(k)}_3\big)^{1/2} \|^2  \\
    &\geq \sigma_{\min}^2 \bigg( \| e_{(j-1)p_3 + k}\t \mathbf{V}_1  \|^2 + \|e_{(k-1)p_1 + i} \mathbf{V}_2 \|^2 + \| e_{(i-1)p_2 + j} \mathbf{V}_3 \|^2 \bigg) \\
    &\gg \frac{\sigma^2\mu_0^2 r}{p} \bigg(  \|e_{(j-1)p_3 + k} \t \mathbf{V}_1 \|^2 + \|e_{(k-1)p_1 + i}\t \mathbf{V}_2 \|^2 + \| e_{(i-1)p_2 + j} \mathbf{V}_3 \|^2 \bigg) \bigg),
\end{align*}
since $\sigma/\sigma_{\min} = O(1)$ and $\mu_0^2 r \lesssim \sqrt{p}$.  Consequently, 
\begin{align}
    \frac{\mathrm{Var}(\xi_{ijk})}{s^2_{ijk}} &= 1 + O\bigg( \frac{\mu_0^2 r}{p} \bigg) \label{eq:taylor1}
\end{align}
which will be useful later on.
% The following calculation will be useful when applying the Berry-Eseen Theorem later on.  By Taylor's Theorem,
% \begin{align}
%     \frac{s_{ijk}}{\sqrt{\var(\xi_{ijk})}} &=  \frac{1}{\sqrt{1 + \tfrac{\mathrm{Var}(\xi_{ijk}) - s^2_{ijk}}{s^2_{ijk}}}} \nonumber \\
%     &= 1 - \frac{1}{2} \frac{\mathrm{Var}(\xi_{ijk}) - s^2_{ijk}}{s^2_{ijk}} + o\bigg( \frac{\mathrm{Var}(\xi_{ijk}) - s^2_{ijk}}{s^2_{ijk}} \bigg) \nonumber \\
%     &= 1 + C\frac{\mu_0^2 r}{p}. \label{eq:taylor1}
% \end{align}
\item 
\textbf{Step 2: Third Moment Calculation}: 
In order to apply the Berry-Esseen Theorem, we will a bound on the third absolute moment of the sum of the independent random variables in question.  To avoid complicated notation, let $(b,c)$ be the index of the first matricization corresponding to its $(i,b,c)$ entry, and similarly for $(a,c)$ and $(a,b)$ (with second and third matricization and $j$ and $k$ replaced respectively).  We can then write
\begin{align*}
    \xi_{ijk} &= \sum_{b=1}^{p_2} \sum_{c=1}^{p_3} \mathcal{Z}_{ibc}\bigg[ (\mathbf{V}_1 \mathbf{V}_1\t)_{(b,c),(j-1)p_3 + k}  + \mathbb{I}_{\{b=j\}} \big( \mathbf{V}_2 \mathbf{V}_2\t \big)_{(a,c),(k-1)p_1 + i} \\
    &\quad + \mathbb{I}_{\{c = k \}} \big( \mathbf{V}_3 \mathbf{V}_3\t \big)_{(a,b),(i-1)p_2 + j} \bigg] \\
    &\quad + \sum_{a\neq i, a=1}^{p_1} \sum_{c=1}^{p_3} \mathcal{Z}_{ajc} \bigg[ \big(\mathbf{V}_2 \mathbf{V}_2\t \big)_{(a,c),(k-1)p_1 + i} + \mathbb{I}_{\{c = k \}} \big( \mathbf{V}_3 \mathbf{V}_3\t \big)_{(a,b),(i-1)p_2 + j} \bigg] \\
    &\quad + \sum_{a\neq i, a=1}^{p_1} \sum_{b\neq j, b=1}^{p_2} \mathcal{Z}_{abk} \big( \mathbf{V}_3 \mathbf{V}_3\t \big)_{(a,b),(i-1)p_2+j},
\end{align*}
which, when written in this form, is precisely a sum of independent random variables. There are $O(p^2)$ many terms in this sum.  We will need to bound
\begin{align}
    \sum_{b=1}^{p_2}& \sum_{c=1}^{p_3}\E \bigg| \mathcal{Z}_{ibc}\bigg[ (\mathbf{V}_1 \mathbf{V}_1\t)_{(b,c),(j-1)p_3 + k}  + \mathbb{I}_{\{b=j\}} \big( \mathbf{V}_2 \mathbf{V}_2\t \big)_{(a,c),(k-1)p_1 + i} \\
    &\quad + \mathbb{I}_{\{c = k \}} \big( \mathbf{V}_3 \mathbf{V}_3\t \big)_{(a,b),(i-1)p_2 + j} \bigg]\bigg|^3 \nonumber \\
    &\quad + \sum_{a\neq i, a=1}^{p_1} \sum_{c=1}^{p_3} \E \bigg| \mathcal{Z}_{ajc} \bigg[ \big(\mathbf{V}_2 \mathbf{V}_2\t \big)_{(a,c),(k-1)p_1 + i} + \mathbb{I}_{\{c = k \}} \big( \mathbf{V}_3 \mathbf{V}_3\t \big)_{(a,b),(i-1)p_2 + j} \bigg]\bigg|^3 \nonumber \\
    &\quad + \sum_{a\neq i, a=1}^{p_1} \sum_{b\neq j, b=1}^{p_2} \E \bigg| \mathcal{Z}_{abk} \big( \mathbf{V}_3 \mathbf{V}_3\t \big)_{(a,b),(i-1)p_2+j} \bigg|^3. \label{eq:toboundthirdmoment}
\end{align}
By subgaussianity (e.g., \citet{vershynin_high-dimensional_2018}, Proposition 2.5.2), it holds that
\begin{align*}
    \E \bigg|& \mathcal{Z}_{ibc}\bigg[ (\mathbf{V}_1 \mathbf{V}_1\t)_{(b,c),(j-1)p_3 + k}  + \mathbb{I}_{\{b=j\}} \big( \mathbf{V}_2 \mathbf{V}_2\t \big)_{(a,c),(k-1)p_1 + i} \\
    &\quad + \mathbb{I}_{\{c = k \}} \big( \mathbf{V}_3 \mathbf{V}_3\t \big)_{(a,b),(i-1)p_2 + j} \bigg]\bigg|^3 \\
    &\leq \bigg| (\mathbf{V}_1 \mathbf{V}_1\t)_{(b,c),(j-1)p_3 + k}  + \mathbb{I}_{\{b=j\}} \big( \mathbf{V}_2 \mathbf{V}_2\t \big)_{(a,c),(k-1)p_1 + i}  \\
    &\quad + \mathbb{I}_{\{c = k \}} \big( \mathbf{V}_3 \mathbf{V}_3\t \big)_{(a,b),(i-1)p_2 + j} \bigg|^3 \E | \mathcal{Z}_{ibc}|^3 \\
    &\leq C \sigma^3 \bigg| (\mathbf{V}_1 \mathbf{V}_1\t)_{(b,c),(j-1)p_3 + k}  + \mathbb{I}_{\{b=j\}} \big( \mathbf{V}_2 \mathbf{V}_2\t \big)_{(a,c),(k-1)p_1 + i}  \\
    &\quad + \mathbb{I}_{\{c = k \}} \big( \mathbf{V}_3 \mathbf{V}_3\t \big)_{(a,b),(i-1)p_2 + j} \bigg|^3 \\
    &\leq C' \sigma^3 \bigg( \bigg| (\mathbf{V}_1 \mathbf{V}_1\t)_{(b,c),(j-1)p_3 + k} \bigg|^3 +  \bigg| \mathbb{I}_{\{b=j\}} \big( \mathbf{V}_2 \mathbf{V}_2\t \big)_{(a,c),(k-1)p_1 + i}  \bigg|^3  \\
    &\quad +  \bigg| \mathbb{I}_{\{c = k \}} \big( \mathbf{V}_3 \mathbf{V}_3\t \big)_{(a,b),(i-1)p_2 + j} \bigg|^3 \bigg) % \\
    %&\leq C' \sigma^3 \max\bigg\{  \bigg| (\mathbf{V}_1 \mathbf{V}_1\t)_{(b,c),(j-1)p_3 + k} \bigg|^3, \bigg| \mathbb{I}_{\{b=j\}} \big( \mathbf{V}_2 \mathbf{V}_2\t \big)_{(a,c),(k-1)p_1 + i}  \bigg|^3,\bigg| \mathbb{I}_{\{c = k \}} \big( \mathbf{V}_3 \mathbf{V}_3\t \big)_{(a,b),(i-1)p_2 + j} \bigg|^3 \bigg\} 
\end{align*}
Substituting this bound into \cref{eq:toboundthirdmoment} and rearranging yields the upper bound
\begin{align}
    C' &\sigma^3 \Bigg[ \sum_{b,c} \big| (\mathbf{V}_1 \mathbf{V}_1\t)_{(b,c),(j-1)p_3 + k} \big|^3 + \sum_{a,c} \big|\big(\mathbf{V}_2 \mathbf{V}_2\t \big)_{(a,c),(k-1)p_1 + i}\big|^3 \nonumber \\
    &\quad + \sum_{a,b} \big|\big( \mathbf{V}_3 \mathbf{V}_3\t \big)_{(a,b),(i-1)p_2+j} \big|^3 \Bigg] \nonumber \\
    &\leq C' \sigma^3 \bigg[ \max_{b,c} \big| (\mathbf{V}_1 \mathbf{V}_1\t)_{(b,c),(j-1)p_3 + k} \big| \sum_{b,c} \big| (\mathbf{V}_1 \mathbf{V}_1\t)_{(b,c),(j-1)p_3 + k} \big|^2 \nonumber \\
    &\qquad + \max_{a,c}  \big|\big(\mathbf{V}_2 \mathbf{V}_2\t \big)_{(a,c),(k-1)p_1 + i}\big| \sum_{a,c} \big|\big(\mathbf{V}_2 \mathbf{V}_2\t \big)_{(a,c),(k-1)p_1 + i}\big|^2 \nonumber \\
    &\qquad + \max_{a,b} \big|\big( \mathbf{V}_3 \mathbf{V}_3\t \big)_{(a,b),(i-1)p_2+j} \big| \sum_{a,b} \big|\big( \mathbf{V}_3 \mathbf{V}_3\t \big)_{(a,b),(i-1)p_2+j} \big|^2 \Bigg]  \nonumber \\
    &\lesssim \sigma^3 \frac{\mu_0^2 r}{p^2} \bigg( \| \mathbf{V}_1 \mathbf{V}_1\t e_{(j-1)p_3 + k} \|^2 + \| \mathbf{V}_2 \mathbf{V}_2\t e_{(k-1)p_1 + i} \|^2 + \| \mathbf{V}_3 \mathbf{V}_3\t e_{(i-1)p_2 + j} \|^2 \bigg) \label{eq:todivide}
\end{align}
In addition, we note that from the previous step it holds that
\begin{align*}
    \mathrm{Var}(\xi_{ijk}) &= s_{ijk}^2 + C \frac{\sigma^2 \mu_0^2 r}{p} \bigg( \left\|e_{(j-1) p_{3}+k}^{\top} \mathbf{V}_{1}\right\|^{2}+\left\|e_{(k-1) p_{1}+i}^{\top} \mathbf{V}_{2}\right\|^{2}+\left\|e_{(i-1) p_{2}+j} \mathbf{V}_{3}\right\|^{2} \bigg).
\end{align*}
This implies that there is some constant $c > 0$ such that
\begin{align*}
    \mathrm{\var}(\xi_{ijk})^{3/2} &\geq c s_{ijk}^{3}.
\end{align*}
Moreover,
\begin{align*}
    s_{ijk}^3 \geq \sigma_{\min}^3 \bigg( \left\|e_{(j-1) p_{3}+k}^{\top} \mathbf{V}_{1}\right\|^{2}+\left\|e_{(k-1) p_{1}+i}^{\top} \mathbf{V}_{2}\right\|^{2}+\left\|e_{(i-1) p_{2}+j} \mathbf{V}_{3}\right\|^{2} \bigg)^{3/2}.
\end{align*}
Therefore, dividing \eqref{eq:todivide} by the $\mathrm{Var}(\xi_{ijk})^{3/2}$ yields that
\begin{align*}
    &\frac{1}{\mathrm{Var}(\xi_{ijk})^{3/2}} \sigma^3 \frac{\mu_0^2 r}{p^2} \bigg( \| \mathbf{V}_1 \mathbf{V}_1\t e_{(j-1)p_3 + k} \|^2 + \| \mathbf{V}_2 \mathbf{V}_2\t e_{(k-1)p_1 + i} \|^2 + \| \mathbf{V}_3 \mathbf{V}_3\t e_{(i-1)p_2 + j} \|^2 \bigg) \\
    &\lesssim \frac{1}{cs_{ijk}^3} \sigma^3 \frac{\mu_0^2 r}{p^2} \bigg( \| \mathbf{V}_1 \mathbf{V}_1\t e_{(j-1)p_3 + k} \|^2 + \| \mathbf{V}_2 \mathbf{V}_2\t e_{(k-1)p_1 + i} \|^2 + \| \mathbf{V}_3 \mathbf{V}_3\t e_{(i-1)p_2 + j} \|^2 \bigg)  \\
    &\lesssim \frac{\mu_0^2 r}{p^2} \frac{\sigma^3}{\sigma_{\min}^3} \ \bigg( \left\|e_{(j-1) p_{3}+k}^{\top} \mathbf{V}_{1}\right\|^{2}+\left\|e_{(k-1) p_{1}+i}^{\top} \mathbf{V}_{2}\right\|^{2}+\left\|e_{(i-1) p_{2}+j} \mathbf{V}_{3}\right\|^{2} \bigg)^{-1/2} \\
    &\lesssim \frac{\mu_0^2 r}{p^2} \bigg( \left\|e_{(j-1) p_{3}+k}^{\top} \mathbf{V}_{1}\right\|^{2}+\left\|e_{(k-1) p_{1}+i}^{\top} \mathbf{V}_{2}\right\|^{2}+\left\|e_{(i-1) p_{2}+j} \mathbf{V}_{3}\right\|^{2} \bigg)^{-1/2} \\
    &\lesssim \frac{\mu_0^2 r }{p^2} \frac{p^{3/2}}{\kappa \mu_0^3 r^{3/2} \sqrt{\log(p)}} \\
    &\lesssim \frac{1}{\sqrt{p\log(p)}},
\end{align*}
where we have used the condition 
\begin{align*}
     \bigg( \left\|e_{(j-1) p_{3}+k}^{\top} \mathbf{V}_{1}\right\|^{2}+\left\|e_{(k-1) p_{1}+i}^{\top} \mathbf{V}_{2}\right\|^{2}+\left\|e_{(i-1) p_{2}+j} \mathbf{V}_{3}\right\|^{2} \bigg)^{1/2} \gg \frac{\kappa \mu_0^3 r^{3/2} \sqrt{\log(p)}}{p^{3/2}}.
\end{align*}
\item 
\textbf{Step 3: Putting It All Together}: 
By the Berry-Esseen Theorem, for any $t \in \mathbb{R}$, it holds that
\begin{align*}
   \Bigg| \p\bigg\{ \frac{\xi_{ijk}}{s_{ijk}} \leq t \bigg\} - \Phi(t) \Bigg| &=  \Bigg| \p\bigg\{ \frac{\xi_{ijk}}{\sqrt{\mathrm{Var}(\xi_{ijk})}} \leq t \frac{s_{ijk}}{\sqrt{\mathrm{Var}(\xi_{ijk})}} \bigg\} - \Phi(t) \bigg| \\
    &\leq  \bigg| \p\bigg\{ \frac{\xi_{ijk}}{\sqrt{\mathrm{Var}(\xi_{ijk})}} \leq t \frac{s_{ijk}}{\sqrt{\mathrm{Var}(\xi_{ijk})}} \bigg\} - \Phi\bigg( t \frac{s_{ijk}}{\sqrt{\mathrm{Var}(\xi_{ijk})}} \bigg) \bigg| \\
    &\quad + \bigg| \Phi(t) - \Phi\bigg( t \frac{s_{ijk}}{\sqrt{\mathrm{Var}(\xi_{ijk})}} \bigg) \bigg| \\
    &\leq \frac{C}{\sqrt{p\log(p)}}  + \bigg| \Phi(t) - \Phi\bigg( t \frac{s_{ijk}}{\sqrt{\mathrm{Var}(\xi_{ijk})}} \bigg) \bigg|.
    \end{align*}
    Note that 
    \begin{align*}
        \sup_{t\in \mathbb{R}} \bigg| \Phi(t) - \ \Phi\bigg( t \frac{s_{ijk}}{\sqrt{\mathrm{Var}(\xi_{ijk})}} \bigg) \bigg| &\leq \sup_{ \mathcal{A} \ \mathrm{ measurable}} \bigg| \p\bigg( Z \in \mathcal{A} \bigg) - \p\bigg( \frac{\sqrt{\mathrm{Var}(\xi_{ijk})}}{s_{ijk}} Z' \in \mathcal{A} \bigg) \bigg| \\
        &= D_{TV} \bigg( Z, \frac{\sqrt{\mathrm{Var}(\xi_{ijk})}}{s_{ijk}} Z' \bigg)
    \end{align*}
    where $Z$ and $Z'$ are independent standard Gaussians and $D_{TV}$ is the total variation distance.  By Theorem 1.3 of \citet{devroye_total_2022}, it holds that
    \begin{align*}
        D_{TV} \bigg( Z, \frac{\sqrt{\mathrm{Var}(\xi_{ijk})}}{s_{ijk}} Z' \bigg) &\leq \frac{2}{3} \bigg(1 - \frac{\mathrm{Var}(\xi_{ijk})}{s^2_{ijk}} \bigg) \\
        &= O( \mu_0^2 \frac{r}{p}),
    \end{align*}
    where we have used \eqref{eq:taylor1}.  Therefore,
    \begin{align*}
        \sup_{t\in \mathbb{R}}\bigg| \p\bigg\{ \frac{\xi_{ijk}}{s_{ijk}} \leq t \bigg\} - \Phi(t) \bigg| &\leq \frac{C_1}{\sqrt{p\log(p)}} + \frac{C_2 \mu_0^2 r}{p}.
    \end{align*}
This completes the proof.
\end{itemize}
\end{proof}

\section{Proof of Validity of Confidence Intervals (Theorem \ref{thm:civalidity2} and Theorem \ref{thm:civalidity})} \label{sec:ciproofs}
In this section we prove the validity of our confidence regions (\cref{thm:civalidity2}) and intervals (\cref{thm:civalidity}). % and \cref{thm:civalidity2}), 
 In \cref{sec:civaliditylemmas} we state several preliminary lemmas needed to guarantee good approximation of our plug-in estimates, and their proofs are in \cref{sec:prelimproofsvalidity}.  In \cref{sec:civalidityproofs} we prove \cref{thm:civalidity}, and in
\cref{sec:civalidity2proofs} we prove \cref{thm:civalidity2}. % Finally, in \cref{sec:testingproof} we prove \cref{cor:testing}.  

\subsection{Preliminary Lemmas: Plug-In Estimate Proximity} \label{sec:civaliditylemmas}

First  we show that our estimates $\mathbf{\hat V}_k$ and $\mathbf{\hat \Lambda}_k$ from \cref{al:ci_eigenvector} and \cref{al:ci_entries} are sufficiently close to $\mathbf{V}_k$ and $\mathbf{\Lambda}_k$ with respect to both $\|\cdot\|$ and $\|\cdot\|_{2,\infty}$. In what follows, recall that $\mathbf{\hat V}_k$ and $\mathbf{\hat \Lambda}_k$ are defined as the leading $r_k$ right singular vectors and singular values of the matrix $\mathcal{M}_k(\mathcal{\tilde T}) \bigg( \big( \uhat_{k+1} \uhat_{k+1}\t \big) \otimes \big( \uhat_{k+2} \uhat_{k+2}\t \big) \bigg)$.

\begin{lemma} \label{lem:Vmatrixcloseness} Instate the conditions of \cref{thm:eigenvectornormality}, and suppose that $$\kappa^2 \mu_0^2 r^{3/2} \sqrt{\log(p)} \lesssim p^{1/4}.$$
Let $\mathbf{\hat V}_k$ and $\mathbf{\hat \Lambda}_k$ be as in \cref{al:ci_eigenvector} and \cref{al:ci_entries}.  Let $\mathbf{W}_{\mathbf{V}_k} \coloneqq \mathrm{sgn}(\mathbf{ V}_k\t \mathbf{\hat V}_k)$.  Then with probability at least $1 - O(p^{-9})$,
\begin{align*}
\max_k \| \sin\Theta(\mathbf{\hat V}_k,\mathbf{V}_k) \| &\lesssim \frac{\sigma \sqrt{pr} + \kappa^2 \sigma \sqrt{p\log(p)}}{\lambda}; \\
\max_k   \| \mathbf{\hat V}_k - \mathbf{V}_k \mathbf{W}_{\mathbf{V}_k} \|_{2,\infty} &\lesssim \frac{\kappa^2 \sigma \mu_0^2 r^{3/2} \sqrt{\log(p)}}{\lambda \sqrt{p}}; \\
\max_k  \| \mathbf{W}_{\mathbf{V}_k}\t \mathbf{\Lambda}_k\inv - \mathbf{\hat{\Lambda}}_k\inv \mathbf{\hat W}_k\t \| &\lesssim \frac{1}{\lambda} \bigg( \frac{\sigma \sqrt{pr} + \kappa^2 \sigma \sqrt{p\log(p)}}{\lambda} \bigg).
\end{align*}
\end{lemma}
\noindent 
Next, in order to prove \cref{thm:civalidity2}, we will need the following concentration inequality for the estimated matrix $\mathbf{\hat \Gamma}^{(m)}_k$ versus the true matrix $\mathbf{\Gamma}^{(m)}_k$.  
\begin{lemma}\label{lem:gammacloseness}
Instate the conditions of \cref{thm:civalidity2}, and define $\mathbf{\hat \Gamma}^{(m)}_k$ as in \cref{al:ci_eigenvector}.  Then with probability at least $1 - O(p^{-6})$ it holds that
\begin{align*}
%    \| \mathbf{\hat W}_k \mathbf{\hat \Gamma}^{(m)}_k \mathbf{\hat W}_k\t - \mathbf{\Gamma}^{(m)}_k \| &\lesssim \frac{\sigma^2 \kappa \mu_0 \sqrt{r} \log(p)}{\lambda^2 p}+ \frac{\sigma^3 \sqrt{rp}}{\lambda^3} \bigg( \sqrt{r} + \kappa^2 \sqrt{\log(p)} \bigg). \\
 \| \mathbf{\hat W}_k&\big(\mathbf{\hat \Gamma}^{(m)}_k \big)^{1/2}\mathbf{\hat W}_k\t -\big(\mathbf{\Gamma}^{(m)}_k \big)^{1/2} \| \\
 &\lesssim \frac{\sigma}{\lambda} \bigg( \frac{\sigma r \sqrt{p}\log(p) + \sigma \kappa^2 \sqrt{rp} \log^{3/2}(p)}{\lambda} + \frac{\kappa \mu_0 \sqrt{r} \log(p)}{p} + \frac{\mu_0 r^{3/2} \sqrt{\log(p)}}{p} \bigg),
 \end{align*}
 where $\mathbf{\hat W}_k = \mathrm{sgn}(\U_k\t \uhat_k)$.
\end{lemma}

\subsection{Proof of \cref{thm:civalidity2}} \label{sec:civalidity2proofs}
We now prove \cref{thm:civalidity2}.

\begin{proof}[Proof of \cref{thm:civalidity2}]
Here we suppress the dependence on $t$. \cref{lem:gammacloseness} reveals that
\begin{align*}
   \| \mathbf{\hat W}_k&\big(\mathbf{\hat \Gamma}^{(m)}_k \big)^{1/2}\mathbf{\hat W}_k\t -\big(\mathbf{\Gamma}^{(m)}_k \big)^{1/2} \|  \\
   &\lesssim \frac{\sigma}{\lambda} \bigg( \frac{\sigma r \sqrt{p}\log(p) + \sigma \kappa^2 \sqrt{rp} \log^{3/2}(p)}{\lambda} + \frac{\kappa \mu_0 \sqrt{r} \log(p)}{p} + \frac{\mu_0 r^{3/2} \sqrt{\log(p)}}{p} \bigg) \\
   &= \frac{\sigma}{\lambda} \times o(1)
\end{align*}
with probability at least $1 - O(p^{-6})$.  Snce $\mathbf{\Gamma}_k^{(m)}$ has smallest eigenvalue at most $c\frac{\sigma^2}{\lambda}$ (see the proof of \cref{lem:gammacloseness}), by Weyl's inequality it therefore holds that
\begin{align}
    \lambda_{\min} \bigg( \big(\mathbf{\hat \Gamma}^{(m)}_k \big)^{1/2}\bigg) \gtrsim \frac{\sigma_{\min}}{\lambda}. \label{lowerbdongamma}
\end{align}
Hence, with probability at least $1 - O(p^{-6})$ it holds that
\begin{align*}
    \| e_m\t \bigg( &\uhat_k \mathbf{W}_k\t- \U_k   \bigg) \big( (\mathbf{\Gamma}^{(m)}_k)^{-1/2} - (\mathbf{W}_k \mathbf{\hat \Gamma}^{(m)}_k \mathbf{W}_k\t )^{-1/2} \big) \| \\
    &\leq \| e_m\t \big(  \uhat_k \mathbf{W}_k\t- \U_k   \big) \| \| (\mathbf{\Gamma}^{(m)}_k)^{-1/2} - (\mathbf{W}_k \mathbf{\hat \Gamma}^{(m)}_k \mathbf{W}_k\t )^{-1/2} \| \\
    &\leq \| e_m\t \big( \uhat_k- \U_k  \mathbf{W}_k \big) \| \| (\mathbf{\Gamma}^{(m)}_k)^{-1/2} - (\mathbf{W}_k \mathbf{\hat \Gamma}^{(m)}_k \mathbf{W}_k\t )^{-1/2} \| \\
    &\leq \| \uhat_k - \U_k \mathbf{W}_k \|_{2,\infty} \| (\mathbf{\Gamma}^{(m)}_k)^{-1/2} \bigg( (\mathbf{\Gamma}^{(m)}_k)^{1/2} - \mathbf{W}_k(\mathbf{\hat \Gamma}^{(m)}_k)^{1/2}\mathbf{W}_k\t  \bigg) ( \mathbf{W}_k \mathbf{\hat \Gamma}^{(m)}_k\mathbf{W}_k\t )^{-1/2} \| \\
    &\lesssim \frac{\kappa \sigma\mu_0 \sqrt{r\log(p)}}{\lambda} \frac{1}{\lambda_{\min}^{1/2}(\mathbf{\Gamma}^{(m)}_k)} \frac{1}{\lambda_{\min}^{1/2}(\mathbf{\hat \Gamma}^{(m)}_k)} \| (\mathbf{\Gamma}^{(m)}_k)^{1/2} - \mathbf{W}_k(\mathbf{\hat \Gamma}^{(m)}_k)^{1/2}\mathbf{W}_k\t \| \\
    &\lesssim \frac{\kappa \sigma \mu_0 \sqrt{r\log(p)}}{\lambda} \frac{\lambda^2}{\sigma_{\min}^2} \| (\mathbf{\Gamma}^{(m)}_k)^{1/2} - \mathbf{W}_k(\mathbf{\hat \Gamma}^{(m)}_k)^{1/2}\mathbf{W}_k\t \| \\
    &\lesssim \kappa \mu_0 \sqrt{r\log(p)} \bigg( \frac{\sigma r \sqrt{p}\log(p) + \sigma \kappa^2 \sqrt{rp} \log^{3/2}(p)}{\lambda} + \frac{\kappa \mu_0 \sqrt{r} \log(p)}{p} + \frac{\mu_0 r^{3/2} \sqrt{\log(p)}}{p} \bigg) \\
    &\asymp \frac{ \sigma \kappa \mu_0 r^{3/2} \sqrt{p} \log^{3/2}(p) + \sigma \kappa^3 r \mu_0 \sqrt{p} \log^2(p)}{\lambda} + \frac{\kappa^2 \mu_0^2 r \log^{3/2}(p)}{p} + \frac{\kappa \mu_0^2 r^2 \log(p)}{p}\\
    &\coloneqq \varepsilon,
\end{align*}
where in the fourth line we have implicitly used \cref{thm:twoinfty}.  For a convex set $A$, we denote $A^{\varepsilon}$ as the $\varepsilon$-enlargement via
\begin{align*}
    A^{\varepsilon} &\coloneqq \{ x: d(x,A) \leq \varepsilon \}.
\end{align*}
By Theorem 1.2 of \citet{raicv_multivariate_2019}, if $Z$ is an isotropic $\mathbb{R}^{r_k}$ dimensional random vector, it holds that
\begin{align*}
    \p\bigg( Z \in A^{\varepsilon} \setminus A \bigg) &\lesssim r^{1/4} \varepsilon.
\end{align*}
The proof is now straightforward.  Define $A_{\alpha}$ as the confidence region such that
\begin{align*}
    \p( Z \in A_{\alpha} ) = 1 - \alpha,
\end{align*}
where $Z \sim N(0, I_{r_k})$.  Then by \cref{thm:eigenvectornormality2}, 
\begin{align*}
    \bigg| \p\bigg\{ e_m\t& \U_k \mathbf{\hat W}_k \in \mathrm{C.I.}_{\alpha} (\uhat_k) \bigg\}  - (1-\alpha) \bigg| \\
    &= \bigg| \p\bigg\{ e_m\t \bigg(  \U_k \mathbf{\hat W}_k - \uhat_k \bigg) \in \big(\mathbf{\hat \Gamma}^{(m)}_k\big)^{1/2} A_{\alpha} \bigg\} - (1 - \alpha) \bigg| \\
    &=\bigg| \p\bigg\{ e_m\t \bigg( \U_k - \uhat_k \mathbf{\hat W}_k\t \bigg) \in  \mathbf{\hat W}_k \big(\mathbf{\hat \Gamma}^{(m)}_k\big)^{1/2} \mathbf{\hat W}_k\t A_{\alpha} \bigg\} - (1-\alpha) \bigg| \\
    &\leq \bigg| \p\Bigg\{ \bigg\{ e_m\t \bigg( \U_k - \uhat_k \mathbf{\hat W}_k\t \bigg) \in \big( \mathbf{\Gamma}^{(m)}_k\big)^{1/2} A_{\alpha}\bigg\} \\
    &\quad \bigcap \bigg\{ \bigg\| e_m\t \bigg( \U_k - \uhat_k \mathbf{\hat W}_k\t \bigg) \bigg( \big( \mathbf{\Gamma}^{(m)}_k \big)^{-1/2} - \mathbf{W}_k \big(\mathbf{\hat \Gamma}^{(m)}_k \big)^{-1/2} \mathbf{W}_k\t \bigg) \bigg\| \leq \varepsilon \bigg\} \Bigg\} \\
    &\qquad - (1- \alpha) \bigg| \\
    &\qquad + \p\bigg\{ \bigg\| e_m\t \bigg( \U_k - \uhat_k \mathbf{\hat W}_k\t \bigg) \bigg( \big( \mathbf{\Gamma}^{(m)}_k \big)^{-1/2} - \mathbf{W}_k \big(\mathbf{\hat \Gamma}^{(m)}_k \big)^{-1/2} \mathbf{W}_k\t \bigg) \bigg\| >  \varepsilon \bigg\} \\
    &\leq \bigg| \p\bigg\{ e_m\t \bigg( \U_k - \uhat_k \mathbf{\hat W}_k\t \bigg) \in \big(\mathbf{\Gamma}^{(m)}_k \big)^{1/2} A_{\alpha}^{\varepsilon} - (1- \alpha) \bigg| \\
    &\quad + p^{-6} \\
    &\lesssim \bigg| \p\bigg\{ Z \in A_{\alpha}^{\varepsilon} \bigg\} - (1 - \alpha) \bigg| \\
    &\quad + \frac{\mu_0 r^2}{p} + \frac{\sigma \kappa^3 \mu_0^2 \log(p) r^{3/2} \sqrt{p}}{\lambda} + \frac{\mu_0 r^{3/2} \kappa}{\sqrt{p}} + p^{-6} \\
    &\lesssim \bigg| \p\bigg( Z \in A_{\alpha} \bigg) - (1 - \alpha) \bigg| \\
    &\quad + \p\bigg\{ Z \in A_{\alpha}^{\varepsilon} \setminus A \bigg\} \\
    &\quad +  \frac{\mu_0 r^2}{p} + \frac{\sigma \kappa^3 \mu_0^2 \log(p) r^{3/2} \sqrt{p}}{\lambda} + \frac{\mu_0 r^{3/2} \kappa}{\sqrt{p}} + p^{-6}\\
    &\lesssim r^{1/4} \varepsilon + + \frac{\mu_0 r^2}{p} + \frac{\sigma \kappa^3 \mu_0^2 \log(p) r^{3/2} \sqrt{p}}{\lambda} + \frac{\mu_0 r^{3/2} \kappa}{\sqrt{p}} + p^{-6} \\
    &\lesssim r^{1/2} \varepsilon + + \frac{\mu_0 r^2}{p} + \frac{\sigma \kappa^3 \mu_0^2 \log(p) r^{3/2} \sqrt{p}}{\lambda} + \frac{\mu_0 r^{3/2} \kappa}{\sqrt{p}} + p^{-6} \\
    &\lesssim r^{1/2} \bigg( \frac{ \sigma \kappa \mu_0 r^{3/2} \sqrt{p} \log^{3/2}(p) + \sigma \kappa^3 r \mu_0 \sqrt{p} \log^2(p)}{\lambda} + \frac{\kappa^2 \mu_0^2 r \log^{3/2}(p)}{p} + \frac{\kappa \mu_0^2 r^2 \log(p)}{p} \bigg) \\
    &\quad +\frac{\mu_0 r^2}{p} + \frac{\sigma \kappa^3 \mu_0^2 \log(p) r^{3/2} \sqrt{p}}{\lambda} + \frac{\mu_0 r^{3/2} \kappa}{\sqrt{p}} + p^{-6} \\
    &= o(1),
\end{align*}
where the term is $o(1)$ as long as
\begin{align*}
    \lambda/\sigma \gg \kappa^3 \mu_0^2 \log^2(p) r^2 \sqrt{p},
\end{align*}
which is true by assumption, and 
\begin{align*}
    \kappa^2 \mu_0^2 r^{3/2} \log^{3/2}(p) + \kappa \mu_0^2 r^{5/2} \log(p) \ll p
\end{align*}
which holds as long as $\kappa^2 \mu_0^2 r^{3/2} \sqrt{\log(p)} \lesssim p^{1/4}$, which is also by assumption. This completes the proof. 
\end{proof}

\subsection{Proof of \cref{thm:civalidity}} \label{sec:civalidityproofs}
%We are now prepared to prove \cref{thm:civalidity}.  

\begin{proof}[Proof of \cref{thm:civalidity}]
We will model the argument in the proof of Theorem 4.11 of \citet{chen_spectral_2021}, where we will argue that $\hat s^2_{ijk} - s_{ijk}^2$ is sufficiently small.  We first introduce an auxiliary term
\begin{align*}
    \tilde s^2_{ijk} &\coloneqq \sum_{a} \big( \mathbf{Z}_1 \big)_{ia}^2 \big(\mathbf{V}_1 \mathbf{V}_1\t \big)_{a,(j-1)p_3+k}^2 +\sum_{b} \big( \mathbf{Z}_2 \big)^2_{jb} \big( \mathbf{V}_2 \mathbf{V}_2\t \big)_{b,(k-1)p_1 + i}^2  \\
    &\quad + \sum_c \big( \mathbf{Z}_3 \big)^2_{kc} \big( \mathbf{V}_3 \mathbf{V}_3\t \big)_{c,(i-1)p_2 + j}^2.
\end{align*}
We will compare both $\hat s^2_{ijk}$ and $s_{ijk}^2$ to $\tilde s^2_{ijk}$.
\begin{itemize}
    \item \textbf{Step 1: Showing $\hat s^2_{ijk} \approx \tilde s^2_{ijk}$:} 
Observe that
\begin{align*}
  | \hat s^2_{ijk} - \tilde s^2_{ijk} | &\leq \bigg|  \sum_{a} \big( \mathbf{Z}_1 \big)_{ia}^2 \big(\mathbf{V}_1 \mathbf{V}_1\t \big)_{a,(j-1)p_3+k}^2 -\big( \mathbf{\hat Z}_1 \big)_{ia}^2 \big(\mathbf{\hat V}_1 \mathbf{\hat V}_1\t \big)_{a,(j-1)p_3+k}^2 \bigg| \\
  &\quad + \bigg| \sum_{b} \big( \mathbf{Z}_2 \big)^2_{jb} \big( \mathbf{V}_2 \mathbf{V}_2\t \big)_{b,(k-1)p_1 + i}^2 -  \big( \mathbf{\hat Z}_2 \big)^2_{jb} \big( \mathbf{\hat V}_2 \mathbf{\hat V}_2\t \big)_{b,(k-1)p_1 + i}^2 \bigg|\\
  &\quad +\bigg|  \sum_c \big( \mathbf{Z}_3 \big)^2_{kc} \big( \mathbf{V}_3 \mathbf{V}_3\t \big)_{c,(i-1)p_2 + j}^2 - \big( \mathbf{\hat Z}_3 \big)^2_{kc} \big( \mathbf{\hat V}_3 \mathbf{\hat V}_3\t \big)_{c,(i-1)p_2 + j}^2 \bigg|.
\end{align*}
We will focus on the first term, since the other terms will follow by symmetry.  Note that
\begin{align}
    \bigg| \sum_{a} & \big( \mathbf{Z}_1 \big)_{ia}^2 \big(\mathbf{V}_1 \mathbf{V}_1\t \big)_{a,(j-1)p_3+k}^2 -\big( \mathbf{\hat Z}_1 \big)_{ia}^2 \big(\mathbf{\hat V}_1 \mathbf{\hat V}_1\t \big)_{a,(j-1)p_3+k}^2 \bigg| \nonumber\\
    &\leq \sum_{a} \big| (\mathbf{\hat Z}_1)_{ia}^2 - (\mathbf{Z}_1)_{ia}^2 \big| \big( \mathbf{\hat V}_1 \mathbf{\hat V}_1\t \big)_{a,(j-1)p_3 + k}^2 \\
    &\quad + \sum_a (\mathbf{Z}_1 )_{ia}^2 \bigg| \big(\mathbf{V}_1 \mathbf{V}_1\t\big)^2_{a,(j-1)p_3 + k} - \big(\mathbf{\hat V}_1 \mathbf{\hat V}_1\t\big)^2_{a,(j-1)p_3 + k} \bigg| \nonumber\\
    &\leq \max_{a} | (\mathbf{\hat Z}_1)_{ia}^2 - (\mathbf{Z}_1)_{ia}^2 | \sum_a \big(\mathbf{\hat V}_1 \mathbf{\hat V}_1\t \big)_{a,(j-1)p_3 + k}^2 \nonumber\\
    &\quad + \max_a | (\mathbf{Z}_1 )_{ia}^2 | p^2 \max_{a} \bigg| \big(\mathbf{V}_1 \mathbf{V}_1\t\big)^2_{a,(j-1)p_3 + k} - \big(\mathbf{\hat V}_1 \mathbf{\hat V}_1\t\big)^2_{a,(j-1)p_3 + k} \bigg|\nonumber \\
    &\leq \bigg( \| \mathbf{\hat Z}_1 \|_{\max} + \| \mathbf{Z}_1 \|_{\max} \bigg) \| \mathbf{\hat Z}_1 - \mathbf{Z}_1 \|_{\max} \| \mathbf{\hat V}_1 \|_{2,\infty}^2 \nonumber\\
    &\quad + p^2 \| \mathbf{Z}_1 \|_{\max}^2 \bigg( \| \mathbf{V}_1 \mathbf{V}_1\t \|_{\max} + \| \mathbf{\hat V}_1 \mathbf{\hat V}_1\t \|_{\max} \bigg) \| \mathbf{V}_1 \mathbf{V}_1\t - \mathbf{\hat V}_1 \mathbf{\hat V}_1\t \|_{\max}. \label{v1v2closeness}
\end{align}
It is straightforward to  note that with probability at least $1 - O(p^{-9})$ that
\begin{align*}
    \| \mathbf{Z}_1 \|_{\max} &= \max_{a,b,c} | \mathcal{Z}_{abc} | \\
    &\lesssim \sigma \sqrt{\log(p)}.
\end{align*}
In addition, it holds that
\begin{align*}
    \| \mathbf{V}_1 \mathbf{V}_1\t \|_{\max} &= \max_{i,j} | e_i\t \mathbf{V}_1 \mathbf{V}_1\t e_j | \\
    &\leq \| \mathbf{V}_1 \|_{2,\infty}^2 \\
    &\leq \mu_0^2 \frac{r}{p^2}.
\end{align*}
Similarly,
\begin{align*}
    \| \mathbf{\hat V}_1 \mathbf{\hat V}_1 \|_{\max} &\leq \| \mathbf{\hat V}_1 \|_{2,\infty}^2 \\
    &\leq \bigg( \| \mathbf{\hat V}_1 - \mathbf{V}_1 \mathbf{W}_{\mathbf{V}_1} \|_{2,\infty} + \mu_0 \frac{\sqrt{r}}{p} \bigg)^2.
\end{align*}
We  note that by \cref{lem:Vmatrixcloseness}, it holds that
\begin{align*}
  \| \mathbf{\hat V}_1 - \mathbf{V}_1 \mathbf{W}_{\mathbf{V}_1} \|_{2,\infty} &\lesssim \frac{ \kappa^2 \mu_0 r \sqrt{p\log(p)}}{\lambda} \mu_0 \frac{\sqrt{r}}{p} \\
  &\leq \mu_0 \frac{\sqrt{r}}{p}
\end{align*}
provided that 
\begin{align}
    \lambda/\sigma \gtrsim \kappa^2 \mu_0 r \sqrt{p\log(p)}. \label{snrcondition}
\end{align}
Note that we require that
\begin{align*}
    \kappa^2 \mu_0^2 r^{3/2} \sqrt{\log(p)} \lesssim p^{1/4}.
\end{align*}
In addition, \cref{thm:twoinfty} requires that $\lambda/\sigma \gtrsim \kappa p^{3/4} \sqrt{\log(p)}$.  Therefore, together these imply that
\begin{align*}
    \lambda/\sigma &\gtrsim \kappa p^{1/4} \sqrt{p\log(p)} \\
    &\gtrsim \kappa^3 r^{3/2} \mu_0^2 \log(p) \sqrt{p} \\
    &\gg \kappa^2 \mu_0 r \sqrt{p\log(p)},
\end{align*}
so that \cref{snrcondition} holds. Therefore,
\begin{align*}
    \| \mathbf{\hat V}_1 \mathbf{\hat V}_1\t \|_{\max} &\lesssim \mu_0^2 \frac{r}{p^2}.
\end{align*}
Therefore, the bound in \eqref{v1v2closeness} reduces to
\begin{align}
     \bigg( \| \mathbf{\hat Z}_1 \|_{\max} + \sigma \sqrt{\log(p)} \bigg) \| \mathbf{\hat Z}_1 - \mathbf{Z}_1 \|_{\max} \mu_0^2 \frac{r}{p^2} + \sigma^2 \log(p) \mu_0^2 r  \| \mathbf{V}_1 \mathbf{V}_1\t - \mathbf{\hat V}_1 \mathbf{\hat V}_1\t \|_{\max}.  \label{v1v2bd2}
\end{align}
We now note that
\begin{align*}
    \| \mathbf{\hat Z}_1 - \mathbf{Z}_1 \|_{\max} &= \| \mathcal{\hat Z} - \mathcal{Z} \|_{\max} \\
    &= \| \mathcal{T + Z - \hat T - Z} \|_{\max} \\
    &= \| \mathcal{T - \hat T} \|_{\max} \\
  % &\lesssim \frac{\mu_0 \sigma \sqrt{r\log(p)}}{p} + \frac{\sigma \kappa \mu_0^3 r^{3/2} \sqrt{\log(p)}}{p^{3/2}} + \frac{\sigma^2 \mu_0^4 \kappa^3 r^{3} \log(p)}{\lambda \sqrt{p}} \\
    &\lesssim  \frac{\mu_0\sigma \kappa  \sqrt{r\log(p)}  }{p} +  \frac{\sigma^2 \mu_0^4 \kappa^3 r^3 \log(p)}{\lambda \sqrt{p}}, % \\
%    &\lesssim \sqrt{\log(p)}
\end{align*}
which holds with probability at least $1 - O(p^{-6})$, which holds by \cref{cor:maxnormbound}. As a byproduct, we also obtain that
\begin{align*}
    \| \mathbf{\hat Z}_1 \|_{\max} &\leq \| \mathbf{\hat Z}_1 - \mathbf{Z}_1 \|_{\max} + \sigma \sqrt{\log(p)} \\
    &\lesssim \sigma\sqrt{\log(p)}
\end{align*}
provided that
\begin{align*}
    \lambda/\sigma \gtrsim \mu_0^2 \kappa^3 r^2 \sqrt{\log(p)},
\end{align*}
which is guaranteed by the conditions in \cref{thm:twoinfty} as well as the assumption $\kappa^2 \mu_0^2 r^{3/2} \sqrt{\log(p)} \lesssim p^{1/4}$.  Finally, by \cref{lem:Vmatrixcloseness}, it holds that
\begin{align*}
    \| \mathbf{V}_1 \mathbf{V}_1\t -\mathbf{\hat V}_1 \mathbf{\hat V}_1\t \|_{\max} &\leq \| \mathbf{V}_1 \mathbf{V}_1\t - \mathbf{\hat V}_1 (\mathbf{W}_{\mathbf{V}_1}\t) \mathbf{V}_1\t \|_{\max} + \| \mathbf{\hat V}_1 (\mathbf{W}_{\mathbf{V}_1}\t) \mathbf{V}_1\t - \mathbf{\hat V}_1 \mathbf{\hat V}_1\t \|_{\max} \\
    &\leq \| \mathbf{V}_1 \|_{2,\infty} \| \mathbf{V}_1 - \mathbf{\hat V}_1 \mathbf{W}_{\mathbf{V}_1}\t \|_{2,\infty} + \| \mathbf{\hat V}_1 \|_{2,\infty} \| \mathbf{V}_1 \mathbf{W}_{\mathbf{V}_1} - \mathbf{\hat V}_1 \|_{2,\infty} \\
    &\lesssim \mu_0 \frac{\sqrt{r}}{p} \frac{\kappa^2 \sigma \mu_0^2 r^{3/2} \sqrt{\log(p)}}{\lambda \sqrt{p}} \\
    &\lesssim \frac{ \sigma \mu_0^3 \kappa^2 r^2 \sqrt{\log(p)}}{\lambda p^{3/2}}.
\end{align*}
Therefore, plugging in these estimates to \eqref{v1v2bd2}, we see that with probability at least $1 - O(p^{-6})$,
\begin{align}
    | \hat s^2_{ijk} - \tilde s_{ijk}^2 | &\lesssim  \sigma \sqrt{\log(p)} \bigg( \frac{\mu_0 \sigma \kappa \sqrt{r\log(p)}}{p} + \frac{\sigma^2 \mu_0^4 \kappa^3 r^3 \log(p)}{\lambda \sqrt{p}} \bigg) \mu_0^2 \frac{r}{p^2} \nonumber \\
    &\quad + \sigma^2 \log(p) \mu_0^2 r \frac{\sigma \mu_0^3 \kappa^2 r^2 \sqrt{\log(p)}}{\lambda p^{3/2}} \nonumber \\
    &\lesssim \frac{\sigma^2 \kappa \mu_0^2 r^{3/2} \log(p)}{p^3} + \frac{\sigma^3 \mu_0^6 \kappa^3 r^4 \log^{3/2}(p)}{\lambda p^{5/2}} \nonumber \\
    &\quad + \frac{\sigma^3 \mu_0^5 r^3 \kappa^2 \log^{3/2}(p)}{\lambda p^{3/2}} \nonumber \\
    &\lesssim \frac{\sigma^2 \kappa \mu_0^3 r^{3/2} \log(p)}{p^3} + \frac{\sigma^3 \mu_0^5 r^3 \kappa^2 \log^{3/2}(p)}{\lambda p^{3/2}}.\label{vhatvtildebound}
\end{align}
%\textcolor{black}{Note: the second term is suboptimal relative to the lower bound in \cref{thm:asymptoticnormalityentries}, and this is due to the bounding technique.  However, it is not clear if there is another way to bound this term that eliminates this extra factor.  Moreover, since $\lambda/\sigma \gtrsim p^{3/4}$ by assumption, this bound reduces to $p^{9/4}$ when $\lambda/\sigma \asymp p^{3/4}$, which  still allows unbalanced-ness, but not as much as \cref{thm:asymptoticnormalityentries}.}
\item 
\textbf{Step 2: Showing $\tilde s^2_{ijk} \approx s^2_{ijk}$}: 
Note that the term $\tilde s^2_{ijk}$ is a sum of independent subexponential random variables, so we will apply Bernstein's inequality to it.  In order to avoid additional cross-term covariance factors, we will apply it to each of the three separate terms. 

Define
\begin{align*}
    (s^{(i)})^2 &= \left\|e_{(j-1) p_{3}+k}^{\top} \mathbf{V}_{1} \mathbf{V}_{1}^{\top} \big(\Sigma^{(i)}_1\big)^{1/2}\right\|^{2}; \\
    (s^{(j)})^2  &\coloneqq \left\|e_{(k-1) p_{1}+i}\t \mathbf{V}_{2} \mathbf{V}_{2}^{\top} \big(\Sigma^{(j)}_2\big)^{1/2}\right\|^{2}; \\
    (s^{(k)})^2 &\coloneqq \left\|e_{(i-1)p_{2}+j}\t \mathbf{V}_{3} \mathbf{V}_{3}^{\top} \big(\Sigma^{(k)}_3\big)^{1/2}\right\|^{2},
\end{align*}
so that $s^2_{ijk} = (s^{(i)})^2 + (s^{(j)})^2 + (s^{(k)})^2.$ Define $(\tilde s^{(i)})^2, (\tilde s^{(j)})^2,$ and $(\tilde s^{(k)})^2$ similarly.  Then
\begin{align*}
    \p\bigg\{ |\tilde s^2_{ijk} - s^2_{ijk} | \geq t \bigg\} &\leq 3 \max_{c \in ijk} \p\bigg\{ | (s^{(c)})^2 - (\tilde s^{(c)})^2| \geq t/3 \bigg\}.
\end{align*}
Without loss of generality, we focus on the first term; i.e., $s^{(i)}$.  Observe that
\begin{align*}
    \E (\tilde s^{(i)})^2 = (s^{(i)})^2,
\end{align*}
so that the difference $(\tilde s^{(i)})^2 - (s^{(i)})^2$ term is a sum of mean-zero random variables.  In order to apply Bernstein's inequality (Theorem 2.8.1 in \citet{vershynin_high-dimensional_2018}), we need to bound:
\begin{align*}
A_1 &\coloneqq \sum_{a} \| (( \mathbf{Z}_1)_{ia}^2 - \sigma^2_{ia} ) (\mathbf{V}_1 \mathbf{V}_1\t)_{a,(j-1)p_3 + k}^2 \|_{\psi_1}^2; \\
A_2 &\coloneqq \max_{l}\| (( \mathbf{Z}_1)_{ia}^2 - \sigma^2_{ia} ) (\mathbf{V}_1 \mathbf{V}_1\t)_{a,(j-1)p_3 + k} ^2\|_{\psi_1}.
\end{align*}
However,
\begin{align*}
    \sum_{a} \|& (( \mathbf{Z}_1)_{ia}^2 - \sigma^2_{ia} ) (\mathbf{V}_1 \mathbf{V}_1\t)_{a,(j-1)p_3 + k}^2 \|_{\psi_1}^2   \\
     &= \sum_{a}(\mathbf{V}_1 \mathbf{V}_1\t)_{a,(j-1)p_3 + k}^4 \| (\mathbf{Z}_1)_{ia}^2 - \sigma^2_{ia} \|_{\psi_1}^2 \\
    &\leq \sigma^4 \max_{a}  (\mathbf{V}_1 \mathbf{V}_1\t)_{a,(j-1)p_3 + k}^2 \sum_{a} (\mathbf{V}_1 \mathbf{V}_1\t)_{a,(j-1)p_3 + k}^2 \\
    &\leq \sigma^4 \| \mathbf{V}_1 \|_{2,\infty}^2 \| e_{(j-1)p_3 + k}\t \mathbf{V}_1 \|^2 \| (\mathbf{V}_1 \mathbf{V}_1\t)_{\cdot,(j-1)p_3 + k } \|^2 \\
    &\leq C \sigma^4 \frac{ \mu_0^2 r}{p^2} \| e_{(j-1)p_3 + k}\t \mathbf{V}_1 \|^4,
    \end{align*}
and
\begin{align*}
    \max_{a}\|& (( \mathbf{Z}_1)_{ia}^2 - \sigma^2_{ia} ) (\mathbf{V}_1 \mathbf{V}_1\t)_{a,(j-1)p_3 + k}^2 \|_{\psi_1}   \\
     &\leq \max_{a} | (\mathbf{V}_1 \mathbf{V}_1\t)_{a,(j-1)p_3 + k}^2 | \|(\mathbf{Z}_1)_{ia}^2 - \sigma^2_{ia} \|_{\psi_1} \\
    &\leq \sigma^2 \| \mathbf{V}_1 \|_{2,\infty}^2 \| \| e_{(j-1)p_3 + k}\t \mathbf{V}_1 \|^2  \\
    &\leq C\sigma^2 \mu_0^2 \frac{r}{p^2} \|  e_{(j-1)p_3 + k}\t \mathbf{V}_1 \|^2.
\end{align*}
Therefore, by Bernstein's inequality, 
\begin{align*}
    \p\bigg\{ &| (\tilde s^{(i)})^2 - (s^{(i)})^2 | \geq C t \bigg\} \\
    &\leq 2 \exp \bigg\{ -c \min\bigg( \frac{t^2}{\sigma^4 \| e_{(j-1)p_3 + k}\t \mathbf{V}_1 \|^4 \frac{\mu_0^2  r}{p^2} }, \frac{t }{ \sigma^2 \frac{\mu_0^2 r}{p^2} \|  e_{(j-1)p_3 + k}\t \mathbf{V}_1 \|^2   } \bigg) \bigg\}.
\end{align*}
Taking
\begin{align*}
    t &= C\sigma^2 \| e_{(j-1)p_3 + k}\t \mathbf{V}_1 \|^2 \mu_0 \frac{\sqrt{r\log(p)}}{p},
\end{align*}
yields that
\begin{align*}
     \p\bigg\{ | (\tilde s^{(i)})^2 - (s^{(i)})^2 | &\geq C\sigma^2 \| e_{(j-1)p_3 + k}\t \mathbf{V}_1 \|^2 \mu_0 \frac{\sqrt{r\log(p)}}{p} \bigg\} \\
     &\leq 2 \exp \bigg\{ -c \min\bigg( C^2\log(p) , C\frac{p\sqrt{\log(p)}}{\mu_0 \sqrt{r}} \bigg) \bigg\} \\
     &\leq Cp^{-20},
\end{align*}
since $\mu_0^2 r \lesssim \sqrt{p}$.  Consequently, by symmetry and the union bound, we obtain with this same probability that
\begin{align*}
    | \tilde s_{ijk}^2 - s_{ijk}^2 | &\lesssim \sigma^2 \mu_0 \frac{\sqrt{r\log(p)}}{p} \bigg(  \| e_{(j-1)p_3 + k}\t \mathbf{V}_1 \|^2 + \| e_{(k-1)p_1 + i}\t \mathbf{V}_2 \|^2 + \| e_{(i-1)p_2 + j}\t \mathbf{V}_3 \|^2 \bigg) \\
    &\ll s^2_{ijk},
\end{align*}
since $s^2_{ijk}$ satisfies the lower bound
\begin{align*}
    s^2_{ijk} \geq \sigma_{\min}^2 \bigg( \| e_{(j-1)p_3 + k}\t \mathbf{V}_1 \|^2 + \| e_{(k-1)p_1 + i}\t \mathbf{V}_2 \|^2 + \| e_{(i-1)p_2 + j}\t \mathbf{V}_3 \|^2 \bigg),
\end{align*}
and $\sigma/\sigma_{\min} = O(1)$ by assumption.  
\item \textbf{Step 3: Combining These Bounds}: 
Combining steps 1 and 2, we see that with probability at least $1 - O(p^{-6})$ that
\begin{align*}
    |s^2_{ijk} - \hat s^2_{ijk} | &\lesssim \sigma^2 \mu_0 \frac{\sqrt{r\log(p)}}{p} \bigg( \| e_{(j-1)p_3 + k}\t \mathbf{V}_1 \|^2 + \| e_{(k-1)p_1 + i}\t \mathbf{V}_2 \|^2 + \| e_{(i-1)p_2 + j}\t \mathbf{V}_3 \|^2 \bigg) \\
    &\quad + \frac{\sigma^2 \kappa \mu_0^3 r^{3/2} \log(p)}{p^3} + \frac{\sigma^3 \mu_0^5 r^3 \kappa^2 \log^{3/2}(p)}{\lambda p^{3/2}},
\end{align*}
where we used the bound \eqref{vhatvtildebound}.  Therefore,
\begin{align*}
    |s^2_{ijk} - \hat s^2_{ijk} | &\ll s^2_{ijk}
\end{align*}
under the assumption that
\begin{align*}
    s^2_{ijk} \gg \max\bigg\{  \frac{\sigma^{2} \kappa \mu_{0}^{3} r^{3 / 2} \log (p)}{p^{3}},\frac{\sigma^{3} \mu_{0}^{5} r^{3} \kappa^{2} \log ^{3 / 2}(p)}{\lambda p^{3 / 2}} \bigg\}. \numberthis \label{3202023}
\end{align*}
The second term is guaranteed by the assumption in \cref{thm:civalidity}.  For the first term, recall that the assumption in \cref{thm:asymptoticnormalityentries} implies that
\begin{align*}
    s_{ijk}^2 &\gg \frac{\sigma^2 \kappa^4 \mu_0^6 r^4 \log(p)}{p^3},
\end{align*}
so that the assumption in \eqref{3202023} is met.  

 From this expansion, it holds that
\begin{align*}
    \hat s_{ijk} &= s_{ijk}(1 + o(1))
\end{align*}
with probability at least $1 - O(p^{-6})$.  Therefore, by \cref{thm:asymptoticnormalityentries}, we have that 
\begin{align*}
    \bigg| \p\bigg\{ &\mathcal{T}_{ijk} \in \mathrm{C.I.}_{\alpha}( \mathcal{\hat T}_{ijk} ) \bigg\} - (1 - \alpha) \bigg| \\
    &= \bigg| \p\bigg\{ | \mathcal{\hat T}_{ijk} - \mathcal{T}_{ijk} | \leq z_{\alpha/2} \sqrt{\hat s^2_{ijk}} \bigg\} - (1 - \alpha ) \bigg| \\
    &\leq \bigg| \p\bigg\{ |\mathcal{\hat T}_{ijk} - \mathcal{T}_{ijk} | \leq z_{\alpha/2} s_{ijk} (1 + o(1)) \bigg\} - (1-\alpha) \bigg| + o(1) \\
    &\leq \bigg| \Phi\big( (1+o(1)\big) z_{\alpha/2} - \Phi\big(-(1+o(1))z_{\alpha/2} \big) - (1-\alpha) \bigg| + o(1) \\
    &\leq \bigg| \Phi(z_{\alpha/2}) - \Phi(-z_{\alpha/2}) - (1-\alpha) \bigg| + 2 \bigg| \Phi\big(  (1+ o(1))z_{\alpha/2} \big) - \Phi(z_{\alpha/2}) \bigg| + o(1) \\
    &\leq o(1),
\end{align*}
which follows by the Lipschitz continuity of $\Phi$.  This completes the proof.
\end{itemize}
\end{proof}

\subsection{Proof of Preliminary Lemmas from \cref{sec:civaliditylemmas}}
\label{sec:prelimproofsvalidity}

This section contains all the proofs from \cref{sec:civaliditylemmas}.

\subsubsection{Proof of \cref{lem:Vmatrixcloseness}}

\begin{proof}[Proof of \cref{lem:Vmatrixcloseness}]
Without loss of generality, we consider $k = 1$, and recall we assume $t$ is such that $t = t_0 +1$, where $t_0$ is such that \cref{thm:twoinfty} holds.  We note that $\mathbf{\hat V}_1$ is the right orthonormal matrix in the truncated SVD of the matrix
\begin{align*}
    \bigg(\mathbf{T}_1 + \mathbf{Z}_1 \bigg) \bigg( \uhat_2^{(t-1)} \big( \uhat_2^{(t-1)} \big)\t \otimes  \uhat_3^{(t-1)} \big( \uhat_3^{(t-1)} \big)\t \bigg).
\end{align*}
Note that
\begin{align*}
    \mathbf{E}\t &\coloneqq \mathbf{T}_1 - \bigg( \mathbf{T}_1  + \mathbf{Z}_1 \bigg) \bigg( \uhat_2^{(t-1)} \big( \uhat_2^{(t-1)} \big)\t \otimes  \uhat_3^{(t-1)} \big( \uhat_3^{(t-1)} \big)\t \bigg) \\
    &= \mathbf{T}_1 - \mathbf{T}_1 \bigg( \uhat_2^{(t-1)} \big( \uhat_2^{(t-1)} \big)\t \otimes  \uhat_3^{(t-1)} \big( \uhat_3^{(t-1)} \big)\t \bigg) \\
    &\quad - \mathbf{Z}_1 \bigg( \uhat_2^{(t-1)} \big( \uhat_2^{(t-1)} \big)\t \otimes  \uhat_3^{(t-1)} \big( \uhat_3^{(t-1)} \big)\t \bigg) \\
    &= \mathbf{T}_1 \bigg( \mathcal{P}_{\U_2} \otimes \mathcal{P}_{\U_3}  - \bigg( \uhat_2^{(t-1)} \big( \uhat_2^{(t-1)} \big)\t \otimes  \uhat_3^{(t-1)} \big( \uhat_3^{(t-1)} \big)\t \bigg) \bigg) \\
    &\quad - \mathbf{Z}_1 \bigg( \uhat_2^{(t-1)} \big( \uhat_2^{(t-1)} \big)\t \otimes  \uhat_3^{(t-1)} \big( \uhat_3^{(t-1)} \big)\t \bigg).
\end{align*}
Taking norms, it holds that
\begin{align}
    \| \mathbf{E} \| &\leq \| \mathbf{T}_1 \| \bigg\| \mathcal{P}_{\U_2} \otimes \mathcal{P}_{\U_3}  - \bigg( \uhat_2^{(t-1)} \big( \uhat_2^{(t-1)} \big)\t \otimes  \uhat_3^{(t-1)} \big( \uhat_3^{(t-1)} \big)\t \bigg) \bigg\| \nonumber \\
    &\quad + \|  \mathbf{Z}_1 \bigg( \uhat_2^{(t-1)} \big( \uhat_2^{(t-1)} \big)\t \otimes  \uhat_3^{(t-1)} \big( \uhat_3^{(t-1)} \big)\t \bigg) \| \nonumber \\
    &\leq 2 \lambda_1 \max \bigg\{ \| \sin\Theta( \uhat_2^{(t-1)}, \U_2 ) \|, \| \sin\Theta(\uhat_3^{(t-1)},\U_3 ) \| \bigg\} \nonumber \\
    &\quad + C \sigma \sqrt{pr} \nonumber \\
    &\lesssim  \kappa^2  \sigma \sqrt{p\log(p)} + \sigma \sqrt{pr} \label{Ebound}  \\
    &\ll \lambda, \nonumber
\end{align}
which holds with probability at least $1 - O(p^{-9})$ \textcolor{black}{by \eqref{sinthetaegood}}.  Here we have used the fact that $t_0 + 1 \leq t \leq t_{\max}$, as well as the bounds on the term $\|\mathbf{Z}_1 \big( \uhat_2^{(t-1)} (\uhat_2^{(t-1)})\t \otimes \uhat_3^{(t-1)} (\uhat_3^{(t-1)})\t \| \leq C \sigma \sqrt{pr}$, which holds on the event $\mathcal{E}_{\mathrm{Good}}$ (defined in \cref{sec:notation2}).  Therefore, by the Davis-Kahan Theorem, it holds that
\begin{align*}
    \| \sin\Theta\big( \mathbf{\hat V}_1 ,\mathbf{V}_1 \big) \| &\lesssim \frac{\kappa^2 \sigma \sqrt{p\log(p)} + \sigma \sqrt{pr}}{\lambda},
\end{align*}
with probability at least $1 - O(p^{-9})$. The same argument goes through for the other modes as well. 

We now consider the $\ell_{2,\infty}$ error for $\mathbf{\hat V}_1$.  We will apply Theorem 3.7 of \citet{cape_two--infinity_2019}, (with $\mathbf{X}$ therein defined as $\mathbf{T}_1\t$ and $\U$ and $\mathbf{V}$ switched from their notation) to see that % there exists an orthogonal matrix $\mathbf{W}_{\mathbf{V}_1}$ such that
\begin{align}
    \| \mathbf{\hat V}_1 - \mathbf{V}_1 \mathbf{W}_{\mathbf{V}_1} \|_{2,\infty} &\leq 2 \bigg( \frac{ \| ( \mathbf{I} - \mathbf{V}_1 \mathbf{V}_1\t ) \mathbf{E} \U_1 \U_1\t \|_{2,\infty}}{\lambda} \bigg)  \nonumber\\
    &\quad + 2 \bigg( \frac{ \| ( \mathbf{I} - \mathbf{V}_1 \mathbf{V}_1\t ) \mathbf{E} ( \mathbf{I} - \U_1 \U_1\t) \|_{2,\infty}}{\lambda} \bigg) \| \sin\Theta(\uhat_1^{(t)}, \U_1 ) \|  \nonumber\\
    &\quad + 2 \bigg( \frac{  \| ( \mathbf{I} - \mathbf{V}_1 \mathbf{V}_1\t ) \mathbf{T}_1\t ( \mathbf{I} - \U_1 \U_1\t) \|_{2,\infty}}{\lambda} \bigg) \| \sin\Theta( \uhat_1^{(t)}, \U_1 ) \| \nonumber\\
    &\quad + \| \sin\Theta( \mathbf{\hat V}_1, \mathbf{V}_1 ) \|^2 \| \mathbf{V}_1 \|_{2,\infty} \nonumber\\
    &\leq 2 \bigg( \frac{ \| ( \mathbf{I} - \mathbf{V}_1 \mathbf{V}_1\t ) \mathbf{E} \U_1 \U_1\t \|_{2,\infty}}{\lambda} \bigg)  \nonumber\\
    &\quad + 2 \bigg( \frac{ \| ( \mathbf{I} - \mathbf{V}_1 \mathbf{V}_1\t ) \mathbf{E} ( \mathbf{I} - \U_1 \U_1\t) \|_{2,\infty}}{\lambda} \bigg) \| \sin\Theta(\uhat_1^{(t)}, \U_1 ) \|  \nonumber\\
    &\quad + \| \sin\Theta( \mathbf{\hat V}_1, \mathbf{V}_1 ) \|^2 \| \mathbf{V}_1 \|_{2,\infty}, \\
    &\coloneqq (I) + (II) + (III),
    \label{eq:v1twoinfty}
\end{align}
where we have used the fact that $(\mathbf{I}- \mathbf{V}_1 \mathbf{V}_1\t) \mathbf{T}_1\t = 0$.  Therefore, it suffices to bound each of the other three terms above.  
\begin{itemize}
    \item 
\textbf{The term $(I)$}: We note that numerator satisfies
\begin{align*}
    &\| ( \mathbf{I} - \mathbf{V}_1 \mathbf{V}_1\t ) \mathbf{E} \U_1 \U_1\t \|_{2,\infty} \\
    &= \|  ( \mathbf{I} - \mathbf{V}_1 \mathbf{V}_1\t ) \bigg[ \uhat_2^{(t-1)} (\uhat_2^{(t-1)})\t \otimes \uhat_3^{(t-1)} (\uhat_3^{(t-1)})\t \big(  \mathbf{T}_1\t + \mathbf{Z}_1\t \big) - \mathbf{T}_1\t \bigg] \U_1 \U_1\t \|_{2,\infty} \\
    &\leq \|  ( \mathbf{I} - \mathbf{V}_1 \mathbf{V}_1\t )  \bigg( \uhat_2^{(t-1)} (\uhat_2^{(t-1)})\t \otimes \uhat_3^{(t-1)}( \uhat_3^{(t-1)})\t - \U_2 \U_2\otimes \U_3 \U_3\t \bigg) \mathbf{T}_1\t  \U_1 \U_1\t \|_{2,\infty} \\
    &\quad + \|  ( \mathbf{I} - \mathbf{V}_1 \mathbf{V}_1\t )\bigg(\uhat_2^{(t-1)} (\uhat_2^{(t-1)})\t \otimes \uhat_3^{(t-1)}( \uhat_3^{(t-1)})\t\bigg) \mathbf{Z}_1\t  \U_1 \U_1\t \|_{2,\infty} \\
    &\leq \|    \bigg( \uhat_2^{(t-1)} (\uhat_2^{(t-1)})\t \otimes \uhat_3^{(t-1)}( \uhat_3^{(t-1)})\t- \U_2 \U_2\otimes \U_3 \U_3\t \bigg) \mathbf{T}_1\t   \|_{2,\infty} \\
    &\quad +\|  \mathbf{V}_1 \mathbf{V}_1\t \bigg( \uhat_2^{(t-1)} (\uhat_2^{(t-1)})\t \otimes \uhat_3^{(t-1)}( \uhat_3^{(t-1)})\t - \U_2 \U_2\otimes \U_3 \U_3\t \bigg) \mathbf{T}_1\t  \|_{2,\infty} \\
    &\quad + \|   \bigg(\uhat_2^{(t-1)} (\uhat_2^{(t-1)})\t \otimes \uhat_3^{(t-1)}( \uhat_3^{(t-1)})\t\bigg) \mathbf{Z}_1\t  \U_1 \U_1\t \|_{2,\infty} \\
     &\quad + \|   \mathbf{V}_1 \mathbf{V}_1\t \bigg( \uhat_2^{(t-1)} (\uhat_2^{(t-1)})\t \otimes \uhat_3^{(t-1)}( \uhat_3^{(t-1)})\t\bigg) \mathbf{Z}_1\t \U_1 \U_1\t \|_{2,\infty} \\
     &\leq \|  \uhat_2^{(t-1)} (\uhat_2^{(t-1)})\t \otimes \uhat_3^{(t-1)}( \uhat_3^{(t-1)})\t- \U_2 \U_2\otimes \U_3 \U_3\t \|_{2,\infty} \lambda_1 \\
     &\quad + \| \mathbf{V}_1 \|_{2,\infty} \| \uhat_2^{(t-1)} (\uhat_2^{(t-1)})\t \otimes \uhat_3^{(t-1)}( \uhat_3^{(t-1)})\t- \U_2 \U_2\otimes \U_3 \U_3\t \| \lambda_1 \\
     &\quad + \| \uhat_2^{(t-1)} (\uhat_2^{(t-1)})\t \otimes \uhat_3^{(t-1)}( \uhat_3^{(t-1)})\t\|_{2,\infty} \| \bigg( \uhat_2^{(t-1)} (\uhat_2^{(t-1)})\t \otimes \uhat_3^{(t-1)}( \uhat_3^{(t-1)})\t \bigg) \mathbf{Z}_1\t \| \\
     &\quad + \| \mathbf{V}_1 \|_{2,\infty} \| \bigg( \uhat_2^{(t-1)} (\uhat_2^{(t-1)})\t \otimes \uhat_3^{(t-1)}( \uhat_3^{(t-1)})\t \bigg) \mathbf{Z}_1\t \|,
\end{align*}
where the final inequality is due to the fact that $\uhat_2^{(t-1)} (\uhat_2^{(t-1)})\t \otimes \uhat_3^{(t-1)}( \uhat_3^{(t-1)})\t$ is an orthogonal projection matrix  and hence equals its square. 
From the proof of \cref{lem:zentrywiselemma}, under the condition $\kappa^2 \mu_0^2 r^{3/2} \sqrt{\log(p)} \lesssim p^{1/4}$, \cref{cor:asymptoticnormality_projection} implies that with probability at least $1 - O(p^{-9})$,
\begin{align*}
    \max_k \| \uhat_k \uhat_k\t - \U_k \U_k\t \|_{2,\infty} \lesssim \frac{\sigma \mu_0 \sqrt{r\log(p)}}{\lambda}.
\end{align*}
Therefore,
\begin{align*}
    \| \uhat_2 \uhat_2\t \otimes \uhat_3\uhat_3\t - \U_2 \U_2\t \otimes \U_3\U_3\t \|_{2,\infty} &\leq \| \uhat_2 \uhat_2\t \otimes \big( \uhat_3 \uhat_3\t - \U_3 \U_3\t \big) \|_{2,\infty} \\
    &\quad + \| \big( \uhat_2 \uhat_2 - \U_2 \U_2\t \big) \otimes \U_3 \U_3\t \|_{2,\infty} \\
    &\lesssim \mu_0 \sqrt{\frac{r}{p}} \frac{\sigma \mu_0 \sqrt{r\log(p)}}{\lambda} \\
    &\lesssim \frac{ \sigma \mu_0^2 r \sqrt{\log(p)}}{\lambda \sqrt{p}}.
\end{align*}
In addition, with this same probability, \textcolor{black}{by \eqref{sinthetaegood}} it holds that
\begin{align*}
    \max_k \| \uhat_k \uhat_k\t - \U_k \U_k\t \| &\lesssim \max_k \|\sin\Theta(\uhat_k,\U_k) \| \lesssim \frac{\kappa \sqrt{p\log(p)}}{\lambda/\sigma}.
\end{align*}
Plugging in all these bounds yields
\begin{align*}
      \| ( \mathbf{I} - &\mathbf{V}_1 \mathbf{V}_1\t ) \mathbf{E} \U_1 \U_1\t \|_{2,\infty} \\
      &\lesssim \lambda_1 \frac{ \sigma \mu_0^2 r \sqrt{\log(p)}}{\lambda \sqrt{p}} + \lambda_1 \mu_0 \frac{\sqrt{r}}{p} \frac{\kappa \sigma \sqrt{p\log(p)}}{\lambda} \\
      &\quad + \bigg(\| \uhat_2 \uhat_2\t \otimes \uhat_3 \uhat_3\t \|_{2,\infty} + \| \mathbf{V}_1 \|_{2,\infty} \bigg) \| \bigg( \uhat_2 \uhat_2\t \otimes \uhat_3 \uhat_3\t \bigg) \mathbf{Z}_1\t \| \\
      &\lesssim \lambda_1 \frac{ \sigma \mu_0^2 r \sqrt{\log(p)}}{\lambda \sqrt{p}} + \lambda_1 \mu_0 \frac{\sqrt{r}}{p} \frac{\kappa \sigma \sqrt{p\log(p)}}{\lambda} \\
      &\quad + \mu_0^2 \frac{r}{p} \| \bigg( \uhat_2 \uhat_2\t \otimes \uhat_3 \uhat_3\t \bigg) \mathbf{Z}_1\t \|.
\end{align*}
On the event $\mathcal{E}_{\mathrm{Good}}$, it holds that
\begin{align*}
     \| \mathbf{Z}_1 \bigg( \uhat_2 \uhat_2\t \otimes \uhat_3 \uhat_3\t \bigg) \| &\lesssim \sigma\sqrt{pr};
     \end{align*}
Therefore, with probability at least $1 - O(p^{-9})$,
\begin{align}
    &\frac{ \| ( \mathbf{I} - \mathbf{V}_1 \mathbf{V}_1\t ) \mathbf{E} \U_1 \U_1\t \|_{2,\infty} }{\lambda} \nonumber \\
    &\lesssim  \frac{1}{\lambda} \bigg\{ \lambda_1 \frac{ \sigma \mu_0^2 r \sqrt{\log(p)}}{\lambda \sqrt{p}} + \lambda_1 \mu_0 \frac{\sqrt{r}}{p} \frac{\kappa \sigma \sqrt{p\log(p)}}{\lambda}  + \mu_0^2 \frac{r}{p} \sigma \sqrt{pr} \bigg\}\nonumber \\
     &\lesssim \frac{\kappa \sigma \mu_0^2 r \sqrt{\log(p)}  + \kappa^2 \sigma \mu_0 \sqrt{r\log(p)} + \sigma \mu_0^2 r^{3/2}}{\lambda\sqrt{p}  } \nonumber \\
     &\lesssim \frac{\kappa^2 \sigma \mu_0^2 r^{3/2} \sqrt{\log(p)}}{\lambda \sqrt{p}}\label{v1bd1}
     \end{align}
\item  \textbf{The Term $(II)$:} We note that the numerator of the second term in \eqref{eq:v1twoinfty} is of the form
\begin{align*}
    \| ( \mathbf{I} - \mathbf{V}_1 \mathbf{V}_1\t ) \mathbf{E}( \mathbf{I} - \U_1 \U_1\t) \|_{2,\infty} &\leq   \| ( \mathbf{I} - \mathbf{V}_1 \mathbf{V}_1\t ) \mathbf{E} \|.
\end{align*}
By repeating the argument above, this term satisfies the same upper bound as $(I)$ as the  $\sin\Theta$ distance is upper bounded by one. 
\item \textbf{The Term $(III)$:} The only remaining term is the term
\begin{align*}
    \| \sin\Theta( \mathbf{\hat V}_1 , \mathbf{V}_1 ) \|^2 \|\mathbf{V}_1 \|_{2,\infty} &\lesssim \mu_0 \frac{\sqrt{r}}{p} \frac{ \| \mathbf{E} \|^2}{\lambda^2}
\end{align*}
by the Davis-Kahan Theorem.  We have already showed in \eqref{Ebound} that
\begin{align*}
    \| \mathbf{E} \| &\lesssim \sigma \sqrt{pr} +  \kappa^2 \sigma \sqrt{p\log(p)}.
\end{align*}
Therefore,
\begin{align*}
   \mu_0 \frac{\sqrt{r}}{p} \frac{ \| \mathbf{E} \|^2}{\lambda^2} &\lesssim \mu_0 \frac{\sqrt{r}}{p}  \frac{ \big( \sigma \sqrt{pr} +  \kappa^2 \sigma \sqrt{p\log(p)} \big)^2}{\lambda^2} \\
    &\lesssim \mu_0 \frac{\sqrt{r}}{p}  \frac{\sigma^2 pr + \kappa^4 \sigma^2 p\log(p) + \sigma^2 \kappa^2 p \sqrt{r\log(p)}}{\lambda^2} \\
    &\lesssim \frac{\mu_0 \sigma^2 \sqrt{r} \big( r + \kappa^4 \log(p) + \kappa^2 \sqrt{r\log(p)}\big)}{\lambda^2} \\
    &\lesssim \frac{\mu_0 \sigma^2 r^{3/2} \kappa^4 \log(p)}{\lambda^2}.
\end{align*}
This bound is smaller than the bound in \eqref{v1bd1} as long as
\begin{align*}
    \lambda/\sigma \gtrsim \frac{\kappa^2}{\mu_0} \sqrt{p\log(p)}.
\end{align*}
Recall that we assume that $\lambda/\sigma \gtrsim \kappa p^{3/4} \sqrt{\log(p)}$ and that $\kappa \lesssim p^{1/4}$.  Therefore, \eqref{v1bd1} dominates this upper bound.
\end{itemize}
Putting these bounds together, with probability at least $1 - O(p^{-9})$ it holds that
\begin{align*}
    \| \mathbf{\hat V}_1 - \mathbf{V}_1 \mathbf{W}_{\mathbf{V}_1} \|_{2,\infty} &\lesssim \frac{\kappa^2 \sigma \mu_0^2 r^{3/2} \sqrt{\log(p)}}{\lambda \sqrt{p}},
\end{align*}
which proves the second assertion.

Next, we show that
\begin{align*}
     \| \mathbf{W}_{\mathbf{V}_k}\t \mathbf{\Lambda}_k\inv - \mathbf{\hat{\Lambda}}_k\inv \mathbf{\hat W}_k\t \| &\lesssim \frac{1}{\lambda} \bigg( \frac{\sigma \sqrt{pr} + \kappa^2 \sigma \sqrt{p\log(p)}}{\lambda} \bigg)
\end{align*}
with probability at least $1 - O(p^{-9})$.  
Since $\mathbf{\hat{\Lambda}}_k$ are the leading $r_k$ singular values of the matrix 
\begin{align*}
  \big( \mathbf{T}_k + \mathbf{Z}_k \big) \big( \uhat_{k+1}\uhat_{k+1}\t \otimes \uhat_{k+2} \uhat_{k+2}\t \big),
\end{align*}
then it holds that
\begin{align*}
(\mathbf{\hat{\Lambda}}_k)_{r_k} &\geq (\mathbf{\Lambda}_k)_{r_k}  - 2 \| \mathbf{E}_k \| \\
&\gtrsim (\mathbf{\Lambda}_k)_{r_k},
\end{align*}
where $\mathbf{E}_k$ is as in the previous part of the lemma (where above we suppressed the dependence of $\mathbf{E}_k$ on $k$). Hence $\| \mathbf{\hat{\Lambda}}_k\inv \| \lesssim \lambda\inv$, which will be useful in the sequel. 

Observe that
\begin{align*}
    \| \mathbf{W}_{\mathbf{V}_k}\t \mathbf{\Lambda}_k\inv - \mathbf{\hat{\Lambda}}_k\inv \mathbf{\hat W}_k\t \| &\leq \| \mathbf{\Lambda}_k\inv ( \mathbf{\hat W}_k - \U_k\t \uhat_k ) \| + \| \mathbf{\Lambda}_k\inv \U_k\t \uhat_k - \mathbf{V}_k\t \mathbf{\hat V}_k \mathbf{\hat{\Lambda}}_k\inv \| \\
    &\quad + \|\big( \mathbf{W}_{\mathbf{V}_k}- \mathbf{V}_k\t \mathbf{\hat V}_k \big)\mathbf{\hat{\Lambda}}_k\inv \| \\
    &\lesssim \frac{\| \mathbf{\hat W}_k - \U_k\t \uhat_k \| + \|\mathbf{W}_{\mathbf{V}_k}- \mathbf{V}_k\t \mathbf{\hat V}_k \| }{\lambda} \\
    &\quad + \| \mathbf{\Lambda}_k\inv\big( \U_k\t \uhat_k \mathbf{\hat{\Lambda}}_k - \mathbf{\Lambda}_k \mathbf{V}_k\t \mathbf{\hat V}_k \big) \mathbf{\hat{\Lambda}}_k\inv \| \\
    &\lesssim \frac{\| \mathbf{\hat W}_k - \U_k\t \uhat_k \| + \|\mathbf{W}_{\mathbf{V}_k}- \mathbf{V}_k\t \mathbf{\hat V}_k \| }{\lambda} \\
    &\quad + \frac{ \| \U_k\t \uhat_k \mathbf{\hat{\Lambda}}_k - \mathbf{\Lambda}_k \mathbf{V}_k\t \mathbf{\hat V}_k \|}{\lambda^2}.
\end{align*}
Note that since $\mathbf{\hat W}_k = \mathrm{sgn}(\U_k,\uhat_k)$, it holds on the event $\mathcal{E}_{\mathrm{Good}}$ \textcolor{black}{by \cref{sinthetaegood} and the same argument as \eqref{sintheta14}} that
\begin{align*}
    \| \mathbf{\hat W}_k - \U_k\t \uhat_k \| &\lesssim \| \sin\Theta(\uhat_k,\U_k) \|^2 \\
    &\lesssim \bigg( \frac{\kappa \sigma \sqrt{p\log(p)}}{\lambda} \bigg)^2;
\end{align*}
and similarly that
\begin{align*}
    \| \mathbf{W}_{\mathbf{V}_k} - \mathbf{V}_k\t\mathbf{\hat V}_k \| &\leq \| \sin\Theta(\mathbf{V}_k ,\mathbf{\hat V}_k ) \|^2 \\
    &\lesssim \bigg( \frac{\kappa^2 \sigma \sqrt{p\log(p)} + \sigma \sqrt{pr}}{\lambda} \bigg)^2,
\end{align*}
which holds with probability at least $1 - O(p^{-9})$ by the previous part of this proof.  For the remaining term, we note that by the eigenvector eigenvalue equation, it holds that $\uhat_k \mathbf{\hat{\Lambda}}_k = (\mathbf{T}_k + \mathbf{E}_k) \mathbf{\hat V}_k$, and hence that
\begin{align*}
    \| \U_k\t \uhat_k \mathbf{\hat{\Lambda}}_k - \mathbf{\Lambda}_k \mathbf{V}_k\t \mathbf{\hat V}_k \| &= \| \U_k\t \big( \mathbf{T}_k + \mathbf{E}_k \big) \mathbf{\hat V}_k - \U_k\t \mathbf{T}_k \mathbf{\hat V}_k \| \\
    &= \| \U_k\t \mathbf{E}_k \mathbf{\hat V}_k \| \\
    &\leq \| \mathbf{E}_k \| \\
    &\lesssim \sigma \sqrt{pr} + \kappa^2 \sigma \sqrt{p\log(p)},
\end{align*}
where the final inequality holds with probability at least $1 - O(p^{-9})$ by \eqref{Ebound}.  Putting these all together shows that
\begin{align*}
    \| \mathbf{W}_{\mathbf{V}_k}\t \mathbf{\Lambda}_k\inv - \mathbf{\hat{\Lambda}}_k\inv \mathbf{\hat W}_k\t \| &\lesssim \frac{1}{\lambda} \bigg( \frac{\kappa^2 \sigma \sqrt{p\log(p)} + \sigma \sqrt{pr}}{\lambda} \bigg)^2 + \frac{\sigma \sqrt{pr} + \kappa^2 \sigma \sqrt{p\log(p)}}{\lambda^2} \\
    &\lesssim \frac{1}{\lambda} \bigg( \frac{\sigma \sqrt{pr} + \kappa^2 \sigma \sqrt{p\log(p)}}{\lambda} \bigg),
\end{align*}
which completes the proof.
\end{proof}

\subsubsection{Proof of \cref{lem:gammacloseness}}
\begin{proof}[Proof of \cref{lem:gammacloseness}]
First we will show that
\begin{align*}
        \| \mathbf{\hat W}_k &\mathbf{\hat \Gamma}^{(m)}_k \mathbf{\hat W}_k\t - \mathbf{\Gamma}^{(m)}_k \|\\
        &\lesssim \frac{\sigma^2}{\lambda^2} \bigg( \frac{\sigma r \sqrt{p}\log(p) + \sigma \kappa^2 \sqrt{rp} \log^{3/2}(p)}{\lambda} + \frac{\kappa \mu_0 \sqrt{r} \log(p)}{p} + \frac{\mu_0 r^{3/2} \sqrt{\log(p)}}{p} \bigg)
\end{align*}
with probability at least $1 - O(p^{-6})$. Define $\tilde \Sigma^{(m)}_k$ as the diagonal matrix whose diagonal entries are the squared entries of $e_m\t \mathbf{Z}_k$.  We will proceed in steps. First, we note that
\begin{align*}
   \bigg\|  \mathbf{\hat W}_k\mathbf{\hat \Gamma}^{(m)}_k  \mathbf{\hat W}_k\t - \mathbf{\Gamma}^{(m)}_k  \bigg\| &= \bigg\|   \mathbf{\hat W}_k\mathbf{\hat{\Lambda}}_k\inv \mathbf{\hat V}_k\t \hat \Sigma^{(m)}_k \mathbf{\hat V}_k \mathbf{\hat{\Lambda}}_k\inv \mathbf{\hat W}_k\t - \mathbf{\Lambda}_k\inv \mathbf{V}_k\t \Sigma^{(m)}_k \mathbf{V}_k \mathbf{\Lambda}_k\inv \bigg\| \\
   &\leq\bigg\|   \mathbf{\hat W}_k\mathbf{\hat{\Lambda}}_k\inv \mathbf{\hat V}_k\t \hat \Sigma^{(m)}_k \mathbf{\hat V}_k \mathbf{\hat{\Lambda}}_k\inv \mathbf{\hat W}_k\t - \mathbf{\Lambda}_k\inv \mathbf{V}_k\t  \tilde \Sigma^{(m)}_k \mathbf{V}_k \mathbf{\Lambda}_k\inv \bigg\| \\
   &\quad + \bigg\|   \mathbf{\Lambda}_k\inv \mathbf{V}_k\t\bigg( \tilde \Sigma^{(m)}_k - \Sigma^{(m)}_k \bigg)\mathbf{V}_k \mathbf{\Lambda}_k\inv \bigg\| \\
   &\leq \bigg\| \bigg( \mathbf{\hat W}_k \mathbf{\hat{\Lambda}}_k\inv \mathbf{\hat V}_k\t - \mathbf{\Lambda}_k\inv \mathbf{V}_k\t \bigg) \hat \Sigma^{(m)}_k \mathbf{\hat V}_k \mathbf{\hat{\Lambda}}_k\inv \mathbf{\hat W}_k\t \bigg\| \\
   &\quad + \bigg\| \mathbf{\Lambda}_k\inv \mathbf{V}_k\t \bigg( \tilde \Sigma^{(m)}_k \mathbf{V}_k \mathbf{\Lambda}_k\inv - \hat \Sigma^{(m)}_k \mathbf{\hat V}_k \mathbf{\hat{\Lambda}}_k\inv \mathbf{\hat W}_k\t \bigg) \bigg\| \\
   &\quad + \bigg\|   \mathbf{\Lambda}_k\inv \mathbf{V}_k\t\bigg( \tilde \Sigma^{(m)}_k - \Sigma^{(m)}_k \bigg)\mathbf{V}_k \mathbf{\Lambda}_k\inv \bigg\| \\
   &\coloneqq \alpha_1 + \alpha_2 + \alpha_3.
\end{align*}
In the subsequent steps, we bound $\alpha_1$, $\alpha_2$, and $\alpha_3$, but first we obtain several preliminary bounds.
\begin{itemize}
    \item \textbf{Step 1: Initial Bounds:} First, it holds that
\begin{align*}
\| \hat \Sigma^{(m)}_k - \tilde \Sigma^{(m)}_k \| &= \max_{l} |  \big(\mathbf{\hat Z}_k\big)_{ml}^2 - \big(\mathbf{ Z}_k\big)_{ml}^2  | \\
    &\leq \bigg( \| \mathcal{\hat Z} \|_{\max} + \| \mathcal{Z} \|_{\max} \bigg) \| \mathcal{\hat Z} - \mathcal{Z} \|_{\max}.
\end{align*}
By \cref{cor:maxnormbound}, with probability at least $1 - O(p^{-6})$ it holds that
\begin{align*}
    \| \mathcal{\hat Z - Z} \|_{\max} &= \| \mathcal{\hat T - T} \|_{\max} \\
    &\lesssim \frac{\sigma \kappa \mu_{0} \sqrt{r \log (p)}}{p}+\frac{\sigma^{2} \mu_{0}^{4} \kappa^{3} r^{3} \log (p)}{\lambda \sqrt{p}}.
\end{align*}
Since $\kappa^2 \mu_0^2 r^{3/2} \sqrt{\log(p)} \lesssim p^{1/4}$, and $\lambda/\sigma \gtrsim \kappa p^{3/4} \sqrt{\log(p)}$, it also holds that
\begin{align*}
    \| \mathcal{\hat Z} \|_{\max} &\leq \| \mathcal{\hat Z - Z} \|_{\max} + \| \mathcal{Z} \|_{\max} \\
    &\lesssim \sigma \sqrt{\log(p)}
\end{align*}
with probability at least $1 - O(p^{-6})$.  Therefore, with this same probability,
\begin{align}
    \| \hat \Sigma^{(m)}_k - \tilde \Sigma^{(m)}_k \| &\lesssim \bigg( \| \mathcal{\hat Z} \|_{\max} + \| \mathcal{Z} \|_{\max} \bigg) \| \mathcal{\hat Z} - \mathcal{Z} \|_{\max} \nonumber \\
    &\lesssim \sigma \sqrt{\log(p)} \| \mathcal{\hat T - T} \|_{\max} \nonumber \\
    &\lesssim \sigma \sqrt{\log(p)} \bigg(\frac{\sigma \kappa \mu_{0} \sqrt{r \log (p)}}{p}+\frac{\sigma^{2} \mu_{0}^{4} \kappa^{3} r^{3} \log (p)}{\lambda \sqrt{p}}\bigg) \nonumber \\
    &\asymp \frac{\sigma^2 \kappa \mu_0 \sqrt{r} \log(p)}{p} + \frac{\sigma^3 \mu_0^4 \kappa^3 r^3 \log^{3/2}(p)}{\lambda \sqrt{p}}. \label{sigmahatminussigmatilde}
\end{align}
This argument reveals that with probability at least $1 - O(p^{-6})$,
\begin{align}
    \| \hat \Sigma^{(m)}_k \| &\leq \| \hat \Sigma^{(m)}_k - \tilde \Sigma^{(m)}_k \| + \| \tilde \Sigma^{(m)}_k \| \nonumber \\
    &\lesssim \frac{\sigma^2 \kappa \mu_0 \sqrt{r} \log(p)}{p} + \frac{\sigma^3 \mu_0^4 \kappa^3 r^3 \log^{3/2}(p)}{\lambda \sqrt{p}} + \max_{l} | (\mathbf{Z}_k)_{ml}^2 | \nonumber \\
    &\lesssim \frac{\sigma^2 \kappa \mu_0 \sqrt{r} \log(p)}{p} + \frac{\sigma^3 \mu_0^4 \kappa^3 r^3 \log^{3/2}(p)}{\lambda \sqrt{p}} + \sigma^2 \log(p) \nonumber \\
    &\asymp \sigma^2 \log(p), \label{sigmahatbound}
\end{align}
where the penultimate line is due to the fact that with probability at least $1- O(p^{-6})$, \begin{align*}
    \max_l | \big(\mathbf{Z}_k\big)_{ml}^2 | &\leq \sigma^2 + \max_l | \big(\mathbf{Z}_k\big)_{ml}^2 - \sigma^2_{ml} | \\
    &\lesssim \sigma^2 \log(p),
\end{align*}
together with the assumption that $\kappa^2 \mu_0^2 r^{3/2} \sqrt{\log(p)} \lesssim p^{1/4}$ and $\lambda/\sigma \gtrsim \kappa p^{3/4} \sqrt{\log(p)}$, so that the first two terms are less than $\sigma^2 \log(p)$.

 In addition, by \cref{lem:Vmatrixcloseness} (whose statement does not depend on this lemma), it holds that
\begin{align}
    \| \mathbf{W}_{\mathbf{V}_k}\t  \mathbf{\Lambda}_k\inv - \mathbf{\hat{\Lambda}}_k\inv \mathbf{\hat W}_k\t \| &\lesssim  \frac{1}{\lambda} \bigg( \frac{\sigma \sqrt{pr} + \kappa^2 \sigma \sqrt{p\log(p)}}{\lambda} \bigg) \label{approxcommute1}
\end{align}
with probability $1 - O(p^{-9})$.  In addition,
\begin{align}
    \| \mathbf{\hat V}_k \mathbf{W}_{\mathbf{V}_k}\t - \mathbf{V}_k \| &= \| \mathbf{\hat V}_k  - \mathbf{V}_k \mathbf{W}_{\mathbf{V}_k}\|  \nonumber \\
    &\leq  \| \mathbf{\hat V}_k  - \mathbf{V}_k \mathbf{W}_{\mathbf{V}_k}\|_F\nonumber  \\
    &= \| \sin\Theta(\mathbf{V}_k ,\mathbf{\hat V}_k ) \|_F \nonumber \\
    &\leq\sqrt{r} \| \sin\Theta(\mathbf{V}_k ,\mathbf{\hat V}_k ) \| \nonumber \\
    &\lesssim \sqrt{r} \frac{\sigma \sqrt{pr} + \kappa^2 \sigma \sqrt{p\log(p)}}{\lambda} \nonumber \\
    &\lesssim \frac{\sigma r \sqrt{p} + \kappa^2 \sigma \sqrt{rp\log(p)}}{\lambda}. \label{vsintheta}
\end{align}
Finally, by \eqref{approxcommute1} and \eqref{vsintheta}, we have that
\begin{align}
     \| \mathbf{\hat V}_k \mathbf{\hat{\Lambda}}_k\inv \mathbf{\hat W}_k\t - \mathbf{V}_k \mathbf{\Lambda}_k\inv \| &\leq  \| \mathbf{\hat V}_k \big( \mathbf{W}_{\mathbf{V}_k}\t \mathbf{\Lambda}_k\inv - \mathbf{\hat{\Lambda}}_k\inv  \mathbf{\hat W}_k\t \big) \| + \| \big(\mathbf{V}_k - \mathbf{\hat V}_k \mathbf{W}_{\mathbf{V}_k}\t \big) \mathbf{\Lambda}_k\inv  \|  \nonumber \\
     &\lesssim  \frac{\sigma r \sqrt{p} + \kappa^2 \sigma \sqrt{rp\log(p)}}{\lambda^2} \label{vhatminusv2}.
\end{align}
\item \textbf{Step 2: Bounding $\alpha_1$:} We have that
\begin{align}
    \alpha_1 &\coloneqq \bigg\| \bigg( \mathbf{\hat W}_k \mathbf{\hat{\Lambda}}_k\inv \mathbf{\hat V}_k\t - \mathbf{\Lambda}_k\inv \mathbf{V}_k\t \bigg) \hat \Sigma^{(m)} \mathbf{\hat V}_k \mathbf{\hat{\Lambda}}_k\inv \mathbf{\hat W}_k\t \bigg\| \nonumber \\
    &\lesssim   \frac{\sigma^2 \log(p)}{\lambda} \| \mathbf{\hat V}_k \mathbf{\hat{\Lambda}}_k\inv \mathbf{\hat W}_k\t - \mathbf{V}_k \mathbf{\Lambda}_k\inv \|   \label{sigmahatboundused1} \\
    %&\lesssim  \frac{\sigma^2 \log(p)}{\lambda}  \bigg( \| \mathbf{\hat V}_k \big( \mathbf{W}_{\mathbf{V}_k}\t \mathbf{\Lambda}_k\inv - \mathbf{\hat{\Lambda}}_k\inv  \mathbf{\hat W}_k\t \big) \| + \| \big(\mathbf{V}_k - \mathbf{\hat V}_k \mathbf{W}_{\mathbf{V}_k}\t \big) \mathbf{\Lambda}_k\inv  \| \bigg) \nonumber \\
    %&\lesssim \frac{\sigma^2 \log(p)}{\lambda} \bigg( \| \mathbf{W}_{\mathbf{V}_k}\t \mathbf{\Lambda}_k\inv - \mathbf{\hat{\Lambda}}_k\inv  \mathbf{\hat W}_k\t \| + \frac{ \| \mathbf{V}_k - \mathbf{\hat V}_k \mathbf{W}_{\mathbf{V}_k}\t\|}{\lambda} \bigg)\nonumber \\
    %&\lesssim \frac{\sigma^2 \log(p)}{\lambda} \bigg( \frac{1}{\lambda} \bigg( \frac{\sigma \sqrt{p r}+\kappa^{2} \sigma \sqrt{p \log (p)}}{\lambda} \bigg) + \frac{\sigma r \sqrt{p}+\kappa^{2} \sigma \sqrt{r p \log (p)}}{\lambda^2} \bigg) \label{alpha12} \\
    &\lesssim \frac{\sigma^2 \log(p)}{\lambda} \bigg(  \frac{\sigma r \sqrt{p} + \kappa^2 \sigma \sqrt{rp\log(p)}}{\lambda^2} \bigg) \label{vhatminusv3} \\
    &\asymp \frac{\sigma^3 r \sqrt{p} \log(p) + \sigma^3 \kappa^2 \sqrt{rp} \log^{3/2}(p)}{\lambda^3} \label{alpha1bound}
\end{align}
where \eqref{sigmahatboundused1} holds by \eqref{sigmahatminussigmatilde} and \eqref{vhatminusv3} follows from \eqref{vhatminusv2}.  This bound holds with probability at least $1 - O(p^{-6})$. 
\item \textbf{Step 3: Bounding $\alpha_2$:} Note that
\begin{align}
    \alpha_2 &\coloneqq \left\|\mathbf{\Lambda}_{k}^{-1} \mathbf{V}_{k}^{\top}\left(\tilde{\Sigma}^{(m)}_k \mathbf{V}_{k} \mathbf{\Lambda}_{k}^{-1}-\hat{\Sigma}^{(m)}_k \hat{\mathbf{V}}_{k} \hat{\Lambda}_{k}^{-1} \hat{\mathbf{W}}_{k}^{\top}\right)\right\| \nonumber \\
    &\leq \frac{1}{\lambda} \bigg\| \tilde \Sigma^{(m)}_k \mathbf{V}_k \mathbf{\Lambda}_k\inv - \hat \Sigma^{(m)}_k \mathbf{\hat V}_k \mathbf{\hat{\Lambda}}_k\inv \mathbf{\hat W}_k\t \bigg\| \nonumber \\
    &\leq \frac{1}{\lambda} \bigg(  \frac{ \| \tilde \Sigma^{(m)}_k - \hat \Sigma^{(m)}_k\|}{\lambda} + \| \hat \Sigma^{(m)}_k \| \| \mathbf{V}_k \mathbf{\Lambda}_k\inv - \mathbf{\hat V}_k \mathbf{\hat{\Lambda}}_k\inv \mathbf{\hat W}_k\t \| \bigg) \nonumber \\
    &\lesssim  \frac{1}{\lambda^2} \bigg( \frac{\sigma^{2} \kappa \mu_{0} \sqrt{r} \log (p)}{p}+\frac{\sigma^{3} \mu_{0}^{4} \kappa^{3} r^{3} \log ^{3 / 2}(p)}{\lambda \sqrt{p}} \bigg) \label{zhatminusz} \\
    &\quad + \frac{\sigma^2 \log(p)}{\lambda} \|  \mathbf{V}_k \mathbf{\Lambda}_k\inv - \mathbf{\hat V}_k \mathbf{\hat{\Lambda}}_k\inv \mathbf{\hat W}_k\t \| \label{zhatminusz2} \\
    &\lesssim  \frac{1}{\lambda^2} \bigg( \frac{\sigma^{2} \kappa \mu_{0} \sqrt{r} \log (p)}{p}+\frac{\sigma^{3} \mu_{0}^{4} \kappa^{3} r^{3} \log ^{3 / 2}(p)}{\lambda \sqrt{p}} \bigg) \nonumber \\
    &\quad + \frac{\sigma^2 \log(p)}{\lambda} \bigg( \frac{\sigma r \sqrt{p}+\kappa^{2} \sigma \sqrt{r p \log (p)}}{\lambda^{2}}\bigg) \label{vhatminusv4} \\
    &\asymp \frac{\sigma^3 r \sqrt{p} \log(p) + \sigma^3 \kappa^2 \sqrt{rp} \log^{3/2}(p)}{\lambda^3} + \frac{\sigma^2 \kappa \mu_0 \sqrt{r} \log(p)}{\lambda^2 p} \label{alpha2bound}
\end{align}
where \eqref{zhatminusz} follows from \eqref{sigmahatminussigmatilde},  \eqref{zhatminusz2} follows from \eqref{sigmahatbound}, and \eqref{vhatminusv4} follows from \eqref{vhatminusv2}. We note that \eqref{alpha2bound} follows since
\begin{align*}
    \frac{\sigma^3 \mu_0^4 \kappa^3 r^3 \log^{3/2}(p)}{\lambda^3 \sqrt{p}} \lesssim \frac{\sigma^3 \kappa^2 \sqrt{rp} \log^{3/2}(p)}{\lambda^3},
\end{align*}
which holds on the assumption $\kappa^2 \mu_0^2 r^{3/2} \sqrt{\log(p)} \lesssim p^{1/4}$. 
\item 
\textbf{Step 4: Bounding $\alpha_3$:} Finally, we note that
\begin{align*}
    \| \mathbf{V}_k\t \mathbf{\Lambda}_k\inv \bigg( \tilde \Sigma^{(m)}_k - \Sigma^{(m)}_k \bigg) \mathbf{\Lambda}_k\inv \mathbf{V}_k \| &\leq r  \| \mathbf{V}_k\t \mathbf{\Lambda}_k\inv \bigg( \tilde \Sigma^{(m)}_k - \Sigma^{(m)}_k \bigg) \mathbf{\Lambda}_k\inv \mathbf{V}_k \|_{\max}.
\end{align*}
We now note that the $i_1,i_2$ entry of the above matrix can be written as
\begin{align*}
    \sum_{j_1,j_2} &\big( \mathbf{V}_k \mathbf{\Lambda}_k\inv \big)_{j_1 i_1} \big( \mathbf{V}_k \mathbf{\Lambda}_k\inv\big)_{j_2 i_2} \bigg( \tilde \Sigma^{(m)}_k - \Sigma^{(m)}_k \bigg)_{j_1 j_2}\\
    &= \sum_{j_1}\big( \mathbf{V}_k \mathbf{\Lambda}_k\inv \big)_{j_1 i_1} \big( \mathbf{V}_k \mathbf{\Lambda}_k\inv\big)_{j_1 i_2} \bigg( \tilde \Sigma^{(m)}_k - \Sigma^{(m)}_k \bigg)_{j_1 j_1} \\
    &= \sum_{j_1}\big( \mathbf{V}_k \mathbf{\Lambda}_k\inv \big)_{j_1 i_1} \big( \mathbf{V}_k \mathbf{\Lambda}_k\inv\big)_{j_1 i_2} \bigg(\big( \mathbf{Z}_k\big)_{mj_1}^2 - \E \big( \mathbf{Z}_k\big)_{mj_1}^2 \bigg) \\
    &\coloneqq \eta_{i_1i_2}.
\end{align*}
This is a sum of $p_{-k}$ independent subexponential random variables, so we will apply Bernstein's inequality (Theorem 2.8.1 of \citet{vershynin_high-dimensional_2018}).  We note that
\begin{align*}
    \sum_{j_1} \bigg\| \big( \mathbf{V}_k \mathbf{\Lambda}_k\inv &\big)_{j_1 i_1} \big( \mathbf{V}_k \mathbf{\Lambda}_k\inv\big)_{j_1 i_2} \bigg(\big( \mathbf{Z}_k\big)_{mj_1}^2 - \E \big( \mathbf{Z}_k\big)_{mj_1}^2 \bigg) \bigg\|_{\psi_1}^2 \\
    &\leq \sum_{j_1} \big( \mathbf{V}_k \mathbf{\Lambda}_k\inv \big)_{j_1 i_1}^2 \big( \mathbf{V}_k \mathbf{\Lambda}_k\inv\big)_{j_1 i_2}^2 \|  \big( \mathbf{Z}_k\big)_{mj_1}^2 - \E \big( \mathbf{Z}_k\big)_{mj_1}^2 \|_{\psi_1}^2 \\
    &\lesssim\sigma^4 \max_{j_1} \| e_{j_1}\t\mathbf{V}_k \mathbf{\Lambda}_k\inv \|^2 \sum_{j_1} \big( \mathbf{V}_k \mathbf{\Lambda}_k\inv\big)_{j_1 i_2}^2  \\
    &\lesssim \frac{\sigma^4}{\lambda^4} \| \mathbf{V}_k \|_{2,\infty}^2 \sum_{j_1} (\mathbf{V}_k)_{j_1 i_2}^2 \\
    &\leq C\frac{\sigma^4}{\lambda^4} \mu_0^2 \frac{r}{p^2}
\end{align*}
In addition,
\begin{align*}
    \max_{j_1} \bigg\| \big( \mathbf{V}_k \mathbf{\Lambda}_k\inv \big)_{j_1 i_1} \big( \mathbf{V}_k \mathbf{\Lambda}_k\inv\big)_{j_1 i_2} \bigg(\big( \mathbf{Z}_k\big)_{mj_1}^2 - \E \big( \mathbf{Z}_k\big)_{mj_1}^2 \bigg) \bigg\|_{\psi_1} &\leq C \frac{\sigma^2}{\lambda^2} \mu_0^2 \frac{r}{p^2}.
\end{align*}
By Bernstein's inequality, it holds that
\begin{align*}
    \p\bigg\{ |\eta_{i_1i_2} | \geq t \bigg\} &\leq 2 \exp\bigg( -c \min\bigg\{ \frac{t^2}{ \frac{\sigma^4}{\lambda^4} \mu_0^2 \frac{r}{p^2}}, \frac{t}{\frac{\sigma^2}{\lambda^2} \mu_0^2 \frac{r}{p^2}} \bigg\} \bigg).
\end{align*}
Let $t = C\frac{\sigma^2}{\lambda^2} \mu_0 \frac{\sqrt{r}}{p} \sqrt{\log(p)}$.  Then 
\begin{align*}
      \p\bigg\{ |\eta_{i_1i_2} | \geq C \frac{\sigma^2}{\lambda^2} \mu_0 \frac{\sqrt{r}}{p} \bigg\} &\leq 2 \exp\bigg( -c \min\bigg\{ C^2 \log(p), \frac{C p \sqrt{\log(p)}}{\mu_0} \bigg\} \bigg) \\
      &\lesssim O(p^{-8}).
\end{align*}
Taking a union bound over all $r^2$ entries and noting that $r^2 \lesssim \sqrt{p}$, it holds with probability at least $1 - O(p^{-7})$ that
\begin{align}
    \alpha_3 &\lesssim \mu_0 \frac{r^{3/2}}{p} \frac{\sigma^2}{\lambda^2} \sqrt{\log(p)}. \label{alpha3bound}
\end{align}
\item \textbf{Step 5: Putting It All Together:} Combining \eqref{alpha1bound}, \eqref{alpha2bound}, and \eqref{alpha3bound}, we have that
\begin{align*}
   & \left\|\hat{\mathbf{W}}_{k} \hat{\Gamma}^{(m)} \hat{\mathbf{W}}_{k}^{\top}-\mathbf{\Gamma}^{(m)}_k\right\| \\&\lesssim \frac{\sigma^{3} r \sqrt{p} \log (p)+\sigma^{3} \kappa^{2} \sqrt{r p} \log ^{3 / 2}(p)}{\lambda^{3}} +\frac{\sigma^{2} \kappa \mu_{0} \sqrt{r} \log (p)}{\lambda^{2} p} +  \mu_0 \frac{r^{3/2}}{p} \frac{\sigma^2}{\lambda^2} \sqrt{\log(p)} \\
    &\leq \frac{\sigma^2}{\lambda^2} \bigg( \frac{\sigma r \sqrt{p}\log(p) + \sigma \kappa^2 \sqrt{rp} \log^{3/2}(p)}{\lambda} + \frac{\kappa \mu_0 \sqrt{r} \log(p)}{p} + \frac{\mu_0 r^{3/2} \sqrt{\log(p)}}{p} \bigg)
\end{align*}
with probability at least $1 - O(p^{-6})$. We now complete the proof of the lemma.  By Theorem 6.2 of \citet{higham_functions_2008}, it holds that
\begin{align*}
    \| &\mathbf{\hat W}_k\big(\mathbf{\hat \Gamma}^{(m)}_k \big)^{1/2}\mathbf{\hat W}_k\t -\big(\mathbf{\Gamma}^{(m)}_k \big)^{1/2} \| \\
    &\leq \frac{1}{ \lambda_{\min}^{1/2}(\mathbf{\hat \Gamma}^{(m)}_k) + \lambda_{\min}^{1/2}( \mathbf{\Gamma}^{(m)}_k)} \| \mathbf{\hat W}_k \mathbf{\hat \Gamma}^{(m)}_k \mathbf{\hat W}_k\t - \mathbf{\Gamma}^{(m)}_k \| \\
    &\lesssim \frac{1}{\lambda_{\min}^{1/2}( \mathbf{\Gamma}^{(m)}_k)}  \frac{\sigma^2}{\lambda^2} \bigg( \frac{\sigma r \sqrt{p}\log(p) + \sigma \kappa^2 \sqrt{rp} \log^{3/2}(p)}{\lambda} + \frac{\kappa \mu_0 \sqrt{r} \log(p)}{p} + \frac{\mu_0 r^{3/2} \sqrt{\log(p)}}{p} \bigg) \\
    &\lesssim \frac{\sigma}{\lambda} \bigg( \frac{\sigma r \sqrt{p}\log(p) + \sigma \kappa^2 \sqrt{rp} \log^{3/2}(p)}{\lambda} + \frac{\kappa \mu_0 \sqrt{r} \log(p)}{p} + \frac{\mu_0 r^{3/2} \sqrt{\log(p)}}{p} \bigg),
\end{align*}
where we have used the observation that
\begin{align*}
     \lambda_{\min}( \mathbf{\Gamma}^{(m)}_k ) \geq \frac{\sigma_{\min}^2}{\lambda},
 \end{align*}
 (see the proof of \cref{thm:eigenvectornormality2}), together with the assumption $\sigma /\sigma_{\min} = O(1)$.  This bound holds with probability at least $1 - O(p^{-6})$, which completes the proof. 
 \end{itemize}
\end{proof}

\section{Proofs of Applications and Some More General Theorems} \label{sec:applicationproofs}
In this section we prove the results in \cref{sec:consequences}.  In \cref{sec:testingproof} we prove \cref{cor:testing}.  Next, in \cref{sec:generalizationapplication} we provide slightly more general statements of \cref{thm:simultaneousinference_v1} and \cref{thm:entrytesting_v1}; i.e., \cref{thm:simultaneousinference} and \cref{thm:entrytesting}.  We then prove \cref{thm:simultaneousinference} and \cref{thm:entrytesting} the subsequent two subsections.  The proofs of these results rely heavily on the previous proofs.

\subsection{Proof of \cref{cor:testing}} \label{sec:testingproof}
\begin{proof}[Proof of \cref{cor:testing}]
%First, will prove the results with $n = p_1 = p_2$ and $L = p_3$ and $K = r_1 = r_2$ and $r = r_3$.  
Recall that throughout $r_k = O(1)$.  We will also assume that $\mu_0,\kappa = O(1)$, and that \begin{align}
    \lambda/\sigma \gg p^{3/4} \sqrt{\log(p)}.\label{snrconditiontoverify} %\asymp p^{3/4} \sqrt{\log(p)}, 
\end{align}
We will verify these conditions the end of the proof.  Without loss of generality we prove the result for $k = 1$.  For convenience we will suppress the dependence of $\mathbf{\Pi}_1$ and $\U_1$ on the index $k$.  

By \cref{lem:gammacloseness}, it holds that with probability at least $1 - O(p^{-6})$ that
\begin{align*}
    &\left\|\hat{\mathbf{W}}_{1} \big(\mathbf{\hat{\Gamma}}^{(m)}_1\big)^{1/2} \hat{\mathbf{W}}_{1}^{\top}-\big(\mathbf{\Gamma}^{(m)}_1\big)^{1/2}\right\| \\
    &\lesssim \frac{\sigma}{\lambda}\left(\frac{\sigma r \sqrt{p} \log (p)+\sigma \kappa^{2} \sqrt{r p} \log ^{3 / 2}(p)}{\lambda}+\frac{\kappa \mu_{0} \sqrt{r} \log (p)}{p}+\frac{\mu_{0} r^{3 / 2} \sqrt{\log (p)}}{p}\right) \\
    &\asymp \frac{\sigma}{\lambda} \bigg( \frac{\sqrt{p} \log^{3/2}(p)}{\lambda/\sigma} \bigg) \\
    &\ll \frac{\sigma_{\min}}{\lambda}
\end{align*}
since $\sigma/\sigma_{\min} = O(1)$. This implies that
\begin{align*}
  \frac{  \left\|\hat{\mathbf{W}}_{1} \big(\hat{\Gamma}^{(m)}\big)^{1/2} \hat{\mathbf{W}}_{1}^{\top}-\big(\mathbf{\Gamma}^{(m)}_k\big)^{1/2}\right\|}{\lambda_{\min}^{1/2}(\mathbf{\Gamma}^{(m)}_k)} \to 0 
\end{align*}
almost surely, where we implicitly use the fact that $\lambda_{\min}(\mathbf{\Gamma}^{(m)}_k) \gtrsim \frac{\sigma_{\min}^2}{\lambda^2}$.  Consequently, we have the almost sure convergence
\begin{align}
    \| \hat{\mathbf{W}}_{1}\big(\mathbf{\hat \Gamma}^{(m)}_k\big)^{1/2}\hat{\mathbf{W}}_{1}\t \big(\mathbf{\Gamma}^{(m)}_k\big)^{-1/2} - \mathbf{I}_{r_1}\| \to 0. \label{gammaconv1}
\end{align}
Next, by \cref{thm:eigenvectornormality2}, it holds that
\begin{align}
   \mathbf{\hat W}_{1}\t \bigg( \Gamma^{(i)} + \Gamma^{(j)} \bigg)^{-1/2}  \mathbf{\hat W}_{1} \big( \uhat_{i\cdot} - \uhat_{j\cdot} -  \big(\U \mathbf{\hat W}_{1}\big)_{i\cdot} + \big( \U  \mathbf{\hat W}_{1}\big)_{j\cdot} \big) \to N(0, \mathbf{I}_{r_1}), \label{normalconv2}
\end{align}
where we note that the proof of \cref{thm:eigenvectornormality2} reveals that rows $i$ and $j$ of $\uhat$ are asymptotically independent.  
Therefore, \eqref{normalconv2}, \eqref{gammaconv1}, and the continuous mapping theorem imply
\begin{align*}
    \big(\hat \Gamma^{(i)} + \hat \Gamma^{(j)} \big)^{-1/2}\big( \uhat_{i\cdot} - \uhat_{j\cdot} -  \big(\U\mathbf{\hat W}_{1}\big)_{i\cdot} + \big( \U \mathbf{\hat W}_{1}\big)_{j\cdot} \big) \to N(0, \mathbf{I}_{r_1} ).
\end{align*}
Under $H_0$, by Proposition 2 of \citet{agterberg_estimating_2022} it holds that $(\U \mathbf{\hat W}_{1})_{i\cdot} = (\U \mathbf{\hat W}_{1})_{j\cdot}$.  Therefore, under the null it holds that
\begin{align*}
    \big(\hat \Gamma^{(i)} + \hat \Gamma^{(j)} \big)^{-1/2}\big( \uhat_{i\cdot} - \uhat_{j\cdot} \big) \to N(0, \mathbf{I}_{r_1}),
\end{align*}
and hence that
\begin{align*}
    \hat T_{ij} &= \|  \big(\hat \Gamma^{(i)} + \hat \Gamma^{(j)} \big)^{-1/2}\big( \uhat_{i\cdot} - \uhat_{j\cdot} \big) \|^2 \to \chi^2_{r_1}.
\end{align*}
Next, under any alternative, Slutsky's Theorem and \cref{gammaconv1} imply
\begin{align*}
    \big(\hat \Gamma^{(i)} &+ \hat \Gamma^{(j)} \big)^{-1/2}\big( \uhat_{i\cdot} - \uhat_{j\cdot}-  \big(\hat \Gamma^{(i)} + \hat \Gamma^{(j)} \big)^{1/2} \mathbf{\hat W}_{1}\t \bigg( \Gamma^{(i)} + \Gamma^{(j)} \bigg)^{-1/2}  \mathbf{\hat W}_{1}  \big(\U\mathbf{\hat W}_{1}\big)_{i\cdot} + \big( \U \mathbf{\hat W}_{1}\big)_{j\cdot} \big)\\
    &\qquad \qquad \to N(0,\mathbf{I}_{r_1}),
\end{align*}
which further implies that
\begin{align*}
     \big(\hat \Gamma^{(i)} &+ \hat \Gamma^{(j)} \big)^{-1/2}\big( \uhat_{i\cdot} - \uhat_{j\cdot} \big) -  \mathbf{\hat W}_{1}\t \bigg( \Gamma^{(i)} + \Gamma^{(j)} \bigg)^{-1/2}  \mathbf{\hat W}_{1}  \big(\U\mathbf{W}_{1}\big)_{i\cdot} + \big( \U \mathbf{W}_{1}\big)_{j\cdot} \big) \\
     &\to N(0, \mathbf{I}_{r_1}).
\end{align*}
Therefore, the distribution of the random variable
\begin{align*}
   \big(\hat \Gamma^{(i)} + \hat \Gamma^{(j)} \big)^{-1/2}\big( \uhat_{i\cdot} - \uhat_{j\cdot} \big)
\end{align*}
is asymptotically equivalent to a Gaussian random variable with mean
\begin{align*}
    \mathbf{\hat W}_{1}\t \bigg( \Gamma^{(i)} + \Gamma^{(j)} \bigg)^{-1/2}  \mathbf{\hat W}_{1}  \big(\U\mathbf{\hat W}_{1}\big)_{i\cdot} + \big( \U \mathbf{\hat W}_{1}\big)_{j\cdot} \big).
\end{align*}
Therefore, let $\eta_{ij}$ denote this Gaussian random variable. By the Delta method, we have that
\begin{align*}
    \hat T_{ij} &= \| \eta_{ij}  \|^2,
\end{align*}
where the equality is in distribution.  It is straightforward to see that $\eta_{ij}^2$ is a noncentral $\chi^2$ distributed random variable with noncentrality parameter
\begin{align*}
    \|   \mathbf{\hat W}_{1}\t &\bigg( \Gamma^{(i)} + \Gamma^{(j)} \bigg)^{-1/2}  \mathbf{\hat W}_{1}  \big(\U\mathbf{\hat W}_{1}\big)_{i\cdot} + \big( \U \mathbf{\hat W}_{1}\big)_{j\cdot} \big) \|^2\\
    &= \big( \U_{i\cdot} - \U_{j\cdot} \big)\t \bigg( \Gamma^{(i)} + \Gamma^{(j)} \bigg)^{-1}  \big( \U_{i\cdot} - \U_{j\cdot} \big).
\end{align*}
When this term converges to $\gamma < \infty$, we immediately obtain the result under the alternative.  Therefore, it suffices to show that
\begin{align}
     \big( \U_{i\cdot} - \U_{j\cdot} \big)\t \bigg( \Gamma^{(i)} + \Gamma^{(j)} \bigg)^{-1}  \big( \U_{i\cdot} - \U_{j\cdot} \big) \to \infty \label{snronditiontoverify2}
\end{align}
whenever
\begin{align}
    \frac{\lambda_{\min}(\mathcal{S})}{\sigma} p \| \mathbf{\Pi}_{i\cdot} - \mathbf{\Pi}_{j\cdot} \| \to \infty. \label{snrsnr}
\end{align}
We now verify this condition, as well the condition \eqref{snrconditiontoverify}. To examine the SNR condition, we note that Lemma 1 of \citet{agterberg_estimating_2022} shows that
\begin{align*}
    \lambda \asymp \lambda_{\min}(\mathcal{S}) p^{3/2}
\end{align*}
as long as  $r$ is bounded.  Therefore, the signal-strength condition 
\begin{align*}
    \frac{\lambda_{\min}(\mathcal{S})^2}{\sigma^2} \gg \frac{1}{p^{3/2}} \log(p)
\end{align*}
is equivalent to the condition
\begin{align*}
    \frac{\lambda^2}{\sigma^2} \gg p^{3/2} \log(p),
\end{align*}
which is guaranteed whenever
\begin{align*}
    \lambda/\sigma \gg p^{3/4} \sqrt{\log(p)},
\end{align*}
which is what is needed for \cref{thm:civalidity2}.  In addition, we note that the incoherence condition $\mu_0 = O(1)$ is guaranteed by Lemma 1 of \citet{agterberg_estimating_2022}, as well as the fact that $\kappa = O(1)$.  This verifies \eqref{snrconditiontoverify}.

 For the condition \eqref{snronditiontoverify2}, we note that
 \begin{align*}
    \big( \U_{i\cdot} - \U_{j\cdot} \big)\t \bigg( \Gamma^{(i)} + \Gamma^{(j)} \bigg)^{-1}  \big( \U_{i\cdot} - \U_{j\cdot} \big) &\gtrsim \| \U_{i\cdot} - \U_{j\cdot} \|^2 \frac{\lambda^2}{\sigma^2}
 \end{align*}
 Therefore, this term diverges as long as 
 \begin{align}
      \| \U_{i\cdot} - \U_{j\cdot} \| \gg \frac{\sigma}{\lambda} \asymp \frac{\sigma}{\lambda_{\min}(\mathcal{S}) p^{3/2}} \label{snrconditiontoverify3}
 \end{align}
% We follow the proof of Lemma 7 of \citet{fan_simple_2022}.  First, observe that $\U_1$ can be obtained via the eigendecomposition of the matrix
% \begin{align*}
% \mathbb{Q} &=   \mathbf{\Pi} \bigg( \mathbf{M}^{(1)}     \mathbf{\Pi}\t     \mathbf{\Pi} \big(  \mathbf{M}^{(1)}  \big)\t \bigg) \mathbf{\Pi}\t.
% \end{align*}
% Since $\mathbb{Q} = \U \mathbf{\Lambda}_1 \U\t,$ it holds that
% \begin{align*}
%     \U = \Pi \mathbb{M} \Pi\t \U \mathbf{\Lambda}_1\inv,
% \end{align*}
% where
% \begin{align*}
%     \mathbb{M} &\coloneqq\mathbf{M}^{(1)}     \mathbf{\Pi}\t     \mathbf{\Pi}\big(  \mathbf{M}^{(1)}  \big)\t.
% \end{align*}
% since $\mathbb{M}$ is full-rank, we see that there must exist an invertible matrix $\mathbf{R}$ such that
% \begin{align*}
%     \U = \mathbf{\Pi} \mathbf{R}.
% \end{align*}
% Moreover, since $\U\t \U = \mathbf{I}_{r_1}$, it holds that $\mathbf{RR\t \Pi\t \Pi R R\t} = \mathbf{RR\t}$ and hence
% \begin{align*}
%     \mathbf{RR}\t = (\mathbf{\Pi\t\Pi})\inv.
% \end{align*}
% Therefore, since $\mathbf{R}$ is a $r_1\times r_1$ matrix and $\lambda_{\min}(\mathbf{\Pi\t \Pi}) \gtrsim p$, it holds that the entries of $\mathbf{R}$ are of order $\frac{1}{\sqrt{p}}$.  Consequently,
By Proposition 2 of \citet{agterberg_estimating_2022}, one has that $\U_k = \mathbf{\Pi}_k \U\pure$ and that $\U\pure$ has entries of order $\frac{1}{\sqrt{p}}$.  Hence,
\begin{align*}
    \| \U_{i\cdot} - \U_{j\cdot} \| &= \| \big( \mathbf{\Pi }\U\pure \big)_{i\cdot} - \big( \mathbf{\Pi } \U\pure \big)_{j\cdot} \| \\
    &\asymp \frac{1}{\sqrt{p}} \| \mathbf{\Pi}_{i\cdot} - \mathbf{\Pi}_{j\cdot} \|.
\end{align*}
Consequently, 
\begin{align*}
    \frac{\lambda_{\min}(\mathcal{S})}{\sigma} p  \| \mathbf{\Pi}_{i\cdot} - \mathbf{\Pi}_{j\cdot} \| &\asymp   \frac{\lambda_{\min}(\mathcal{S})}{\sigma}p^{3/2}  \| \U_{i\cdot} - \U_{j\cdot} \|
\end{align*}
and therefore \eqref{snrconditiontoverify3} holds whenever \eqref{snrsnr} holds, which completes the proof.
\end{proof}
\subsection{Generalizations of Theorems \ref{thm:simultaneousinference_v1} and \ref{thm:entrytesting_v1}}
\label{sec:generalizationapplication}
In this section we state generalizations of the applications in \cref{sec:introapplications}.  The following result generalizes \cref{thm:simultaneousinference_v1} to the setting where $\mu_0$ and $\kappa$ are permitted to grow.

\begin{theorem}[Generalization of \cref{thm:simultaneousinference_v1}] \label{thm:simultaneousinference}
Instate the conditions in \cref{thm:eigenvectornormality}, and suppose that
\begin{align*}
    \mu_0^2 \kappa^2 r^{3/2} \sqrt{\log(p)} \lesssim p^{1/4}.
\end{align*}
Let $J$ be a given index set with $|J| = o(p^{1/6})$. .  Define the $|J|\times |J|$ matrix $S_J$ via
\begin{align*}
    (S_J)_{\{i,j,k\},\{i',j',k'\}} &\coloneqq  \mathbb{I}_{\{i= i'\}} e_{(j-1)p_3 + k}\t  \mathbf{V}_1 \mathbf{V}_1\t \Sigma_1^{(i)} \mathbf{ V}_1 \mathbf{ V}_1\t e_{(j'-1)p_3 + k'} \\
    &\quad + \mathbb{I}_{\{j= j'\}} e_{(k-1)p_1 + i}\t  \mathbf{ V}_2 \mathbf{ V}_2\t \Sigma_2^{(j)} \mathbf{ V}_2 \mathbf{ V}_2\t e_{(k'-1)p_3 + i'}  \\
    &\quad + \mathbb{I}_{\{k= k'\}} e_{(i-1)p_2 + j}\t  \mathbf{ V}_3 \mathbf{ V}_3\t \Sigma_3^{(k)} \mathbf{ V}_3 \mathbf{ V}_3\t e_{(i'-1)p_2 + '},
\end{align*}
where $\Sigma_k^{(m)}$ is as in \cref{thm:asymptoticnormalityentries}.  Suppose $S_J$ is invertible, and let $s^2_{\min}$ denote its smallest eigenvalue.  Suppose that 
\begin{align*}
    s_{\min}/\sigma &\gg \max\bigg\{ |J|^{3/2}  \frac{\kappa^2 \mu_0^3 r^{3/2} \sqrt{\log(p)}}{p^{3/2}}, |J|^{3/2} \frac{ \mu_0^4 \kappa^3 r^2 \log(p)}{(\lambda/\sigma) \sqrt{p}}, |J| \frac{ \mu_0^{5/2} r^{3/2} \kappa \log^{3/4}(p)}{(\lambda/\sigma)^{1/2} p^{3/4}}, \\
      &\qquad \qquad \qquad  |J|^{1/6} \frac{\kappa \mu_0^{3/2} r^{3/2} \sqrt{\log(p)}}{p^{4/3}}, |J|^{1/6} \frac{\kappa^{5/3}\mu_0^{3} r^{7/6} \log^{5/6}(p)}{(\lambda/\sigma)^{1/3} p^{5/6}}\bigg\}.
\end{align*}
Let $\mathrm{C.I.}_J^{\alpha}(\mathcal{\hat T})$ denote the output of \cref{al:simultaneousci}.
%Let $\hat \Sigma^{(i)}$ denote the $p_2 \times p_3$ diagonal matrix with entries consisting of the squared values of $e_i\t \mathbf{\hat Z}_1$, and define $\hat \Sigma^{(j)}$ and $\hat \Sigma^{(k)}$ similarly.  Define the matrix $\hat S_J$ via
%\begin{align*}
%    \big(\hat S_{J} \big)_{\{i,j,k\},\{i',j',k'\}} &\coloneqq \mathbb{I}_{\{i= i'\}} e_{(j-1)p_3 + k}\t  \mathbf{\hat V}_1 \mathbf{\hat V}_1\t \hat \Sigma^{(i)} \mathbf{\hat V}_1 \mathbf{\hat V}_1\t e_{(j'-1)p_3 + k'} \\
%    &\quad + \mathbb{I}_{\{j= j'\}} e_{(k-1)p_1 + i}\t  \mathbf{\hat V}_2 \mathbf{\hat V}_2\t\hat \Sigma^{(j)} \mathbf{\hat V}_2 \mathbf{\hat V}_2\t e_{(k'-1)p_3 + i'}  \\
%    &\quad + \mathbb{I}_{\{k= k'\}} e_{(i-1)p_2 + j}\t  \mathbf{\hat V}_3 \mathbf{\hat V}_3\t \hat \Sigma^{(k)} \mathbf{\hat V}_3 \mathbf{\hat V}_3\t e_{(i'-1)p_2 +j'}.
%\end{align*}
 Then it holds that
\begin{align*}
 \p\bigg\{ \mathrm{Vec}(\mathcal{ T}_J) \in \mathrm{C.I.}_{\alpha}(\mathcal{\hat T}_J) \bigg\}= 1- \alpha -  o(1).
\end{align*}
\end{theorem}

\noindent The following result generalizes \cref{thm:entrytesting_v1}.  

\begin{theorem}[Generalization of \cref{thm:entrytesting_v1}]\label{thm:entrytesting}
Instate the conditions of \cref{thm:eigenvectornormality}, and suppose that $\kappa ^2 \mu_0^2 r^{3/2} \sqrt{\log(p)} \lesssim p^{1/4}$.  
% Define %the test statistic
% %\begin{align*}
% %    \hat T_{\{ijk\},\{i'j'k'\}} &= \frac{\mathcal{\hat T}_{ijk} - \mathcal{\hat T}_{i'j'k'}}{\hat s_{\{ijk\}\{i'j'k'\}}},
% %\end{align*}
% %where
% \begin{align*}
%     \hat s_{\{ijk\}\{i'j'k'\}}^2 &= \hat s^2_{ijk} + \hat s^2_{i'j'k'}  \\
%     &\quad - \mathbb{I}_{\{i= i'\}} e_{(j-1)p_3 + k}\t  \mathbf{\hat V}_1 \mathbf{\hat V}_1\t \hat \Sigma^{(i)} \mathbf{\hat V}_1 \mathbf{\hat V}_1\t e_{(j'-1)p_3 + k'} \\
%     &\quad - \mathbb{I}_{\{j= j'\}} e_{(k-1)p_1 + i}\t  \mathbf{\hat V}_2 \mathbf{\hat V}_2\t\hat \Sigma^{(j)} \mathbf{\hat V}_2 \mathbf{\hat V}_2\t e_{(k'-1)p_3 + i'}  \\
%     &\quad - \mathbb{I}_{\{k= k'\}} e_{(i-1)p_2 + j}\t  \mathbf{\hat V}_3 \mathbf{\hat V}_3\t \hat \Sigma^{(k)} \mathbf{\hat V}_3 \mathbf{\hat V}_3\t e_{(i'-1)p_2 +j'},
% \end{align*}
% with $\hat s^2_{ijk}$ defined as in \cref{thm:asymptoticnormalityentries_v1}. 
Suppose that  \begin{align*}
 \min\bigg\{ \|& e_{(j-1)p_3 + k }\t \mathbf{V}_1\|^2, \| e_{(j'-1)p_3 + k'}\t \mathbf{V}_1\|^2, \| e_{(k-1)p_1 + i}\t \mathbf{V}_2 \|^2,\\
 &\| e_{(k'-1)p_1 + i}\t \mathbf{V}_2 \|^2, \| e_{(i-1)p_2 + j}\t \mathbf{V}_3 \|^2, \| e_{(i'-1)p_2 + j'}\t \mathbf{V}_3 \|^2 \bigg\} \\
    &\gg\max\bigg\{ \frac{\kappa \mu_0^2 r^{3/2} \log(p)}{p^3} , \frac{\mu_0^5 r^3 \kappa^2 \log^{3/2}(p)}{(\lambda/\sigma)p^{3/2}}, \frac{\kappa^4 \mu_0^6 r^3 \log(p)}{p^3}, \frac{\mu_0^8 \kappa^6 r^3 \log^2(p)}{(\lambda/\sigma)^2 p} \bigg\} .
    \end{align*} 
     Let $\mathrm{C.I.}^{\alpha}_{\{ijk\},\{i'j'k'\}}( \mathcal{\hat T}) $ denote the output of \cref{al:entrytesting}. 
Then it holds that
\begin{align*}
    \frac{\mathcal{\hat T}_{ijk} - \mathcal{\hat T}_{i'j'k'} - \big( \mathcal{T}_{ijk} - \mathcal{T}_{i'j'k'}\big)}{\hat s_{\{ijk\}\{i'j'k'\}}} \to N(0,1).
\end{align*}
\end{theorem}

\subsection{Proof of \cref{thm:simultaneousinference}}

\begin{proof}[Proof of \cref{thm:simultaneousinference}]
The proof is similar to the proof of \cref{thm:civalidity}. First, by the proof of \cref{thm:asymptoticnormalityentries}, by applying the main expansion \eqref{mainentrywiseexpansion} to each entry separately, we have that with probability at least $1 - O(|J|p^{-9})$, it holds that
\begin{align*}
    \mathcal{T}_{J} &= \xi_{J} + O\bigg( |J|\frac{\sigma \kappa^2 \mu_0^3 r^{3/2} \sqrt{\log(p)}}{p^{3/2}} + |J| \frac{\sigma^2 \mu_0^4 \kappa^3 r^2 \log(p)}{\lambda \sqrt{p}} \bigg),
\end{align*}
where we set $\xi_J$ as the random variable with entries $\xi_{ijk}$ from \cref{lem:xiijkgaussian}.  We will show that $\xi_J$ is asymptotically Gaussian with covariance $S_J$, and we will demonstrate that $\hat S_J$ approximates $S_J$, which will yield the result.  
\begin{itemize}
\item \textbf{Step 1: Limiting Covariance Structure and First-Order Approximation}: 
Recall we define $s^2_{\min}$ via
\begin{align*}
    s^2_{\min} \coloneqq \lambda_{\min}(S_J),
\end{align*}
If it holds that
\begin{align}
    s_{\min} &\gg |J| \sigma\max\bigg\{ \frac{ \kappa^2 \mu_0^3 r^{3/2} \sqrt{\log(p)}}{p^{3/2}}, \frac{\sigma \mu_0^4 \kappa^3 r^2 \log(p)}{\lambda \sqrt{p}}\bigg\}, \label{smin0}
\end{align}
then 
\begin{align*}
    S_J^{-1/2} \mathrm{Vec}(\mathcal{T}_J) &= S_{J}^{-1/2} \xi_J + o(1).
\end{align*}
with probability at least $1 - O(|J|p^{-9})$.  It now remains to show that $\mathrm{Cov}(\xi_J) \approx S_J$.  By the proof of \cref{lem:xiijkgaussian}, it holds that
\begin{align*}
    \mathrm{Var}(\xi_{ijk}) &= s^2_{ijk} \bigg( 1 + O\big( \frac{\mu_0^2 r}{p} \big) \bigg).
\end{align*}
Therefore, it suffices to consider the covariance terms.  Observe that
\begin{align*}
    &\mathrm{Cov}\bigg( \xi_{ijk} \xi_{i'j'k'}\bigg) \\&= \mathbb{E}\bigg( {e}_{i}^{\top} \mathbf{Z}_{1} \mathbf{V}_{1} \mathbf{V}_{1}^{\top} e_{(j-1) p_{3}+k}+e_{j}^{\top} \mathbf{Z}_{2} \mathbf{V}_{2} \mathbf{V}_{2}^{\top} e_{(k-1) p_{1}+i}+e_{k}^{\top} \mathbf{Z}_{3} \mathbf{V}_{3} \mathbf{V}_{3}^{\top} e_{(i-1) p_{2}+j} \bigg)\\
    &\times \bigg( {e}_{i'}^{\top} \mathbf{Z}_{1} \mathbf{V}_{1} \mathbf{V}_{1}^{\top} e_{(j'-1) p_{3}+k'}+e_{j'}^{\top} \mathbf{Z}_{2} \mathbf{V}_{2} \mathbf{V}_{2}^{\top} e_{(k'-1) p_{1}+i'}+e_{k'}^{\top} \mathbf{Z}_{3} \mathbf{V}_{3} \mathbf{V}_{3}^{\top} e_{(i'-1) p_{2}+j'} \bigg).
\end{align*}
By a similar argument to Step 1 of the proof of \cref{lem:xiijkgaussian}, it holds that the cross term satisfies
\begin{align*}
    \mathbb{E}&\bigg( e_i\t \mathbf{Z}_1 \mathbf{V}_1 \mathbf{V}_1\t e_{(j-1)p_3 + k}\bigg) \bigg( e_{j'}\t \mathbf{Z}_2 \mathbf{V}_2 \mathbf{V}_2\t e_{(k'-1)p_1 + i'} \bigg) \\
    &= O\bigg( \sigma^2 \mu_0^2 \frac{r}{p} \bigg( \| e_{(j-1)p_3 + k} \mathbf{V}_1 \|^2 + \| e_{(k'-1)p_1 + i'} \mathbf{V}_2 \|^2 \bigg) \bigg) \\
    &= O\bigg(  \sigma^2 \mu_0^4 \frac{r^2}{p^2}  \bigg).
\end{align*}
The other cross terms can be handled similarly.  Hence, the only remaining terms are those such that $i = i'$, $j = j'$, or $k = k'$.  If $i= i'$, we have
\begin{align*}
    \mathbb{E}\bigg( &e_i\t \mathbf{Z}_1 \mathbf{V}_1\t e_{(j-1)p_3 + k} \bigg)\bigg( e_i\t \mathbf{Z}_1 \mathbf{V}_1\t e_{(j'-1)p_3 + k'} \bigg) \mathbb{I}_{\{i=i'\}}\\
    &=  e_{(j-1)p_2 + k}\t \mathbf{V}_1 \mathbf{V}_1\t \big( \Sigma^{(i)} \big)  \mathbf{V}_1 \mathbf{V}_1\t e_{(j'-1)p_3 + k'},
\end{align*}
with similar values if $j = j'$ or $k = k'$; in particular, these are the entries of $S_J$ by definition.  Hence
%so we see that
% \begin{align*}
%      \mathrm{Cov}(\xi_{ijk} \xi_{i'j'k'}) &=  \mathbb{I}_{\{i= i'\}} e_{(j-1)p_3 + k}\t  \mathbf{V}_1 \mathbf{V}_1\t \Sigma^{(i)} \mathbf{\hat V}_1 \mathbf{\hat V}_1\t e_{(j'-1)p_3 + k'} \\
%     &\quad + \mathbb{I}_{\{j= j'\}} e_{(k-1)p_1 + i}\t  \mathbf{\hat V}_2 \mathbf{\hat V}_2\t \Sigma^{(j)} \mathbf{\hat V}_2 \mathbf{\hat V}_2\t e_{(k'-1)p_3 + i'}  \\
%     &\quad + \mathbb{I}_{\{k= k'\}} e_{(i-1)p_2 + j}\t  \mathbf{\hat V}_3 \mathbf{\hat V}_3\t \Sigma^{(k)} \mathbf{\hat V}_3 \mathbf{\hat V}_3\t e_{(i'-1)p_2 + '} \\
%     &\qqyad + 
% \end{align*}
% Consequently, by the definition of $S_J$ it holds that
\begin{align}
    \mathrm{Cov}(\xi_J) &= S_J + O\bigg( \frac{|J|^2\mu_0^2 r}{p} s_{\min}^2 \bigg),\label{Sjapprox}
\end{align}
provided that $s_{\min}^2 \gg \sigma^2 \mu_0^2 r/p$.  
\item \textbf{Step 2: Gaussian Approximation}: We now study the Gaussian approximation of the vector $S_{J}^{-1/2} \mathrm{Vec}(\mathcal{T}_J - \mathcal{\hat T}_J)$.    We will apply Corollary 2.2 of \citet{shao_berryesseen_2022}.  Define, for some sufficiently large constant $C$,
\begin{align*}
    \Delta = \Delta^{(i)} &= C  \left(|J| \frac{\sigma \kappa^2 \mu_{0}^{3} r^{3 / 2} \sqrt{\log (p)}}{p^{3 / 2}}+|J| \frac{\sigma^{2} \mu_{0}^{4} \kappa^{3} r^{2} \log (p)}{\lambda \sqrt{p}}\right).
\end{align*}
Note that $\mathrm{Cov}(\xi_J)$ is invertible with smallest eigenvalue at least $s^2_{\min}(1- o(1))$ since $S_J$ is invertible provided that $|J|^2  \ll p /(\mu_0^2 r)$.  Consequently, we have that
\begin{align*}
    \lambda_{\min}\big( \mathrm{Cov}^{1/2}(\xi_J)\big) &\gtrsim s_{\min} \gg \Delta
\end{align*}
as long as
\begin{align}
    s_{\min} &\gg |J| \max \bigg\{   \frac{\sigma \kappa^2 \mu_{0}^{3} r^{3 / 2} \sqrt{\log (p)}}{p^{3 / 2}}, \frac{\sigma^{2} \mu_{0}^{4} \kappa^{3} r^{2} \log (p)}{\lambda \sqrt{p}}\bigg\}. \label{smin1}
\end{align}
Hence, it holds that
\begin{align*}
    \mathrm{Cov}(\xi_J)^{-1/2}
    \mathrm{Vec}(\mathcal{T}_J) &= \mathrm{Cov}(\xi_J)^{-1/2} \xi_J + \frac{\Delta}{s_{\min}}
\end{align*}
with probability at least $1 - O(|J|p^{-9})$.  Let this event be denoted $\mathcal{A}$.  By Corollary 2.2 of \citet{shao_berryesseen_2022}, it holds that
\begin{align*}
    \sup_{A \in \mathcal{A} } \bigg| &\p\bigg\{ \mathrm{Cov}^{-1/2} \mathrm{Vec}(\mathcal{T}_J - \mathcal{\hat T}_J) \in A \bigg\} - \p\bigg\{ Z \in A \bigg\} \bigg| \\
    &\lesssim |J|^{1/2}\gamma + \frac{\Delta}{s_{\min}} \E \| \mathrm{Cov}^{-1/2}\xi_J \|+ |J|p^{-9}, 
\end{align*}
where $\gamma$ is the sum of the third moments of the independent random variables in $\xi_J$ (see step 2 of the proof of \cref{lem:xiijkgaussian} for the explicit definition in terms of the indices of $\xi_J$).  A straightforward modification of the proof of \cref{lem:xiijkgaussian} shows that
\begin{align*}
    \gamma &\lesssim  \frac{|J|}{\sqrt{p\log(p)}}
\end{align*}
as long as $s_{\min} \gg \Delta$, which is true by assumption.  In addition, by subgaussianity, it holds that
\begin{align*}
    \mathbb{E} \| \mathrm{Cov}^{-1/2}( \xi_J) \xi_J \| &\lesssim | J |^{1/2}.
\end{align*}
Consequently, we have that
\begin{align*}
      \sup_{A \in \mathcal{A} } &\bigg| \p\bigg\{ \mathrm{Cov}(\xi_J)^{-1/2} \mathrm{Vec}(\mathcal{T}_J - \mathcal{\hat T}_J) \in A \bigg\} - \p\bigg\{ Z \in A \bigg\} \bigg| \\
      &\lesssim \frac{|J|^{3/2}}{\sqrt{p\log(p)}} + |J|p^{-9} + |J|^{1/2} \Delta \\
      &\lesssim \frac{|J|^{3/2}}{\sqrt{p \log(p)}} + \frac{ |J|^{3/2}}{s_{\min}} \left( \frac{\sigma \kappa^2 \mu_{0}^{3} r^{3 / 2} \sqrt{\log (p)}}{p^{3 / 2}}+ \frac{\sigma^{2} \mu_{0}^{4} \kappa^{3} r^{2} \log (p)}{\lambda \sqrt{p}}\right),
\end{align*}
which holds as long as
\begin{align}
    s_{\min} \gg |J| \frac{\sigma \kappa^2 \mu_{0}^{3} r^{3 / 2} \sqrt{\log (p)}}{p^{3 / 2}}+|J| \frac{\sigma^{2} \mu_{0}^{4} \kappa^{3} r^{2} \log (p)}{\lambda \sqrt{p}}. \label{smin2}
\end{align}
Therefore, for any $A \in \mathcal{A}$,
\begin{align*}
  \bigg| &\p\bigg\{ S_J^{-1/2} \mathrm{Vec}(\mathcal{T}_J - \mathcal{\hat T}_J) \in A \bigg\} - 
    \p \bigg\{ Z \in A \bigg\} \bigg| \\
    &\leq \bigg| \p\bigg\{ \mathrm{Cov}(\xi_J)^{-1/2}(\mathrm{Vec}(\mathcal{T}_J - \mathcal{\hat T}_J) \in \Cov(\xi_J)^{-1/2} S_J^{1/2} A \bigg\} - \p\bigg\{ Z' \in \Cov(\xi_J)^{-1/2} S_J^{1/2} A \bigg\} \bigg| \\
    &\quad + \bigg| \p\bigg\{ Z' \in \Cov(\xi_J)^{-1/2} S_J^{1/2} A \bigg\} - \p\bigg\{ Z \in A \bigg\} \bigg| \\
    &\lesssim \frac{|J|^{3/2}}{\sqrt{p\log(p)}} + \frac{ |J|^{3/2}}{s_{\min}}\bigg( \frac{\sigma \kappa^2 \mu_0^3 r^{3/2} \sqrt{\log(p)}}{p^{3/2}} + \frac{\sigma^2 \mu_0^4 \kappa^3 r^2 \log(p)}{\lambda \sqrt{p}} \bigg) \\
    &\quad + \bigg| \p\bigg\{ Z' \in \Cov(\xi_J)^{-1/2} S_J^{1/2} A \bigg\} - \p\bigg\{ Z \in A \bigg\} \bigg|
\end{align*}
where $Z'$ is an independent $|J|$-dimensional standard Gaussian random variable.  To bound the remaining term, by Theorem 1.3 of \citet{devroye_total_2022}, it holds that
\begin{align*}
    \sup_{A \in \mathcal{A}} \bigg| \p\bigg\{ S_J^{-1/2} \Cov(\xi_J)^{1/2}Z \in A \bigg\} - \p\bigg\{ Z' \in  A \bigg\} \bigg| &\leq d_{TV} (  S_J^{-1/2} \Cov(\xi_J)^{1/2}Z, Z' ) \\
    &\lesssim \| I - S_J\inv \Cov(\xi_J) \| \\
    &\lesssim |J|^2 \mu_0^2 \frac{r}{p},
\end{align*}
where the final inequality holds by \eqref{Sjapprox}.  Putting it all together, it holds that
\begin{align}
    \sup_{A \in \mathcal{A}} \bigg| &\p\bigg\{ S_J^{-1/2} \mathrm{Vec}(\mathcal{T}_J - \mathcal{\hat T}_J) \in A \bigg\} - \p\bigg\{ Z \in A \bigg\} \bigg| \nonumber \\&\lesssim \frac{|J|^{3/2}}{\sqrt{p\log(p)}} + \frac{|J|^2\mu_0^2 r}{p}  +\frac{ |J|^{3/2}}{s_{\min}}\bigg( \frac{\sigma \kappa^2 \mu_0^3 r^{3/2} \sqrt{\log(p)}}{p^{3/2}} + \frac{\sigma^2 \mu_0^4 \kappa^3 r^2 \log(p)}{\lambda \sqrt{p}} \bigg). \label{tjapprox}
\end{align}
We note that this bound is only non-vacuous if these terms are smaller than one, which requires that
\begin{align}
   s_{\min} &\gg |J|^{3/2} \max\bigg\{ \frac{\sigma \kappa^2 \mu_0^3 r^{3/2} \sqrt{\log(p)}}{p^{3/2}} , \frac{\sigma^2 \mu_0^4 \kappa^3 r^2 \log(p)}{\lambda \sqrt{p}} \bigg\}, \label{smin4}
\end{align}
which is the stronger than the requirements \eqref{smin0},\eqref{smin1}, and \eqref{smin2} by factors of $|J|$. 
\item 
\textbf{Step 3: Covariance Estimation}: We now consider the plug-in estimation of the covariance $S_J$.  The proof of \cref{thm:civalidity} shows that with probability at least $1 - O(p^{-6})$
\begin{align*}
    |(s^{(i)})^2 - (\hat s^{(i)})^2 |  &\lesssim \sigma^2 \mu_0 \frac{\sqrt{r \log(p)}}{p}  \| e_{(j-1)p_3 + k}\t \mathbf{V}_1 \|^2 + \frac{\sigma^2 \kappa \mu_0^2 r^{3/2} \log(p)}{p^3} \\
    &\quad + \frac{\sigma^3 \mu_0^5 r^3 \kappa^2 \log^{3/2}(p)}{\lambda p^{3/2}},
\end{align*}
with similar bounds for $s^{(j)}$ and $s^{(k)}$, where the notation is defined in the proof of \cref{thm:civalidity} (see \cref{sec:civalidityproofs}). % \textcolor{black}{todo: modify for other spectral structure of $S_J$}
A straightfoward modification of the same proof for the cross-terms reveals that with probability at least $1 - O(|J|^2 p^{-6})$
\begin{align*}
    \| &S_J - \hat S_J \| \\
    &\leq |J|^2 \| S_J - \hat S_J \|_{\max} \\
   % &\lesssim |J|^2 \sigma^2 \mu_0 \frac{\sqrt{r \log(p)}}{p} \mu_0^2 \frac{r}{p^2}  + |J|^2\bigg( \frac{\sigma^2 \kappa \mu_0^2 r^{3/2} \log(p)}{p^3} + \frac{\sigma^3 \mu_0^5 r^3 \kappa^2 \log^{3/2}(p)}{\lambda p^{3/2}} \bigg) \\
    &\lesssim %|J|^2 \mu_0^3 \frac{ r^{3/2} \sqrt{ \log(p)}}{p^3}  + 
    |J|^2 \sigma^2 \mu_0^3 \frac{ r^{3/2} \sqrt{\log(p)}}{p^3}  + |J|^2\bigg( \frac{\sigma^2 \kappa \mu_0^2 r^{3/2} \log(p)}{p^3} + \frac{\sigma^3 \mu_0^5 r^3 \kappa^2 \log^{3/2}(p)}{\lambda p^{3/2}} \bigg) \\
   % &\leq  o(1) s^2_{\min} \\
    &\ll s^2_{\min}, %\\
    %&\lesssim \lambda_{\min}(S_J)
\end{align*}
where the penultimate inequality holds as long as 
\begin{align}
    s_{\min}^2 &\gg |J|^2 \max\bigg\{ \frac{\sigma^2 \mu_0^3 r^{3/2} \sqrt{\log(p)}}{p^3}, \frac{\sigma^2 \kappa \mu_0^2 r^{3/2} \log(p)}{p^3}, \frac{\sigma^3 \mu_0^5 r^3 \kappa^2 \log^{3/2}(p)}{\lambda p^{3/2}} \bigg\}. \label{sminsecond}
\end{align}
Denote
\begin{align}
    M &\coloneqq % |J|^2 \mu_0 \frac{\sqrt{r \log(p)}}{p} s_{\min}^2 +
    |J|^2 \sigma^2 \mu_0^3 \frac{ r^{3/2} \sqrt{\log(p)}}{p^3} %\nonumber %\\ &\quad 
    + |J|^2\bigg( \frac{\sigma^2 \kappa \mu_0^2 r^{3/2} \log(p)}{p^3} 
     + \frac{\sigma^3 \mu_0^5 r^3 \kappa^2 \log^{3/2}(p)}{\lambda p^{3/2}} \bigg), \label{Mdef}
\end{align}
so that it holds that
\begin{align*}
    \| S_J - \hat S_J \| \lesssim M
\end{align*}
with probability at least $1 - O(|J|^2 p^{-6})$.  By Theorem 6.2 of \citet{higham_functions_2008}, it holds that
\begin{align*}
    \| \hat S_J^{1/2} - S_J^{1/2} \| &\leq \frac{1}{ \lambda_{\min}^{1/2}(\hat S_J) + \lambda_{\min}^{1/2}(S_J)} \| \hat S_J - S_J \| \\
    &\lesssim \frac{1}{s_{\min}} M \\
    &\ll s_{\min}
\end{align*}
provided $M \ll s^2_{\min}$. Therefore, by Weyl's inequality, $\hat S_J$ is invertible and
\begin{align*}
    \| \bigg( S_{J}^{-1/2} - \hat S_{J}^{-1/2} \bigg) (\mathcal{\hat T}_J - \mathcal{T}_J ) \| &\leq |J | \| \mathcal{\hat T} - \mathcal{T} \|_{\max} \| S_{J}^{-1/2} - \hat S_{J}^{-1/2} \| \\
    &\leq |J| \| S_{J}^{-1/2} ( \hat S_{J}^{1/2} - S_{J}^{1/2}) \hat S_{J}^{-1/2} \| \| \mathcal{\hat T} - \mathcal{T} \|_{\max} \\
    &\lesssim |J| \| \mathcal{\hat T} - \mathcal{T} \|_{\max} \frac{1}{s_{\min}^2} \| \hat S_J^{1/2} - S_J^{1/2} \| \\
    &\lesssim |J| \| \mathcal{\hat T} - \mathcal{T} \|_{\max} \frac{1}{s_{\min}^3} M \\
    &\lesssim  \frac{|J| M }{s_{\min}^3} \bigg( \frac{\sigma \kappa \mu_{0} \sqrt{r \log (p)}}{p}+\frac{\sigma^{2} \mu_{0}^{4} \kappa^{3} r^{3} \log (p)}{\lambda \sqrt{p}} \bigg) \\
    &=: \eps,
\end{align*}
where we note that the final inequality holds with probability at least $1 - O(p^{-6})$ by \cref{cor:maxnormbound}.  
\item 
\textbf{Step 4: Completing the Proof}: Just as in the proof of \cref{thm:civalidity2}, for a convex set $A$ denote $A^{\eps}$ as the $\eps$-enlargment via
\begin{align*}
    A^{\eps} &\coloneqq \{x : d(x,A) \leq \eps \}.
\end{align*}
By Theorem 1.2 of \citet{raicv_multivariate_2019}, if $A$ is an isotropic $\mathbb{R}^{|J|}$-dimensional random gaussian vector, it holds that
\begin{align*}
    \p\bigg( Z \in A^{\eps}\setminus A\bigg) &\lesssim |J|^{1/4} \eps \lesssim |J|^{1/2} \eps.
\end{align*}
We are now prepared to complete the proof.  Let $A_{\alpha}$ denote the confidence region such that $\p(Z \in A_{\alpha}) = 1 - \alpha$, where $Z \sim N(0, \mathbf{I}_{|J|})$.  Then
\begin{align*}
    \bigg|\p&\bigg\{  \mathrm{Vec}(\mathcal{T}_J) \in \mathrm{C.I.}_{\alpha}(\mathcal{\hat T}_J) \bigg\} - (1 - \alpha) \bigg| \\
    &= \bigg| \p\bigg\{ \hat S_{J}^{-1/2}\bigg( \mathrm{Vec}\big( \mathcal{\hat T}_J - \mathcal{T}_J \big) \bigg) \in \mathcal{A}_{\alpha} \bigg\} - (1 - \alpha) \bigg| \\
    &\leq \bigg| \p\bigg\{ S_{J}^{-1/2} \bigg( \mathrm{Vec}\big( \mathcal{\hat T}_J - \mathcal{T}_J \big) \bigg) \in \mathcal{A}_{\alpha}^{\eps} \bigg\} \cap \bigg\{\bigg|\big( S_{J}^{-1/2} - \hat S_{J}^{-1/2}\big) \bigg( \mathrm{Vec}\big( \mathcal{\hat T}_J - \mathcal{T}_J \big) \bigg)\bigg| \leq \eps \bigg\} \\
    &\qquad\qquad- (1 - \alpha) \bigg| \\
    &\quad + \p\bigg\{ \bigg|\big( S_{J}^{-1/2} - \hat S_{J}^{-1/2}\big) \bigg( \mathrm{Vec}\big( \mathcal{\hat T}_J - \mathcal{T}_J \big) \bigg) \big| > \eps \bigg\} \\
    &\leq \bigg| \p\bigg\{  S_{J}^{-1/2} \bigg( \mathrm{Vec}\big( \mathcal{\hat T}_J - \mathcal{T}_J \big) \bigg) \in \mathcal{A}_{\alpha}^{\eps} \bigg\} - (1- \alpha) \bigg|+ O(|J|^2p^{-6}) \\
    &\leq \bigg| \p\bigg\{  S_{J}^{-1/2} \bigg( \mathrm{Vec}\big( \mathcal{\hat T}_J - \mathcal{T}_J \big) \bigg) \in \mathcal{A}_{\alpha}^{\eps} \bigg\} - \p\bigg\{ Z \in A_\alpha^{\eps} \bigg\} \bigg| + \bigg| \p\bigg\{ Z \in A_{\alpha}^{\eps} \bigg\} - (1 - \alpha) \bigg| \\
    &\quad + O(|J|^2p^{-6}) \\
    &\lesssim \frac{|J|^{3/2}}{\sqrt{p\log(p)}} + \frac{|J|^2\mu_0^2 r}{p}  +\frac{ |J|^{3/2}}{s_{\min}}\bigg( \frac{\sigma \kappa^2 \mu_0^3 r^{3/2} \sqrt{\log(p)}}{p^{3/2}} + \frac{\sigma^2 \mu_0^4 \kappa^3 r^2 \log(p)}{\lambda \sqrt{p}} \bigg) \\
    &\quad + \bigg| \p\bigg\{ Z \in A_{\alpha}^{\eps} \bigg\} - \p\bigg\{ Z \in A_{\alpha} \bigg\} \bigg| + O(|J|^2p^{-6}) \\
    &\lesssim \frac{|J|^{3/2}}{\sqrt{p\log(p)}} + \frac{|J|^2\mu_0^2 r}{p}  +\frac{ |J|^{3/2}}{s_{\min}}\bigg( \frac{\sigma \kappa^2 \mu_0^3 r^{3/2} \sqrt{\log(p)}}{p^{3/2}} + \frac{\sigma^2 \mu_0^4 \kappa^3 r^2 \log(p)}{\lambda \sqrt{p}} \bigg) \\
    &\quad + \p( Z \in A^{\eps}_{\alpha} \setminus A_{\alpha}) + O(|J|^2p^{-6}) \\
    &\lesssim \frac{|J|^{3/2}}{\sqrt{p\log(p)}} + \frac{|J|^2\mu_0^2 r}{p}  +\frac{ |J|^{3/2}}{s_{\min}}\bigg( \frac{\sigma \kappa^2 \mu_0^3 r^{3/2} \sqrt{\log(p)}}{p^{3/2}} + \frac{\sigma^2 \mu_0^4 \kappa^3 r^2 \log(p)}{\lambda \sqrt{p}} \bigg) \\
    &\quad + |J|^{1/2} \eps + O(|J|^2p^{-6}),
\end{align*}
where we have used \eqref{tjapprox}.  
In order to complete the proof we need that this final bound is of order $o(1)$.  For the first two terms to be $o(1)$, we see that $J$ needs to satisfy
\begin{align*}
    |J| &\ll \min\bigg\{ p^{1/6}, \frac{p^{1/2}}{\mu_0 r^{1/2}} \bigg\}; 
\end{align*}
This is guaranteed if $|J| = o(p^{1/6})$, since we assume that $\mu_0^2 r \lesssim \sqrt{p}$, so that $\frac{p^{1/2}}{\mu_0 r^{1/2}} \gtrsim p^{1/4}$. This also is sufficient for the probability terms to be $o(1)$. 

We now translate the conditions on $s_{\min}$.   First we collect all our requirements from the previous steps.  From steps 1 and 2, \eqref{smin4} requires that
\begin{align*}
    s_{\min} &\gg  |J|^{3/2} \sigma \max\bigg\{ \frac{\kappa^2 \mu_0^3 r^{3/2} \sqrt{\log(p)}}{p^{3/2}}, \frac{\sigma \mu_0^4 \kappa^3 r^2 \log(p)}{\lambda \sqrt{p}} \bigg\}.
\end{align*}
In addition, from \eqref{sminsecond}, we require
\begin{align}
      s_{\min}^2 &\gg |J|^2 \max\bigg\{ \frac{\sigma^2 \mu_0^3 r^{3/2} \sqrt{\log(p)}}{p^3}, \frac{\sigma^2 \kappa \mu_0^2 r^{3/2} \log(p)}{p^3}, \frac{\sigma^3 \mu_0^5 r^3 \kappa^2 \log^{3/2}(p)}{\lambda p^{3/2}} \bigg\}. \label{sminsquared}
\end{align}
Finally, we also require that $|J|^{1/2} \eps = o(1)$, which translates to the condition
\begin{align*}
    s_{\min}^3 &\gg \sigma |J|^{3/2} M\max\bigg\{ \frac{ \kappa \mu_0 \sqrt{r\log(p)}}{p}, \frac{\sigma \mu_0^4 \kappa^3 r^3 \log(p)}{\lambda \sqrt{p}} \bigg\}. 
\end{align*}
Recalling the definition of $M$ in \eqref{Mdef}, this condition is equivalent to the condition
\begin{align*}
    s_{\min}^3 &\gg  \sigma |J|^{7/2}  \max\bigg\{ \frac{ \kappa \mu_0 \sqrt{r\log(p)}}{p}, \frac{\sigma \mu_0^4 \kappa^3 r^3 \log(p)}{\lambda \sqrt{p}} \bigg\} \\
    &\times \bigg( \mu_0 \frac{\sqrt{r \log(p)}}{p} s_{\min}^2 +  \sigma^2 \mu_0^3 \frac{ r^{3/2} \sqrt{\log(p)}}{p^3} \nonumber+  \frac{\sigma^2 \kappa \mu_0^2 r^{3/2} \log(p)}{p^3} 
     + \frac{\sigma^3 \mu_0^5 r^3 \kappa^2 \log^{3/2}(p)}{\lambda p^{3/2}} \bigg).
\end{align*}
When $|J| = o(p^{1/6})$ and $\kappa^2 \mu_0^2 r^{3/2} \sqrt{\log(p)} \lesssim p^{1/4}$ this condition is satisfied when
\begin{align*}
    s_{\min}^3 &\gg \sigma |J|^{7/2} \max\bigg\{ \frac{\kappa \mu_0 \sqrt{r\log(p)}}{p}, \frac{\sigma \mu_0^4 \kappa^3 r^2 \log(p)}{\lambda \sqrt{p}} \bigg\} \\
    &\quad \times \max\bigg\{ \frac{\sigma^2 \mu_0^3 r^{3/2} \sqrt{\log(p)}}{p^3}, \frac{\sigma^2 \kappa \mu_0^2 r^{3/2}\log(p)}{p^3}, \frac{\sigma^3 \mu_0^5 \kappa^2 \log^{3/2}(p)}{\lambda p^{3/2}} \bigg\} \\
    &\asymp \sigma^3 |J|^{7/2} \max\bigg\{ \frac{\kappa \mu_0^4 r^2 \log(p)}{p^4}, \frac{\kappa^2 \mu_0^3 r^2 \log^{3/2}(p)}{p^4}, \frac{\kappa^3 \mu_0^6 \sqrt{r} \log^2(p)}{ (\lambda/\sigma) p^{5/2}},\\
    &\qquad \qquad \qquad \qquad \qquad \frac{\kappa^3 \mu_0^7 r^{7/2} \log^{3/2}(p)}{ (\lambda/\sigma) p^{7/2}}, \frac{\kappa^4 \mu_0^6 r^{7/2} \log^2(p)}{(\lambda/\sigma) p^{7/2}}, \frac{\kappa^5 \mu_0^9 r^2 \log^{5/2}(p)}{(\lambda/\sigma)^2 p^{5/2}} \bigg\}
\end{align*}
 Putting it together and putting terms on the same scale, we see that we have the three conditions
 \begin{align*}
  s_{\min}/\sigma  &\gg  |J|^{3/2}\max\bigg\{ \frac{\kappa^2 \mu_0^3 r^{3/2} \sqrt{\log(p)}}{p^{3/2}}, \frac{\sigma \mu_0^4 \kappa^3 r^2 \log(p)}{\lambda \sqrt{p}} \bigg\}; \\
        s_{\min}/\sigma &\gg |J| \max\bigg\{ \frac{ \mu_0^{3/2} r^{3/4} \log^{1/4}(p)}{p^{3/2}}, \frac{ \kappa^{1/2} \mu_0 r^{3/4} \sqrt{\log(p)}}{p^{3/2}}, \frac{ \mu_0^{5/2} r^{3/2} \kappa \log^{3/4}(p)}{(\lambda/\sigma)^{1/2} p^{3/4}} \bigg\}; \\
        s_{\min}/\sigma &\gg |J|^{1/6} \max\bigg\{ \frac{\kappa^{1/3}\mu_0^{4/3} r^{2/3} \log^{1/3}(p)}{p^{4/3}}, \frac{\kappa^{2/3} \mu_0 r^{2/3} \log^{1/2}(p)}{p^{4/3}}, \frac{\kappa \mu_0^2 r^{1/6} \log^{2/3}(p)}{(\lambda/\sigma)^{1/3} p^{5/6}}, \\
        &\qquad \qquad \qquad \qquad \frac{\kappa \mu_0^{7/3} r^{7/6} \log^{1/2}(p)}{(\lambda/\sigma)^{1/3} p^{7/6}}, \frac{\kappa^{4/3} \mu_0^2 r^{7/6} \log^{2/3}(p)}{(\lambda/\sigma)^{1/3}p^{7/6}}, \frac{\kappa^{5/3}\mu_0^3 r^{2/3} \log^{5/6}(p)}{(\lambda/\sigma)^{2/3} p^{5/6}} \bigg\}.
 \end{align*}
 Removing redundant conditions shows that we require
 \begin{align*}
     s_{\min}/\sigma &\gg \max\bigg\{ |J|^{3/2}  \frac{\kappa^2 \mu_0^3 r^{3/2} \sqrt{\log(p)}}{p^{3/2}}, |J|^{3/2} \frac{ \mu_0^4 \kappa^3 r^2 \log(p)}{(\lambda/\sigma) \sqrt{p}}, |J| \frac{ \mu_0^{5/2} r^{3/2} \kappa \log^{3/4}(p)}{(\lambda/\sigma)^{1/2} p^{3/4}}, \\
     &\qquad \qquad \qquad  |J|^{1/6} \frac{\kappa \mu_0^{3/2} r^{3/2} \sqrt{\log(p)}}{p^{4/3}}, |J|^{1/6} \frac{\kappa^{5/3}\mu_0^{3} r^{7/6} \log^{5/6}(p)}{(\lambda/\sigma)^{1/3} p^{5/6}}\bigg\}.
 \end{align*}
%  Now recall that from the definition of $s_{\min}$, it holds that
%  \begin{align*}
%      s_{\min}/\sigma &\gtrsim  \min\bigg\{ \left\|e_{(j-1) p_{3}+k}^{\top} \mathbf{V}_{1}\right\|,\left\|e_{(k-1) p_{1}+i} \mathbf{V}_{2}\right\|,\left\|e_{(i-1) p_{2}+j} \mathbf{V}_{3}\right\|\bigg\}.
%  \end{align*}
%  Consequently, the results hold as long as
%  \begin{align*}
%      \min\bigg\{& \left\|e_{(j-1) p_{3}+k}^{\top} \mathbf{V}_{1}\right\|,\left\|e_{(k-1) p_{1}+i} \mathbf{V}_{2}\right\|,\left\|e_{(i-1) p_{2}+j} \mathbf{V}_{3}\right\|\bigg\} \\
%      &\gg \max\bigg\{ |J|^{3/2}  \frac{\kappa \mu_0^3 r^{3/2} \sqrt{\log(p)}}{p^{3/2}}, |J|^{3/2} \frac{ \mu_0^4 \kappa^3 r^2 \log(p)}{(\lambda/\sigma) \sqrt{p}}, |J| \frac{ \mu_0^{5/2} r^{3/2} \kappa \log^{3/4}(p)}{(\lambda/\sigma)^{1/2} p^{3/4}}, \\
%      &\qquad \qquad \qquad  |J|^{1/6} \frac{\kappa \mu_0^{3/2} r^{3/2} \sqrt{\log(p)}}{p^{4/3}}, |J|^{1/6} \frac{\kappa^{5/3}\mu_0^{3} r^{7/6} \log^{5/6}(p)}{(\lambda/\sigma)^{1/3} p^{5/6}}\bigg\},
%  \end{align*}
 which holds under the conditions of \cref{thm:simultaneousinference}.  
 \end{itemize}
\end{proof}

\subsection{Proof of \cref{thm:entrytesting}}

\begin{proof}[Proof of \cref{thm:entrytesting}]
Similar to the previous proof, by the proof of \cref{thm:asymptoticnormalityentries}, by \eqref{mainentrywiseexpansion} it holds that
\begin{align*}
    \mathcal{\hat T}_{ijk} - \mathcal{T}_{ijk} = \xi_{ijk} + O\bigg( \frac{\sigma \kappa^2 \mu_0^3 r^{3/2} \sqrt{\log(p)}}{p^{3/2}} + \frac{\sigma^2 \mu_0^4 \kappa^3 r^2 \log(p)}{\lambda \sqrt{p}} \bigg)
\end{align*}
with probability at least $1 - O(p^{-9})$, where $\xi_{ijk}$ is defined as in \cref{lem:xiijkgaussian}.  A similar expansion holds for $\mathcal{\hat T}_{i'j'k'} - \mathcal{T}_{i'j'k'}$, which demonstrates that
\begin{align*}
    \mathcal{\hat T}_{ijk} - \mathcal{\hat T}_{i'j'k'} - (\mathcal{T}_{ijk} - \mathcal{\hat T}_{i'j'k'}) &= \xi_{ijk} - \mathcal{\xi}_{i'j'k'} + O\bigg( \frac{\sigma \kappa^2 \mu_0^3 r^{3/2} \sqrt{\log(p)}}{p^{3/2}} + \frac{\sigma^2 \mu_0^4 \kappa^3 r^2 \log(p)}{\lambda \sqrt{p}} \bigg).
\end{align*}
Our proof now proceeds in a similar manner to \cref{thm:simultaneousinference}.

\begin{itemize}
    \item \textbf{Step 1: Limiting Variance Structure}: We first calculate the variance of $\xi_{ijk}$.  However, through precisely the same analysis as in step 1 of \cref{thm:simultaneousinference}, we have that
\begin{align*}
    \mathrm{Cov}\bigg( \xi_{ijk} \xi_{i'j'k'} \bigg) &= \mathbb{I}_{\{i= i'\}} e_{(j-1)p_3 + k}\t  \mathbf{\hat V}_1 \mathbf{\hat V}_1\t \hat \Sigma^{(i)} \mathbf{\hat V}_1 \mathbf{\hat V}_1\t e_{(j'-1)p_3 + k'} \\
    &\quad + \mathbb{I}_{\{j= j'\}} e_{(k-1)p_1 + i}\t  \mathbf{\hat V}_2 \mathbf{\hat V}_2\t\hat \Sigma^{(j)} \mathbf{\hat V}_2 \mathbf{\hat V}_2\t e_{(k'-1)p_3 + i'}  \\
    &\quad + \mathbb{I}_{\{k= k'\}} e_{(i-1)p_2 + j}\t  \mathbf{\hat V}_3 \mathbf{\hat V}_3\t \hat \Sigma^{(k)} \mathbf{\hat V}_3 \mathbf{\hat V}_3\t e_{(i'-1)p_2 +j'} \\
    &\qquad + O\bigg(  \mu_0^2 \frac{r}{p} \bigg[ s^2_{ijk} + s^2_{i'j'k'}\bigg] \bigg).
\end{align*}
Therefore,
\begin{align*}
    \mathrm{Var}( \xi_{ijk} - \xi_{i'j'k'}) &= s^2_{\{ijk\}\{i'j'k'\}} \bigg(1 + O\bigg( \mu_0^2 \frac{r}{p} \bigg) \bigg).
\end{align*}
\item 
\textbf{Step 2: Gaussian Approximation and Variance Approximation}: 
Define, for some sufficiently large constant $C$,
\begin{align*}
    \Delta = \Delta^{(i)} = C \bigg( \frac{\sigma \kappa^2 \mu_0^3 r^{3/2} \sqrt{\log(p)}}{p^{3/2}} + \frac{\sigma^2 \mu_0^4 \kappa^3 r^2 \log(p)}{\lambda \sqrt{p}} \bigg).
\end{align*}
Observe that we have the expansion
\begin{align*}
    \frac{\mathcal{\hat T}_{ijk} - \mathcal{\hat T}_{i'j'k'} - ( \mathcal{T}_{ijk} - \mathcal{T}_{i'j'k'})}{s_{\{ijk\}\{i'j'k'\}}} &= \frac{\xi_{ijk} - \xi_{i'j'k'}}{s_{\{ijk\}\{i'j'k'\}}} + \frac{\Delta}{s_{\{ijk\}\{i'j'k'\}}},
\end{align*}
which holds with probability at least $1 - O(p^{-9})$.  By modifying the proof of \cref{lem:xiijkgaussian}, it is straightforward to demonstrate that
\begin{align*}
    \sup_{t\in \mathbb{R}} \bigg| \p\bigg\{ \frac{\xi_{ijk} - \xi_{i'j'k'}}{s_{\{ijk\}\{i'j'k'\}}} \bigg\} - \Phi(t) \bigg| &\leq \frac{C_1}{\sqrt{p\log(p)}} + \frac{C_2 \mu_0^2 r}{p} \\
    &\coloneqq \eta.
\end{align*}
Hence it holds that
\begin{align*}
      \p\bigg\{  \frac{ \mathcal{\hat T}_{ijk} - \mathcal{\hat T}_{i'j'k'} - (\mathcal{T}_{ijk} - \mathcal{\hat T}_{i'j'k'})}{s_{\{ijk\}\{i'j'k'\}}} \leq t \bigg\} &\leq \p\bigg\{ \frac{\xi_{ijk} - \xi_{i'j'k'}}{s_{\{ijk\}\{i'j'k'\}}} \leq t + \frac{\Delta}{s_{\{ijk\}\{i'j'k'\}}} \bigg\} + C_3 p^{-9} \\
      &\leq \Phi\bigg\{ t +   \frac{\Delta}{s_{\{ijk\}\{i'j'k'\}}} \bigg\} + C_3 p^{-9} + \eta \\
      &\leq \Phi(t) + \frac{\Delta}{s_{\{ijk\}\{i'j'k'\}}}+ C_3 p^{-9} + \eta \\
      &= \Phi(t) + o(1),
\end{align*}
where the final bounds holds as long as  $ s_{\{ijk\}\{i'j'k'\}} \gg \Delta$, which we justify at the end of the proof.
In addition, it is straightforward to modify the proof of \cref{thm:civalidity} to demonstrate that as long as
\begin{align*}
    s^2_{\{ijk\}\{i'j'k'\}} &\gg \max\bigg\{ \frac{\sigma^2 \kappa \mu_0^2 r^{3/2} \log(p)}{p^3}, \frac{\sigma^3 \mu_0^5 r^3 \kappa^2 \log^{3/2}(p)}{\lambda p^{3/2}} \bigg\} \numberthis \label{222check}
\end{align*}
it holds that $\hat s_{\{ijk\}\{i'j'k'\}} = s_{\{ijk\}\{i'j'k'\}}(1 +o(1))$ with probability at least $1 - O(p^{-6})$. We will check this condition at the end of the proof. 
Denote the region
\begin{align*}
    \mathcal{R}_{\alpha} &\coloneqq \bigg(\frac{ \mathcal{\hat T}_{ijk} - \mathcal{\hat T}_{i'j'k'}}{\hat s_{\{ijk\}\{i'j'k'\}}} - z_{\alpha/2} , \frac{ \mathcal{\hat T}_{ijk} - \mathcal{\hat T}_{i'j'k'}}{\hat s_{\{ijk\}\{i'j'k'\}}} + z_{\alpha/2}\bigg).
\end{align*}
Consequently, by similar manipulations to the end of the proof of \cref{thm:civalidity},
\begin{align*}
    \bigg|  \frac{ (\mathcal{T}_{ijk} - \mathcal{ T}_{i'j'k'})}{s_{\{ijk\}\{i'j'k'\}}}  \in \mathcal{R}_{\alpha} \bigg\} - (1-\alpha) \bigg| &= o(1). 
\end{align*}
This completes the proof, provided we can justify all of the conditions on $s_{\{ijk\}\{i'j'k'\}}$.
\item \textbf{Step 3: Checking Conditions}: Observe that \eqref{222check} requires that
\begin{align*}
    s^2_{\{ijk\}\{i'j'k'\}} &\gg \max\bigg\{ \frac{\sigma^2 \kappa \mu_0^2 r^{3/2} \log(p)}{p^3}, \frac{\sigma^3 \mu_0^5 r^3 \kappa^2 \log^{3/2}(p)}{\lambda p^{3/2}} \bigg\}.
\end{align*}
In addition, we require that  $s_{\{ijk\}\{i'j'k'\}}\gg \Delta$.  It is sufficient to require that
\begin{align*}
    s^2_{\{ijk\}\{i'j'k'\}} &\gg \max\bigg\{ \frac{\sigma^2 \kappa^4 \mu_0^6 r^3 \log(p)}{p^3}, \frac{\sigma^4 \mu_0^8 \kappa^6 r^4 \log^2(p)}{\lambda^2 p} \bigg\}.
\end{align*}
Observe that $s^2_{\{ijk\}\{i'j'k'\}}$ satisfies
\begin{align*}
    &s^2_{\{ijk\}\{i'j'k'\}} \\
    &\gtrsim \sigma^2 \min\bigg\{ \| e_{(j-1)p_3 + k }\t \mathbf{V}_1\|, \| e_{(j'-1)p_3 + k'}\t \mathbf{V}_1\|,\\
    &\quad \| e_{(k-1)p_1 + i}\t \mathbf{V}_2 \|, \| e_{(k'-1)p_1 + i}\t \mathbf{V}_2 \|, \| e_{(i-1)p_2 + j}\t \mathbf{V}_3 \|, \| e_{(i'-1)p_2 + j'}\t \mathbf{V}_3 \| \bigg\}.
\end{align*}
To see this, note that no matter how many overlapping indices there, there is always one set of indices that are not overlapping, and hence the variance can be lower bounded by at least one of
\begin{align*}
    %s^2_{\{ijk\},\{ij'k'\}} &\geq \\
 \sigma_{\min}^2 \times   \bigg\{ &\| e_{(j-1)p_3 + k }\t \mathbf{V}_1\|^2, \| e_{(j'-1)p_3 + k'}\t \mathbf{V}_1\|^2,\\
 &\quad \| e_{(k-1)p_1 + i}\t \mathbf{V}_2 \|^2, \| e_{(k'-1)p_1 + i}\t \mathbf{V}_2 \|^2, \| e_{(i-1)p_2 + j}\t \mathbf{V}_3 \|^2, \| e_{(i'-1)p_2 + j'}\t \mathbf{V}_3 \|^2 \bigg\}.
\end{align*}
Therefore, as long as 
\begin{align*}
    \min\bigg\{ &\| e_{(j-1)p_3 + k }\t \mathbf{V}_1\|^2, \|e_{(j'-1)p_3 + k'}\t \mathbf{V}_1\|^2, \| e_{(k-1)p_1 + i}\t \mathbf{V}_2 \|^2, \\
    &\quad \| e_{(k'-1)p_1 + i}\t \mathbf{V}_2 \|^2, \| e_{(i-1)p_2 + j}\t \mathbf{V}_3 \|^2, \| e_{(i'-1)p_2 + j'}\t \mathbf{V}_3 \|^2 \bigg\} \\
    &\gg \max\bigg\{ \frac{\kappa \mu_0^2 r^{3/2} \log(p)}{p^3} , \frac{\mu_0^5 r^3 \kappa^2 \log^{3/2}(p)}{(\lambda/\sigma)p^{3/2}}, \frac{\kappa^4 \mu_0^6 r^3 \log(p)}{p^3}, \frac{\mu_0^8 \kappa^6 r^3 \log^2(p)}{(\lambda/\sigma)^2 p} \bigg\}.
    \end{align*}
    we see all our results continue to hold. In particular, when $\kappa,\mu_0 = O(1)$, it is sufficient to have that
    \begin{align*}
         \min\bigg\{ &\| e_{(j-1)p_3 + k }\t \mathbf{V}_1\|^2, \| e_{(j'-1)p_3 + k'}\t \mathbf{V}_1\|^2, \| e_{(k-1)p_1 + i}\t \mathbf{V}_2 \|^2, \\
         &\quad \| e_{(k'-1)p_1 + i}\t \mathbf{V}_2 \|^2, \| e_{(i-1)p_2 + j}\t \mathbf{V}_3 \|^2, \| e_{(i'-1)p_2 + j'}\t \mathbf{V}_3 \|^2 \bigg\} \\
    &\gg \max\bigg\{ \frac{r^{3} \log(p)}{p^3} , \frac{r^{3} \log^{3/2}(p)}{(\lambda/\sigma)p^{3/2}}\bigg\},
    \end{align*}
    which is precisely the condition in \cref{thm:entrytesting_v1}. 
    \end{itemize}
\end{proof}

\section{Proofs of Lower Bounds} \label{sec:efficiencyproof}
In this section we prove all of our lower bound results.  In the subsequent section we prove both  \cref{thm:efficiencyloadings} and \cref{thm:efficiency}, and in \cref{sec:prooforders} we prove \cref{thm:efficiency_order,thm:efficiency_order_ijk}.
\subsection{Proofs of \cref{thm:efficiencyloadings} and \cref{thm:efficiency}}
In this section we prove %\cref{thm:efficiency} and 
\cref{thm:efficiencyloadings} and \cref{thm:efficiency}.  We prove both results simultaneously in a self-contained manner.  

\begin{proof}[Proofs of \cref{thm:efficiencyloadings} and \cref{thm:efficiency}]
Without loss of generality we assume that $\sigma^2 \equiv \sigma^2_{\min}$, since otherwise  we only increase the variance.  In addition, we assume that the core tensor $\mathcal{C} \in \mathbb{R}^{r_1 \times r_2 \times r_3}$ is known, since not knowing it will also only increase the variance. 

Next, by the Cramer-Rao lower bound, for any unbiased estimator $(\tilde \U_1, \tilde \U_2, \tilde \U_3)$ of the parameter $(\U_1, \U_2, \U_3)$, it holds that
\begin{align*}
    \mathrm{Var} (\tilde \U_1, \tilde \U_2, \tilde \U_3) &\succcurlyeq \mathcal{I}\inv,
\end{align*}
where $\mathcal{I}$ is defined via
\begin{align*}
    \mathcal{I} &\coloneqq \E \bigg[\nabla_{\U_1, \U_2,\U_3} \log \mathcal{L}(\mathcal{\hat T}; \U_1, \U_2, \U_3)\bigg] \bigg[ \nabla_{\U_1, \U_2,\U_3} \log \mathcal{L}(\mathcal{\hat T}; \U_1, \U_2, \U_3)\bigg]\t,
\end{align*}
% where $\mathcal{L}(\mathcal{\hat T};\U_1,\U_2,\U_3)$ is the likelihood function.  Therefore 
By the delta method, for any unbiased estimator $\mathcal{\check T}_{ijk}$ of $\mathcal{T}_{ijk}$ it holds that
\begin{align*}
    \mathrm{Var}(\mathcal{\check T}_{ijk}) \geq \nabla_{\U_1,\U_2,\U_3} \mathcal{T}_{ijk}\t  \mathcal{I}\inv\nabla_{\U_1,\U_2,\U_3} \mathcal{T}_{ijk}, 
\end{align*}
where $\nabla_{\U_1,\U_2,\U_3} \mathcal{T}_{ijk}$ is the gradient of $\mathcal{T}_{ijk}$ with respect to the (vectorized) parameter $(\U_1, \U_2,\U_3)$.  
We now proceed in several steps. 
\begin{itemize}
    \item \textbf{Step 1: Calculating the Relevant Quantities}: 
 First, we note that
\begin{align}
    \frac{\del \mathcal{T}_{ijk}}{\del (\U_1)_{l_1,s}} &= \frac{\del}{\del (\U_1)_{l_1,s}} \sum_{a,b,c} \mathcal{C}_{abc} (\U_1)_{ia} (\U_2)_{jb} (\U_3)_{kc} \nonumber \\
    &= \mathbb{I}_{i = l_1} \sum_{b,c} \mathcal{C}_{sbc} (\U_2)_{jb} (\U_3)_{kc}. \label{eq:tijkpartial}
\end{align}
A similar calculation can be made for the other modes, with appropriate replacements.

We now calculate the Fisher information matrix.  To do so, we will use the convention that the vectorized matrices $\U_1, \U_2,$ and $\U_3$ are in the order such that they are first indexed by row then by column, which will be a useful convention later on. We note that
\begin{align*}
\log f( \mathcal{\hat T}; \U_1, \U_2, \U_3, \mathcal{C} ) &=\frac{-1}{2\sigma^2} \sum_{i_1,i_2,i_3} \bigg( \mathcal{\hat T}_{i_1 i_2 i_3} - \sum_{a,b,c} \mathcal{C}_{abc} (\U_1)_{i_1 a} (\U_2)_{i_2 b} (\U_3 )_{i_3 c} \bigg)^2; 
\end{align*}
and hence that
\begin{align*}
    &\frac{\del }{\del (\U_1)_{l a'}}\log f( \mathcal{\hat T}; \U_1, \U_2, \U_3, \mathcal{C} )  \\
&= \frac{\del }{\del (\U_1)_{l a'}}\frac{-1}{2\sigma^2} \sum_{i_1,i_2,i_3} \bigg( \mathcal{\hat T}_{i_1 i_2 i_3} - \sum_{a,b,c} \mathcal{C}_{abc} (\U_1)_{i_1 a} (\U_2)_{i_2 b} (\U_3 )_{i_3 c} \bigg)^2 \\
%&= \frac{\del }{\del (\U_1)_{l a'}}\frac{-1}{2\sigma^2} \sum_{i_2,i_3} \bigg( \mathcal{\hat T}_{l i_2 i_3} - \sum_{a,b,c} \mathcal{C}_{abc} (\U_1)_{l a} (\U_2)_{i_2 b} (\U_3 )_{i_3 c} \bigg)^2 \\
&= \frac{-1}{\sigma^2} \sum_{i_2,i_3} \bigg( \mathcal{\hat T}_{l i_2 i_3} - \sum_{a,b,c} \mathcal{C}_{abc} (\U_1)_{l a} (\U_2)_{i_2 b} (\U_3 )_{i_3 c}  \bigg)\\
&\quad \times \frac{\del}{\del(\U_1)_{la'}} \bigg( \mathcal{\hat T}_{l i_2 i_3} - \sum_{a,b,c} \mathcal{C}_{abc} (\U_1)_{l a} (\U_2)_{i_2 b} (\U_3 )_{i_3 c}  \bigg)  \\
&= \frac{1}{\sigma^2} \sum_{i_2,i_3} \bigg( \mathcal{\hat T}_{l i_2 i_3} - \sum_{a,b,c} \mathcal{C}_{abc} (\U_1)_{l a} (\U_2)_{i_2 b} (\U_3 )_{i_3 c}  \bigg)  \bigg( \sum_{b,c} \mathcal{C}_{a' bc} (\U_2)_{i_2 b} (\U_3 )_{i_3 c}  \bigg)  \\
&= \frac{1}{\sigma^2} \sum_{i_2, i_3} \mathcal{Z}_{l i_2 i_3}   \bigg( \sum_{b,c} \mathcal{C}_{a' bc} (\U_2)_{i_2 b} (\U_3 )_{i_3 c}  \bigg).
\end{align*}
Therefore,\begin{align}
\E &\Bigg[  \frac{\del }{\del (\U_1)_{l_1 a_1}}\log f( \mathcal{\hat T}; \U_1, \U_2, \U_3, \mathcal{C} ) \frac{\del }{\del (\U_1)_{l_2 a_2}}\log f( \mathcal{\hat T}; \U_1, \U_2, \U_3, \mathcal{C} ) \Bigg] \nonumber \\
&= \frac{1}{\sigma^4} \E \Bigg[ \sum_{i_2, i_3} \mathcal{Z}_{l_1 i_2 i_3}   \bigg( \sum_{b,c} \mathcal{C}_{a_1 bc} (\U_2)_{i_2 b} (\U_3 )_{i_3 c}  \bigg) \Bigg] \Bigg[  \sum_{i_2, i_3} \mathcal{Z}_{l_2 i_2 i_3}   \bigg( \sum_{b,c} \mathcal{C}_{a_2 bc} (\U_2)_{i_2 b} (\U_3 )_{i_3 c} \Bigg] \nonumber \\
&= \frac{1}{\sigma^2} \mathbb{I}_{l_1 = l_2}\sum_{i_2, i_3} \bigg( \sum_{b,c} \mathcal{C}_{a_1 bc} (\U_2)_{i_2 b} (\U_3 )_{i_3 c}  \bigg) \bigg(  \sum_{b,c} \mathcal{C}_{a_2 bc} (\U_2)_{i_2 b} (\U_3 )_{i_3 c} \bigg) \nonumber \\
&= \frac{1}{\sigma^2} \mathbb{I}_{l_1 = l_2}\sum_{b,c} \mathcal{C}_{a_1 bc} \mathcal{C}_{a_2 bc}\label{grammatrixidentity}
\end{align}
where the second inequality uses the independence of $\mathcal{Z}_{i_1i_2i_3}$, and the final inequality is by orthonormality.  Similarly,
\begin{align}
&\E \Bigg[  \frac{\del }{\del (\U_1)_{l_1 a_1}}\log f( \mathcal{\hat T}; \U_1, \U_2, \U_3, \mathcal{C} ) \frac{\del }{\del (\U_2)_{l_2 b_1}}\log f( \mathcal{\hat T}; \U_1, \U_2, \U_3, \mathcal{C} ) \Bigg] \nonumber \\
&=  \frac{1}{\sigma^4} \E \Bigg[ \sum_{i_2, i_3} \mathcal{Z}_{l_1 i_2 i_3}   \bigg( \sum_{b,c} \mathcal{C}_{a_1 bc} (\U_2)_{i_2 b} (\U_3 )_{i_3 c}  \bigg) \Bigg]  \Bigg[  \sum_{i_1, i_3} \mathcal{Z}_{i_1 l_2 i_3}   \bigg( \sum_{a,c} \mathcal{C}_{a b_1 c} (\U_1)_{i_1 a} (\U_3 )_{i_3 c} \Bigg] \nonumber\\
&= \frac{1}{\sigma^2} \sum_{i_3} \bigg(\sum_{b,c} \mathcal{C}_{a_1 bc} (\U_2)_{l_2 b} (\U_3)_{i_3 c} \bigg) \bigg( \sum_{a,c} \mathcal{C}_{a b_1 c} (\U_1)_{l_1 a} (\U_3 )_{i_3 c} \bigg) \label{eq:Hidentity}
\end{align}
Similar derivations can be obtained for $\U_2$ and $\U_3$. 

Define the $r_1$ vectors $\{\mathbf{c}_{a_i}\}_{i=1}^{r_1} \in \mathbb{R}^{r_2 r_3}$ by setting its $(b,c)$ entry equal to $\mathcal{C}_{a_i bc}$, and define the $r_2$ and $r_3$ vectors $\{\mathbf{c}_{b_j}\}_{j=1}^{r_2}$ and $\{ \mathbf{c}_{c_k}\}_{k=1}^{r_3}$ similarly.  With this definition, we note that \eqref{grammatrixidentity} is simply the equation for the sample gram matrix for the $r_1$ vectors $\{\mathbf{c}_{a_i}\}_{i=1}^{r_1} \in \mathbb{R}^{r_2 r_3}$ (with similar observations for the other two modes).  Therefore, abusing notation slightly, let $\mathbf{I}_1 \in \mathbb{R}^{p_1 r_1 \times p_1 r_1}$ be the block-diagonal matrix with $r_1 \times r_1$ blocks   given by this sample gram matrix, and let $\mathbf{I}_2$ and $\mathbf{I}_3$ be defined similarly.  With these notations in place, we see that $\mathcal{I}$ can be written in the following form:
\begin{align*}
   \sigma^2 \mathcal{I} &= \begin{pmatrix}
    \mathbf{I}_1 & 0 &0 \\ 0 & \mathbf{I}_2 & 0 \\
    0 & 0 &\mathbf{I}_3
    \end{pmatrix} + \begin{pmatrix} 0 & \mathbf{H}_{12} & \mathbf{H}_{13} \\ \mathbf{H}_{12}\t & 0 & \mathbf{H}_{23} \\ \mathbf{H}_{13}\t & \mathbf{H}_{23}\t & 0
    \end{pmatrix},
\end{align*}
where $\mathbf{H}_{12}, \mathbf{H}_{13}$ and $\mathbf{H}_{23}$ are of appropriate dimension with entries given by the identity in \eqref{eq:Hidentity}.  
\item 
\textbf{Step 2: Three Tensor Algebra Identities}: 
We now simplify the expressions in \eqref{eq:tijkpartial},
\eqref{grammatrixidentity} and \eqref{eq:Hidentity}.  
We recall from \eqref{eq:tijkpartial} that
\begin{align}
     \frac{\del \mathcal{T}_{ijk}}{\del (\U_1)_{i,s}} &= \sum_{b,c} \mathcal{C}_{sbc} (\U_2)_{jb} (\U_3)_{kc} \nonumber\\
     &= \sum_{b,c} \bigg( \sum_{i_1i_2i_3} \mathcal{T}_{i_1 i_2 i_3} (\U_1)_{i_1 s} (\U_2)_{i_2 b} (\U_3)_{i_3 c} \bigg) (\U_2)_{jb} (\U_3)_{kc} \nonumber\\
     &= \sum_{b,c} (\U_2)_{jb} (\U_3)_{kc} \bigg( \sum_{i_1} (\U_1)_{i_1 s} \sum_{i_2 i_3} \mathcal{T}_{i_1 i_2 i_3} (\U_2)_{i_2 b} (\U_3)_{i_3 c} \bigg)\nonumber \\
     &= \sum_{b,c} (\U_2)_{jb} (\U_3)_{kc} \bigg( \U_1\t \mathbf{T}_1 (\U_2 \otimes \U_3) \bigg)_{s, (b,c)} \nonumber\\
     &= \Bigg[ \U_1\t \mathbf{T}_1 (\U_2 \otimes \U_3) (\U_2 \otimes \U_3)\t\Bigg]_{s,(j-1)p_3 + k} \nonumber\\
     &= \Bigg[ \U_1\t \U_1 \mathbf{\Lambda}_1 \mathbf{V}_1 \mathcal{P}_{\U_1 \otimes \U_2} \Bigg]_{s,(j-1)p_3 +k} \nonumber \\
     &= \bigg( \mathbf{\Lambda}_1 \mathbf{V}_1\t \bigg)_{s,(j-1)p_3 +k}. \label{tijkpartial2}
\end{align}
Here we have denoted the index $(b,c)$ to be the matricization index corresponding to the $s,b,c$ element of the underlying tensor.

We now consider the form of the blocks in the definition of $\mathbf{I}_1$.  Recall that each block is given by the matrix whose $(a_1, a_2)$ entry is equal to
\begin{align*}
    \sum_{b,c} \mathcal{C}_{a_1 bc} \mathcal{C}_{a_2bc}.
\end{align*}
We have that
\begin{align*}
     &\sum_{b,c} \mathcal{C}_{a_1 bc} \mathcal{C}_{a_2bc} \\
     &= \sum_{b,c} \bigg(  \sum_{i_1,i_2,i_3}  \mathcal{T}_{i_1 i_2 i_3} (\U_1)_{i_1 a_1} (\U_2)_{i_2 b} (\U_3)_{i_3 c} \bigg)\bigg(  \sum_{j_1,j_2,j_3}  \mathcal{T}_{j_1 j_2 j_3} (\U_1)_{j_1 a_2} (\U_2)_{j_2 b} (\U_3)_{j_3 c} \bigg) \\
     &= \sum_{b,c}\sum_{i_1, i_2,i_3} \mathcal{T}_{i_1i_2i_3} (\U_1)_{i_1 a_1} (\U_2 )_{i_2 b} (\U_3 )_{i_3}\bigg[ \sum_{j_1} (\U_1)_{j_1 a_2} \sum_{j_2 j_3} \mathcal{T}_{j_1 j_2 j_3} (\U_2)_{j_2b} (\U_3)_{j_3 c} \bigg] \\
     &= \sum_{b,c}\sum_{i_1, i_2,i_3} \mathcal{T}_{i_1i_2i_3} (\U_1)_{i_1 a_1} (\U_2 )_{i_2 b} (\U_3 )_{i_3} \bigg[ \sum_{j_1} (\U_1)_{j_1 a_2} \bigg( \mathbf{T}_1 (\U_2 \otimes \U_3 ) \bigg)_{j_1,(b,c)} \bigg] \\
     &= \sum_{b,c}\sum_{i_1, i_2,i_3} \mathcal{T}_{i_1i_2i_3} (\U_1)_{i_1 a_1} (\U_2 )_{i_2 b} (\U_3 )_{i_3} \bigg( \U_1 \t \mathbf{T}_1 (\U_2 \otimes \U_3)\bigg)_{a_2, (b,c)} \\
     &=\sum_{b,c} \bigg( \U_1 \t \mathbf{T}_1 (\U_2 \otimes \U_3)\bigg)_{a_2, (b,c)} \sum_{i_1, i_2,i_3} \mathcal{T}_{i_1i_2i_3} (\U_1)_{i_1 a_1} (\U_2 )_{i_2 b} (\U_3 )_{i_3} \\
    &=\sum_{b,c} \bigg( \U_1 \t \mathbf{T}_1 (\U_2 \otimes \U_3)\bigg)_{a_2, (b,c)} \bigg( \U_1\t \mathbf{T}_1 (\U_2 \otimes \U_3) \bigg)_{a_1, (b,c)} \\
    &= \bigg(\U_1\t \mathbf{T}_1 ( \U_2 \otimes \U_3 ) (\U_2 \otimes \U_3)\t \mathbf{T}_1\t \U_1 \bigg)_{a_1a_2} \\
    &= \bigg( \U_1\t \U_1 \mathbf{\Lambda}_1 \mathbf{V}_1 \mathcal{P}_{\U_2 \otimes \U_3} \mathbf{V}_1\t \mathbf{\Lambda}_1 \U_1\t \U_1 \bigg)_{a_1 a_2} \\
    &= (\mathbf{\Lambda}_1^2)_{a_1 a_2},
\end{align*}
where we have used the fact that $\mathbf{T}_1 \mathcal{P}_{\U_2 \otimes \U_3} = \mathbf{T}_1$ and the fact that $\mathcal{P}_{\U_2 \otimes \U_3} = (\U_2 \otimes \U_3) (\U_2 \otimes \U_3)\t$.  This calculation reveals that $\mathbf{I}_1$ is simply a diagonal matrix with repeated diagonal blocks of $\mathbf{\Lambda}_1^2$, with a similar observation holding for $\mathbf{I}_2$ and $\mathbf{I}_3$, with diagonal blocks equal to $\mathbf{\Lambda}_2^2$ and $\mathbf{\Lambda}_3^2$ respectively.

Finally, we note that from \eqref{eq:Hidentity} that
\begin{align*}
\sigma^2 \E \Bigg[  \frac{\del }{\del (\U_1)_{l_1 a_1}}&\log f( \mathcal{\hat T}; \U_1, \U_2, \U_3, \mathcal{C} ) \frac{\del }{\del (\U_1)_{l_2 a_2}}\log f( \mathcal{\hat T}; \U_1, \U_2, \U_3, \mathcal{C} ) \Bigg] \\
&=  \sum_{i_3} \bigg(\sum_{b,c} \mathcal{C}_{a_1 bc} (\U_2)_{l_2 b} (\U_3)_{i_3 c} \bigg) \bigg( \sum_{a,c} \mathcal{C}_{a b_1 c} (\U_1)_{l_1 a} (\U_3 )_{i_3 c} \bigg).
\end{align*}
Note that
\begin{align*}
    \sum_{b,c} \mathcal{C}_{a_1bc} (\U_2)_{l_2 b} (\U_3)_{i_3 c} %&= \sum_{b,c} (\U_2)_{l_2 b} (\U_3)_{i_3 c}  \\
    &= \sum_{b,c} (\U_2)_{l_2 b} (\U_3)_{i_3 c} \bigg( \sum_{j_1j_2j_3} \mathcal{T}_{j_1j_2j_3} (\U_1)_{j_1a_1} (\U_2)_{j_2 b} (\U_3)_{j_3c} \bigg) \\
    &= \sum_{b,c} (\U_2)_{l_2 b} (\U_3)_{i_3 c} \bigg( \U_1\t \mathbf{T}_1 (\U_2 \otimes \U_3)\bigg)_{a_1,(b,c)} \\
    &= \bigg( \U_1\t \mathbf{T}_1 (\U_2 \otimes \U_3) (\U_2 \otimes \U_3)\t \bigg)_{a_1,(l_2,i_3)} \\
    &= \bigg( \mathbf{\Lambda}_1 \mathbf{V}_1\t \bigg)_{a_1,(l_2,i_3)} \\
    &= (\mathbf{\Lambda}_1)_{a_1} \big( \mathbf{V}_1\big)_{(l_2-1)p_3 + i_3, a_1}.
\end{align*}
By a similar argument,
\begin{align*}
    \bigg( \sum_{a,c} \mathcal{C}_{a b_1 c} (\U_1)_{l_1 a} (\U_3 )_{i_3 c} \bigg) &= (\mathbf{\Lambda}_2)_{b_1} \big( \mathbf{V}_2 \big)_{(l_1 - 1)p_2 + i_3}.
\end{align*}
Therefore,
\begin{align}
    \sum_{i_3} &\bigg(\sum_{b,c} \mathcal{C}_{a_1 bc} (\U_2)_{l_2 b} (\U_3)_{i_3 c} \bigg) \bigg( \sum_{a,c} \mathcal{C}_{a b_1 c} (\U_1)_{l_1 a} (\U_3 )_{i_3 c} \bigg) \nonumber \\
    &= \sum_{i_3}(\mathbf{\Lambda}_2)_{b_1} \big( \mathbf{V}_2 \big)_{(l_1 - 1)p_2 + i_3}(\mathbf{\Lambda}_1)_{a_1} \big( \mathbf{V}_1\big)_{(l_2-1)p_3 + i_3, a_1}, \label{hdef}
\end{align}
with appropriate replacements for different modes. 
\item \textbf{Step 3: Inverting the Fisher Information Matrix}:
Recall that we have that
\begin{align*}
    \sigma^2 \mathcal{I} &= \begin{pmatrix} \mathbf{I}_1 & 0 & 0 \\ 0 &\mathbf{I}_2 & 0 \\ 0 & 0 &\mathbf{I}_3 \end{pmatrix} + \begin{pmatrix} 0 & \mathbf{H}_{12} & \mathbf{H}_{13} \\ \mathbf{H}_{12}\t & 0 & \mathbf{H}_{23} \\ \mathbf{H}_{13}\t & \mathbf{H}_{23}\t & 0
    \end{pmatrix} \\
    &\coloneqq \mathcal{\tilde I} + \mathcal{H}.
\end{align*}
It is straightforward to observe that the matrix $ \mathcal{\tilde I}$ is invertible since it is a diagonal matrix with positive elements.  We note also that
\begin{align*}
\| \mathcal{H} \|&= \Bigg\| \begin{pmatrix} 0 & \mathbf{H}_{12} & \mathbf{H}_{13} \\ \mathbf{H}_{12}\t & 0 & \mathbf{H}_{23} \\ \mathbf{H}_{13}\t & \mathbf{H}_{23}\t & 0
    \end{pmatrix}\Bigg\| \\
    &\leq 3 \max_{jk} \| \mathbf{H}_{jk} \|
\end{align*}
We calculate the maximum spectral norm of each block.  Note that by \eqref{hdef},
\begin{align*}
\|    \mathbf{H}_{12}  \| &\leq \| \mathbf{H}_{12} \|_F \\
&\leq \Bigg[ \sum_{l_1,l_2,a_1,b_1} \bigg( \sum_{i_3}(\mathbf{\Lambda}_2)_{b_1} \big( \mathbf{V}_2 \big)_{(l_1 - 1)p_2 + i_3}(\mathbf{\Lambda}_1)_{a_1} \big( \mathbf{V}_1\big)_{(l_2-1)p_3 + i_3, a_1} \bigg)^2 \Bigg]^{1/2} \\
&\leq \Bigg[ \sum_{l_1,l_2,a_1,b_1,i_3} \bigg( (\mathbf{\Lambda}_2)_{b_1} \big( \mathbf{V}_2 \big)_{(l_1 - 1)p_2 + i_3} \bigg)^2 \bigg( (\mathbf{\Lambda}_1)_{a_1} \big( \mathbf{V}_1\big)_{(l_2-1)p_3 + i_3, a_1}  \bigg)^2 \Bigg]^{1/2} \\
&\leq \kappa^2 \lambda^2 \Bigg[ \sum_{l_1,l_2,i_3} \| \big( \mathbf{V}_2 \big)_{(l_1 - 1)p_2 + i_3,\cdot} \|^2 \| \big( \mathbf{V}_1\big)_{(l_2-1)p_3 + i_3, \cdot} \|^2 \Bigg]^{1/2} \\
&\leq \kappa^2 \lambda^2 \Bigg[ \frac{\mu_0^2 r}{p_2p_3} \sum_{l_1,l_2,i_3}  \| \big( \mathbf{V}_1\big)_{(l_2-1)p_3 + i_3, \cdot} \|^2 \Bigg]^{1/2} \\
&\leq \kappa^2 \lambda^2 \Bigg[ \frac{C \mu_0^2 r}{p} \sum_{l_2,i_3}\| \big( \mathbf{V}_1\big)_{(l_2-1)p_3 + i_3, \cdot} \|^2  \Bigg]^{1/2} \\
&\leq \kappa^2 \lambda^2 \big( \frac{C \mu_0^2 r}{p} \big) \\
&\lesssim \kappa^2 \lambda^2 \mu_0 \sqrt{\frac{r}{p}}.
\end{align*}
Therefore, applying this argument to each block, we obtain
\begin{align*}
    \| \mathcal{H} \| &\lesssim \kappa^2 \lambda^2 \mu_0 \sqrt{\frac{r}{p}}.
\end{align*}
Therefore,
\begin{align*}
    \| \mathcal{\tilde I}\inv \mathcal{H} \| &\leq \| \mathcal{\tilde I}\inv \| \|\mathcal{H} \| \\
    &\leq \frac{1}{\lambda^2} \|\mathcal{H} \| \\
    &\lesssim \kappa^2 \mu_0 \sqrt{\frac{r}{p}} \\
    &\ll 1,
\end{align*}
provided that $\kappa^2  \mu_0 \sqrt{\frac{r}{p}} \ll 1$.

Therefore, the following series expansion is justified:
\begin{align*}
    (\sigma^2\mathcal{I})\inv &= \bigg( \mathcal{I} + \mathcal{H}\bigg)\inv \\
    &= \Bigg[ \mathcal{\tilde I} \bigg( \mathbf{I}_{p_1p_2p_3r_1r_2r_3} + \mathcal{\tilde I}\inv \mathcal{H} \bigg) \Bigg]\inv \\
    &= \bigg(  \mathbf{I}_{p_1p_2p_3r_1r_2r_3} + \mathcal{\tilde I}\inv \mathcal{H} \bigg)\inv \mathcal{\tilde I}\inv \\
    &= \sum_{k=0}^{\infty} \big(  \mathcal{\tilde I}\inv \mathcal{H} \big)^k,
\end{align*}
and hence that
\begin{align*}
    \mathcal{I}\inv &= \sigma^2 \sum_{k=0}^{\infty} \big(  \mathcal{\tilde I}\inv \mathcal{H} \big)^k \mathcal{\tilde I}\inv \\
    &= \sigma^2  \mathcal{\tilde I}\inv + \sigma^2 \sum_{k=1}^{\infty} \big(  \mathcal{\tilde I}\inv \mathcal{H} \big)^k \mathcal{\tilde I}\inv \\
    &= \sigma^2 (1 - o(1)) \mathcal{\tilde I}\inv,
\end{align*}
where the $o(1)$ is taken to be with respect to the positive semidefinite ordering. 
\item 
\textbf{Step 4: Putting It All Together}: First we prove \cref{thm:efficiencyloadings}.   It is immediate that the $m$'th row of any estimator $\mathbf{\tilde U}_k$ of $\U_k$ has covariance lower bounded by the $r_k \times r_k$ submatrix of $\mathcal{I}\inv$.  Therefore,
\begin{align*}
    \mathrm{Var}( e_m\t \utilde_k ) &\succcurlyeq \big(\mathcal{I}\inv\big)_{r_k: r_k} \\
    &\succcurlyeq \sigma^2 (1 - o(1)) \big(\mathcal{\tilde I}\inv\big)_{r_k: r_k} \\
    &\succcurlyeq \sigma^2 (1- o(1)) \mathbf{\Lambda}_k^{-2},
\end{align*}
since the $r_k \times r_k$ submatrix of $\mathcal{\tilde I}\inv$ corresponding to the $m$'th row of $\U_k$ is simply $\mathbf{\Lambda}_k^{-2}$.  This completes the proof of \cref{thm:efficiencyloadings}.

 We now complete the proof of \cref{thm:efficiency}.  We have that
 \begin{align*}
    \big(  &\nabla_{\U_1,\U_2,\U_3} \mathcal{T}_{ijk} \big)\t \mathcal{I}\inv \big(  \nabla_{\U_1,\U_2,\U_3} \mathcal{T}_{ijk} \big) \\
   &= (1- o(1)) \sigma^2  \big(  \nabla_{\U_1,\U_2,\U_3} \mathcal{T}_{ijk} \big)\t \mathcal{\tilde I}\inv \big(  \nabla_{\U_1,\U_2,\U_3} \mathcal{T}_{ijk} \big).
 \end{align*}
 Note that from \eqref{eq:tijkpartial} and \eqref{tijkpartial2} that the vector
 \begin{align*}
     \nabla_{\U_1,\U_2,\U_3} \mathcal{T}_{ijk}
 \end{align*}
 is only nonzero in indices corresponding to rows $i$ $j$ or $k$.  Each of these is then scaled by the diagonal matrix $\mathbf{\Lambda}_1^{-2}$, $\mathbf{\Lambda}_2^{-2}$, or $\mathbf{\Lambda}_3^{-3}$ respectively. 
 Therefore, 
\begin{align*}
   & \big(  \nabla_{\U_1,\U_2,\U_3} \mathcal{T}_{ijk} \big)\t \mathcal{ I}\inv \big(  \nabla_{\U_1,\U_2,\U_3} \mathcal{T}_{ijk} \big) \\
    &= (1- o(1))\sigma^2 \sum_{s_1,s_2} (\mathbf{\Lambda}_1 \mathbf{V}_1\t )_{s_1,(j-1)p_3 + k} (\mathbf{\Lambda}_1)^{-2}_{s_1 s_2} (\mathbf{\Lambda}_1 \mathbf{V}_1\t )_{s_2,(j-1)p_3 + k} \\
    &\quad + (1- o(1))\sigma^2\sum_{s_1, s_2} (\mathbf{\Lambda}_2 \mathbf{V}_2\t )_{s_1,(k-1)p_1 + i} (\mathbf{\Lambda}_2)^{-2}_{s_1 s_2} (\mathbf{\Lambda}_2 \mathbf{V}_2\t )_{s_2,(k-1)p_1 + i} \\
    &\quad + (1- o(1))\sigma^2\sum_{s_1, s_2} (\mathbf{\Lambda}_3 \mathbf{V}_3\t )_{s_1,(i-1)p_2 + j} (\mathbf{\Lambda}_3)^{-2}_{s_1 s_2} (\mathbf{\Lambda}_3 \mathbf{V}_3\t )_{s_2,(i-1)p_2 + j} \\ 
    &= (1- o(1))\sigma^2\sum_{s_1} (\mathbf{\Lambda}_1 \mathbf{V}_1\t )_{s_1,(j-1)p_3 + k} (\mathbf{\Lambda}_1)^{-2}_{s_1 s_1} (\mathbf{\Lambda}_1 \mathbf{V}_1\t )_{s_2,(j-1)p_3 + k} \\
    &\quad +(1- o(1))\sigma^2 \sum_{s_1} (\mathbf{\Lambda}_2 \mathbf{V}_2\t )_{s_1,(k-1)p_1 + i} (\mathbf{\Lambda}_2)^{-2}_{s_1 s_1} (\mathbf{\Lambda}_2 \mathbf{V}_2\t )_{s_1,(k-1)p_1 + i} \\
    &\quad + (1- o(1))\sigma^2\sum_{s_1} (\mathbf{\Lambda}_3 \mathbf{V}_3\t )_{s_1,(i-1)p_2 + j} (\mathbf{\Lambda}_3)^{-2}_{s_1 s_1} (\mathbf{\Lambda}_3 \mathbf{V}_3\t )_{s_2,(i-1)p_2 + j} \\ 
    &= (1- o(1))\sigma^2\sum_{s_1} ( \mathbf{V}_1\t )_{s_1,(j-1)p_3 + k}^2 \\
    &\quad + (1- o(1))\sigma^2\sum_{s_1} ( \mathbf{V}_2\t )_{s_1,(k-1)p_1 + i}^2 \\
    &\quad + (1- o(1))\sigma^2\sum_{s_1} ( \mathbf{V}_3\t )_{s_1,(i-1)p_2 + j}^2 \\ 
    &=(1- o(1))\sigma^2 \bigg( \| e_{(j-1)p_3 + k} \mathbf{V}_1 \|^2 + \| e_{(k-1)p_2 + j} \mathbf{V}_2 \|^2 + \| e_{(i-1)p_2 + j} \mathbf{V}_3 \|^2\bigg).
\end{align*}
This completes the proof of \cref{thm:efficiency}.
\end{itemize}
\end{proof}

\subsection{Proofs of \cref{thm:efficiency_order,thm:efficiency_order_ijk}} \label{sec:prooforders}
In this section we prove the minimax lower bounds for the length of the confidence intervals.  

\subsubsection{Proof of \cref{thm:efficiency_order}}

\begin{proof}[Proof of \cref{thm:efficiency_order}]
%\textcolor{black}{
Without loss of generality we may  assume that $\|\xi\|_{\infty} = \sigma_{\min} = 1$, since the result is invariant to rescaling by these quantities.  We further assume that $k = m = 1$. Our proof mimics the proof of Theorem 3 of \citet{cai_confidence_2017}, but our construction is inspired by a similar construction in \citet{cheng_tackling_2021}.%}

%\textcolor{black}{
Take any $\mathcal{T}$ such that $\mathcal{T} = \mathcal{C} \times_1 \U_1 \times \U_2 \times \U_3 \in \Theta(\lambda,\kappa,\frac{\mu_0}{2},\sigma,\sigma_{\min})$, and suppose $\U_1$ is such that  $\| \big( \U_1 \big)_{\cdot, r_k} \|_{\infty} \leq \frac{\mu_0}{2}\sqrt{\frac{1}{p_1}}$. Let $\mathcal{H}_0$ denote the hypothesis space with point mass at $\mathcal{T}$, and let $\pi_{\mathcal{H}_0}$ denote the prior on this set.  We will construct an alternative hypothesis space $\mathcal{H}_1$ with prior given point mass $\pi_{\mathcal{H}_1}$ of the form
\begin{align*}
    \mathcal{\bar T} = \mathcal{C} \times_1 \mathbf{\bar U}_1 \times_2 \U_2 \times_3 \U_3,
\end{align*}
where $\mathbf{\bar U}_1$ is an orthonormal matrix satisfying certain properties.  By Lemma 1 of \citet{cai_confidence_2017} and properties of the infimum, it holds that 
\begin{align*}
&\inf_{\mathrm{C.I.}_{k,m}^{\alpha}(\xi,\mathcal{Z},\mathcal{T}) \in \mathcal{I}_{\alpha}(\Theta,\xi)} \sup_{\mathcal{T} \in \Theta} \mathbb{E}_{\mathcal{T}} L\big( \mathrm{C.I.}_{k,m}^{\alpha}(\xi,\mathcal{Z},\mathcal{T}) \big) \\ &\geq \inf_{\mathrm{C.I.}_{k,m}^{\alpha}(\xi,\mathcal{Z},\mathcal{T}) \in \mathcal{I}_{\alpha}(\Theta,\xi)} \sup_{\mathcal{T} \in \Theta_0 \cup \Theta_1} \mathbb{E}_{\mathcal{T}} L\big( \mathrm{C.I.}_{k,m}^{\alpha}(\xi,\mathcal{Z},\mathcal{T}) \big) \\
 &\geq \big| \xi\t \big( \mathbf{U}_1 - \mathbf{\bar U}_1 \big)_{1\cdot} \big| \big( 1 - 2\alpha - \mathrm{TV}(f_{\pi_{\mathcal{H}_1}}, f_{\pi_{\mathcal{H}_0}} ) \big),
\end{align*}
where $\mathrm{TV}(\cdot)$ denotes the total variation distance, and $f_{\pi_{\mathcal{H}_i}}$ denotes the density function with the prior $\pi_{\mathcal{H}_i}$ on $\mathcal{H}_i$. We note that
\begin{align*}
    \mathrm{TV}(f_{\pi_{\mathcal{H}_1}}, f_{\pi_{\mathcal{H}_0}} ) &\leq \sqrt{\chi^2(f_{\pi_{\mathcal{H}_1}}, f_{\pi_{\mathcal{H}_0}})}, 
\end{align*}
where
\begin{align*}
    \chi^2(f_1,f_0) &= \int \frac{f_1^2(z)}{f_0(z)} dz- 1.
\end{align*}
We will upper bound $\chi^2(f_{\pi_{\mathcal{H}_1}}, f_{\pi_{\mathcal{H}_0}})$ and lower bound  $\big| \xi\t \big( \mathbf{U}_1 - \mathbf{\bar U}_1 \big)_{1\cdot} \big|$.  %}

%\textcolor{black}{
First we describe our construction.  For a given $\mathbf{U}_1$, let $\mathbf{U'}_1$ be defined as follows. Let $\mathbf{v}$ denote the vector
\begin{align*}
    \mathbf{v} := \frac{1}{\| (\U_1)_{\cdot r_1} + \delta (\mathbf{I} - \U_1 \U_1\t ) e_1\| } \bigg( (\U_1)_{\cdot r_1} + \delta (\mathbf{I} - \U_{1} \U_{1}\t ) e_1 \bigg),
\end{align*}
where we will choose $\delta$ later. We set $\mathbf{\bar U}_1$ to have the first $r_1 - 1$ columns equal to $\U_1$, and we let its $r_1$'th column equal $\mathbf{v}$.  The proof proceeds in steps: first we demonstrate that $\mathcal{\bar T} \in \Theta(\lambda,\kappa,\mu_0,\sigma,\sigma_{\min})$, after which we upper bound the $\chi^2$ distance, and finally we lower bound the difference $\big|\xi\t\big( \U_1 - \mathbf{\bar U}_1\big)_{1\cdot} \big|$.

\begin{itemize}
    \item \textbf{Step 1: Showing $\mathcal{\bar T} \in \Theta(\lambda,\kappa,\mu_0,\sigma,\sigma_{\min}).$} It is clear that the singular values and the singular vectors $\U_2$ and $\U_3$ of $\mathcal{\bar T}$ remain unchanged; consequently, we need only demonstrate that $ \|\mathbf{\bar U}_1\|_{2,\infty} \leq \mu_0 \sqrt{\frac{r_1}{p_1}}$. Therefore, it suffices to demonstrate that $\|\mathbf{v} \|_{\infty} \leq \mu_0 \sqrt{\frac{1}{p_1}}$. First, note that \begin{align*}
    \|  (\U_1)_{\cdot r_1} + \delta ((\mathbf{I} - \U_{1} \U_{1}\t )) e_1 \| = \sqrt{1 + \delta^2 \| ((\mathbf{I} - \U_{1} \U_{1}\t ) ) e_1 \|^2}.
\end{align*}
Consequently,
\begin{align*}
    \| \mathbf{v} \|_{\infty} &\leq \| (\U_1)_{\cdot r_1} - \mathbf{v} \|_{\infty} + \|(\U_1)_{\cdot r_1} \|_{\infty} \\
    &\leq \bigg\| \big(\mathbf{U}_1\big)_{\cdot r_1} - \frac{(\U_1)_{\cdot r_1} + \delta (\mathbf{I} - \U_1 \U_1\t ) e_1}{\sqrt{1 + \delta^2 \| (\mathbf{I} - \U_1 \U_1\t ) e_1 \|^2}} \bigg\| + \frac{\mu_0}{2} \sqrt{\frac{1}{p_1}} \\
    &\leq \bigg\| \big( \U_1 \big)_{\cdot r_1} + \delta \big( \mathbf{I} - \U_1  \U_1\t \big) e_1 \bigg\|_{\infty} \bigg| 1 - \frac{1}{\sqrt{1 + \delta^2 \| (\mathbf{I} - \U_1\U_1\t) e_1\|^2}} \bigg| + \frac{\mu_0}{2} \sqrt{\frac{1}{p_1}} \\
    &\leq \bigg( \frac{\mu_0}{2} \sqrt{\frac{1}{p_1}} + \delta \bigg) \bigg| 1 - \frac{1}{\sqrt{1 + \delta^2 \| (\mathbf{I} - \U_1\U_1\t) e_1\|^2}} \bigg|+ \frac{\mu_0}{2} \sqrt{\frac{1}{p_1}} \\
    &\leq \delta \bigg( \frac{\mu_0}{2} \sqrt{\frac{1}{p_1}} + \delta \bigg) +  \frac{\mu_0}{2} \sqrt{\frac{1}{p_1}}.
\end{align*}
Therefore, $\|\mathbf{v} \|_{\infty} \leq \mu_0 \sqrt{\frac{1}{p_1}}$ as long as 
\begin{align*}
    \delta \bigg( \frac{\mu_0}{2} \sqrt{\frac{1}{p_1}} + \delta \bigg) \leq \frac{\mu_0}{2}  \sqrt{\frac{1}{p_1}}. \numberthis \label{eq:choosedelta}
\end{align*}
We will check this upon choosing $\delta$.
    \item \textbf{Step 2: Upper bounding $\chi^2(f_{\pi_{\mathcal{H}_1}}, f_{\pi_{\mathcal{H}_0}})$}. Under Gaussian noise, we have that
\begin{align*}
    \chi^2(f_{\pi_{\mathcal{H}_1}}, f_{\pi_{\mathcal{H}_0}}) + 1 &= \exp\bigg( \sum_{i,j,k} \frac{(\mathcal{\bar T}_{ijk} - \mathcal{T}_{ijk})^2}{\sigma^2_{ijk}} \bigg) \\
    &\leq \exp\bigg(  \| \mathcal{\bar T} - \mathcal{T} \|_F^2 \bigg) \\
    &\leq \exp\bigg( \|(\mathbf{\bar U}_1 - \U_1) \mathbf{\Lambda}_1 \mathbf{V}_1\t \|_F^2 \bigg) \\
    &\leq \exp\bigg( \|(\mathbf{\bar U}_1 - \U_1) \mathbf{\Lambda}_1  \|_F^2 \bigg) \\
    &\leq \exp\bigg( \lambda_{r_1}^2 \|(\U_1)_{\cdot r_1} - \mathbf{v} \|^2 \bigg),
\end{align*}
where we have used the assumption $\sigma_{\min} = 1$. 
We now bound the quantity inside the exponential. Observe that
\begin{align*}
    \| \big(\mathbf{U}_1\big)_{\cdot r_1} - \mathbf{v} \| &= \bigg\| \big(\mathbf{U}_1\big)_{\cdot r_1} - \frac{(\U_1)_{\cdot r_1} + \delta (\mathbf{I} - \U_1 \U_1\t ) e_1}{\sqrt{1 + \delta^2 \| (\mathbf{I} - \U_1 \U_1\t ) e_1 \|^2}} \bigg\| \\
    &\leq \delta \bigg\|\bigg( \mathbf{I} - \U_1 \U_1\t \bigg) e_1 \bigg\| \\
    &\quad + \bigg( 1 - \frac{1}{\sqrt{1 + \delta^2 \|(\mathbf{I} - \U_1 \U_1\t ) e_1\|}} \bigg) \| (\U_{1})_{\cdot r_1} + \delta (\mathbf{I} - \U_1 \U_1\t ) e_1 \| \\
    &\leq \delta  + \frac{\sqrt{1 + \delta^2\|(\mathbf{I} - \U_1 \U_1\t )e_i\|} - 1}{\sqrt{1 + \delta^2 \| (\mathbf{I} - \U_1 \U_1\t)e_1 \|^2}} (1 + \delta)  \\
    &\leq 3\delta ,
\end{align*}
where the final inequality holds as long as $\delta \leq \frac{1}{12}$.  Therefore,
\begin{align*}
\chi^2(f_{\pi_{\mathcal{H}_1}}, f_{\pi_{\mathcal{H}_0}}) &\leq \exp\bigg( 9\lambda_{r_1}^2 \delta^2 \bigg) - 1.
\end{align*}
\item \textbf{Step 3: Lower bounding $\xi\t (\mathbf{\bar U}_1 - \U_1 \big)_{1\cdot}$. } 
Without loss of generality assume that $(\U_1)_{1 r_1}$ is nonnegative.  Observe that
\begin{align*}
    |\xi\t& (\mathbf{\bar U}_1 - \U_1 \big)_{1\cdot}|\\ &= |\xi_{r_1}| \bigg| (\mathbf{\bar U}_1 - \U_1)_{1r_1} \bigg| \\
    &\geq C \| \xi\|_{\infty} \bigg| (\mathbf{\bar U}_1 - \U_1)_{1r_1} \bigg| \\
    &\geq C\|\xi\|_{\infty} \bigg| \mathbf{v}_{11} - \big(\U_1)_{1r_1} \bigg| \\
    &= C \bigg|\frac{1}{\sqrt{1 + \delta^2 \| (\mathbf{I} - \U_1 \U_1\t ) e_1 \|^2}}\bigg( (\U_1)_{1 r_1} + \delta (\mathbf{I} - \U_1 \U_1\t )_{11} \bigg) - \big(\U_1)_{1r_1} \bigg| \\
    &\geq C\Bigg( \frac{\delta \| |(\mathbf{I} - \U_1 \U_1\t )_{11} \|}{\sqrt{1 + \delta^2 \| (\mathbf{I} - \U_1 \U_1\t ) e_1 \|^2}} - (\U_1)_{11} \bigg( 1 - \frac{1}{\sqrt{1 + \delta^2 \| (\mathbf{I} - \U_1 \U_1\t ) _{11} \|^2}} \bigg) \Bigg) \\
    &\geq C \bigg( \frac{1}{2} \delta \| (\mathbf{I} - \U_1 \U_1\t )_{11} \| - |(\U_1)_{11}| \delta^2 \| (\mathbf{I} - \U_1 \U_1\t )_{11} \|  \bigg) \t \\
    &\geq \frac{C \delta}{4 } \| (\mathbf{I} - \U_1 \U_1\t )_{11} \|.
\end{align*}
The incoherence assumption implies that
\begin{align*}
    \| \U_1 \|_{2,\infty} \leq \frac{\mu_0}{2} \sqrt{\frac{r_1}{p_1}} \leq \frac{1}{2}
\end{align*}
and hence
\begin{align*}
    \| (\mathbf{I} - \U_1 \U_1\t )_{11} \| \geq \frac{1}{2}.
\end{align*}
Therefore,
\begin{align*}
     |\xi\t (\mathbf{\bar U}_1 - \U_1 \big)_{1\cdot}| &\geq \frac{C \delta}{8 }.
     \end{align*}
\item   \textbf{Completing the proof}.     Combining all of our bounds, we obtain that
    \begin{align*}
    \inf \sup    \mathbb{E}  L\big( \mathrm{C.I.}_{1,1}^{\alpha}(\xi,\mathcal{Z},\mathcal{T}) &\geq \big| \xi\t \big( \mathbf{U}_1 - \mathbf{\bar U}_1 \big)_{1\cdot} \big| \big( 1 - 2\alpha - \mathrm{TV}(f_{\pi_{\mathcal{H}_1}}, f_{\pi_{\mathcal{H}_0}} ) \big) \\
        &\geq \frac{C\delta}{9 } \bigg( 1 - 2\alpha - \sqrt{\exp\bigg(9 \lambda_{r_1}^2\delta^2\bigg) - 1} \bigg).
    \end{align*}
    By taking $\delta = \frac{\eps}{\lambda_{r_1}}$ for some sufficiently small constant $\eps$, we complete the proof, provided we can justify \eqref{eq:choosedelta}.  However, with this choice of $\delta$, it holds that
    \begin{align*}
       \delta \bigg( \frac{\mu_0}{2}\sqrt{\frac{1}{p_1}} + \delta \bigg) &= \frac{\eps}{\lambda_{r_k}} \bigg( \frac{\mu_0}{2}\sqrt{\frac{1}{p_1}} +  \frac{\eps}{\lambda_{r_k}}\bigg) \\
       &\leq  \frac{\eps}{\lambda} \bigg( \frac{\mu_0}{2} \sqrt{\frac{1}{p_1}} + \frac{\eps}{\lambda} \bigg) \\
       &\leq \frac{\eps}{\lambda/\sigma}  \bigg( \frac{\mu_0}{2} \sqrt{\frac{1}{p_1}} + \frac{\eps}{\lambda/\sigma} \bigg) \\
       &\leq \frac{\eps}{C_0 \sqrt{p_1}} \bigg( \frac{\mu_0}{2} \sqrt{\frac{1}{p_1}} + \frac{\eps}{C_0 \sqrt{p_1}} \bigg) \\
       &\leq \frac{\mu_0}{2} \sqrt{\frac{1}{p_1}}  \frac{\eps}{C_0 \sqrt{p_1}} \bigg( 1 + \frac{2 \eps}{\mu_0 C_0} \bigg) \\
       &\leq \frac{\mu_0}{2} \sqrt{\frac{1}{p_1}},
    \end{align*}
    where we have used the assumption that $\lambda/\sigma \geq C_0 \sqrt{p_{\max}}$ together with the fact that $\sigma \geq 1$ from the fact that $\sigma_{\min} = 1$.  
    %}
\end{itemize}
This completes the proof.

\end{proof}

\subsubsection{Proof of \cref{thm:efficiency_order_ijk}}
\begin{proof}[Proof of \cref{thm:efficiency_order_ijk}]
The proof is similar to the previous proof, only we generalize our construction slightly.  Without loss of generality we consider $i = j = k = 1$.  First let $\mathcal{T} = \mathcal{C} \times_1 \U_1 \times_2 \U_2 \times_3 \U_3$ be such that $\| \U_k \|_{\infty} \leq \frac{\mu_0}{2} \sqrt{\frac{r_k}{p_k}}$ for each $k$ which is permitted since $\mu_0 > 2$.  Let $\mathcal{T}$ also satisfy 
\begin{align*}
    \sqrt{\| e_1\t \mathbf{V}_1 \|^2 + \|e_1\t \mathbf{V}_2 \|^2 + \| e_1\t \mathbf{V}_3 \|^2} \geq C \mu_0 \frac{r_{\max}^{3/2}\kappa}{\lambda \sqrt{p_{\min}}},
\end{align*}
which is possible whenever $\lambda \geq C_0 \sqrt{p} r \kappa$, which holds under our assumptions on the class $\Theta(\lambda,\kappa,\mu_0,\sigma,\sigma_{\min})$. Define
\begin{align*}
    \mathbf{\bar U}_k = \bigg( \U_k + (\I - \U_k \U_k\t) \Delta_k \bigg) \Bigg[\bigg( \U_k + (\I - \U_k \U_k\t) \Delta_k \bigg)\t \bigg( \U_k + (\I - \U_k \U_k\t) \Delta_k \bigg)  \Bigg]^{-1/2},
\end{align*}
where $\Delta_k$ is the $p_k \times r_k$ matrix whose first row has $l$'th entry equal to $\pm\frac{\eps}{\lambda_l}$ for some constant $\eps$ to be determined later, and whose sign will be chosen later. Define
\begin{align*}
    \mathcal{\bar T} := \mathcal{C} \times_1 \mathbf{\bar U}_1 \times_2 \mathbf{\bar U}_2 \times_3 \mathbf{\bar U}_3.
\end{align*}
 By Lemma 1 of \citet{cai_confidence_2017}, we have that
\begin{align*}
 & \inf_{\mathrm{C.I.}^{\alpha}_{111}(\mathcal{Z},\mathcal{T}) \in \mathcal{I}_{\alpha}(\Theta,\{1,1,1\}} \sup_{\mathcal{T} \in \Theta(\lambda,\kappa,\mu_0)} \mathbb{E}_{\mathcal{T}} L\big( \mathrm{C.I.}_{111}^{\alpha}(\mathcal{Z},\mathcal{T}) \big) \\&\geq | \mathcal{T}_{ijk} - \mathcal{\bar T}_{ijk}| \big( 1 - 2 \alpha - \sqrt{\chi^2(f_{\pi_{\mathcal{H}_1}},f_{\pi_{{\mathcal{H}_0}}}) } \big).
\end{align*}
Similar to the previous proof, we proceed in steps.
\begin{itemize}
    \item \textbf{Step 1: Checking that $\mathcal{\bar T} \in \Theta(\lambda,\kappa,\mu_0)$}. It is evident that $\lambda \leq \lambda_{\min}(\mathcal{\bar T}) \leq \lambda_{\max}(\mathcal{\bar T}) \leq \kappa \lambda$.  Therefore, it suffices to demonstrate that $\|\mathbf{\bar U}_k \|_{2,\infty} \leq \mu_0 \sqrt{\frac{r_k}{p_k}}$.  Observe that
    \begin{align*}
        \mathbf{\bar U}_k &= \U_k + \big( \I - \U_k \U_k\t\big) \Delta_k + \bigg( \U_k + \big( \I - \U_k \U_k\t\big) \Delta_k \bigg) \bigg( \mathbf{C}_{\Delta_k}^{-1/2} - \mathbf{I} \bigg), \numberthis \label{ukbarexpression}
    \end{align*}
    where we have defined $\mathbf{C}_{\Delta_k}$ via
    \begin{align*}
        \mathbf{C}_{\Delta_k} &:= \bigg( \U_k + \big( \I - \U_k \U_k\t \big) \Delta_k \bigg)\t \bigg( \U_k + \big( \I - \U_k \U_k\t \big) \Delta_k \bigg)\\
        &\equiv \mathbf{I} + \Delta_k\t  \big( \I - \U_k \U_k\t \big) \Delta_k,
    \end{align*}
    which is positive definite as long as $\|\Delta_k\t (\I - \U_k \U_k\t) \Delta_k\| < 1,$ which we will demonstrate momentarily.      Therefore,
    \begin{align*}
        \| \mathbf{\bar U}_k \|_{2,\infty} &\leq \| \U_k \|_{2,\infty} + \| (\I - \U_k \U_k\t) \Delta_k \|_{2,\infty} \\
        &\quad + \bigg( \| \U_k \|_{2,\infty} + \| (\I - \U_k \U_k\t) \Delta_k \|_{2,\infty} \bigg) \bigg\| \mathbf{C}_{\Delta_k}^{-1/2} - \mathbf{I} \bigg\|. \numberthis \label{incoherence}
    \end{align*}
    Therefore, it suffices to bound the quantities $\|(\I - \U_k\U_k\t)\Delta_k \|_{2,\infty}$ and $\|\mathbf{C}_{\Delta_k}^{-1/2} - \I \|$. 
    \begin{itemize}
        \item \textbf{Bounding $\|(\I - \U_k\U_k\t)\Delta_k \|_{2,\infty}$}.  First, observe that we can write $\Delta_k$ via $\Delta_k = \eps \begin{pmatrix} \mathbf{v}\t \\ \mathbf{0} \end{pmatrix} \mathbf{\Lambda}_k\inv$, where $\mathbf{v} \in \{-1,1\}^{r_k}$ is a vector of signs.  Therefore,
        \begin{align*}
            \| (\mathbf{I} - \U_k \U_k\t) \Delta_k \|_{2,\infty} &\leq \frac{\eps}{\lambda} \| (\mathbf{I} - \U_k \U_k\t)\begin{pmatrix} \mathbf{v}\t \\ \mathbf{0} \end{pmatrix} \|_{2,\infty}.
        \end{align*}
Since $\mathbf{v}$ is a vector of signs, the matrix $ (\mathbf{I} - \U_k \U_k\t)\begin{pmatrix} \mathbf{v}\t \\ \mathbf{0} \end{pmatrix}$ is simply the rank one matrix whose columns are the entries of the first row of the matrix $(\mathbf{I} - \U_k \U_k\t)$.  Consequently, $\| (\mathbf{I} - \U_k \U_k\t)\begin{pmatrix} \mathbf{v}\t \\ \mathbf{0} \end{pmatrix}\|_{2,\infty} \leq \| (\mathbf{I} - \U_k \U_k\t)\begin{pmatrix} \mathbf{v}\t \\ \mathbf{0} \end{pmatrix}\| \leq \sqrt{r_k}$.  Combining these bounds yields that
\begin{align*}
    \| (\I - \U_k \U_k\t ) \Delta_k \|_{2,\infty} &\leq \frac{\eps \sqrt{r_k}}{\lambda}.
\end{align*}
        \item \textbf{Bounding $\|\mathbf{C}_{\Delta_k}^{-1/2} - \I \|$}. 
        Observe that
        \begin{align*}
            \| \mathbf{C}_{\Delta_k}^{-1/2} - \I \| &\leq \| \mathbf{C}_{\Delta_k}^{-1/2} ( \I - \mathbf{C}_{\Delta_k}^{1/2}) \| \leq \| \mathbf{C}_{\Delta_k}^{-1/2} \| \| \I - \mathbf{C}_{\Delta_k}^{1/2} \|.
        \end{align*}
        By Theorem 6.2 of \citet{higham_functions_2008}, it holds that
        \begin{align*}
            \| \I - \mathbf{C}_{\Delta_k}^{1/2} \| &\leq \frac{1}{\lambda_{\min}^{1/2}(\mathbf{C}_{\Delta_k}) + 1} \| \I - \mathbf{C}_{\Delta_k} \| \\
            &\leq \| \Delta_k\t (\mathbf{I} - \U_k \U_k\t)^2 \Delta_k \| \\
            &\leq \frac{\eps^2 r_k}{\lambda^2},
        \end{align*}
        where we have implicitly implied a similar argument to the previous bound.  As a result, we see that as long as $1 - \frac{\eps^2}{\lambda^2} \geq \frac{1}{2}$, we have that $\lambda_{\min}(\mathbf{C}_{\Delta_k}) \geq \frac{1}{2}$ by Weyl's inequality.  Therefore,
        \begin{align*}
            \| \mathbf{C}_{\Delta_k}^{1/2} - \I \| \leq 2 \frac{\eps^2 r_k}{\lambda^2}.
        \end{align*}
    \end{itemize}
    Combining our bounds and plugging this inequality into \eqref{incoherence}, we have that
    \begin{align*}
        \| \mathbf{\bar U}_k \|_{2,\infty} &\leq \frac{\mu_0}{2} \sqrt{\frac{r_k}{p_k}} + \frac{\eps\sqrt{r_k}}{\lambda} + \bigg( \frac{\mu_0}{2} \sqrt{\frac{r_k}{p_k}} + \frac{\eps\sqrt{r_k}}{\lambda} \bigg) \frac{2\eps^2 r_k}{\lambda^2} \\
        &\leq  \frac{\mu_0}{2} \sqrt{\frac{r_k}{p_k}} + \frac{\eps\sqrt{r_k} }{\lambda/\sigma} +  \bigg( \frac{\mu_0}{2} \sqrt{\frac{r_k}{p_k}} + \frac{\eps \sqrt{r_k}}{\lambda/\sigma} \bigg) \frac{2\eps^2 r_k}{(\lambda/\sigma)^2} \\
        &\leq \frac{\mu_0}{2} \sqrt{\frac{r_k}{p_k}} + \frac{\eps\sqrt{r_k} }{C_0 \sqrt{p_k}} +  \bigg( \frac{\mu_0}{2} \sqrt{\frac{r_k}{p_k}} + \frac{\eps\sqrt{r_k}}{C_0 \sqrt{p_k}} \bigg) \frac{2\eps^2 r_k}{C_0^2 p_k} \\
&\leq\mu_0\sqrt{\frac{r_k}{p_k}},
    \end{align*}
    where the final inequality holds for $\eps$ sufficiently small, the second inequality holds from the fact that $\sigma \geq \sigma_{\min} = 1$, and the penultimate inequality holds from the assumption $\lambda/\sigma \geq C_0 \kappa r_{\max} \sqrt{p_{\max}}$.  
    \item \textbf{Step 2: Upper bounding $\chi^2(f_{{\pi_{\mathcal{H}_1}}},f_{{\pi_{\mathcal{H}_0}}})$} Similar to the proof of \cref{thm:efficiency_order}, it holds that
\begin{align*}
    \chi^2 ( f_{{\pi_{\mathcal{H}_1}}}, f_{\pi_{\mathcal{H}_0}}) &\leq \exp\bigg( \| \mathcal{\bar T} - \mathcal{T} \|_F^2 \bigg) - 1.
\end{align*}
Therefore,
\begin{align*}
    \| \mathcal{\bar T - \mathcal{T}} \|_F &= \| \mathcal{C} \times_1 \mathbf{\bar U}_1 \times_2 \mathbf{\bar U}_2 \times_3 \mathbf{\bar U}_3 - \mathcal{C} \times_1 \U_1 \times_2 \U_2 \times_3 \U_3 \|_F \\
    &\leq \| \mathcal{C} \times_1 (\mathbf{\bar U}_1 - \U_1) \times_2 \mathbf{\bar U}_2 \times_3 \mathbf{\bar U}_3 \|_F \\
    &\quad + \| \mathcal{C} \times_1 \U_1 \times_2 (\mathbf{\bar U}_2 - \U_2) \times_3 \mathbf{\bar U}_3 \|_F \\
    &\quad + \| \mathcal{C} \times_1 \U_1 \times_2 \U_2 \times_3 ( \mathbf{\bar U}_3 - \U_3) \|_F \\
    &= \| (\mathbf{\bar U}_1 - \U_1) \mathcal{M}_1(\mathcal{C}) (\mathbf{\bar U}_2 \otimes \mathbf{\bar U}_3)\t \|_F \\
    &\quad + \| (\mathbf{\bar U}_2 - \U_2) \mathcal{M}_2(\mathcal{C}) (\mathbf{ U}_1 \otimes \mathbf{\bar U}_3)\t \|_F \\
     &\quad +  \| (\mathbf{\bar U}_3 - \U_3) \mathcal{M}_3(\mathcal{C}) (\mathbf{ U}_1 \otimes \mathbf{ U}_2)\t \|_F.
\end{align*}
We now consider an upper bound for the first term; the remaining two terms are similar.  From \eqref{ukbarexpression} we have that
\begin{align*}
    \| (\mathbf{\bar U}_1 &- \U_1) \mathcal{M}_1(\mathcal{C}) (\mathbf{\bar U}_2 \otimes \mathbf{\bar U}_3)\t \|_F \\
    &\leq  \| (\I - \U_1 \U_1\t )\Delta_1 \mathcal{M}_1(\mathcal{C}) (\mathbf{\bar U}_2 \otimes \mathbf{\bar U}_3)\t \|_F \\
    &\quad + \bigg\|\bigg(\U_1+  (\I - \U_1 \U_1\t )\Delta_1 \bigg) \bigg( \mathbf{C}_{\Delta_1}^{-1/2}- \I\bigg) \mathcal{M}_1(\mathcal{C}) (\mathbf{\bar U}_2 \otimes \mathbf{\bar U}_3)\t \bigg\|_F \\
    &\leq \underbrace{\| (\I - \U_1 \U_1\t )\Delta_1 \mathcal{M}_1(\mathcal{C}) (\mathbf{\bar U}_2 \otimes \mathbf{\bar U}_3)\t \|}_{\alpha_1} \\
    &\quad + \underbrace{\bigg\|  (\I - \U_1 \U_1\t )\Delta_1 \bigg( \mathbf{C}_{\Delta_1}^{-1/2}- \I\bigg) \mathcal{M}_1(\mathcal{C}) (\mathbf{\bar U}_2 \otimes \mathbf{\bar U}_3)\t \bigg\|}_{\alpha_2} \\
    &\quad + \underbrace{\bigg\|  \U_1  \bigg( \mathbf{C}_{\Delta_1}^{-1/2}- \I\bigg) \mathcal{M}_1(\mathcal{C}) (\mathbf{\bar U}_2 \otimes \mathbf{\bar U}_3)\t \bigg\|_F}_{\alpha_3},
\end{align*}
where we have passed from the Frobenius norm to the operator norm in the first two terms since the matrix $(\mathbf{I} - \U_1 \U_1\t)\Delta_k$ is rank one.  We now bound $\alpha_1$ through $\alpha_3$.
\begin{itemize}
    \item \textbf{Bounding $\alpha_1$.} Recall that $\Delta_1 = \eps \begin{pmatrix} \mathbf{v}\t \\ 0 \end{pmatrix} \mathbf{\Lambda}_k\inv$, where $\mathbf{v}$ is a matrix of signs.  Therefore,
    \begin{align*}
        \alpha_1 &\leq \eps \| \big( \I - \U_1 \U_1\t \big) \begin{pmatrix} \mathbf{v}\t \\ 0 \end{pmatrix} \mathbf{\Lambda}_k\inv \mathcal{M}_1(\mathcal{C}) \big( \mathbf{\bar U}_2 \otimes \mathbf{\bar U}_3\big)\t \|.
    \end{align*}
    Next, note that $\mathcal{M}_1(\mathcal{C}) (\mathbf{\bar U}_2\otimes \mathbf{\bar U}_3)\t$ has the same nonzero singular values as $\mathcal{M}_1(\mathcal{C})$ and hence also $\mathcal{T}$.  Therefore, $\| \mathbf{\Lambda}_k\inv \mathcal{M}_1(\mathcal{C}) \big( \mathbf{\bar U}_2 \otimes \mathbf{\bar U}_3\big)\t \| =  1$.  As a consequence,
    \begin{align*}
        \alpha_1 \leq \eps.
    \end{align*}
    \item \textbf{Bounding $\alpha_2$.}  We observe that
    \begin{align*}
        \alpha_2 &\leq \frac{\eps \sqrt{r_1}}{\lambda} \lambda_1 \| \mathbf{C}_{\Delta_1}^{-1/2} - \I \| \leq \eps \frac{\eps^2 \kappa r^{3/2}}{\lambda^2}\leq \eps \frac{\eps^2 \kappa r^{3/2}}{(\lambda/\sigma)^2}\leq \eps,
    \end{align*}
    where we have used the assumption that $\lambda/\sigma \geq C_0 r \kappa \sqrt{p}$, together with the previous bounds.
    \item \textbf{Bounding $\alpha_3$.} We have that
    \begin{align*}
        \alpha_3 &\leq \| \mathbf{C}_{\Delta_1}^{-1/2} - \I \| \lambda_1 \sqrt{r_1}\\
        &\leq \eps \frac{\eps r_1^{3/2} \kappa }{\lambda} \\
        &\leq \eps\frac{\eps r_1^{3/2} \kappa }{\lambda/\sigma} \\
        &\leq \eps,
    \end{align*}
    where we have used the assumptions $\lambda/\sigma \geq C_0 \kappa r \sqrt{p}$ and $r_k \leq p_{\max}^{1/2}$.  
\end{itemize}
As a consequence, we have that
\begin{align*}
    \| \mathcal{\bar T} - \mathcal{T} \|_F &\leq 9 \eps.
\end{align*}Therefore,
\begin{align*}
    \chi^2(f_{{\pi_{\mathcal{H}_1}}}, f_{\pi_{\mathcal{H}_0}}) \leq \exp\bigg( 81 \eps^2 \bigg) - 1.
\end{align*}
    \item \textbf{Step 3: Lower Bounding $|\mathcal{T}_{ijk} - \mathcal{\bar T}_{ijk}|$}.
Suppose that
\begin{align*}
   \mathcal{\bar T}_{ijk} - \mathcal{T}_{ijk} &= \eps \bigg(     \mathbf{v}_1\t \mathbf{V}_1\t e_{1} +    \mathbf{v}_2\t  \mathbf{V}_2\t e_{1} + \mathbf{v}_3\t  \mathbf{V}_3\t e_{1} \bigg) \bigg( 1 + o(1) \bigg), \numberthis \label{eq:desiredbound}
\end{align*}
where now we choose the sign of $\mathbf{v}_k \in \{-1,1\}^{r_k}$ by taking the sign of the entries of $e_{1}\t \mathbf{V}_k$. For these choices of $\mathbf{v}_k$, \eqref{eq:desiredbound} implies
\begin{align*}
    |  \mathcal{\bar T}_{111} - \mathcal{T}_{111} | &\geq c \eps \bigg( \big\| e_{1}\t \mathbf{V}_1 \big\|_1 + \big\| e_{1}\t \mathbf{V}_2 \big\|_1 + \big\| e_{1}\t \mathbf{V}_3 \big\|_1 \bigg) \\
    &\geq c \eps \bigg( \big\| e_{1}\t \mathbf{V}_1 \big\|_2 + \big\| e_{1}\t \mathbf{V}_2 \big\|_2 + \big\| e_{1}\t \mathbf{V}_3 \big\|_2 \bigg) \\
    &\geq c \eps \sqrt{ \big\|e_{1}\t \mathbf{V}_1 \big\|_2^2 + \big\| e_{1}\t \mathbf{V}_2 \big\|_2^2 + \big\| e_{1}\t \mathbf{V}_3 \big\|_2^2} . \numberthis \label{lowerboundtoprovetoday}
\end{align*}
We therefore prove \eqref{eq:desiredbound}.  Observe that
\begin{align*}
    \mathcal{\bar T}_{ijk} - \mathcal{T}_{ijk} &= \sum_{l_1=1}^{r_1} \sum_{l_2=1}^{r_2} \sum_{l_3=1}^{r_3} \mathcal{C}_{l_1l_2l_3} \bigg(\big(\mathbf{\bar U}_1\big)_{il_1} \big(\mathbf{\bar U}_2\big)_{jl_2} \big(\mathbf{\bar U}_3\big)_{kl_3} -  \big(\U_1\big)_{il_1} \big(\U_2\big)_{jl_2} \big(\U_3\big)_{kl_3} \bigg).
    \end{align*}
According to \eqref{ukbarexpression} we have that
\begin{align*}
    \mathbf{\bar U}_k - \U_k %&=  \big( \I - \U_k \U_k\t\big) \Delta_k + \bigg( \U_k + \big( \I - \U_k \U_k\t\big) \Delta_k \bigg) \bigg( \mathbf{C}_{\Delta_k}^{-1/2} - \mathbf{I} \bigg) \\
    &= \Delta_k  +  \underbrace{\bigg( \U_k + (\I -\U_k\U_k\t)\Delta_k \bigg) \bigg( \mathbf{C}_{\Delta_k}^{-1/2} - \I \bigg) - \U_k \U_k\t \Delta_k }_{\Gamma_k},
\end{align*}
allowing us to write
\begin{align*}
    \mathbf{\bar U}_k = \U_k +\Delta_k  + \Gamma_k.
\end{align*}
From this expression, we have that
\begin{align*}
  \big(\mathbf{\bar U}_1\big)_{il_1} &\big(\mathbf{\bar U}_2\big)_{jl_2} \big(\mathbf{\bar U}_3\big)_{kl_3} -  \big(\U_1\big)_{il_1} \big(\U_2\big)_{jl_2} \big(\U_3\big)_{kl_3} \\
  &=  \big(\mathbf{ U}_1 + \Delta_1 + \Gamma_1\big)_{il_1} \big(\mathbf{ U}_2 + \Delta_2 + \Gamma_2\big)_{jl_2} \big(\mathbf{ U}_3 + \Delta_3 + \Gamma_3\big)_{kl_3} \\
  &\quad -  \big(\U_1\big)_{il_1} \big(\U_2\big)_{jl_2} \big(\U_3\big)_{kl_3} \\
  %&= \big(\U_1 + (\I-\U_1 \U_1\t)\Delta_1 + \Gamma_1 \big)_{il_1} \big(\U_2 +(\I-\U_2 \U_2\t) \Delta_2 + \Gamma_2 \big)_{jl_2} \big(\U_3 +(\I-\U_3 \U_3\t) \Delta_3 + \Gamma_3\big)_{kl_3} \\
%  &\quad -  \big(\U_1\big)_{il_1} \big(\U_2\big)_{jl_2} \big(\U_3\big)_{kl_3} \\
  &= (\Delta_1)_{1l_1} (\U_2)_{2l_2} (\U_3)_{3l_3} + (\U_1)_{1l_1}  \Delta_2  \U_3 + (\U_1)_{1l_1} \U_2 \Delta_3 + \mathcal{R}_{l_1l_2l_3},
\end{align*}
where $\mathcal{R}_{l_1l_2l_3}$ contains all of the cross terms with at least two appearances of $\Delta$ or at least one appearance of $\Gamma$.  

We will demonstrate that $\mathcal{R}$ is a lower-order term. First, note that $\Delta_k$ and $\Gamma_k$ satisfy
\begin{align*}
    \| e_1\t \Delta_k \| &\leq \frac{\eps \sqrt{r_k}}{\lambda}; \\
    \|e_1\t \Gamma_k \| &\leq \bigg( \| \U_k \|_{2,\infty} + \frac{\eps \sqrt{r_k}}{\lambda} \bigg) \frac{2 \eps^2 r_k}{\lambda^2} + \| \U_k \|_{2,\infty} \frac{\eps \sqrt{r_k}}{\lambda} \\
    &\leq \frac{4 \eps \mu_0 r_k}{\lambda \sqrt{p_k}}. 
\end{align*}
We now bound $\mathcal{R}$ directly.  There are 8 possible subcases over the possible times $\Delta,\Gamma$, and $\U$ appear, subject to the fact that $\Delta$ must appear with either another $\Delta$ or with $\Gamma$, and $\U$ does not appear three times. Consider, for example the case that $\Delta$ appears twice and $\U$ appears once.  In this case we have
    \begin{align*}
        \bigg| \sum_{l_1l_2l_3} \mathcal{C}_{l_1l_2l_3} \big( \Delta_{1} \big)_{1l_1} \big( \Delta_2 \big)_{1l_2} \big( \U_3 \big)_{1l_3} \bigg| &= \bigg| e_1 \t \U_3 \mathcal{M}_3(\mathcal{C}) \big( \Delta_1 \otimes \Delta_2 \big) e_1 \bigg| \\
        &\leq \| e_1\t \U_3 \| \lambda_1 \| e_1\t \Delta_1 \| \| e_1\t \Delta_2 \| \\
        &\leq \mu_0 \sqrt{\frac{r_k}{p_k}} \frac{\eps^2 r_k \kappa }{\lambda}.
    \end{align*}
The other possible subcases are all extremely similar, and all satisfy the same upper bound up to some constant by virtue of the bounds on $\|\Gamma_k\|_{2,\infty}$ and $\|\Delta_k\|_{2,\infty}$.  As a result, there exists some universal constant $C$ such that
\begin{align*}
    \bigg| \sum_{l_1l_2l_3} \mathcal{C}_{l_1l_2l_3} \mathcal{R}_{l_1l_2l_3} \bigg| &\leq C \mu_0 \sqrt{\frac{r_k}{p_k}} \frac{\eps^2 r_k \kappa }{\lambda}.
\end{align*}
% \begin{itemize}
%     \item \textbf{Case 1: $\Delta \Delta \U$}. 
%     \item \textbf{Case 2: $\Delta \Delta \Gamma$}.  Similar to the previous case, we have that
%     \begin{align*}
%         \bigg| \sum_{l_1l_2l_3} \mathcal{C}_{l_1l_2l_3} \big( \Delta_{1} \big)_{1l_1} \big( \Delta_2 \big)_{1l_2} \big( \Gamma_3 \big)_{1l_3} \bigg| &= \bigg| e_1 \t \Gamma_3 \mathcal{M}_3(\mathcal{C}) \big( \Delta_1 \otimes \Delta_2 \big) e_1 \bigg| \\
%         &\leq \| e_1\t \Gamma_3 \| \lambda_1 \| e_1\t \Delta_1 \| \| e_1\t \Delta_2 \| \\
%         &\leq \frac{4\eps \mu_0 r_k}{\lambda \sqrt{p_k}} \frac{\eps^2 r_k \kappa }{\lambda}.
%     \end{align*}
%     \item \textbf{Case 3: $\Delta \Delta \Delta$} Similar to the previous cases, 
%     \begin{align*}
%         \bigg| \sum_{l_1l_2l_3} \mathcal{C}_{l_1l_2l_3} \big( \Delta_{1} \big)_{1l_1} \big( \Delta_2 \big)_{1l_2} \big( \Delta_3 \big)_{1l_3} \bigg| &= \bigg| e_1 \t \Delta_3 \mathcal{M}_3(\mathcal{C}) \big( \Delta_1 \otimes \Delta_2 \big) e_1 \bigg| \\
%         &\leq \| e_1\t \Delta_3 \| \lambda_1 \| e_1\t \Delta_1 \| \| e_1\t \Delta_2 \| \\
%         &\leq \frac{4\eps \mu_0 r_k}{\lambda \sqrt{p_k}} \frac{\eps^2 r_k \kappa }{\lambda}.
%     \end{align*}
%     \item \textbf{Case 4: $\Delta \Gamma \U$}
%     \item \textbf{Case 5: $\Delta \Gamma \Gamma$}
%     \item \textbf{Case 6: $\Gamma \U \U$}
%     \item \textbf{Case 7: $\Gamma \Gamma \U$}
%     \item \textbf{Case 8: $\Gamma \Gamma \Gamma$}
% \end{itemize}
Therefore, we have that
\begin{align*}
    & \mathcal{\bar T}_{ijk} - \mathcal{T}_{ijk} \\
     &= \sum_{l_1=1}^{r_1} \sum_{l_2=1}^{r_2} \sum_{l_3=1}^{r_3} \mathcal{C}_{l_1 l_2 l_3} \bigg( (\Delta_1)_{1l_1} (\U_2)_{1l_2} (\U_3)_{1l_3} + \U_1)_{1l_1} \Delta_2  \U_3 + \U_1)_{1l_1} \U_2 \Delta_3 + \mathcal{R}_{l_1l_2l_3} \bigg) \\
     &= e_1\t \Delta_1 \mathcal{M}_1(\mathcal{C}) (\U_2 \otimes \U_3)\t e_1 + e_1\t \Delta_2 \mathcal{M}_2(\mathcal{C}) (\U_1 \otimes \U_3)\t e_1 +
     e_1\t \Delta_3 \mathcal{M}_3(\mathcal{C}) (\U_2 \otimes \U_1)\t e_1 \\
     &\quad + \sum_{l_1,l_2l_3}\mathcal{C}_{l_1l_2l_3} \mathcal{R}_{l_1l_2l_3}.
\end{align*}
As a consequence,
\begin{align*}
  \bigg|   &\mathcal{\bar T}_{ijk} - \mathcal{T}_{ijk} \bigg| \\&\geq \bigg| e_1\t \Delta_1 \mathcal{M}_1(\mathcal{C}) (\U_2 \otimes \U_3)\t e_1 + e_1\t \Delta_2 \mathcal{M}_2(\mathcal{C}) (\U_1 \otimes \U_3)\t e_1 +
     e_1\t \Delta_3 \mathcal{M}_3(\mathcal{C}) (\U_2 \otimes \U_1)\t e_1 \bigg| \\
     &\qquad - C \mu_0 \sqrt{\frac{r_k}{p_k}} \frac{\eps^2 \mathbf{r}_k \kappa}{\lambda}.
\end{align*}
Next, recall that the $l$'th entry  of the first row of   $\Delta_k$ is of the form $\frac{\pm \eps}{\lambda_l}$, and hence
\begin{align*}
    e_1\t \Delta_1 \mathcal{M}_1(\mathcal{C}) \big( \U_2 \otimes \U_3 \big)\t e_1 &= \eps \mathbf{v}_1\t \mathbf{\Lambda}_k\inv \mathcal{M}_1(\mathcal{C}) \big( \U_2 \otimes \U_3 \big)\t e_1 = \eps \mathbf{v}_1\t \mathbf{\Lambda}_k\inv \mathbf{\Lambda}_k \mathbf{V}_1\t e_1 \\
    &= \eps  \mathbf{v}_1\t \mathbf{V}_1\t e_1,
\end{align*}
where $\mathbf{v}_1$ is a $r_k$-dimensional vector of signs.  Since the signs have not been specified, we are free to select them to match the signs of $e_1\t \mathbf{V}_1$, and hence
\begin{align*}
    \bigg| e_1\t &\Delta_1 \mathcal{M}_1(\mathcal{C}) (\U_2 \otimes \U_3)\t e_1 + e_1\t \Delta_2 \mathcal{M}_2(\mathcal{C}) (\U_1 \otimes \U_3)\t e_1 +
     e_1\t \Delta_3 \mathcal{M}_3(\mathcal{C}) (\U_2 \otimes \U_1)\t e_1 \bigg| \\
     &\geq \eps \bigg( \| e_1 \t \mathbf{V}_1 \|_1 + \| e_1 \t \mathbf{V}_2 \|_1 +\| e_1 \t \mathbf{V}_3 \|_1\bigg) \\
     &\geq \eps \bigg( \| e_1 \t \mathbf{V}_1 \|_2 + \| e_1 \t \mathbf{V}_2 \|_2 +\| e_1 \t \mathbf{V}_3 \|_2\bigg) \\
     &\geq \eps \sqrt{\| e_1 \t \mathbf{V}_1 \|_2^2 + \| e_1 \t \mathbf{V}_2 \|_2^2 +\| e_1 \t \mathbf{V}_3 \|_2^2}.
\end{align*}
In addition, we recall that we have assumed that $\mathcal{T}$ was selected such that
\begin{align*}
    \sqrt{\| e_1 \t \mathbf{V}_1 \|_2^2 + \| e_1 \t \mathbf{V}_2 \|_2^2 +\| e_1 \t \mathbf{V}_3 \|_2^2} &\geq C \mu_0 \frac{r_{\max}^{3/2} \kappa}{\lambda \sqrt{p_{\min}}},
\end{align*}
which is possible whenever $\lambda \geq C_0 r_{\max} \kappa$, which holds under the assumption $\lambda \geq C_0 \kappa \sqrt{p_{\max} }$ and the assumption $r_{\max} \leq p_{\min}^{1/2}$. Therefore,
\begin{align*}
    | \mathcal{T}_{ijk} - \mathcal{\bar T}_{ijk} | %&\geq %\eps \big( 1 - o(1) \big) \sqrt{\| e_1 \t \mathbf{V}_1 \|_2^2 + \| e_1 \t \mathbf{V}_2 \|_2^2 +\| e_1 \t \mathbf{V}_3 \|_2^2} \\
    &\geq \frac{\eps}{2} \sqrt{\| e_1 \t \mathbf{V}_1 \|_2^2 + \| e_1 \t \mathbf{V}_2 \|_2^2 +\| e_1 \t \mathbf{V}_3 \|_2^2}.
\end{align*}
    \item \textbf{Step 4: Completing the proof}.
    Combining all of our results, we have that
    \begin{align*}
          &\inf_{\mathrm{C.I.}^{\alpha}_{111} (\mathcal{Z},\mathcal{T}) \in \mathcal{I}_{\alpha}(\Theta,\{1,1,1\}} \sup_{\mathcal{T} \in \Theta(\lambda,\kappa,\mu_0)} \mathbb{E}_{\mathcal{T}} L\big( \mathrm{C.I.}_{111}^{\alpha}(\mathcal{Z},\mathcal{T}) \big)\\
          &\geq \frac{\eps}{2}\sqrt{\| e_1 \t \mathbf{V}_1 \|_2^2 + \| e_1 \t \mathbf{V}_2 \|_2^2 +\| e_1 \t \mathbf{V}_3 \|_2^2}  \big( 1 - 2 \alpha - \sqrt{\exp\bigg(81 \eps^2 \bigg) - 1} \big) \\
          &\geq c \sqrt{\| e_1 \t \mathbf{V}_1 \|_2^2 + \| e_1 \t \mathbf{V}_2 \|_2^2 +\| e_1 \t \mathbf{V}_3 \|_2^2},
    \end{align*}
    provided $\eps$ is taken sufficiently small. 
\end{itemize}
This completes the proof.    
\end{proof}

%\end{supplement}
%\bibliographystyle{plainnat}
%\bibliographystyle{tensors-perturb-2-infty/plainnat_JA}
%\bibliography{tensor-perturb-2-infinity/reference,tensor-perturb-2-infinity/tensors}
\section{Additional Simulations} \label{sec:additionalsimulations}
In this section we continue the simulation study from \cref{sec:simulations}.  The setup remains the same as in that section.  
\ \\ \ \\
\noindent
\textbf{Approximate Gaussianity of $\uhat_{1}$}: 
First, we examine the approximate Gaussianity implied by \cref{thm:eigenvectornormality2_v1}.  
In \cref{fig:distributionaltheory1}, we plot the approximate Gaussianity of the outputs of \cref{al:tensor-power-iteration} with $p = 100$ and $r= 2$, and in \cref{fig:distributionaltheory2}, we plot the same with $p = 150$.  To obtain this distribution we calculate the true population covariance $\mathbf{\Gamma}_1^{(1)}$ as predicted by \cref{thm:eigenvectornormality2_v1}, and for each iteration we obtain a single point of the form $\big(\mathbf{\Gamma}_1^{(1)}\big)^{-1/2} \big( \uhat_1^{(t)} \mathbf{W} - \U_1 \big)_{1\cdot},$ where $\mathbf{W} = \mathrm{sgn}(\uhat_1^{(t)}, \U_1)$. According to \cref{thm:eigenvectornormality_v1}, this term is approximately Gaussian with $2\times 2$ identity covariance.  We plot both the theoretical (dotted) and empirical (solid) $95\%$ confidence ellipse.  
 \begin{figure*}[htpb]
        \subfloat[$\lambda/\sigma = p^{3/4}$]
                {%
            \includegraphics[width=.3\linewidth]{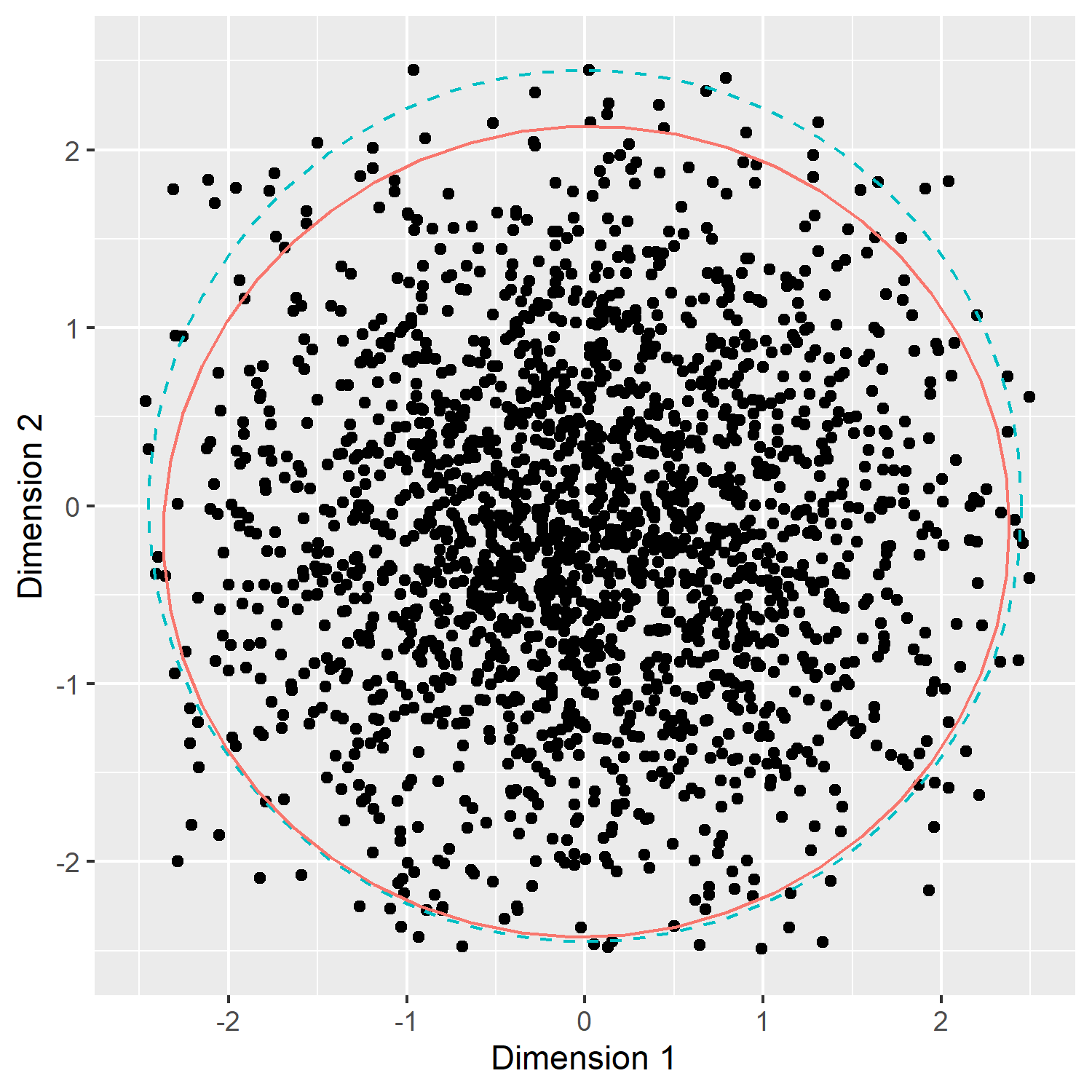}%
            \label{subfig:aveerror}%
        }
        \subfloat[$\lambda/\sigma = p^{7/8}$]{ %
            \includegraphics[width=.3\linewidth]{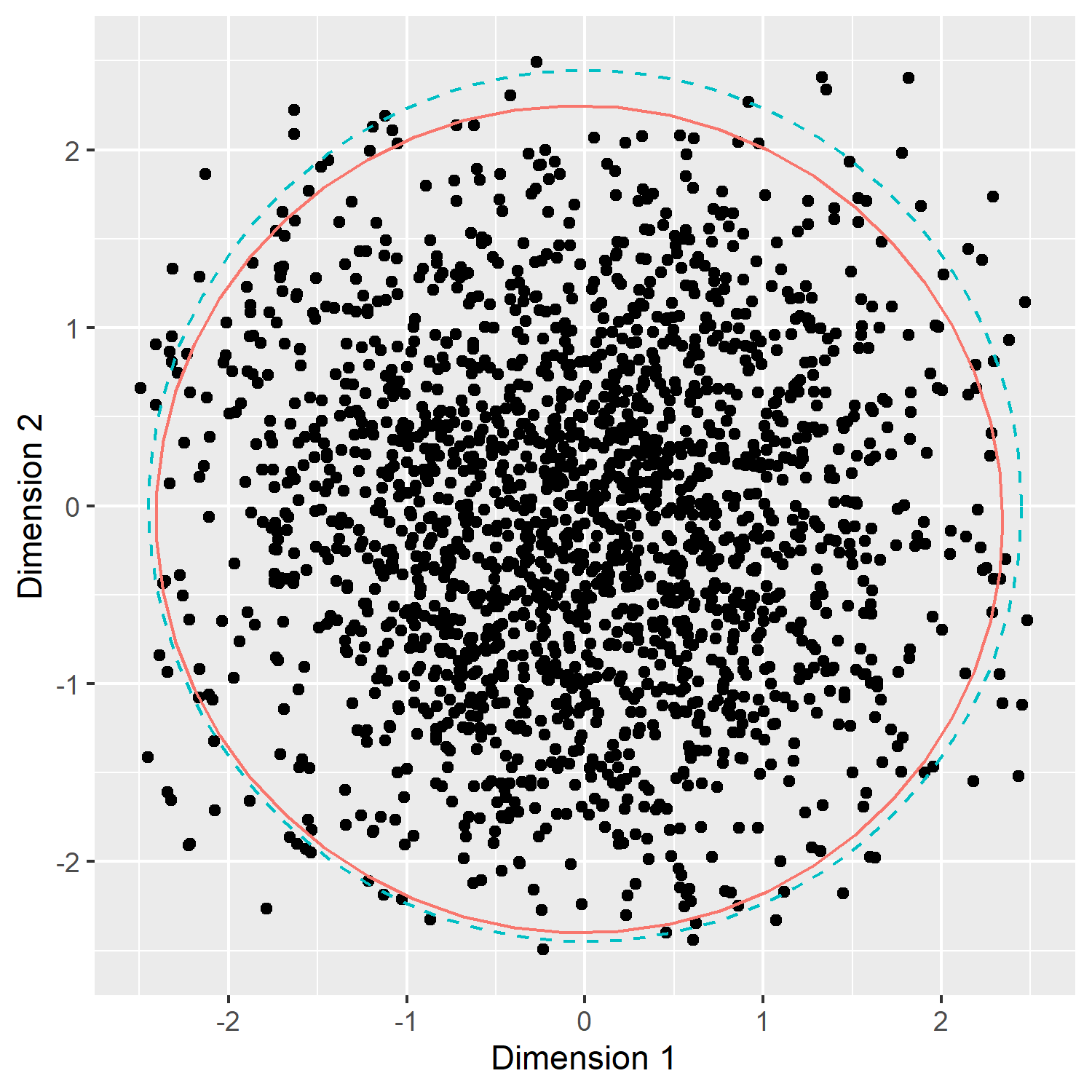}%
            \label{subfig:aveerror}%
     }      
        \subfloat[$\lambda/\sigma = p$]{ %
            \includegraphics[width=.4\linewidth,height=38mm]{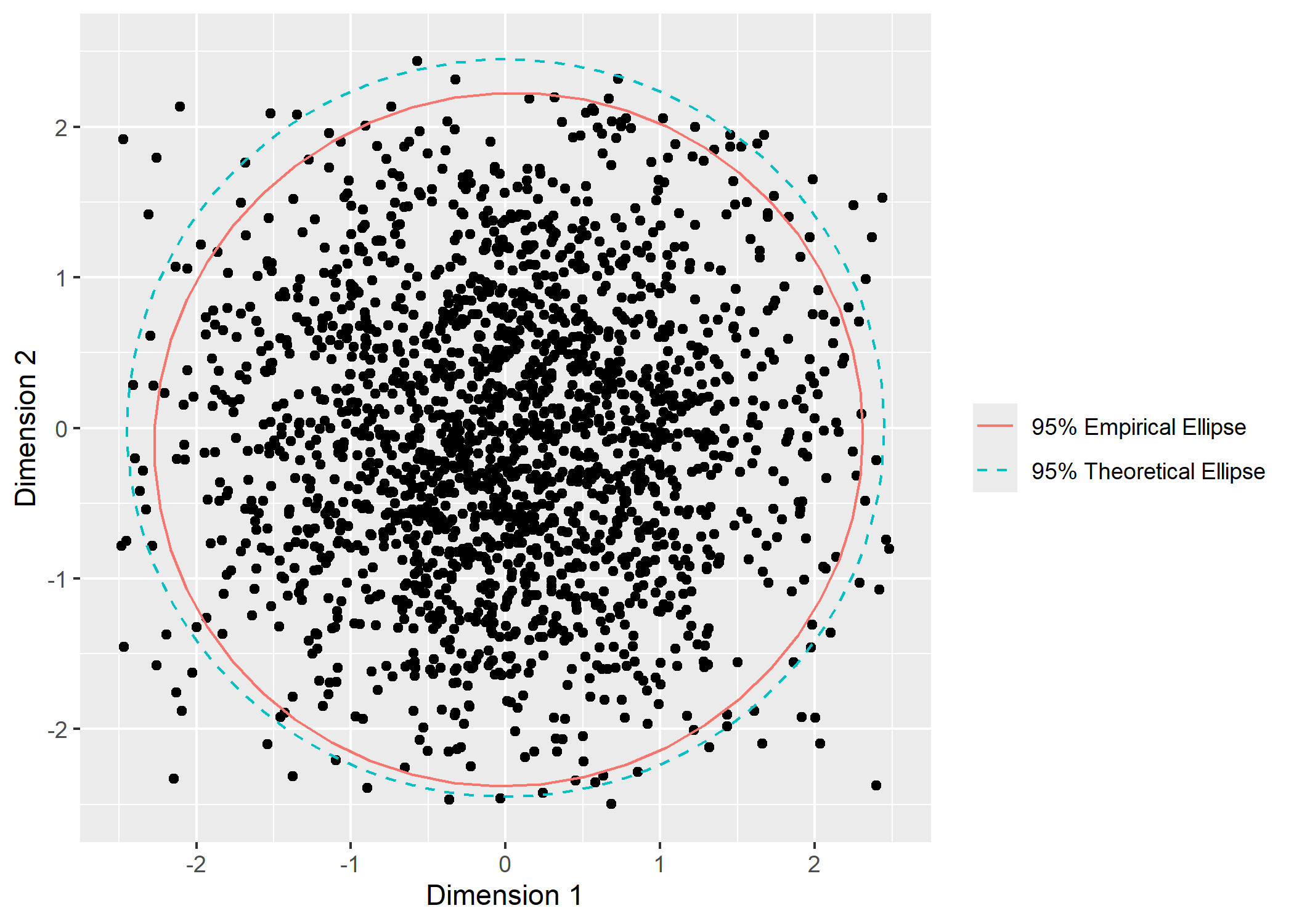}%
            \label{subfig:aveerror}%
     }
        \caption{Simulated distributional results for $\big(\mathbf{\Gamma}^{(1)}_{1}\big)^{-1/2}\big(\uhat_1\mathbf{W} - \U_1\big)_{1\cdot}$, $p = 100$ for varying level of noise.}
        \label{fig:distributionaltheory1}
        \subfloat[$\lambda/\sigma = p^{3/4}$]
                {%
            \includegraphics[width=.3\linewidth]{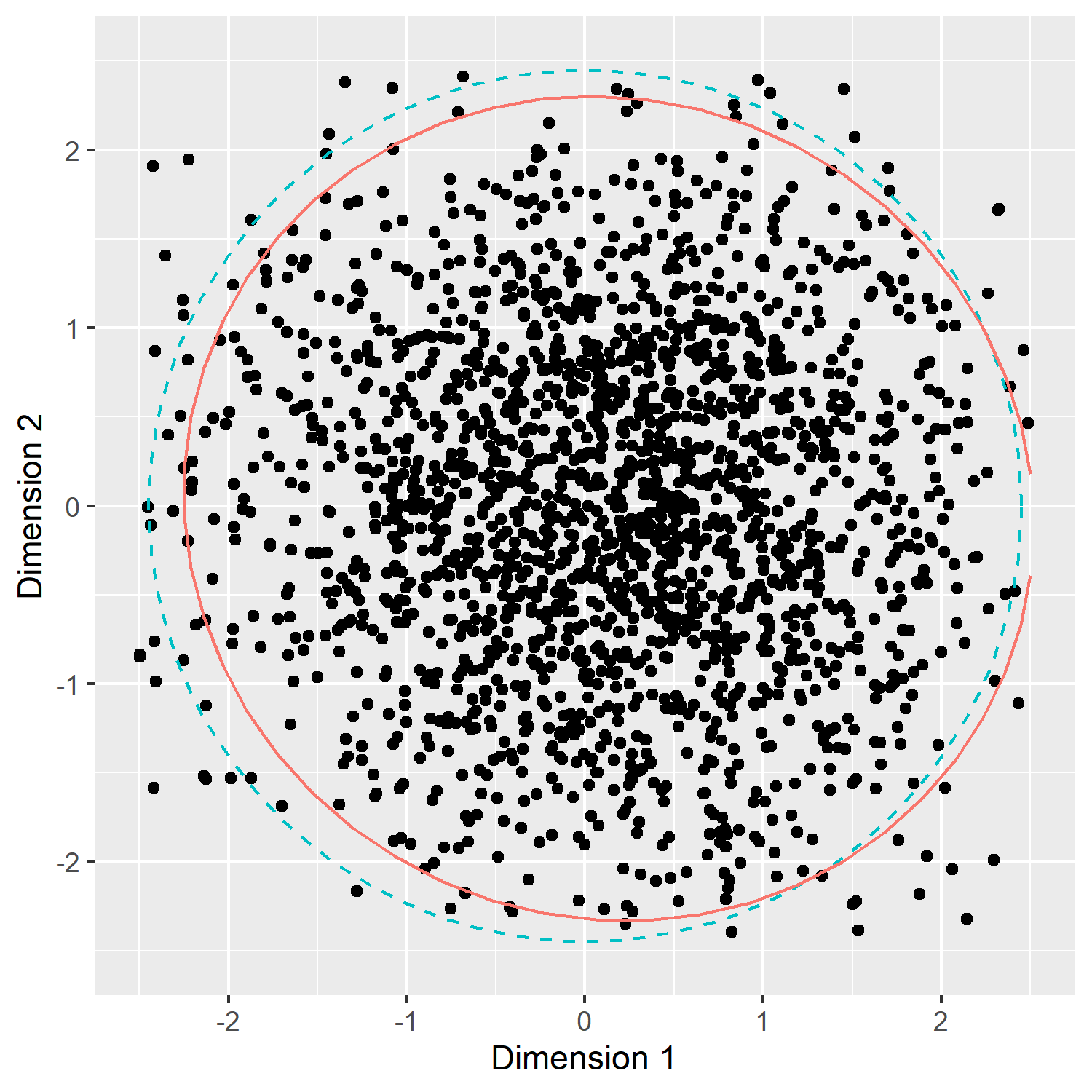}%
            \label{subfig:aveerror}%
        }
        \subfloat[$\lambda/\sigma = p^{7/8}$]{ %
            \includegraphics[width=.3\linewidth]{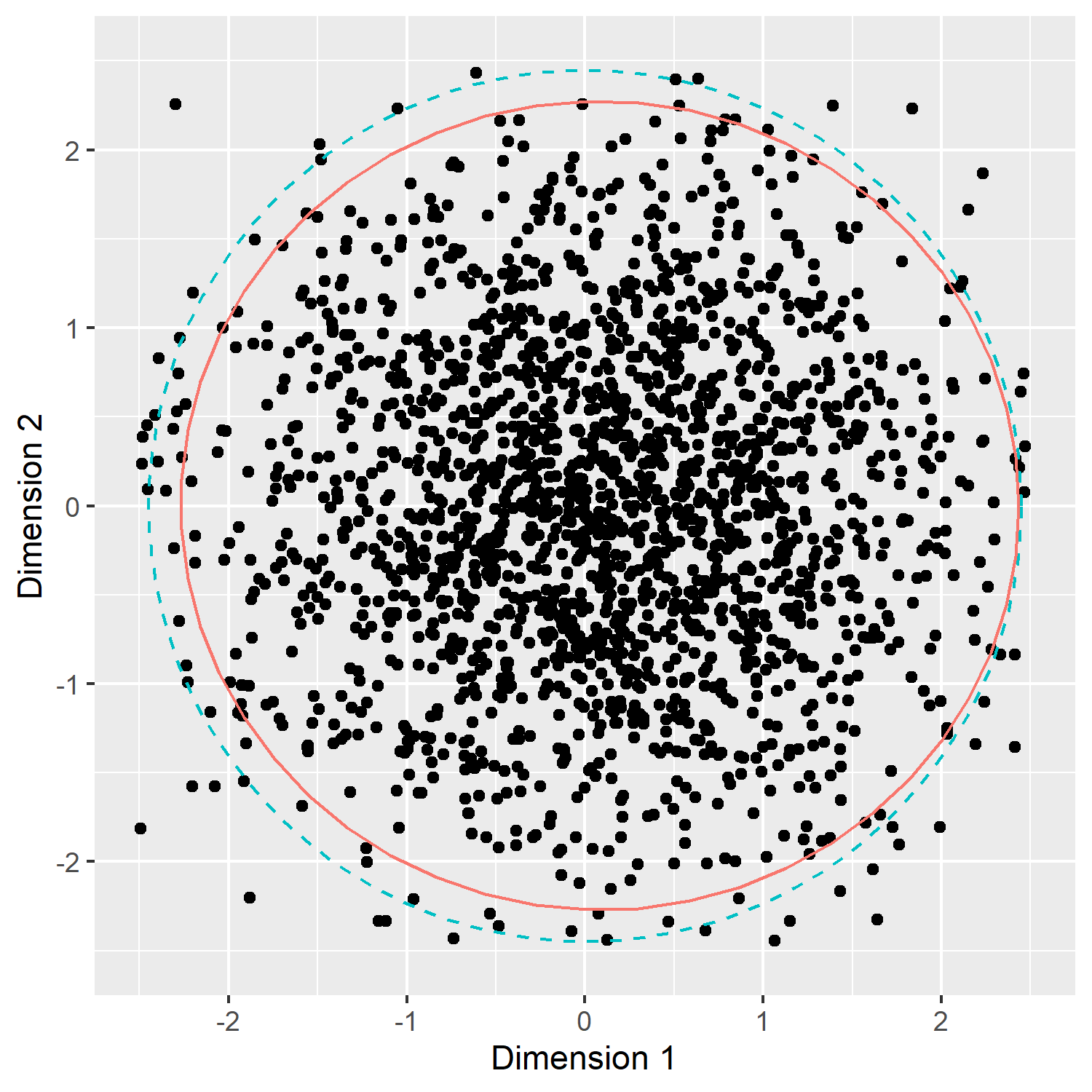}%
            \label{subfig:aveerror}%
     }      
        \subfloat[$\lambda/\sigma = p$]{ %
            \includegraphics[width=.4\linewidth,height=38mm]{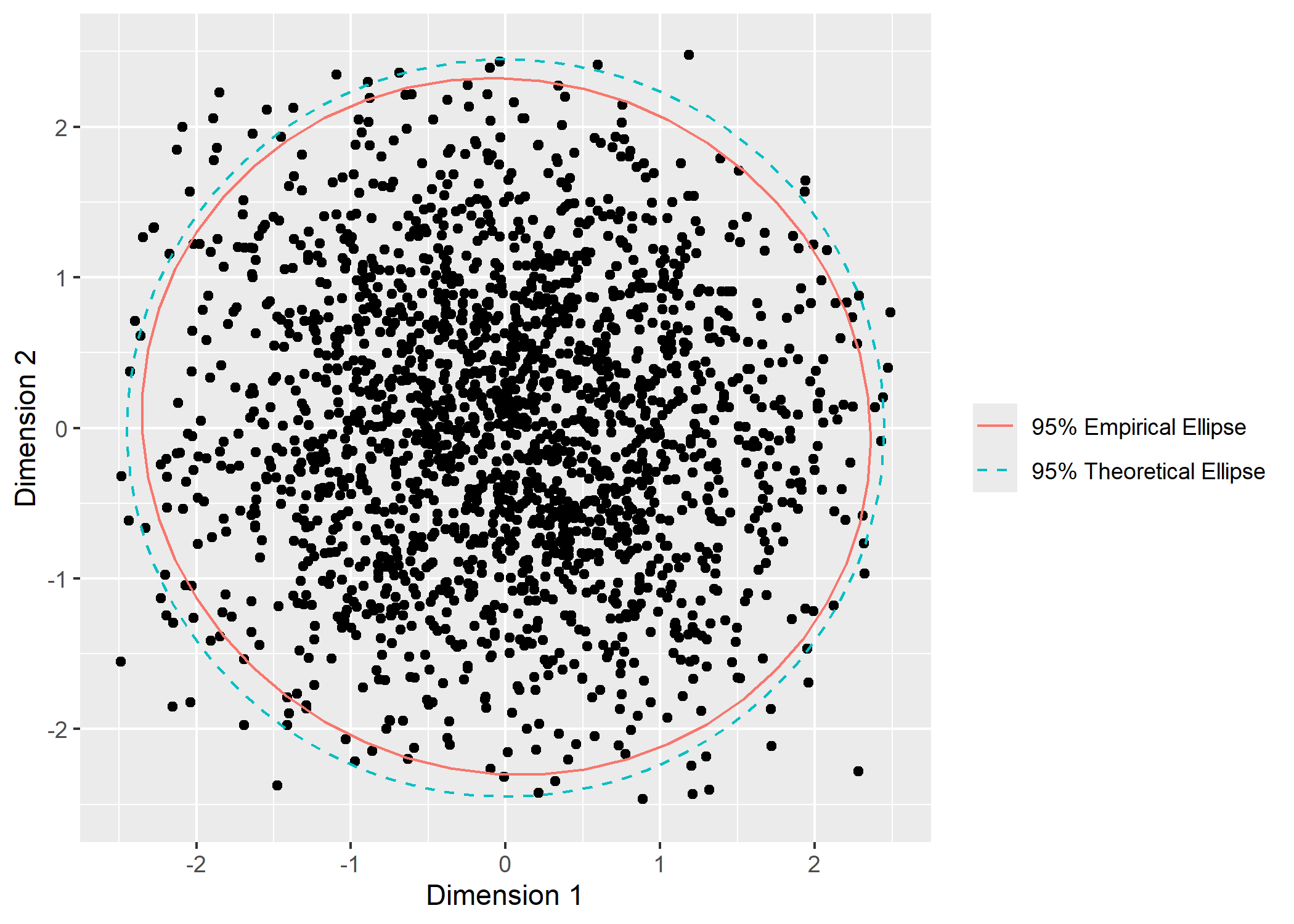}%
            \label{subfig:aveerror}%
     }
        \caption{Simulated distributional results for $\big(\mathbf{\Gamma}^{(1)}_{1}\big)^{-1/2}\big(\uhat_1\mathbf{W} - \U_1\big)_{1\cdot}$, $p = 150$ for varying level of noise.}
        \label{fig:distributionaltheory2}
    \end{figure*}

\ \\ \noindent
\textbf{Approximate Gaussianity of $\mathcal{\hat T}_{111}$}: 
Next, we consider the asymptotic normality  of the estimate $\mathcal{\hat T}_{111}$ as predicted by \cref{thm:asymptoticnormalityentries_v1}.  In \cref{fig:distributionaltheory3} ($p = 100$) and \cref{fig:distributionaltheory4} ($p = 150$) we plot the values of $\frac{\mathcal{\hat T}_{111} - \mathcal{T}_{111}}{s_{111}}$ under the same setup as the previous two figures  with $r = 4$.  The histograms represent the empirical observations, and the overlaid curve represents the density of the standard Gaussian distribution.

\begin{figure*}[htbp]
        \subfloat[$\lambda/\sigma = p^{3/4}$]
                {%
            \includegraphics[width=.3\linewidth]{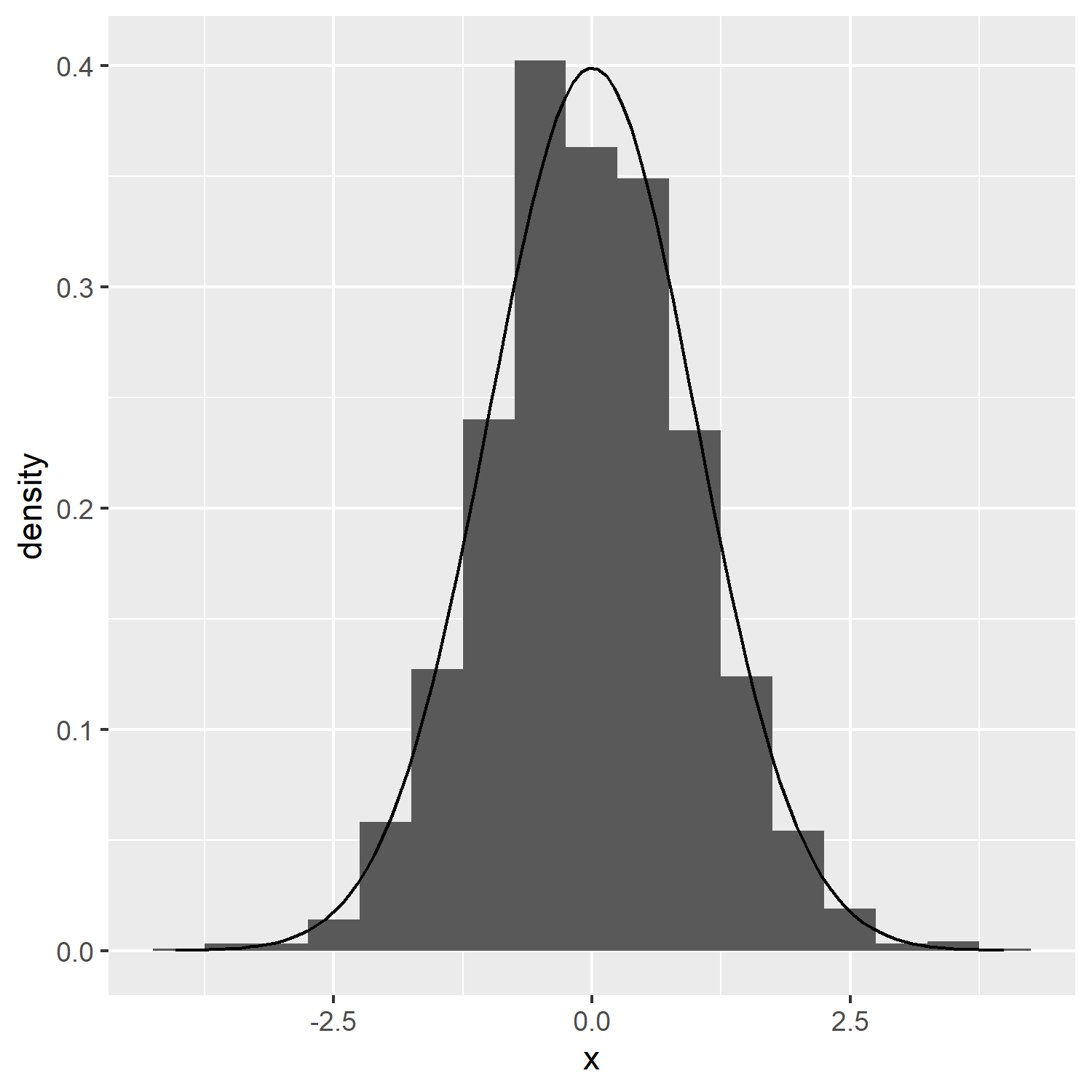}%
            \label{subfig:aveerror}%
        }\hfill
        \subfloat[$\lambda/\sigma = p^{7/8}$]{ %
            \includegraphics[width=.3\linewidth]{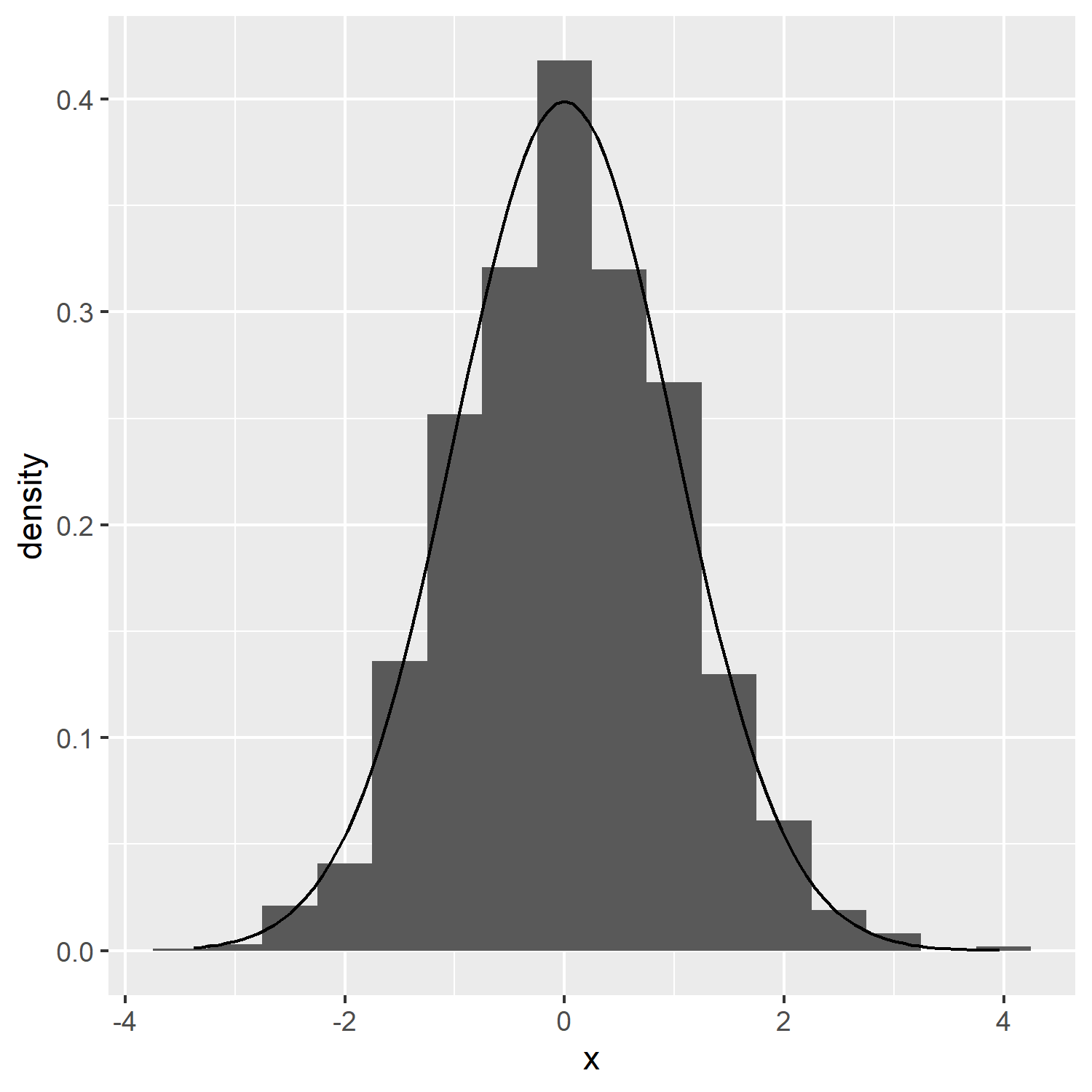}%
            \label{subfig:aveerror}%
     }      \hfill
        \subfloat[$\lambda/\sigma = p$]{ %
            \includegraphics[width=.3\linewidth,height=38mm]{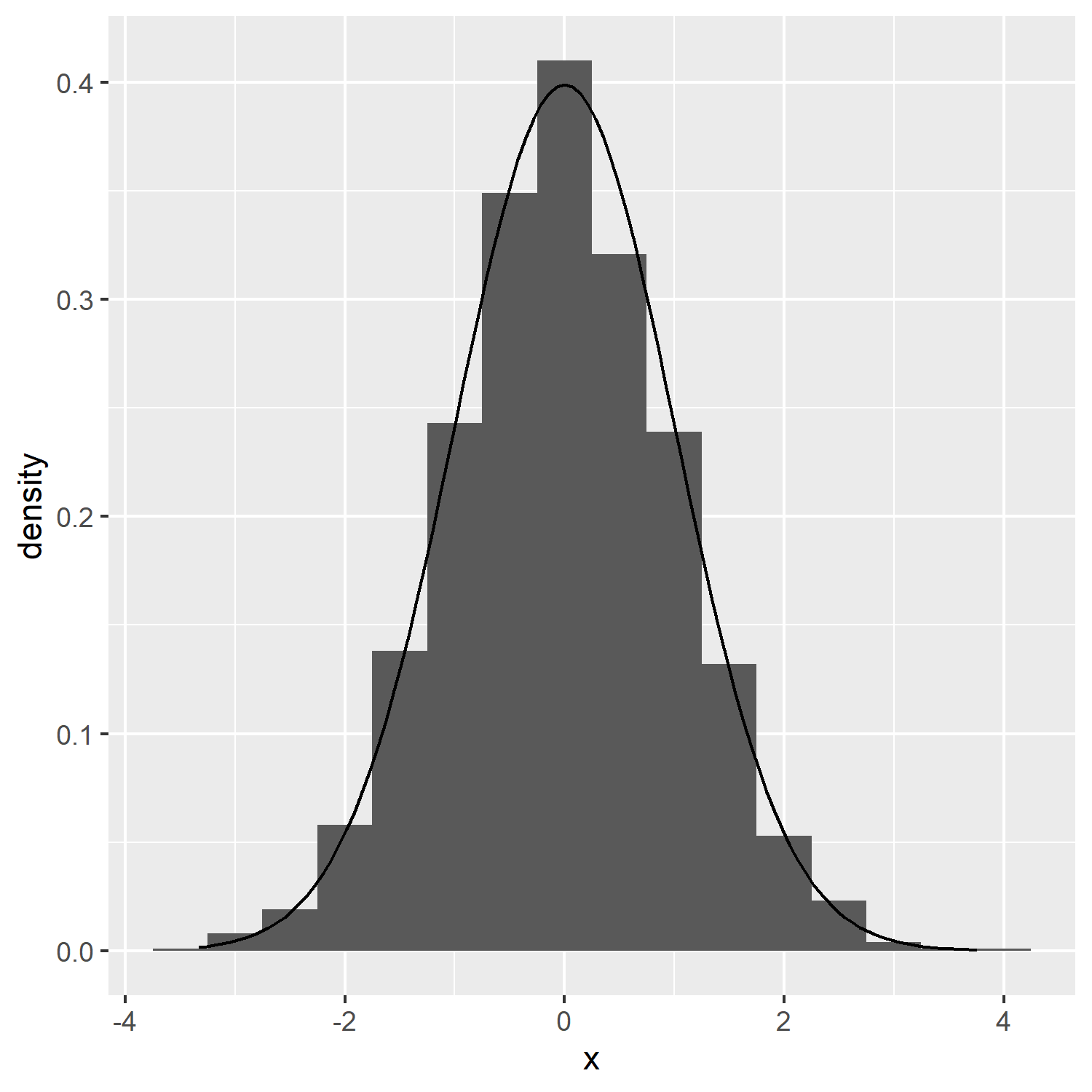}%
            \label{subfig:aveerror}%
     }
        \caption{Simulated distributional results for $\frac{1}{s_{ijk}}(\mathcal{\hat T}_{111} - \mathcal{T}_{111})$, $p = 100$ for varying level of noise.}
        \label{fig:distributionaltheory3}
        \subfloat[$\lambda/\sigma = p^{3/4}$]
                {%
            \includegraphics[width=.3\linewidth]{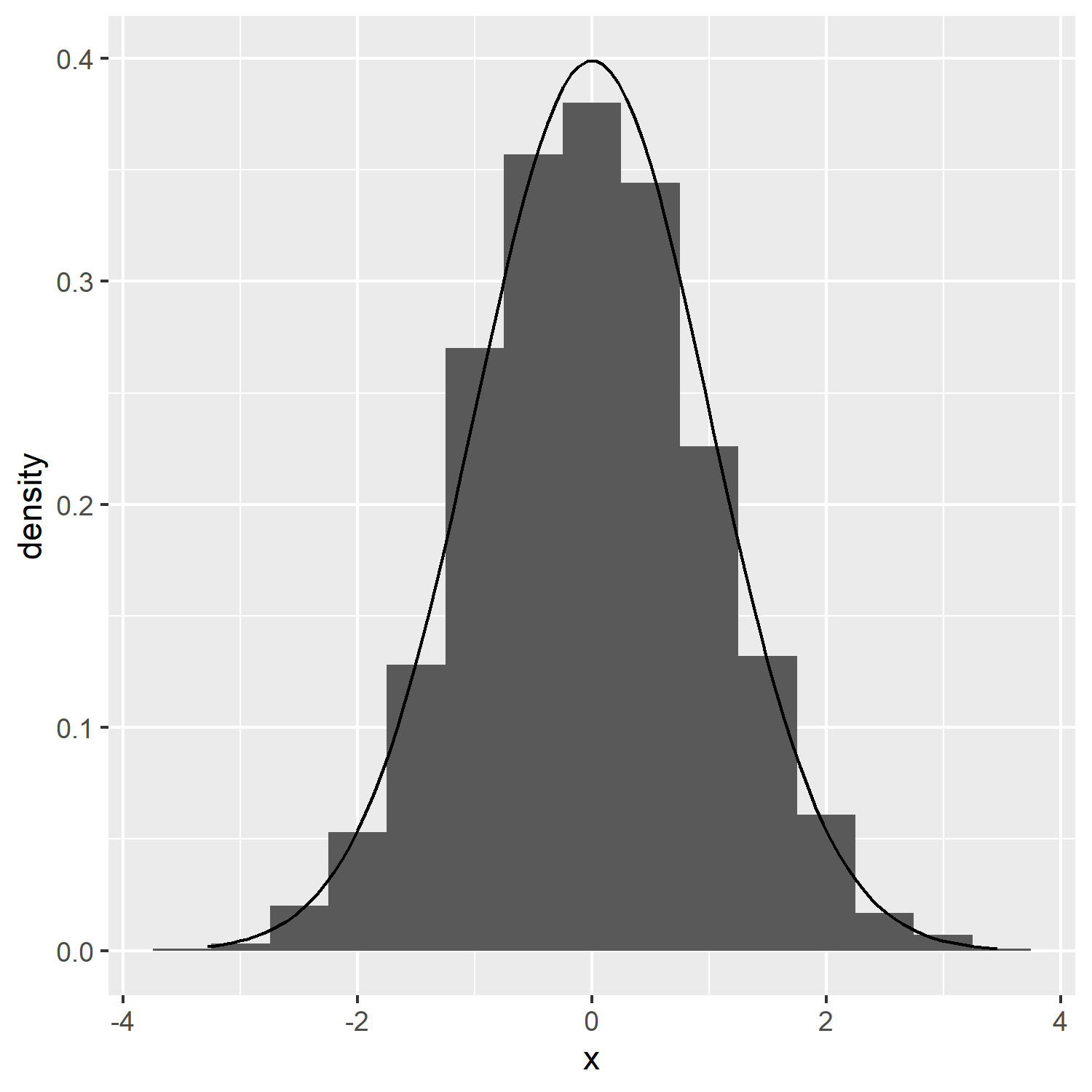}%
            \label{subfig:aveerror}%
        }\hfill
        \subfloat[$\lambda/\sigma = p^{7/8}$]{ %
            \includegraphics[width=.3\linewidth]{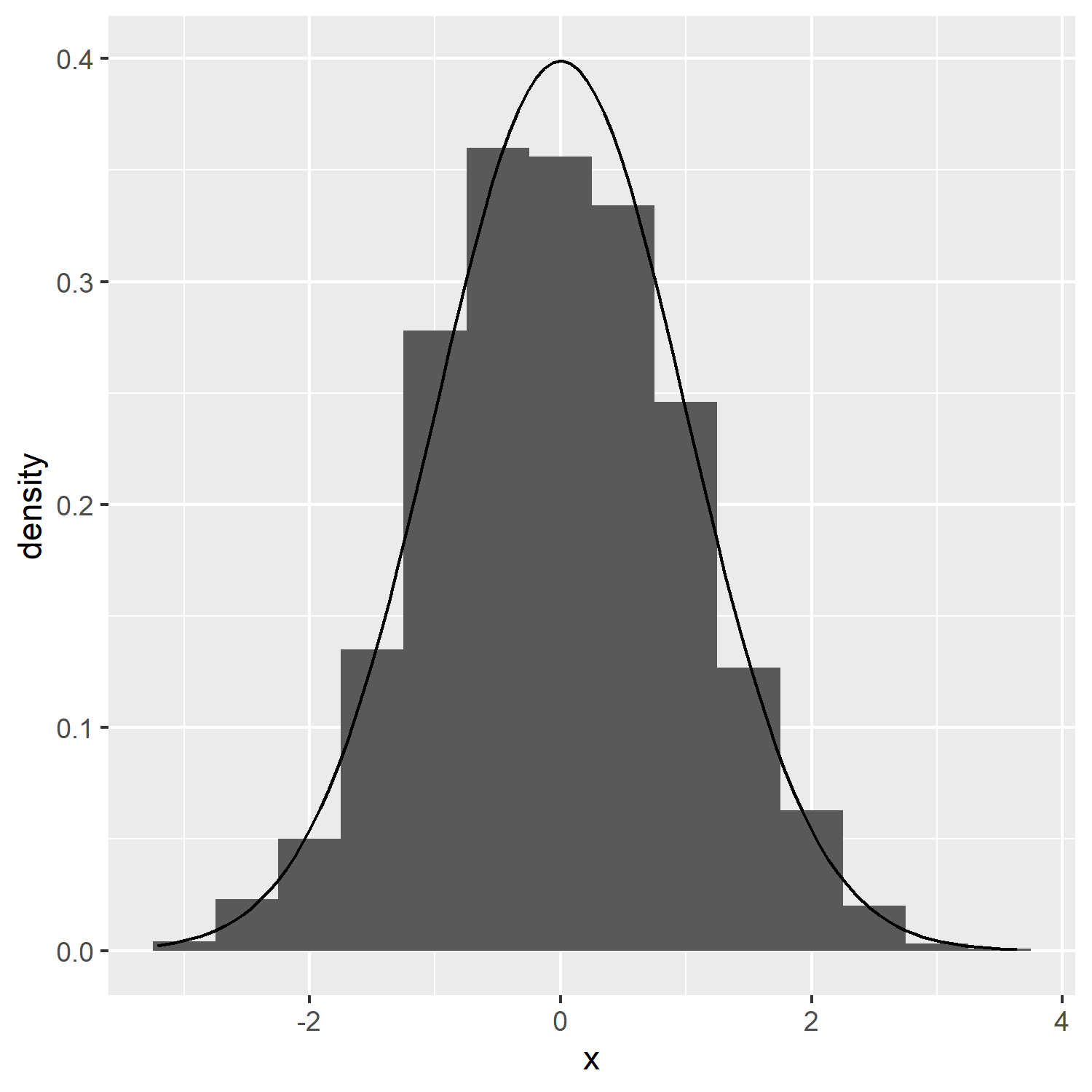}%
            \label{subfig:aveerror}%
     }      \hfill
        \subfloat[$\lambda/\sigma = p$]{ %
            \includegraphics[width=.3\linewidth]{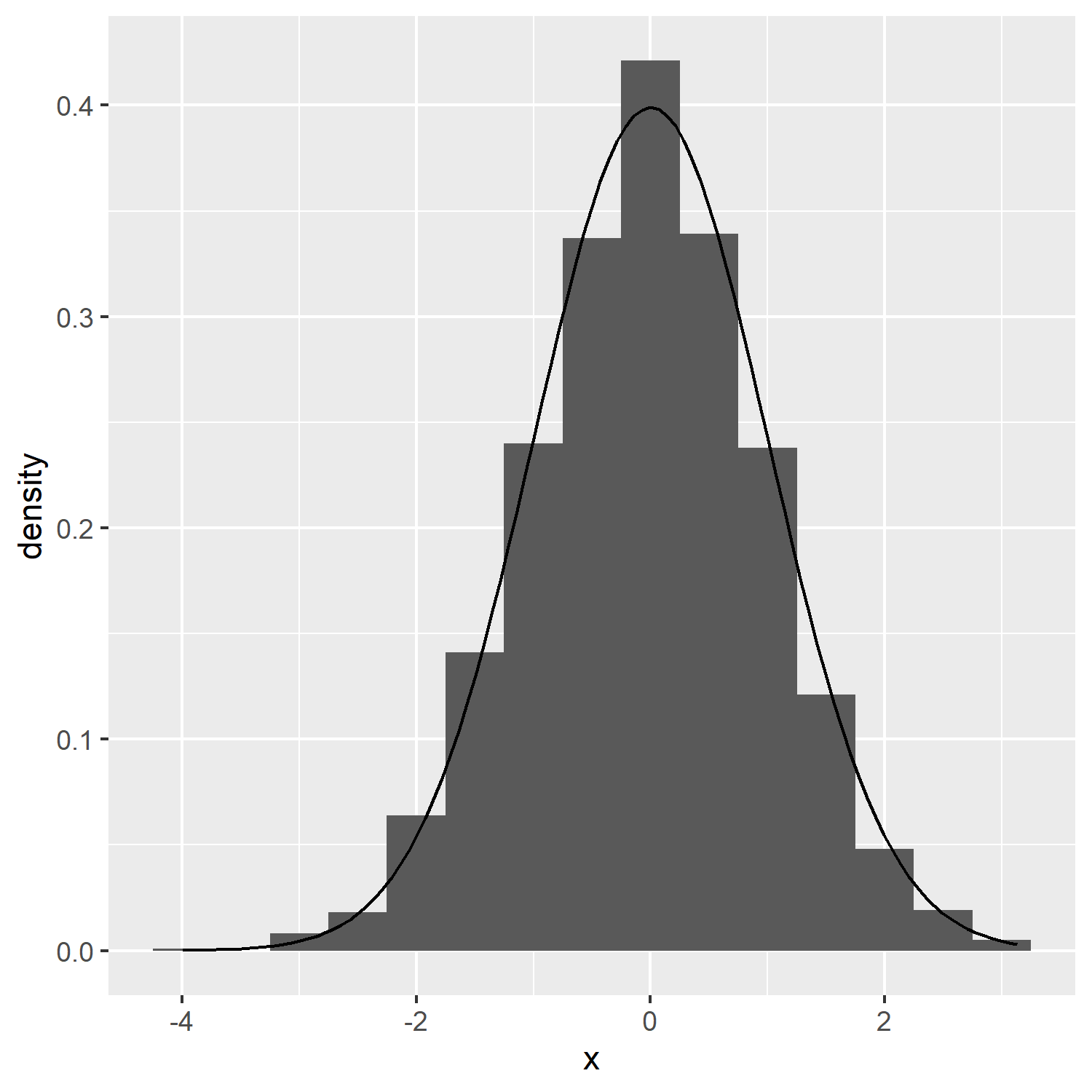}%
            \label{subfig:aveerror}%
     }
        \caption{Simulated distributional results for $\frac{1}{s_{111}}(\mathcal{\hat T}_{111} - \mathcal{T}_{111})$, $p = 150$ for varying level of noise.}
        \label{fig:distributionaltheory4}
    \end{figure*}
 
 \ \\ 

\textbf{Simultaneous Confidence Regions}: In \cref{fig:simultaneous1} and \cref{fig:simultaneous2}, we consider the joint distributions of $\hat S_{J}^{-1/2}\big( \mathcal{\hat T}_{J} - \mathcal{T}_{J} \big)$ with $J = \{111,112\}$ and $J = \{111,122\}$ respectively, where $\hat S_{J}$ is computed via \cref{al:simultaneousci}.  By \cref{thm:simultaneousinference_v1}, the distribution is approximately $N(0, \mathbf{I}_2)$, and we plot both the theoretical (dotted) and empirical (solid) 95\%  confidence ellipses.  
    \begin{figure*}[htpb]
        \subfloat[$\lambda/\sigma = p^{3/4}$]
                {%
            \includegraphics[width=.3\linewidth]{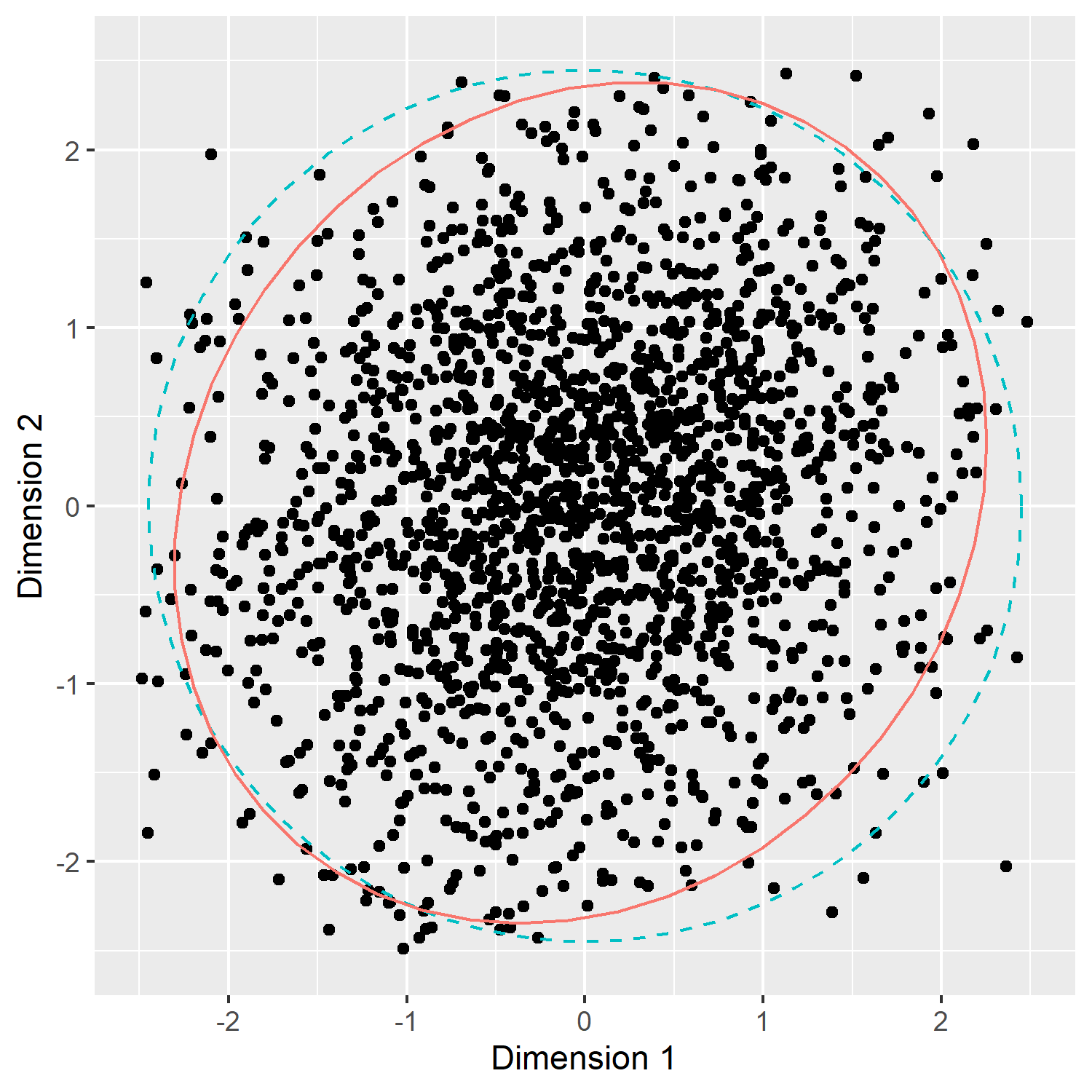}%
            \label{subfig:aveerror}%
        }
        \subfloat[$\lambda/\sigma = p^{7/8}$]{ %
            \includegraphics[width=.3\linewidth]{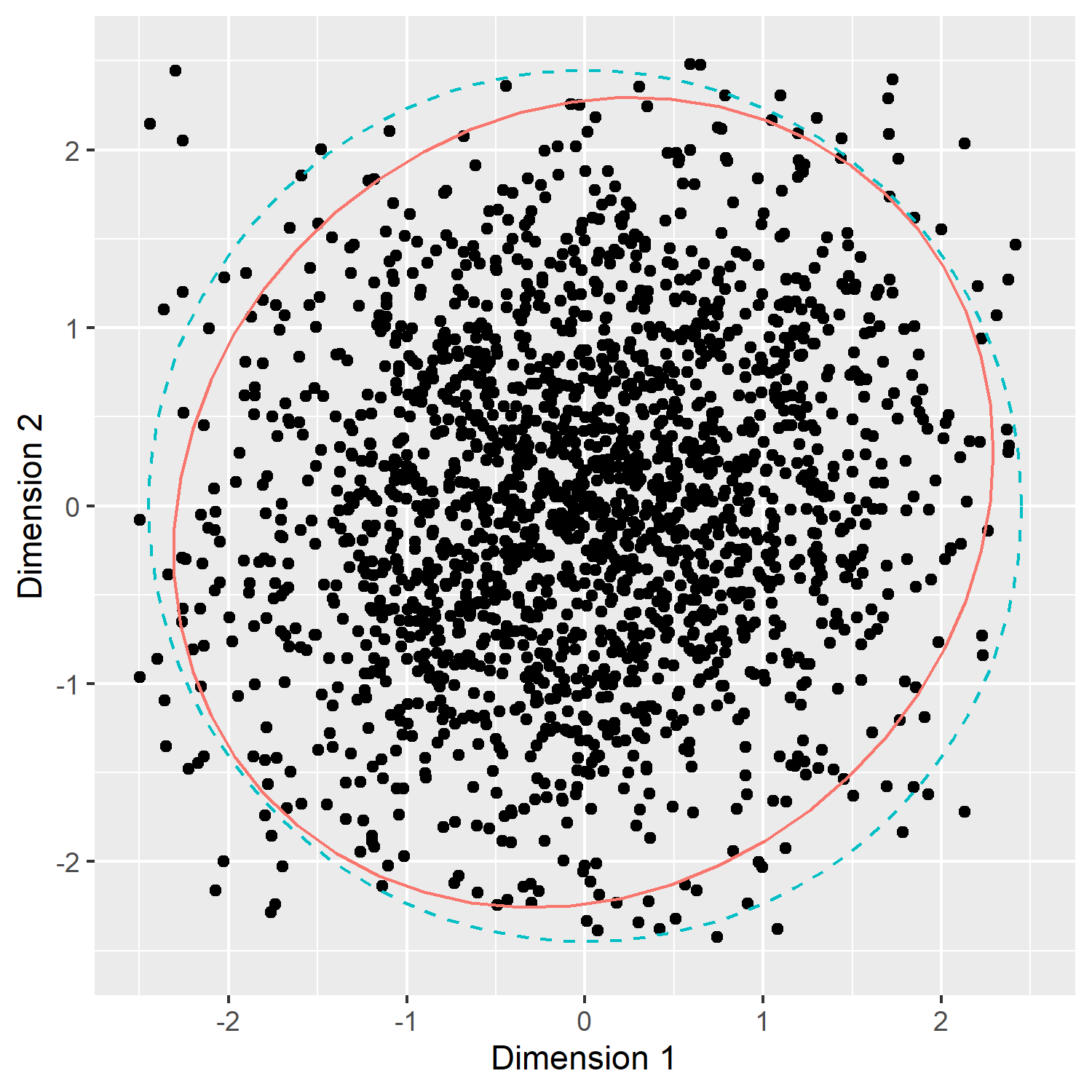}%
            \label{subfig:aveerror}%
     }      
        \subfloat[$\lambda/\sigma = p$]{ %
            \includegraphics[width=.4\linewidth,height=38mm]{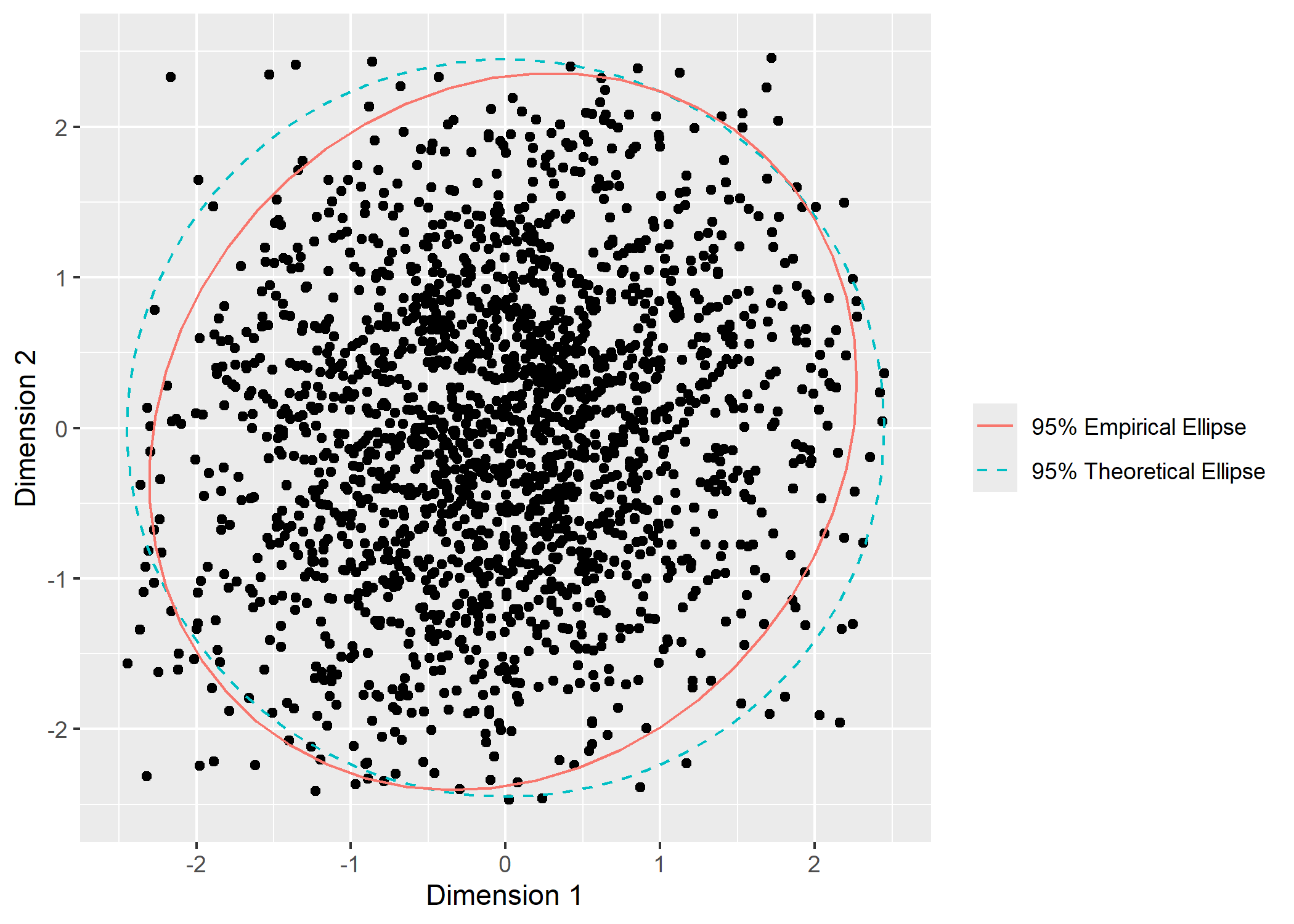}%
            \label{subfig:aveerror}%
     }
        \caption{Simulated distributional results for the joint distribution of $\hat S_{J}^{-1/2}(\mathcal{\hat T}_{111} - \mathcal{T}_{111},\mathcal{\hat T}_{112} - \mathcal{T}_{112})$, $p = 150$ for varying level of noise.}
        \label{fig:simultaneous1}
        \subfloat[$\lambda/\sigma = p^{3/4}$]
                {%
            \includegraphics[width=.3\linewidth]{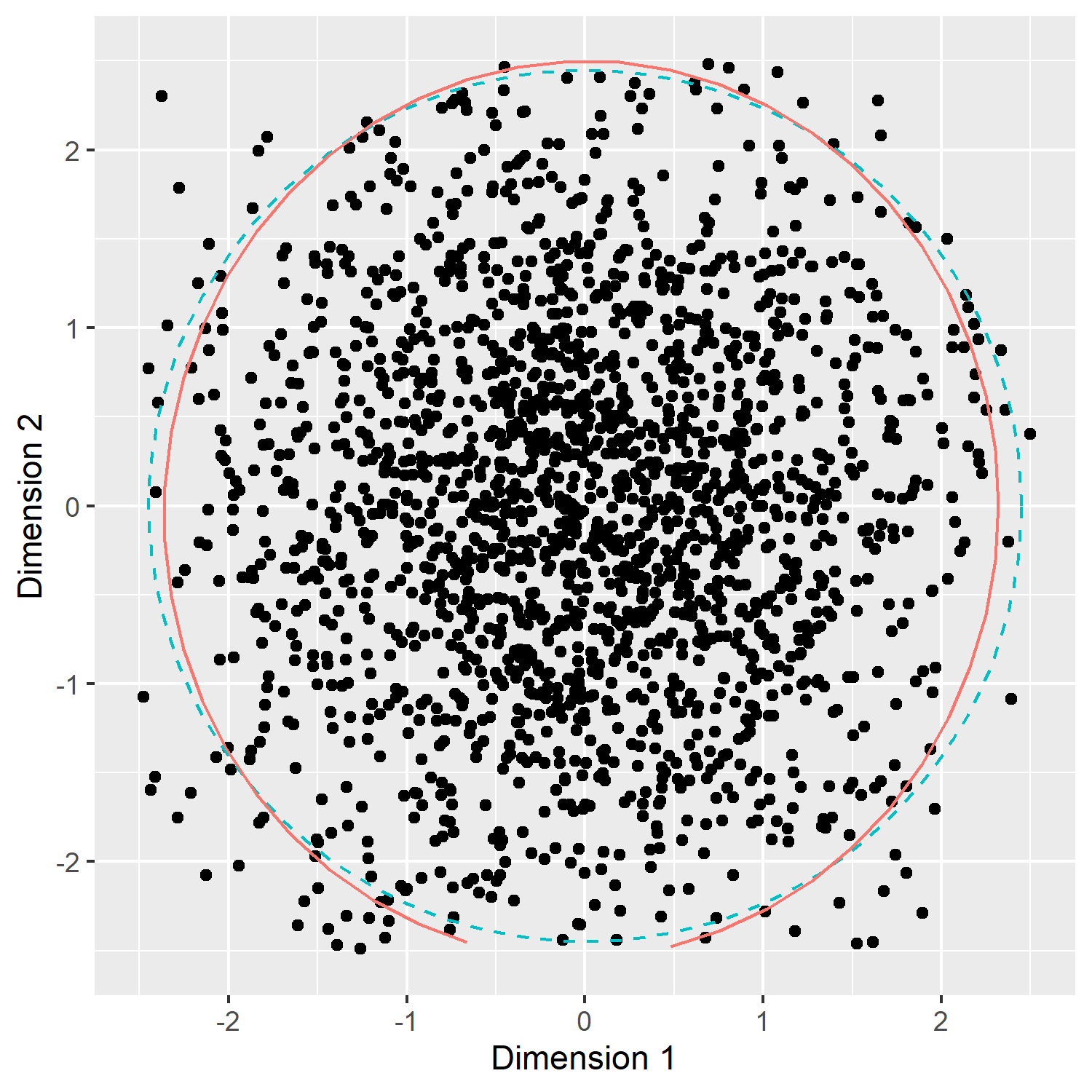}%
            \label{subfig:aveerror}%
        }
        \subfloat[$\lambda/\sigma = p^{7/8}$]{ %
            \includegraphics[width=.3\linewidth]{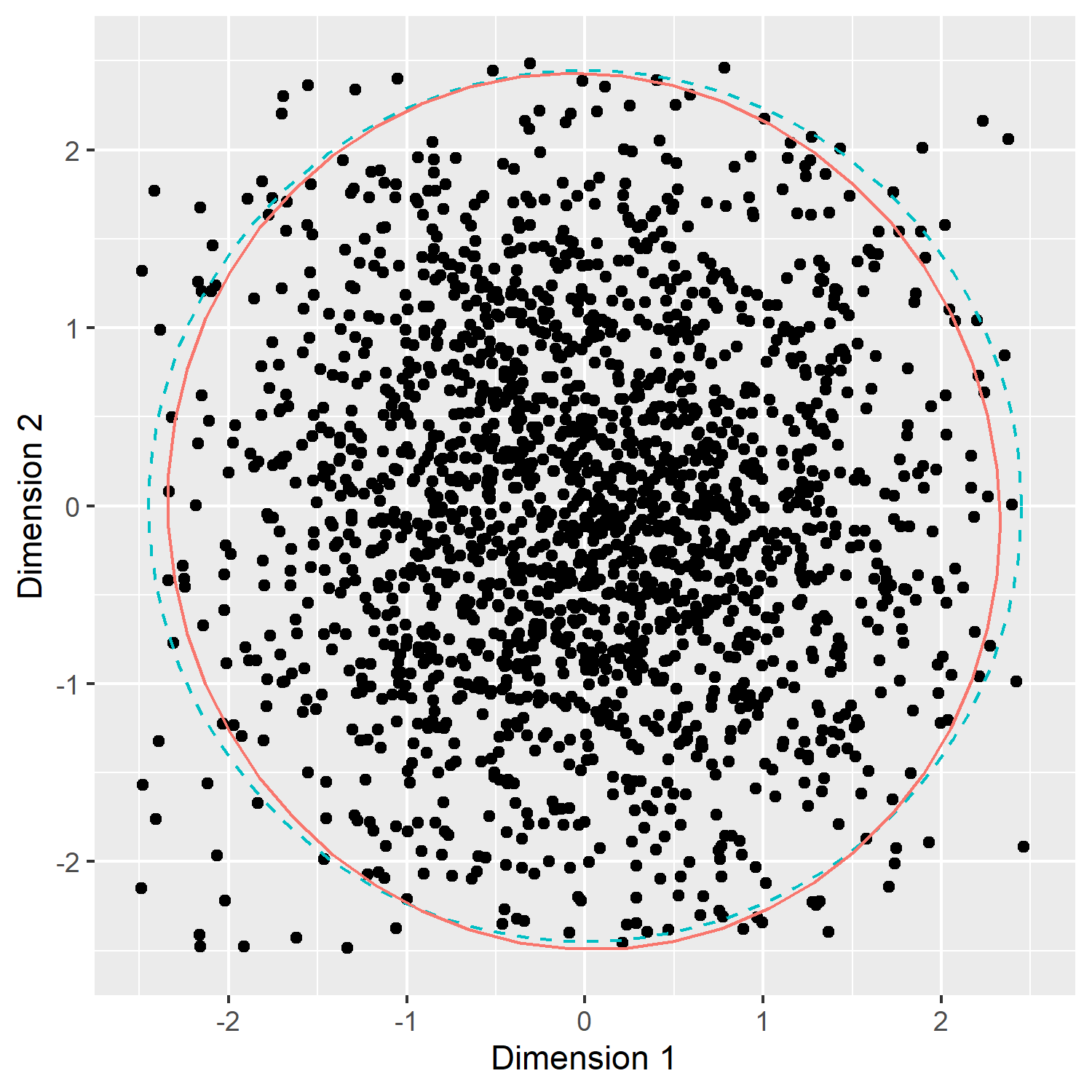}%
            \label{subfig:aveerror}%
     }      
        \subfloat[$\lambda/\sigma = p$]{ %
            \includegraphics[width=.4\linewidth,height=38mm]{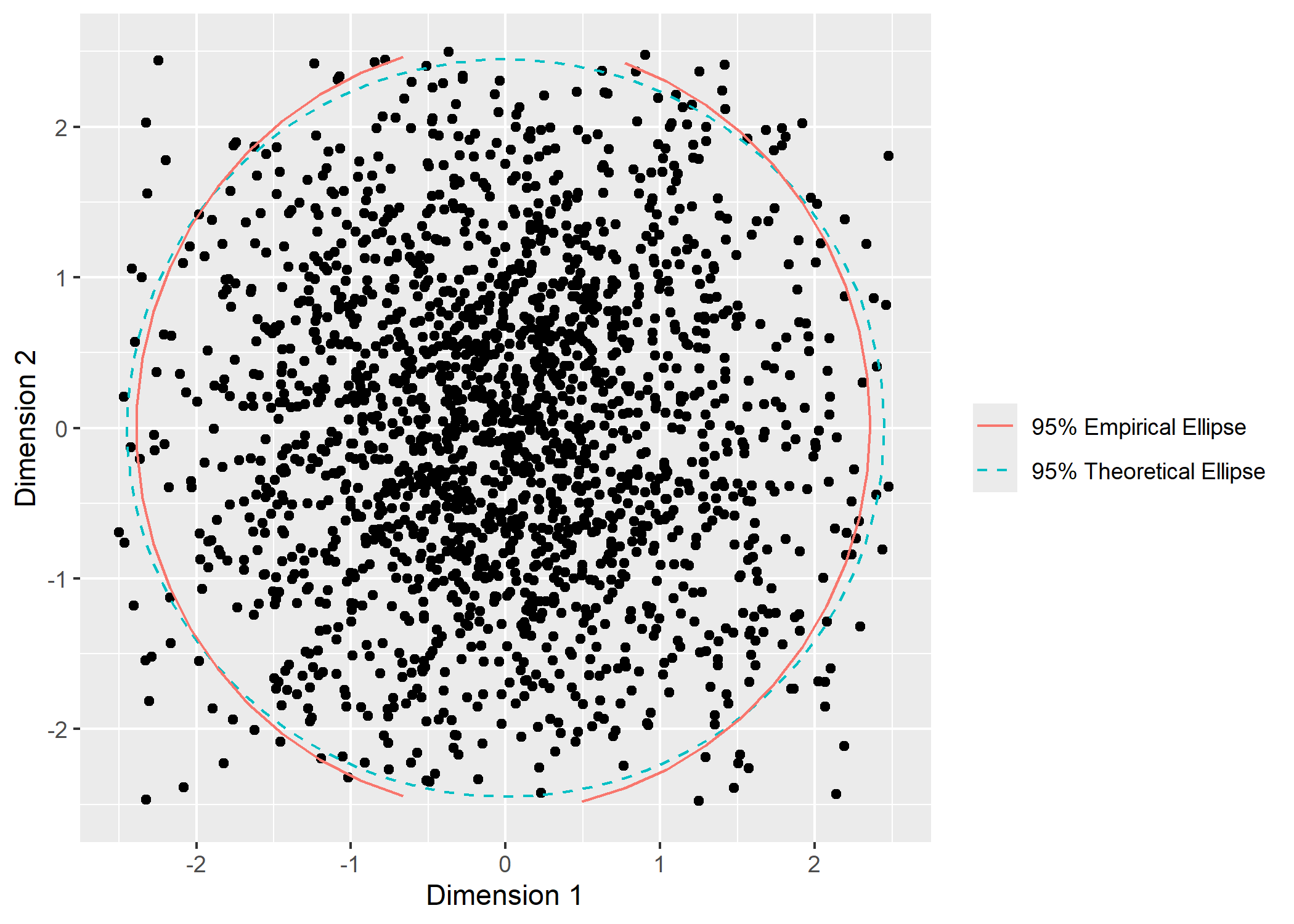}%
            \label{subfig:aveerror}%
     }
        \caption{Simulated distributional results  for the joint distribution of $\hat S_{J}^{-1/2}(\mathcal{\hat T}_{111} - \mathcal{T}_{111},\mathcal{\hat T}_{122} - \mathcal{T}_{122})$, $p = 150$ for varying level of noise.}
        \label{fig:simultaneous2}
    \end{figure*}
\ \\ \ \\
\bibliography{tensor-perturb-2-infinity/reference,tensor-perturb-2-infinity/tensors}

\end{document}